\documentclass[a4paper, 10pt,leqno]{amsart}
\usepackage{amsmath}
\usepackage{amsthm}
\usepackage{amsfonts}
\usepackage{amssymb}
\usepackage{texdraw}

\usepackage[all]{xy}

\usepackage[usenames]{color}

\usepackage{graphicx}
\usepackage{psfrag}

\newtheorem{theorem}{Theorem}[section]{\bf}{\it}
\newtheorem{lemma}[theorem]{Lemma}{\bf}{\it}
\newtheorem{proposition}[theorem]{Proposition}{\bf}{\it}
\newtheorem{corollary}[theorem]{Corollary}{\bf}{\it}
{\bf}{\it} 
{\bf}{\it}
{\bf}{\it}
\newtheorem{example}[theorem]{Example}{\bf}{\it}

\newtheorem{remark}[theorem]{Remark}{\bf}{\it}
\newtheorem{definition}[theorem]{Definition}{\bf}{\it}

{\bf}{\it}
{\bf}{\it}

{\bf}{\it}
\newtheorem*{convention}{Convention}{\bf}{\it}

\newtheorem*{namedtheorem}{\theoremname}
\newcommand{\theoremname}{testing}

\numberwithin{equation}{section}

\newcommand{\diam}{\operatorname{diam}}

\newcommand{\R}{\mathbb R}
\newcommand{\C}{\mathbb C}

\newcommand{\dist}{{\operatorname{dist}\,}}
\newcommand{\id}{{\operatorname{id}}}

\newdimen\vintkern\vintkern11pt
\def\vint{-\kern-\vintkern\int}

\newcommand{\norm}[1]{\lVert #1 \rVert}

\newcommand{\bS}{\mathbb{S}}

\newcommand{\haus}{\mathcal{H}}

\newcommand{\cA}{\mathcal{A}}

\newcommand{\interior}{\mathrm{int}\;}

\newcommand{\Z}{\mathbb{Z}}

\newcommand{\bX}{\mathbb{X}}

\newcommand{\cT}{\mathcal{T}}
\newcommand{\cM}{\mathcal{M}}
\newcommand{\cC}{\mathcal{C}}
\newcommand{\cD}{\mathcal{D}}
\newcommand{\cN}{\mathcal{N}}
\newcommand{\cP}{\mathcal{P}}
\newcommand{\cL}{\mathcal{L}}

\newcommand{\cF}{\mathcal{F}}

\newcommand{\cQ}{\mathcal{Q}}

\newcommand{\cS}{\mathcal{S}}

\newcommand{\fB}{\mathfrak{B}}

\newcommand{\fP}{\mathfrak{P}}

\newcommand{\bU}{\mathbf{U}}
\newcommand{\bV}{\mathbf{V}}
\newcommand{\bW}{\mathbf{W}}

\newcommand{\bB}{\mathbb{B}}

\newcommand{\cB}{\mathcal{B}}

\newcommand{\fL}{\mathfrak{L}}

\newcommand{\cJ}{\mathcal{J}}

\newcommand{\cE}{\mathcal{E}}
\renewcommand{\cA}{\mathcal{A}}

\newcommand{\bfOmega}{\mathbf{\Omega}}

\newcommand{\Cone}{\mathrm{Cone}}

\newcommand{\hull}{\mathrm{hull}}

\newcommand{\cl}{\mathrm{cl}}

\renewcommand{\emptyset}{\varnothing}

\newcommand{\wh}{\widehat}
\newcommand{\wt}{\widetilde}

\newcommand{\icl}{\mathrm{int}}
\newcommand{\ecl}{\mathrm{ext}}
\newcommand{\atom}{\mathrm{atom}}

\newcommand{\LE}{\cA}

\begin{document}

\title{Sharpness of Rickman's Picard theorem in all dimensions}
\date{\today}
\author{David Drasin \and Pekka Pankka}
\address{Department of Mathematics, Purdue University,
150 N. University Street, West Lafayette, IN 47907-2067, USA}
\address{Department of Mathematics and Statistics, P.O. Box 68 (Gustaf H\"allstr\"omin katu 2b), FI-00014 University of Helsinki, Finland \and Department of Mathematics and Statistics, P.O. Box 35, FI-40014 University of Jyv\"askyl\"a, Finland}

\thanks{P.P. has been partly supported by NSF grant DMS-0757732 and the Academy of Finland projects 126836 and 256228.}  
\begin{abstract}
We show that given $n\ge 3$, $q\ge 1$, and a finite set $\{y_1,\ldots, y_q\}$ in $\R^n$ there exists a quasiregular mapping $\R^n \to \R^n$ omitting exactly points $y_1,\ldots, y_q$. 
\end{abstract}

\maketitle

\section{Introduction}
\label{sec:intro}

By the classical Picard theorem an entire holomorphic map $\C \to \C$ omits at most one point if non-constant. The characteristic example of an entire holomorphic map omitting a point is, of course, the exponential function $z\mapsto e^z$, since every entire holomorphic map $\C \to \C$ omitting a point factors through the exponential map.

Liouville's theorem asserts that all entire conformal maps $\R^n \to \R^n$ are M\"obius transformations and, in particular, homeomorphisms for $n\ge 3$. This rigidity of spatial conformal geometry no longer persists in quasiconformal geometry. Reshetnyak in the late 1960's and Martio--Rickman--V\"ais\"al\"a in the early 1970's showed that the rich theory of \emph{mappings of bounded distortion}, or so-called \emph{quasiregular mappings}, is a natural replacement for holomorphic functions in higher dimensions. This advancement raised the question of the existence of Picard type theorems for quasiregular mappings; see e.g.\;Zorich \cite{ZorichV:TheMAL} or V\"ais\"al\"a's survey \cite{VaisalaJ:Surqm}.

Already in his 1967 paper \cite{ZorichV:TheMAL} Zorich gave an example of a quasiregular mapping $\R^n \to \R^n$ omitting the origin. This so-called \emph{Zorich map} is the natural higher-dimensional analog of the exponential function although the mapping is not a local homeomorphism. The branching of the map cannot be avoided by Zorich's \emph{Global Homeomorphism Theorem} from the same article: \emph{For $n\ge 3$, quasiregular local homeomorphisms $\R^n \to \R^n$ are homeomorphisms}. Recall that by Reshetnyak's theorem quasiregular mappings are (generalized) branched covers, that is, discrete and open mappings and hence local homeomorphisms modulo an exceptional set of (topological) codimension at least $2$; we refer to Rickman's monograph \cite{RickmanS:Quam} for the general theory of quasiregular mappings.

A counterpart of Picard's theorem for quasiregular mappings is due to Rickman \cite{RickmanS:Ontnoo}: \emph{Given $K>1$ and $n\ge 2$ there exists $q$ depending only on $K$ and $n$ so that a non-constant $K$-quasiregular mapping $\R^n \to \R^n$ omits at most $q$ points}. The sharpness of Rickman's Picard theorem is known in dimension $n=3$ and is also due to Rickman. In \cite{R} he shows the following existence result: \emph{Given any finite set $P$ in $\R^3$ there exists a quasiregular mapping $\R^3 \to \R^3$ omitting exactly $P$}. 

Holopainen and Rickman generalized the Picard theorem to quasiregular mappings into manifolds with many ends in \cite{HolopainenI:Picttf} and \emph{a fortiori} to quasiregular mappings between manifolds in \cite{HolopainenI:Ricchf}; note also similar results in the sub-Riemannian geometry \cite{HolopainenI:QuamHg}. These result stem from potential theoretic proofs of Rickman's Picard theorem due to Lewis \cite{LewisJ:Pictar} and Eremenko--Lewis \cite{EremenkoA:Unilof}. It can be said that the ramifications of these methods are now well-understood. Recently, Rajala generalized Rickman's Picard theorem to mappings of finite distortion \cite{RajalaK:MapfdR}. Whereas the aforementioned potential-theoretic methods are difficult to adapt to this more general class of mappings, Rajala shows that value distribution theory based on modulus methods is still at our disposal.

The sharpness of these theorems, however, is still mostly unknown and Rickman's three-dimensional construction in \cite{R} provides essentially the only method to produce examples.

In this article we show the precision of Rickman's Picard theorem in all dimensions.

\begin{theorem}
\label{thm:1}
Given $n\ge 3$, $q\ge 2$, and points $y_1,\ldots, y_q$ in $\R^n$ there exists a quasiregular mapping $\R^n\to \R^n$ omitting exactly $y_1,\ldots, y_q$. 
\end{theorem}

It has already been mentioned that the case of dimension $n=3$ was settled by Rickman. For $n=2$ the number of omitted points is at most $1$ by Picard's theorem and the Sto\"ilow factorization; see e.g.\;book of Astala, Iwaniec, and Martin \cite[Section 5.5]{AstalaK:Ellpde}. As discussed above, the case $q=1$ is given by the Zorich map for all $n\ge 3$. Therefore we may restrict to cases $n\ge 4$ and $q\ge 2$. However, it is natural to include $n=3$.

As will become apparent in the following outline of the proof, the proof of Theorem \ref{thm:1} is independent of the analytic theory of quasiregular mappings. 

The general outline follows the idea of Rickman's construction in \cite{R} and both proofs stem from PL-topology. Rickman's original method relies on a very delicate deformation theory of $2$-dimensional branched covers (\cite[Section 5]{R}) which leads to an extension theory of $2$-dimensional branched covers; we refer to \cite{DrasinD:PictRc} for an exposition on Rickman's main ideas. These arguments rely essentially on the discrete nature of the branch set in dimension $2$.  Already when $n=3$, the corresponding deformation theory is much more complicated due to the non-trivial topology of the branch set; see however application of Piergallini's method in \cite{PiergalliniR:Fou4bc} to obtain a quasiregular map $\R^4 \to \bS^2\times \bS^2 \# \bS^2\times \bS^2$ in \cite{RickmanS:Simcqe}. We are not aware of similar deformation theory, based on a detailed analysis of the branch set, in higher dimensions.

The required extension theory is, however, essentially trivial in all dimensions for BLD-mappings. Recall that a mapping $f \colon X\to Y$ between metric spaces $X$ and $Y$ is a \emph{mapping of bounded length distortion} (or a \emph{BLD-map}, for short) if $f$ is open and discrete, and there exists a constant $L\ge 1$ satisfying
\begin{equation}
\label{eq:BLD}
\frac{1}{L} \ell(\gamma) \le \ell(f\circ \gamma) \le L \ell(\gamma)
\end{equation}
for all paths $\gamma$ in $X$, where $\ell(\gamma)$ is the length of $\gamma$. We refer to the seminal paper of Martio and V\"ais\"al\"a \cite{MartioO:Ellqmb} for the discussion of the special r\^ole of BLD-mappings among quasiregular mappings; see also Heinonen--Rickman \cite{HeinonenJ:Geobcb} for the metric theory. 

The BLD-theory in the proof of Theorem \ref{thm:1} brings forth an alternative, and slightly stronger, formulation. We denote by $\bS^n$ and $\bS^{n-1}$ the Euclidean unit spheres in $\R^{n+1}$ and $\R^n$, respectively, and by $B^n(y,\delta)$ the metric ball in $\bS^n$ in the inherited metric.

\begin{theorem}
\label{thm:2}
Let $n\ge 3$, $p\ge 2$, and $y_0,\ldots, y_p$ be points in $\bS^n$. Let also $g$ be a Riemannian metric on $M:=\bS^n\setminus \{y_0,\ldots, y_p\}$ for which $B^n(y_i,\delta)\setminus \{y_i\}$ is isometric, in metric $g$, to $\bS^{n-1}(\delta) \times (0,\infty)$ for some $\delta>0$ and all $0\le i \le p$. Then there exists a surjective BLD-mapping $\R^n \to (M,g)$.
\end{theorem}

Theorem \ref{thm:2} clearly yields Theorem \ref{thm:1} as a corollary. Indeed, let $y_1,\ldots, y_q$ be points in $\R^n$. After identifying $\R^n$ with $\bS^n\setminus \{e_{n+1}\}$ by stereographic projection, we may fix a Riemannian metric $g$ on $M:=\bS^n\setminus\{e_{n+1}, y_1,\ldots,y_q\}$ and a BLD-mapping $f \colon \R^n \to (M,g)$ as in Theorem \ref{thm:2}. It is now easy to verify that the identity map $(M,g) \to \bS^{n}\setminus \{e_{n+1}, y_1,\ldots, y_q\}$ is quasiconformal. Thus $f \colon \R^n \to \R^n\setminus \{y_1,\ldots, y_q\}$ is quasiregular.

We are not aware of other methods of producing examples of BLD-mappings from $\R^n$ into Riemannian manifolds with many ends.

\subsection{Outline of the proof}

Using the framework of Theorem \ref{thm:2}, we outline the construction of a BLD-map $F\colon \R^n \to \bS^n\setminus \{y_0,\ldots, y_p\}$ for $p>2$, and again identify $\R^n$ with $\bS^n\setminus \{e_{n+1}\}$ by stereographic projection. It is no restriction to assume that $y_0 = e_{n+1}$ and  $y_i=(0,t_i)\in \R^{n-1}\times \R \subset \bS^n$ for $-1<t_1<t_2<\ldots < t_p<1$ and we will assume so from now on. 

Setting aside  geometric aspects of the construction, we give first the topological description of $F\colon \R^n \to \bS^n\setminus \{y_0,\ldots, y_p\}$. This description is based on certain essential partitions of $\R^n$ and $\bS^n$. Given a closed set $X$ in $\R^n$ (or in $\bS^n$), we say that a finite collection of closed sets $X_1,\ldots, X_m$ forms an \emph{essential partition of $X$} if $X_1\cup \cdots \cup X_m = X$ and sets $X_i$ have pair-wise disjoint interiors. 

In the target $\bS^n\setminus \{y_0,\ldots, y_p\}$, we fix an essential partition $E_0,\ldots, E_p$ of $\bS^n$ into $n$-cells for which $y_i \in \interior E_i$ for each $0\le i \le p$ and so that
$E_1\cup \cdots \cup E_p = \bar B^n$ and $E_0=\bS^n\setminus B^n$.
We also assume that $E_{i-1}\cap E_i \cap E_{i+1} = \bS^{n-2}$ and $E_i \cap E_{i+1}$ is an $(n-1)$-cell for all $i$ ($\mathrm{mod}\; p+1$); see Figure \ref{fig:Gp_cells_intro}. Denote ${\bf E} = (E_0,\ldots, E_p)$.

\begin{figure}[h!]
\includegraphics[scale=0.60]{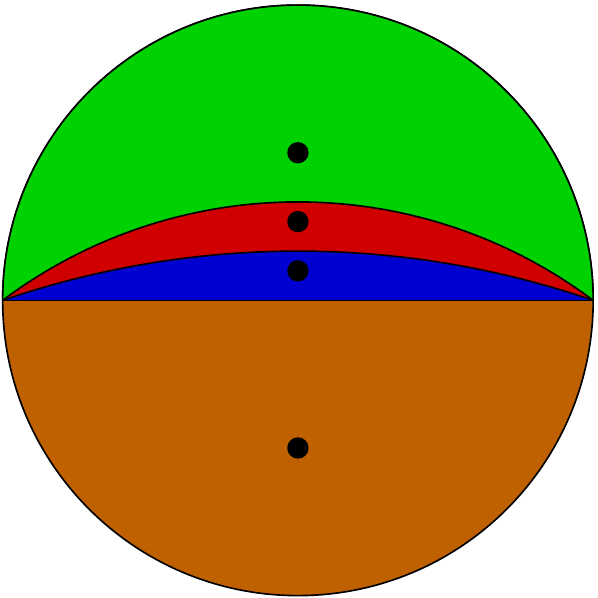}
\caption{Cells $E_1,\ldots, E_4$ with (marked) points $y_1,\ldots, y_4$ for $p=4$ (and $n=2$).} 
\label{fig:Gp_cells_intro}
\end{figure}

The $F$-induced essential partition of $\R^n$ is more complicated. Denote $\R^n_+ = \R^{n-1}\times [0,\infty)$.

Let $E'\subset \R^n$ be a closed set satisfying $E' = \mathrm{cl}(\mathrm{int}E')$. We say that a mapping $\varphi \colon \R^n_+ \to E'$ is a \emph{homeomorphism modulo boundary} if $\varphi|\mathrm{int}\R^n_+\colon \R^{n-1}\times (0,\infty) \to \mathrm{int}E'$ is a homeomorphism and, for every branched cover $\psi \colon \partial E' \to \bS^{n-1}$, the mapping $\psi \circ \varphi|\partial \R^n_+ \colon \R^{n-1}\times \{0\} \to \bS^{n-1}$ is a branched cover. Furthermore, we say that $E'$ is an \emph{half-space modulo boundary} if there exists a homeomorphism modulo boundary $\varphi \colon \R^n_+ \to E'$. Note that  $\partial E'$ need not be homeomorphic to $\R^n_+$; see Figure \ref{fig:cell_mod_boundary}.

\begin{figure}[h!]
\includegraphics[scale=0.45]{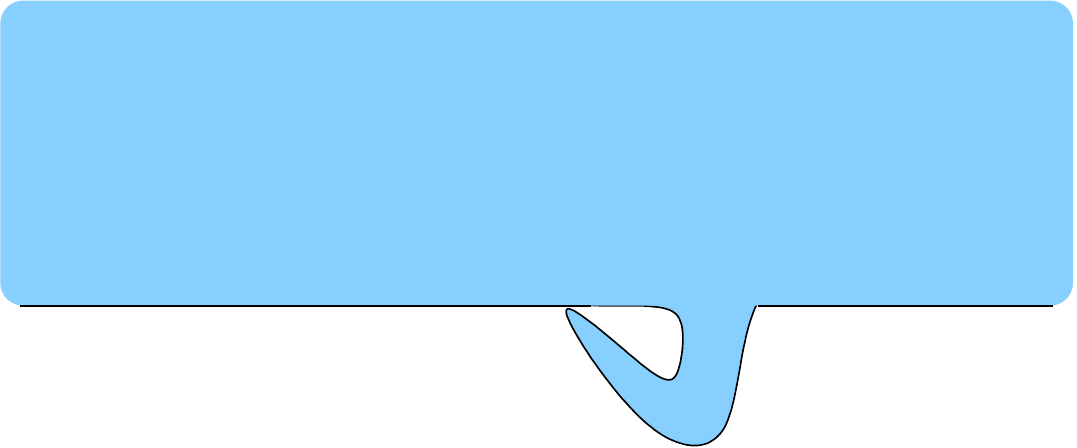}
\caption{A half-space modulo boundary.}
\label{fig:cell_mod_boundary}
\end{figure}

Suppose, for the sake of argument, there is an essential partition $\Omega_0,\ldots, \Omega_p$ of $\R^n$ into closed sets and each $\Omega_i$ has an essential partition $\Omega_{i,1},\ldots, \Omega_{i,j_i}$ into half-spaces modulo boundary. We reduce first the existence of a branched cover $F\colon \R^n \to \bS^n\setminus \{y_0,\ldots, y_p\}$ to an existence of a branched cover $f\colon \partial_\cup \bfOmega \to \partial_\cup {\bf E}$ satisfying $f(\partial \Omega_{i,j})=\partial E_i$. Here, and in what follows, the notation 
\[
\partial_\cup {\bf X} = \bigcup_{i\ne j} X_i \cap X_j
\]
is used whenever ${\bf X} =(X_0,\ldots, X_p)$ is an essential partition.

Suppose $f\colon \partial_\cup \bfOmega \to \partial_\cup {\bf E}$ is a branched cover satisfying the additional condition that $f(\partial \Omega_{i,j}) = \partial E_i$ for every $i=0,\ldots, p$ and $1\le j \le j_i$. Since $\Omega_{i,j}$ is a half-space modulo boundary and $E_i$ is an $n$-cell, we observe that each branched cover $f_{i,j} = f|\partial \Omega_{i,j}$ extends to a branched cover $F_{i,j} \colon \Omega_{i,j} \to E_i\setminus \{y_i\}$. Indeed, we may fix, for every $i$ and $j$, a homeomorphism modulo boundary $\varphi_{i,j} \colon \R^n_+ \to \Omega_{i,j}$ as well as a homeomorphism $\psi_i \colon \bS^{n-1}\times [0,\infty) \to E_i\setminus \{y_i\}$. This means that  $h_{i,j} = \psi_i^{-1} \circ f_{i,j} \circ \varphi_{i,j}|\partial \R^n_+ \colon \R^{n-1}\times \{0\} \to \bS^{n-1}$ is a branched cover. The (trivial) extension $h_{i,j} \times \id \colon \R^n_+ \to \bS^{n-1}\times [0,\infty)$ of $h_{i,j}$ now yields the required extension of $f_{i,j}$ after pre- and post-composition with $\psi_i$ and $\varphi_{i,j}^{-1}|\mathrm{int}\Omega_{i,j}$, respectively. Thus \emph{$f$ extends to a branched cover $F\colon \R^n \to \bS^n\setminus \{y_1,\ldots, y_p\}$}.

Observe also that in forthcoming constructions we may view $\partial_\cup \bfOmega$ and $\partial_\cup {\bf E}$ as branched codimension-$1$ hypersurfaces in $\R^n$ and the map $f$ as a (generalized) Alexander map. In particular, the Zorich map is of this character when $p=2$.

It is crucial that this simple extension is also available for BLD-mappings. It is a simple exercise to observe that the extension $F\colon \R^n \to \bS^n\setminus \{y_0,\ldots, y_p\}$ constructed above will be a BLD-mapping with respect to Riemannian metric $g$ in $\bS^n\setminus \{y_0,\ldots, y_p\}$ if 
\begin{itemize}
\item[(i)] $f\colon \partial_\cup \bfOmega \to \partial_\cup {\bf E}$ is a BLD-map,
\item[(ii)] $\varphi_i \colon \R^n_+ \to \Omega_i$ is BLD modulo boundary and $\varphi_i|\mathrm{int}\R^n_+$ is an embedding,  
\item[(iii)] $\psi_i \colon \bS^{n-1}\times [0,\infty) \to (E_i\setminus \{y_i\},g)$ is bilipschitz.
\end{itemize}
Here and in what follows, we say that a mapping $\varphi \colon \R^n_+ \to \Omega$, where $\Omega$ is a closed set in $\R^n$ with $\Omega = \mathrm{cl}(\mathrm{int}\Omega)$, is \emph{BLD modulo boundary} if the restriction $f|\mathrm{int}\R^n_+ \colon \mathrm{int}\R^n_+ \to \mathrm{int}\Omega$ is BLD, and for every BLD-map $\psi \colon \partial \Omega \to \bS^{n-1}$, the map $\psi \circ \varphi|\partial \R^n_+ \colon \R^{n-1}\times \{0\} \to \bS^{n-1}$ is BLD.

For Riemannian metrics $g$ with cylindrical ends as in Theorem \ref{thm:2}, it is easy to construct homeomorphisms $\psi_i$ satisfying condition (iii), and so this extension argument reduces the proof of Theorem \ref{thm:2} to Theorem \ref{thm:partition_simple}.

A closed set $\Omega$ in $\R^n$ is a \emph{Zorich extension domain} if there exists a map $\R^n_+ \to \Omega$ which is BLD modulo boundary and a homeomorphism in the interior. 

\begin{theorem}
\label{thm:partition_simple}
Given $n\ge 3$ and $p\ge 2$ there is an essential partition $\bfOmega = (\Omega_0,\ldots,\Omega_p)$ of $\R^n$ for which   
\begin{itemize}
\item[(a)] the sets $\Omega_i$ have essential partitions into Zorich extension domains, and
\item[(b)] there exists a BLD-map $f\colon \partial_\cup \bfOmega \to \partial_\cup {\bf E}$ satisfying $f(\partial \Omega_i) = \partial E_i$ for all $i=0,\ldots, p$.
\end{itemize}
\end{theorem}

Essential partitions $\bfOmega$ satisfying both conditions (a) and (b) in Theorem \ref{thm:partition_simple} are called \emph{Rickman partitions}, since the pair-wise common boundary $\partial_\cup \bfOmega$ is analogous to the $2$-dimensional complex Rickman constructs in \cite{R}. The reader may find it interesting to compare Sections \ref{sec:LRA} and \ref{sec:RP} with \cite[Sections 2 and 3]{R}.

The partition in Theorem \ref{thm:partition_simple} is achieved in two stages, with rough Rickman partitions playing an intermediate r\^ole: an essential partition $\wt \bfOmega=(\wt \Omega_0,\ldots, \wt \Omega_p)$ of $\R^n$ is a \emph{rough Rickman partition} if 
\begin{itemize}
\item[(a$'$)] each $\wt \Omega_i$ has an essential partition $(\wt\Omega_{i,1},\ldots, \wt\Omega_{i,j_i})$ with each $\wt\Omega_{i,j}$ BLD-homeomorphic to $\R^{n-1}\times [0,\infty)$, and
\item[(b$'$)] the sets $\partial_\cup \wt\bfOmega$ and $\partial_\cap \wt\bfOmega$ have finite Hausdorff distance, where
\[
\partial_\cap {\wt\bfOmega} = \bigcap_i \wt\Omega_i
\]
is the \emph{common boundary of the partition $\wt\bfOmega$}; $\partial_\cup \wt\bfOmega$ is called the \emph{pair-wise common boundary of $\wt\bfOmega$}.
\end{itemize}

In general, rough Rickman partitions $\wt\bfOmega$ do not admit branched covers $\partial_\cup \wt\bfOmega \to \partial_\cup {\bf E}$. To refine our rough Rickman partition $\wt\bfOmega$ to a Rickman partition $\bfOmega$, we impose an additional compatibility condition, called the tripod property; see Definition \ref{def:tripod} for its precise formulation. These particular rough Rickman partitions, together with a modification of Rickman's \emph{sheet construction} in \cite[Section 7]{R}, then yield the required global partition $\bfOmega$.

In Rickman's original terminology, the construction of rough Rickman partitions is called the \emph{cave construction} and the notion of \emph{cave bases} corresponds to the subdivisions provided by the tripod property. 

We summarize the two parts of the proof of Theorem \ref{thm:partition_simple} as follows. First, we prove the existence of suitable rough Rickman partitions by direct construction.
\begin{theorem}
\label{thm:3}
Given $n\ge 3$ and $p\ge 2$ there exists a rough Rickman partition $\wt\bfOmega = (\wt\Omega_0,\ldots, \wt\Omega_p)$ supporting the tripod property.
\end{theorem}

As in \cite{R} we begin the proof of Theorem \ref{thm:3} by partitioning $\R^n$ with an essential partition $\bfOmega'=(\Omega'_1,\Omega'_2,\Omega'_3)$ with $\Omega'_1$ and $\Omega'_2$ BLD-homeomorphic to $\R^{n-1}\times [0,\infty)$ and $\Omega'_3$ having a partition $(\Omega'_{3,1},\ldots, \Omega'_{3,2^{n-1}})$ into pair-wise disjoint sets, where each $\Omega'_{3,j}$ is BLD-homeomorphic to $\R^{n-1}\times [0,\infty)$. All sets $\Omega'_i$ are unions of unit $n$-cubes $[0,1]^n + v$ where $v\in \Z^n$, and $(\Omega'_1,\Omega'_2,\Omega'_3)$ satisfies the tripod property. This occupies Section \ref{sec:RP}.
The final step in the proof of Theorem \ref{thm:3} is a generalization of this argument. This step is discussed in Section \ref{sec:FT}; see Proposition \ref{prop:3_skew}.

The essential partition $\widetilde{\bfOmega}$ (as well as $\bfOmega'$) has the following geometric property.  Let $X$  be any of the sets $\wt\Omega_0,\wt\Omega_1$ or $\wt\Omega_{2,j}$ for some $1\le j \le 2^{n-1}$, and for each $k\ge 0$ write 
\[
X_k = 3^{-k}X.
\]
By passing to a subsequence if necessary, the sets $X_k$ and their boundaries $\partial X_k\subset \bS^n$ converge in the Hausdorff sense respectively to $X_\infty$ and $\partial X_\infty$, where $\partial X_\infty$ is a ``generalized Alexander horned sphere in $\bS^n$ with infinitely many horns.'' Under the normalization $\wt\bfOmega_k = 3^{-k}\wt\bfOmega$ for $k\ge 0$, in fact $\partial_\cup \bfOmega_\infty = \partial_\cap \bfOmega_\infty$ for any sublimit $\bfOmega_\infty$ of the partitions $\wt\bfOmega_k$, in the Hausdorff sense. This may be considered a \emph{coarse Lakes of Wada} property for the pair-wise common boundary of $\wt\bfOmega$. Of course, this observation applies also to Rickman's original cave construction. We do not discuss this feature of  $\wt\bfOmega$ in more detail, and leave these details to the interested reader.

The second part of the proof of Theorem \ref{thm:partition_simple} is the refinement of rough Rickman partitions to Rickman partitions. This formalizes the effect of the sheet construction (called \emph{pillows} in Section \ref{sec:pillows}) as follows.
\begin{proposition}
\label{prop:4}
Given a rough Rickman partition $\wt\bfOmega = (\wt\Omega_0,\ldots, \wt\Omega_p)$ supporting the tripod property there exists a Rickman partition $\bfOmega = (\Omega_0,\ldots, \Omega_p)$ for which the Hausdorff distance of $\partial_\cup \bfOmega$ and $\partial_\cup \wt\bfOmega$ is at most $1$.
\end{proposition}

We do not explore the geometry of the domains $\Omega_0,\ldots, \Omega_p$ further. However, we do observe that the domains in the Rickman partition can be taken to be uniform domains; see Corollaries \ref{cor:John-domains} and \ref{cor:John-final}.

As discussed in this introduction, Theorem \ref{thm:3} and Proposition \ref{prop:4} together prove Theorem \ref{thm:partition_simple}, and we obtain Theorem \ref{thm:2} from Theorem \ref{thm:partition_simple} and the observation on the existence of BLD-extensions.

\vspace*{5mm}
{\bf Acknowledgments}

Our investigation had as initial impetus many long discussions with Seppo Rickman about his remarkable work \cite{R}; his unflagging patience and optimism provided important stimulus.
The authors are also grateful to Juha Heinonen for his steady encouragement to explore Picard constructions. We thank Pietro Poggi-Corradini, and Kai Rajala for their discussions and interest and we have profited from many questions and suggestions from John Lewis and the referee. Our colleague Jang-Mei Wu warrants special thanks for extensive discussions, notes and suggestions, whose influence is manifest throughout this manuscript.

\setcounter{tocdepth}{2}
\tableofcontents


\section{Preliminaries}
\label{sec:n:pre}

In this section we introduce general metric and combinatorial notions used in the construction. Most discussion is in the ambient space $\R^n$ for some fixed $n\ge 3$.

\subsection{Metric notions}

In $\R^n$, let $d_\infty$ be the $\sup$-metric 
\[
d_\infty(x,y) = \norm{x-y}_\infty
\]
given by the \emph{supremum norm} 
\[
\norm{(x_1,\ldots, x_n)}_\infty = \max_{i} |x_i|.
\]
The metric ball $B_\infty(p,r) = \{ x\in \R^n \colon \norm{p-x}_\infty<r\}$ of radius $r>0$ about $p\in \R^n$ in this metric then is the open cube 
\[
B_\infty(p,r) = p + (-r,r)^n.
\]
Similarly, $\bar B_\infty(p,r) = p + [-r,r]^n$.

Diverting from standard terminology, we apply the term `cube' exclusively to closed $n$-balls $\bar B_\infty(p,r)$. The point $p$ is the center of the cube $\bar B_\infty(p,r)$ and of course the side length of $\bar B_\infty(p,r)$ is $2r$. 

The set $E$ in $\R^n$ is \emph{rectifiably connected} if for all $x,y\in E$ there exists a path $\gamma\colon [0,1]\to E$ of finite length so that $x,y\in \gamma[0,1]$. In this situation, $\gamma$ \emph{connects $x$ and $y$ in $E$}. When $E$ is rectifiably connected in $\R^n$, $d_E$ is its inner metric in $E$; that is, for all $x,y\in E$,
\[
d_E(x,y) = \inf_\gamma \ell(\gamma),
\]
over all paths $\gamma$ connecting $x$ and $y$ in $E$, with $\ell(\gamma)$ the length of $\gamma$. Note that the length of $\gamma$ is in terms of Euclidean distance. The notion of inner metric gives the following characterization of BLD-homeomorphisms: \emph{A homeomorphism $f\colon E\to E'$ between rectifiably connected sets $E$ and $E'$ in $\R^n$ is BLD if and only if $f \colon E \to E'$ is bilipschitz in the inner metric.}

\subsection{Complexes}
\label{sec:complexes}

For a detailed discussion on simplicial complexes we refer to \cite{Hudson} and \cite{RourkeC:Intplt} and merely recall some notation and terminology. Given a simplicial complex $P$ in $\R^n$, $P^{(k)}$ is its \emph{$k$-skeleton}, that is, the collection of all $k$-simplices in $P$.  If $m$ is the largest dimension of simplices in $P$, then $P$ has \emph{dimension} $m$, $m=\dim P$. We consider only \emph{homogeneous} simplicial complexes, that is, every simplex in $P$ is contained in a simplex of dimension $\dim P$. We denote by $|P^{(k)}|$ the subset in $\R^n$ which is the union of all simplices in $P^{(k)}$; thus, $|P|=|P^{(m)}|$. 

Recall that every $k$-simplex $\sigma$ has a standard structure as a simplicial complex having $\sigma$ as its only $k$-simplex and the vertices of $\sigma$ as the $0$-skeleton. The $i$-simplices of this complex form the \emph{$i$-faces} of $\sigma$.

We mainly consider cubical complexes. Much as simplices have a natural structure as a complex, the $k$-dimensional faces of a cube $Q=\bar B_\infty(x,r)$ determine a natural CW complex structure for $Q$. The $k$-dimensional faces of $Q$ are \emph{$k$-cubes}, and a CW complex $P$ is a \emph{cubical complex} if its cells are cubes. Note in particular, that given $i$-cube $Q$ and $j$-cube $Q'$ the intersection $Q\cap Q'$ is a $k$-dimensional, $k\le \min\{i,j\}$, face of both cubes. The $k$-skeleton and its realization are defined for cubical complexes in a manner analogous to simplicial complexes.

A homogeneous cubical complex of dimension $k$ is usually referred to as a \emph{cubical $k$-complex}. A set $E\subset \R^n$ is a \emph{cubical $k$-set} if there is a cubical $k$-complex $P$ with $|P|=E$. Cubical $k$-sets $E$ and $E'$ are \emph{essentially disjoint} if $E\cap E'$ is a cubical set of lower dimension. Given two cubical sets $E$ and $E'$, write
\[
E-E' = \cl(E\setminus E'),
\]
where $\cl(E\setminus E')$ is the closure of $E\setminus E'$. Clearly, $E-E'=E$ if $E'$ has lower dimension than $E$.

A cubical $k$-complex $P$ is \emph{$r$-fine} if all $k$-cubes in $P$ have side length $r$, i.e.\;are congruent to $[0,r]^k\subset \R^k \subset \R^n$. Similarly, a set $E$ in $\R^n$ is \emph{$r$-fine} if $r>0$ is the largest integer for which there exists an $r$-fine cubical complex $P$ with $E=|P|$, and $r$ is called the \emph{side length $\rho(E)$ of $E$}. In what follows, we assume that all cubical complexes are $r$-fine for some integer $r>0$. Given an $r$-fine set $E=|P|$, we tacitly assume that its underlying complex $P$ is also $r$-fine.

Let $P$ be a $3k$-fine cubical $n$-complex for $k\ge 1$, and $\Omega=|P|$. We denote by $\Omega^*$ the subdivision of $\Omega$ into cubes of side length $3$. More formally, there exists a unique $3$-fine cubical $n$-complex $\wt P$ satisfying $\Omega=|\wt P|$;  we denote $\Omega^* = \wt P^{(n)}$ and refer to $\Omega^*$ as the \emph{$3$-fine subdivision of $\Omega$}. We will also need $\Omega^\#$, the \emph{$1$-fine subdivision of $\Omega$}, i.e.\;subdivision of $\Omega$ into unit cubes, and call $\Omega^\#$   the \emph{unit subdivision of $\Omega$}.
In what follows, if $A\subset \R^m$ and $r>0$, we write
\[
rA = \{ rx\in \R^n \colon x \in A\}.
\]

\subsection{Essential partitions}
\label{sec:ep}
Cubical $k$-sets $U_1\ldots,U_m$ induce the \emph{essential partition $\{U_1,\ldots, U_m\}$ of the  cubical set $U$} if $U=U_1\cup\cdots\cup U_m$ and the sets $U_i$ are pairwise essentially disjoint. If the sets $U_1,\ldots, U_m$, and $U$ are $n$-cells, we usually consider the essential partition ordered and denote it $\bU = (U_1,\ldots, U_m)$ as in the introduction.

To simplify notation, for $r>0$ we also denote $r\bU = (rU_1,\ldots,rU_m)$, and given an $n$-cell $E\subset U$, write $\bU \cap E = (U_1 \cap  E,\ldots, U_m\cap E)$ and $\bU - E = (U_1 - E,\ldots, U_m-E)$.

\subsection{Graphs, forests, and adjacency}
\label{sec:gfa}

The pair $G=(V,E)$ is \emph{a graph} if $V$ is a countable set and $E$ is a collection of unoriented pairs of points in $V$; $V$ is the set of \emph{vertices} and $E$ the \emph{edges} of $G$. Note we only allow one edge between two distinct vertices and, in particular, our graphs do not have \emph{loops}, i.e.\;edges from a vertex to itself. 

We use repeatedly the standard fact that a graph contains a maximal tree, that is, given a graph $G=(V,E)$ there is a subtree $T=(V,E')$ containing all vertices of $G$. The length $\ell(G)$ of $G$ is the number of vertices of $G$, the \emph{valence of $G$ at $v$} is $\nu(G,v)$ and $\nu(G) = \max_{v\in G} \nu(G,v)$ is the \emph{(maximal) valence of $G$}. We denote by $d_G(v,v')$ the graph distance of $v$ and $v'$ in $G$, that is, the length of the shortest edge path between $v$ and $v'$ in $G$.

Given a distinguished vertex $v\in G$, the pair $(G,v)$ is called a \emph{rooted graph} and $v$ the \emph{root} of this graph. The \emph{radius $r(G,v)$ of $G$ at $v$} is the largest graph distance between $v$ and a leaf of $G$; a vertex $w\in G$ is a \emph{leaf} if it belongs to exactly one edge, or equivalently, has valence $1$. A vertex which is neither a leaf nor the root is an \emph{inner vertex}.
A subtree $\Gamma\subset G$ connecting the root $v$ to a leaf $w$ of $G$ is a \emph{branch} when all vertices in $\Gamma$ other than $v$ and $w$ have valence $2$.

Let $(G,v)$ be a finite rooted tree and $v'\ne v$ a vertex in $G$. We define the \emph{subtree behind $v'$ in $(G,v)$} as follows. Since $G$ is a tree, there exists unique $v'' \in G$ for which $e=\{v'',v'\}$ is the last edge in the shortest path from $v$ to $v'$. The graph $(V,E\setminus \{e\})$ has two connected components $\Gamma_v$ and $\Gamma_{v'}$ containing $v$ and $v'$, respectively.  Both component are trees; $\Gamma_{v'}$ is the \emph{subtree behind $v'$ in $(G,v)$}. 

A graph $G$ is a \emph{forest} if all of its components are trees. A forest $F\subset G$ is \emph{maximal} if components of $F$ are maximal trees in components of $G$ and $F$ contains all vertices of $G$.

A function $u\colon G \to \R$ on a tree $G$ has the \emph{John property in $G$} if given $v$ and $v'$ in $G$ there exists $0\le j \le d=d_G(v,v')$ so that $u$ is (strictly) increasing on $v_0,\ldots, v_j$ and (strictly) decreasing on $v_{j+1},\ldots, v_d$, where $v=v_0, v_1,\ldots, v_d = v'$ is the unique shortest edge path from $v$ to $v'$ in $G$.

Most graphs we consider are adjacency graphs of collections of $k$-cells in $\R^n$. A set $E\subset \R^n$ is a \emph{$k$-cell} if $E$ is homeomorphic to the closed cube $[0,1]^k$ in $\R^k$; $E$ is a \emph{cubical $k$-cell} if for some $r\ge 1$ there is an $r$-fine homogeneous cubical complex $P$ for which $E=|P|$.

Two $k$-cells $E$ and $E'$ are \emph{adjacent} if $E\cap E'$ is an $(k-1)$-cell. We recall from PL theory that given two adjacent PL $k$-cells $E$ and $E'$ there exists a PL homeomorphism $E\cup E'\to E$ which is identity on $\partial(E\cup E') \cap E$, and refer to \cite{HudsonJ:Piclt} or \cite[Chapter 3]{RourkeC:Intplt} for this and similar results in PL theory.

A collection $\cP$ of $k$-cells in $\R^n$ has the adjacency graph 
\[
\Gamma(\cP)=(\cP,\left\{ \{E,E'\} \colon E\in \cP\ \mathrm{and}\ E'\in \cP\ \mathrm{are\ adjacent}\right\}).
\]
Given a subgraph $\Gamma\subset \Gamma(P)$, we denote $|\Gamma|=\bigcup_{E\in \Gamma} E$; in particular, $|\Gamma(P)|=|P|$. 

\begin{figure}[h!]
\includegraphics[scale=1.2]{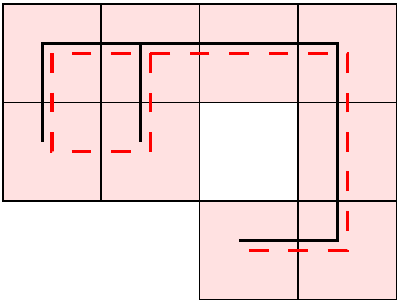}
\caption{A cubical $2$-complex with its adjacency graph and a choice of a maximal tree.}
\label{fig:AG}
\end{figure}

\subsection{Remarks on figures}

Although we consider $n$-cells for $n\ge 3$, we use two-dimensional illustrations related to three-dimensional example configurations, so that often a three-dimensional situation is seen relative to one of its faces. The figures displayed here often have orientations different than suggested by their coordinates in $\R^3$.

In particular, 'fold-out' diagrams illustrate particular cubical $(n-1)$-complexes. To formalize this, suppose $E$ is a cubical $(n-1)$-cell in $\R^n$ with an essential partition $\{E_1,\ldots, E_s\}$ into unit $(n-1)$-cubes and let $\Gamma$ be a maximal tree in $\Gamma(\{E_1,\ldots, E_s\})$. An $(n-1)$-cell $E'$ in $\R^{n-1}$ then is a \emph{fold-out of $E$ (along $\Gamma$)} if $E'$ has a partition $\{E'_1,\ldots, E'_s\}$ with adjacency graph $\Gamma(\{E'_1,\ldots, E'_s\})$ isomorphic to $\Gamma$ and there exists a map $\psi \colon E' \to  E$ which sends each cube $E'_i$ isometrically to $E_i$. We call $\psi$ a \emph{bending of $E'$}. Sometimes, as in Figure \ref{fig:AG2}, a fold will be indicated by a dashed line. 

\begin{figure}[h!]
\includegraphics[scale=0.30]{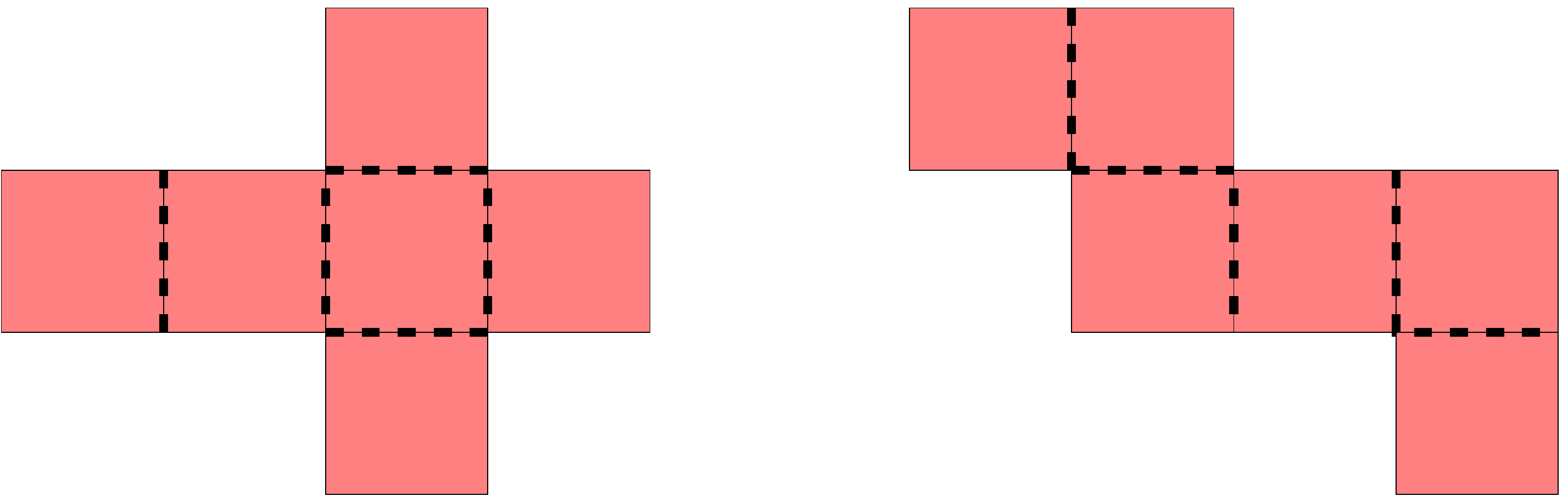}
\caption{Two different fold-outs of faces of a $3$-cube along non-isomorphic maximal trees.}
\label{fig:AG2}
\end{figure}

Fold-out figures, in particular, illustrate $3$-cells contained in $3$-cubes. Most of our figures of this type, e.g.\;in Sections \ref{sec:LRA} and \ref{sec:RP}, are akin to the following two simple examples.

Consider the cube $Q=[0,3]^3$. Then $F=[0,3]^2\times \{0\}$ is a face of $Q$ and the unit cube $q=[1,2]^2\times [0,1]$ is contained in $Q$ and meets $F$ in the face $f=[1,2]^2\times \{0\}$. We illustrate the fact that $q$ meets $F$ by identifying $f$ in $F$ as in Figure \ref{fig:figure_example_1}.

\begin{figure}[h!]
\includegraphics[scale=0.5]{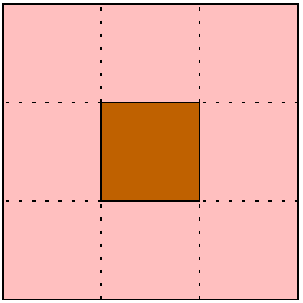}
\caption{Cube $q$ in $Q$ realized as a square $f$ in $F$.}
\label{fig:figure_example_1}
\end{figure}

In our second example, $Q$ and $F$ remain the cube $[0,3]^3$ and its face $[0,3]^2\times \{0\}$ respectively, but $q=[0,1]\times [1,2]\times [0,1]$. Let also $F'$ be the face $\{0\}\times [0,3]^2$ of $Q$. Then $q\cap  (F\cup F')$ is a union of two faces $f=[0,1]\times [1,2]\times \{0\}$ and $f'=\{0\}\times [1,2]\times [0,1]$ of $q$. To indicate how $q$ meets $F\cup F'$ in more than one face, we use the symbol '\emph{\textsf{x}}' to indicate one of the two faces which correspond to $q$ in $Q$ as in Figure \ref{fig:figure_example_2}.
\begin{figure}[h!]
\includegraphics[scale=0.5]{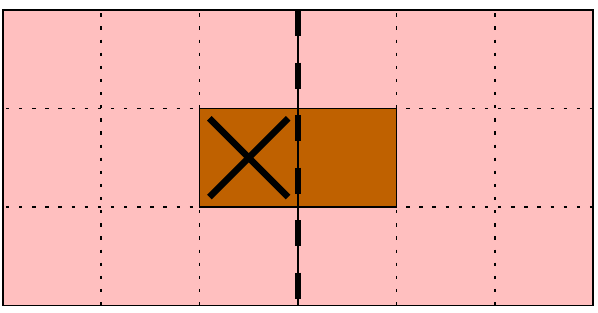}
\caption{Cube $q$ in $Q$ meeting faces $F\cup F'$.}
\label{fig:figure_example_2}
\end{figure}


\section{Atoms and molecules}
\label{sec:RT}

In this section we discuss the elementary BLD-theory of certain cubical $n$-cells. We call these classes of cells \emph{atoms}, \emph{molecules}, \emph{dented atoms} and \emph{dented molecules}.

\begin{definition}
We say that a cubical $n$-cell $A=|P|$ in $\R^n$ is an \emph{atom of length $\ell$} if $A$ is $r$-fine and the adjacency graph $\Gamma(P)$ is a tree of length $\ell$.
\end{definition}

Given an atom $A=|P|$, we denote by $\ell(A)$ its length; i.e.\;$\ell(A)=\ell(\Gamma(P))$. Note also that every $r$-fine atom $A$ has uniquely determined $r$-fine complex $P_A$ with $A=|P_A|$.

\begin{figure}[h!]
\includegraphics[scale=0.5]{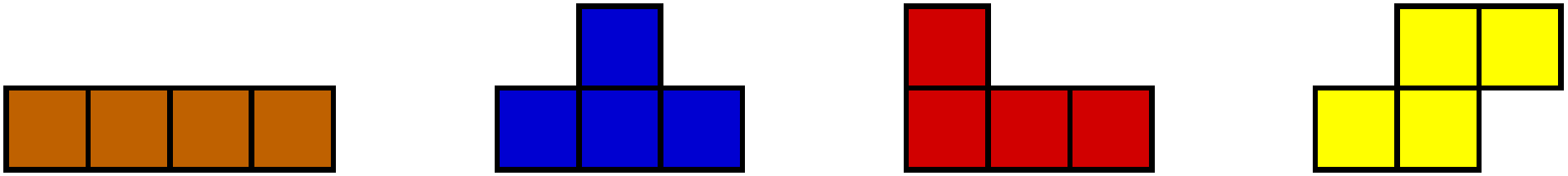}
\caption{Some atoms of length $4$.}
\label{fig:SA}
\end{figure}

Clearly, by finiteness of adjacency trees, every $r$-fine atom of length $\ell$ is uniformly $L$-bilipschitz to the $n$-cube $[0,r]^n$ with $L$ depending only on $n$ and $\ell$. In what follows, we define more complicated cells, using atoms as building blocks. The hierarchy between atoms in these constructions is given by the notion of proper adjacency. Atoms $A=|P|$ and $A'=|P'|$ are \emph{properly $h$-adjacent}, for $h>1$, if
\begin{itemize}
\item[(1)] the side lengths of $A$ and $A'$ satisfy $\rho(A) \ge h\rho(A')$ or $\rho(A')\ge h\rho(A)$, and
\item[(2)] there exist $n$-cubes $Q\in P^{(n)}$ and $Q'\in (P')^{(n)}$ for which $A\cap A' = Q\cap Q'$.
\end{itemize}

Let $\cA$ be a finite collection of properly adjacent atoms so that $\Gamma(\cA)$ is a tree. Suppose also that $|\Gamma(\cA)|$ is \emph{John}, that is, the function $A \mapsto \rho(A)$ is a John function on $\Gamma(\cA)$, so there is a unique $\hat A\in \cA$ with  $\rho(\hat A) = \max_{A\in \cA} \rho(A)$, called the \emph{root} of $\cA$.

We exploit the John property to produce bilipschitz mappings from $|\Gamma(\cA)|$ to proper subdomains of $|\Gamma(\cA)|$. Especially, we construct bilipschitz maps $|\Gamma(\cA)|\to \hat A$, where $\hat A$ is the root of $\cA$. To obtain uniform bounds for the bilipschitz constants, we define a collapsibility condition and introduce a class of $n$-cells called molecules; see Proposition \ref{prop:fRt} for the first bilipschitz contractibility statement for molecules.

Let $A\in \Gamma(\cA)$ be an inner vertex in $(\Gamma(\cA),\hat A)$ and let $\cN(A)$ be neighbors of $A$ in $\Gamma(\cA)$. For each $a\in \cN(A)$, let $q_a\in P_a^{(n-1)}$ be the unique cube satisfying $q_a \cap A = a\cap A$, and denote by $F_A(a)$ the face of $q_a$ containing $q_a\cap A$. Note that, since $|\Gamma(\cA)|$ is John, there exists a unique $A' \in \cN(A)$ so that $\rho(A')>\rho(A)$.

\begin{definition}
\label{def:lambda}
Let $A\in \Gamma(\cA)$ be an inner vertex in $(\Gamma(\cA),\hat A)$ and $A'$ a neighbor of $A$ with $\rho(A')>\rho(A)$. Vertex $A$ is \emph{$\lambda$-collapsible for $\lambda>1$} if  there exists a collection $\{f_a \subset F_{A}(A') \colon a\in \cN(A)\setminus \{A'\}\}$ of essentially pair-wise disjoint $(n-1)$-cubes with $\rho(f_a) = \lambda \rho(F_A(a))$. 
\end{definition}

\begin{definition}
Let $M=|\Gamma(\cA)|=\bigcup_{A\in \cA} A$ be a cubical $n$-cell having an essential partition into finite collection $\cA$ of atoms, $\nu \ge 1$ and $\lambda>1$. Then $M$ is a \emph{$(\nu,\lambda)$-molecule} if
\begin{itemize}
\item[(a)] the adjacency graph $\Gamma(\cA)$ is a tree,
\item[(b)] adjacent atoms in $\cA$ are properly $3$-adjacent,
\item[(c)] $\Gamma(\cA)$ is John, 
\item[(d)] $\Gamma(\cA)$ has valence at most $\nu$, and
\item[(e)] each inner vertex of $(\Gamma(\cA),\hat A)$ is $\lambda$-collapsible, where $\hat A$ is the root of $\cA$.
\end{itemize}
\end{definition}

\begin{remark}
\label{rmk:John}
By (c), $M = |\Gamma(\cA)|$ is a John domain; see e.g.\;\cite{MartioO:Injtps} or \cite{VaisalaJ:UnioJd} for terminology.
\end{remark}

\begin{figure}[h!]
\includegraphics[scale=0.3]{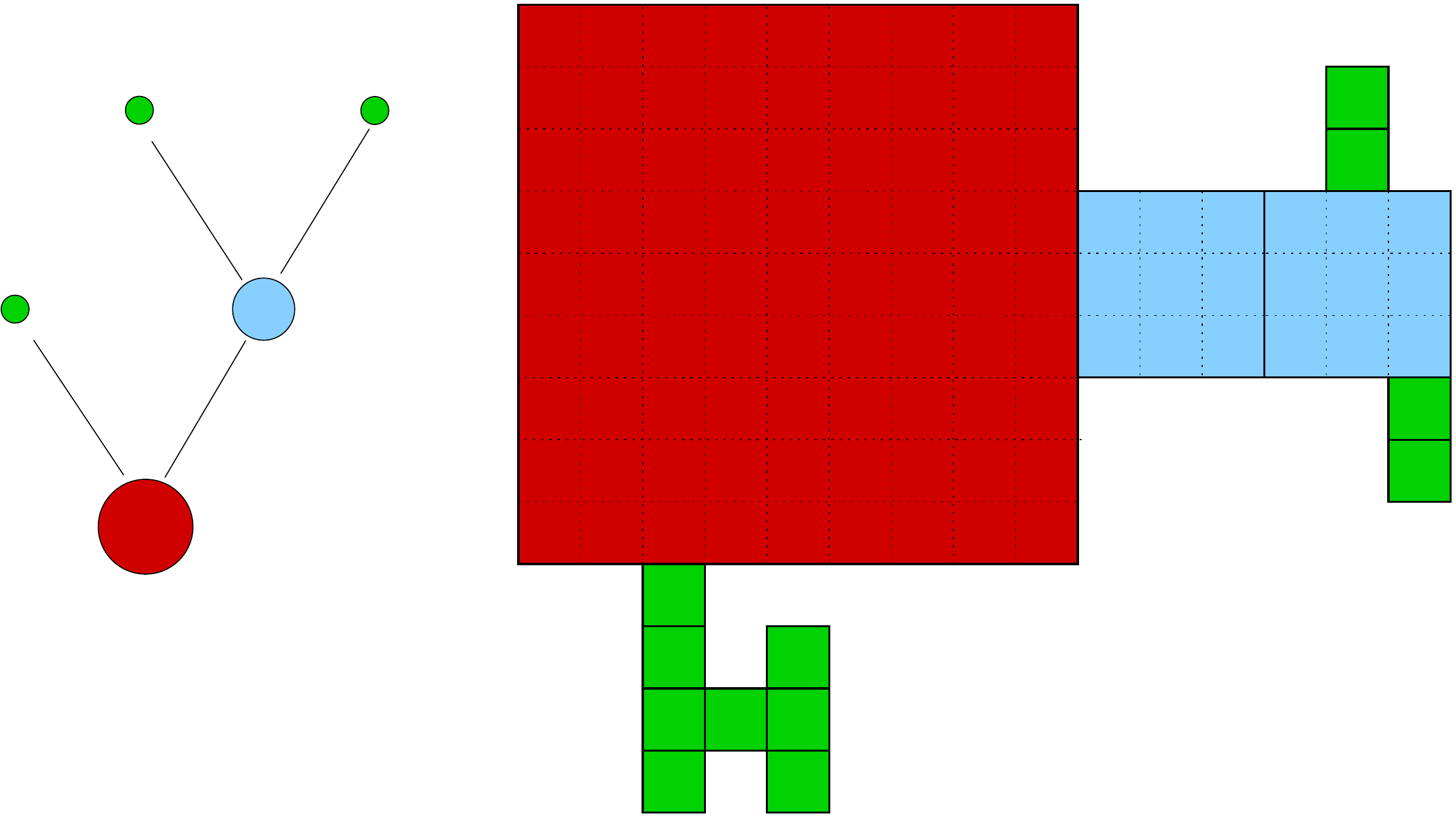}
\caption{Example of a tree $\Gamma(\cA)$ and molecule $M=|\Gamma(\cA)|$.}
\label{fig:molecule}
\end{figure}

Let $M=|\Gamma(\cA)|$ be a molecule. By (b), the atoms in $\cA$ and the tree $\Gamma(\cA)$ are uniquely determined. The tree $\Gamma(M)=\Gamma(\cA)$ is the \emph{atom tree of $M$}, and the root $\hat A$ of $\cA$ is called the \emph{root of $M$}. The tree $\Gamma^\icl(M) = \Gamma\left( \bigcup_{A\in \cA} P_A^{(n)}\right)$ is the \emph{internal tree $\Gamma^\icl(M)$ of $M$}.
In addition, 
\[
\ell_\atom(M) = \max_{A\in \Gamma(\cA)} \ell(A)
\] 
is the \emph{atom length of $M$}, and 
\[
\ell(M) = \ell(\Gamma(\cA))
\]
 is the \emph{(external) length of $M$}.  The \emph{(maximal) side length of $M$} is
\[
\rho(M) = \max_{A\in \Gamma(\cA)} \rho(A).
\]

The main result on molecules is the following bilipschitz contraction property. 
\begin{proposition}
\label{prop:fRt}
Let $M$ be a $(\nu,\lambda)$-molecule with root $\hat A$ in $\R^n$. Then there exists an $L$-bilipschitz homeomorphism 
\[
\phi \colon \left(M, d_M\right) \to (\hat A,d_{\hat A})
\]
which is the identity on $\hat A\cap \partial M$, where $L$ depends only on $n$, $\nu$, $\lambda$, and $\ell_\atom(M)$. 
\end{proposition}

This proposition should not surprise any expert. Its proof is based on the bounded local structure of $\Gamma(M)$ and bilipschitz equivalence of atoms of uniformly bounded length. Due to the specific nature of the statement and its fundamental r\^ole in our arguments, we discuss its proof in detail. We gratefully acknowledge work of Semmes, especially \cite{SemmesS:Fincgs}, as the main source of these ideas.

The proof of Proposition \ref{prop:fRt} is by induction on the size of the tree $\Gamma(M)$. We begin with a lemma corresponding to the induction step of this proof. Given sets $X$ and $Y$ in $\R^n$, the set
\[
X\star Y = \{ tx+(1-t)y \in \R^n \colon x\in X,\ y\in Y,\ t\in [0,1]\},
\]
is the \emph{join of $X$ and $Y$}. If $Q$ is an $n$-cube in $\R^n$, $x_Q$ is its barycenter, that is, $Q=B_\infty(x_Q,r_Q)$, where $r_Q>0$. For an $(n-1)$-cube $F$, the barycenter $x_F$ is defined as the average of the vertices of $F$. Both definitions coincide for $n$-cubes.

\begin{lemma}
\label{lemma:fRt_PL}
Let $Q$ be an $n$-cube and let $M$ be a molecule properly adjacent to $Q$ with $\rho(Q)>\rho(M)$, and let $\nu\ge 1$ and $\lambda>1$. 

Let $F$ be the face of $Q$ containing $M\cap Q$ and let $F_1,\ldots, F_\nu \subset \partial M-Q$ be pair-wise disjoint faces of $n$-cubes $Q_1,\ldots, Q_\nu$ in $\Gamma^\icl(M)$, respectively.  Suppose there exist essentially pair-wise disjoint $(n-1)$-cubes $F'_1,\ldots, F'_\nu$ in $F$ satisfying $\rho(F'_i) = \lambda \rho(F_i)$ for every $i=1,\ldots, \nu$.

Then there exist $L = L(n, \ell_{\atom}(M), \ell(M), \nu,\lambda) \ge 1$ and an $L$-bilipschitz homeomorphism 
\[
\phi \colon (M\cup Q,d_{M\cup Q}) \to Q,
\]
which is the identity on $Q-(F\star \{x_Q\})$ and an isometry on each $F_i \star \{x_{Q_i}\}$. 
\end{lemma}

\begin{proof}
Let $i\in \{1,\ldots, \nu\}$ and set $F''_i = B_\infty(x_{F'_i},\rho(F_i)/2)\cap F\subset F'_i$. Then $F''_i$ is an $(n-1)$-cube in $F$ with the same barycenter as $F'_i$ and the same side length as $F_i$. We denote by $Q''_i\subset Q$ the $n$-cube having $F''_i$ as a face, and set $\Delta_i = F_i \star \{x_{Q_i}\}$, $\Delta''_i = F''_i\star \{x_{Q''_i}\}$. 

By a shelling argument, there exists a PL-homeomorphism $\phi \colon M\cup Q \to Q$ which is the identity in $Q\setminus (F\star \{x_Q\})$ and restricts to an isometry $\phi|\Delta_i \colon \Delta_i \to \Delta''_i$ for every $i=1,\ldots, \nu$; see e.g.\;\cite[Lemma 3.25]{RourkeC:Intplt}. Since it suffices to consider only a finite number of triangulations and PL-homeomorphisms, $\phi$ is uniformly bilipschitz with a constant depending only on $n$, $\ell_\atom(M)$, $\ell(M)$, $\nu$, and $\lambda$.
\end{proof}

\begin{proof}[Proof of Proposition \ref{prop:fRt}]
Let $M=|\Gamma(\cA)|$ be a $(\nu,\lambda)$-molecule with root $\hat A$; see Figure \ref{fig:RC4}. We may assume that $M\ne \hat A$ and, more precisely, that $\Gamma(\cA)$ has inner vertices, since otherwise the claim follows from Lemma \ref{lemma:fRt_PL}.

\begin{figure}[h!]
\includegraphics[scale=0.30]{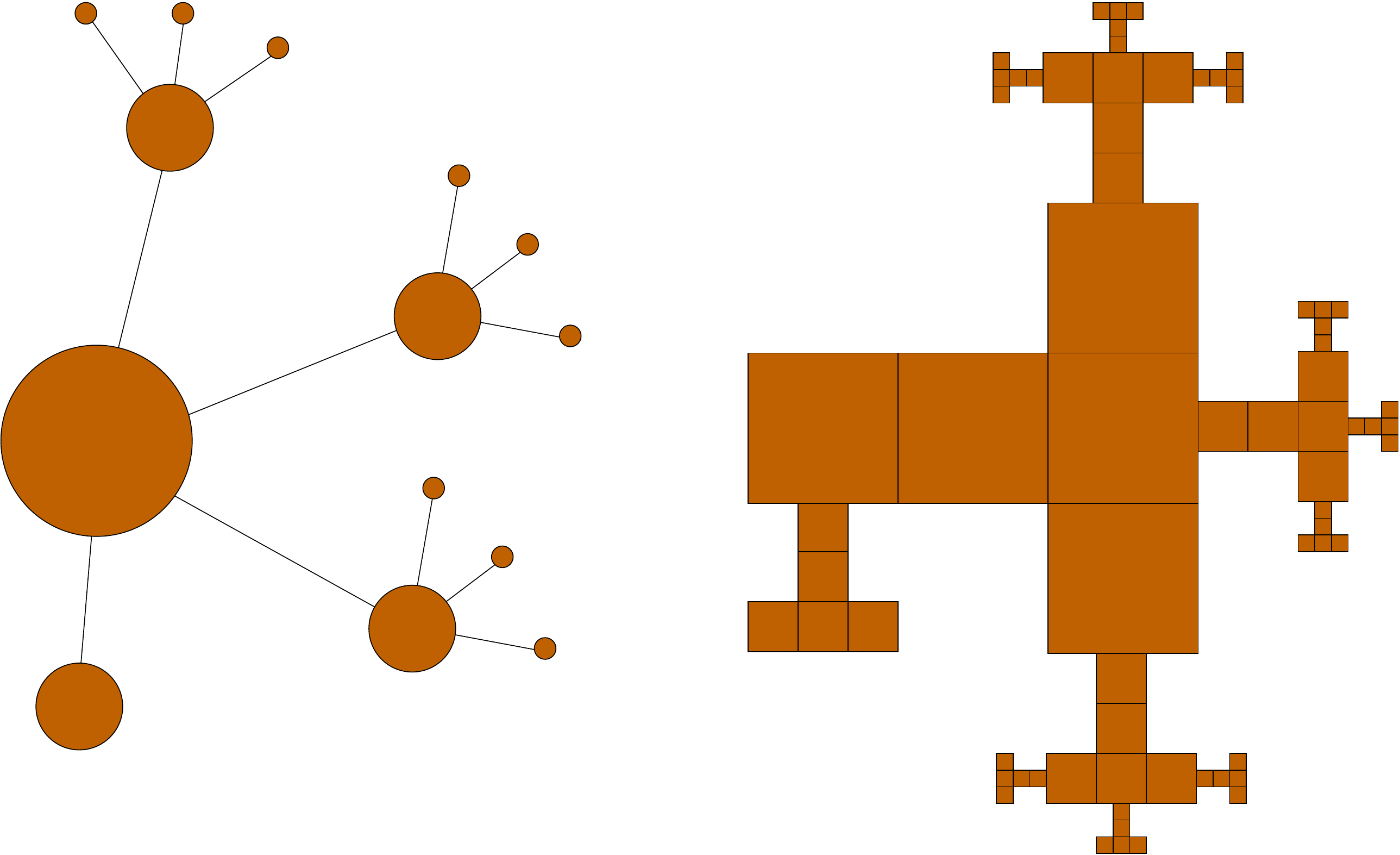}
\caption{A tree $\Gamma(\cA)$ and molecule $M=|\Gamma(\cA)|$.}
\label{fig:RC4}
\end{figure}

To begin the induction, denote $\Gamma_0 = \Gamma(\cA)$, $M_0 = M$, and to each leaf $L\in \Gamma_0$ associate a face $F_L$ of an $n$-cube $Q_L\in \Gamma^\icl(L)$ with $F_L\subset \partial M_0 \cap L$. We denote the set of these chosen faces by $\cF_0$, and for every leaf $L\in\Gamma_0$ set $J_L = F_L \star \{x_{Q_L}\}$.

Fix an atom $A'_0\in \Gamma_0$ which is an inner vertex in $\Gamma_0$ so that the rooted subtree $\Gamma'_0= \Gamma_{A'_0}$ behind $A'_0$ in $(\Gamma_0,\hat A)$ consists of leaves of $\Gamma_0$. Also choose an atom $A_0\in \Gamma_0\setminus \Gamma'_0$ adjacent to $A'_0$ in $\Gamma_0$. Let $Q_0$ be the unique $n$-cube in $A_0$ and $F_0$ the unique face of $Q_0$ which contains $A_0\cap A'_0$; denote $J_0 = F_0 \star \{x_{Q_0}\}$ and $\cF'_0=\{ F_L \colon L \in \Gamma'_0\}$. 

Since $M=|\Gamma(\cA)|$ is a $(\nu,\lambda)$-molecule and $A'_0$ is an inner vertex in $\Gamma(\cA)$, $A'_0$ is $\lambda$-collapsible. Thus there exists a collection $\{F'_L\colon L\in \Gamma'_0\}$ of pair-wise disjoint $(n-1)$-cubes satisfying $\rho(F'_L) = \lambda \rho(F_L)$ for every $L\in \Gamma'_0$. 

By Lemma \ref{lemma:fRt_PL}, there exist a constant $L\ge 1$, depending only on $n$, $\nu$, $\delta$, $\ell_\atom(M)$, and $\ell(M)$, and an $L$-bilipschitz homeomorphism 
\[
\phi_0 \colon (|\Gamma'_0|\cup Q_0, d_{|\Gamma'_0|\cup Q_0}) \to (Q_0,d_{Q_0}),
\]
which is the identity on $Q-(F_0\star \{x_{Q_0}\})$ and an isometry on each join $J_L$ for $L\in \Gamma'_0$.

We now define $\Gamma_1 = \Gamma_0 \setminus \Gamma'_0$ and $\cF_1 = (\cF_0\setminus \cF'_0)\cup \{F_0\}$. Then $M_1=|\Gamma_1|$ is a $(\nu,\lambda)$-molecule with root $\hat A$. In terms of this notation, $\phi_0$ extends, by identity, to an $L$-bilipschitz homeomorphism 
\[
\phi_0 \colon (M_0,d_{M_0}) \to (M_1,d_{M_1}),
\]
which is an isometry on each join $J_L=F_L \star \{x_{Q_L}\}$ for $L\in \Gamma'_0$.

Clearly, $\ell(M_1)<\ell(M_0)$. We iterate this step to obtain a descending sequence of subgraphs $\Gamma_0,\ldots, \Gamma_i$ of $\Gamma(\cA)$ so that every $\Gamma_j$ has at least one vertex fewer than $\Gamma_{j-1}$ for $j=1,\ldots, i$; see Figure \ref{fig:RC5}. Since $\Gamma(\cA)$ is a finite tree, there exists $i_0\ge 1$ depending on $r(\Gamma(\cA),\hat A)$ so that $\Gamma_{i_0}$ consists of only $\hat A$.

\begin{figure}[h!]
\includegraphics[scale=0.3]{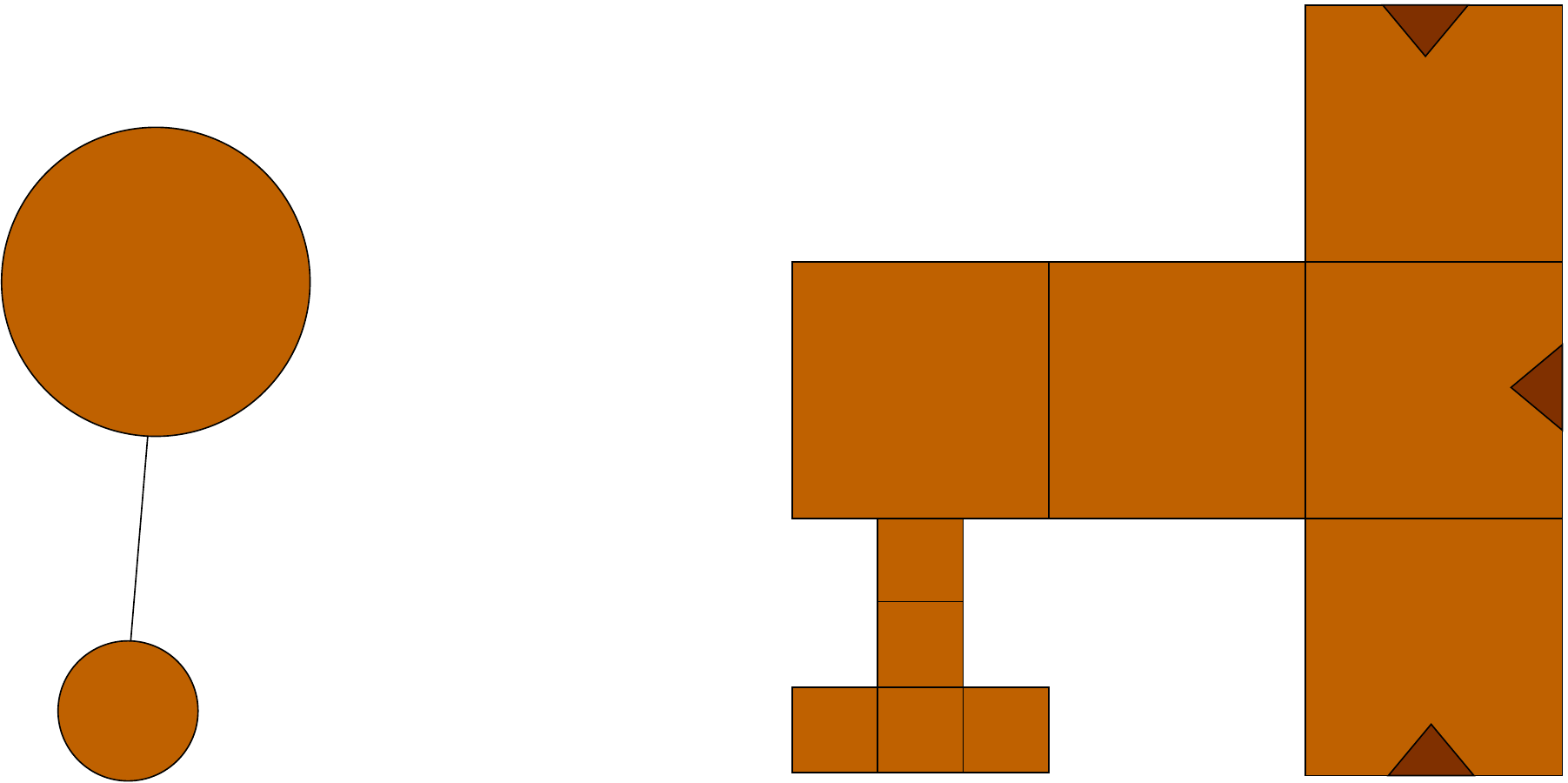}
\caption{An intermediate tree $\Gamma_i$ and cell $|\Gamma_i|$.}
\label{fig:RC5}
\end{figure}

For $i=0,\ldots, i_0$, we also obtain collections of faces $\cF_0,\ldots, \cF_i$ on leaves of graphs $\Gamma_0,\ldots, \Gamma_i$, and $L$-bilipschitz homeomorphisms 
\[
\phi_{j-1} \colon (|\Gamma_{j-1}|,d_{|\Gamma_{j-1}|}) \to (|\Gamma_j|,d_{|\Gamma_j|})
\]
which are isometries on the joins over the faces in $\cF_{j-1}$ for every $j=1,\ldots, i_0$. As in the construction above, $\phi_i(|\Gamma_{i-1}|)$ is contained in a join over a face in $\cF_i$. Thus 
\[
\phi_i \circ \cdots \circ \phi_0 \colon (|\Gamma_0|,d_{|\Gamma_0|}) \to (|\Gamma_i|,d_{|\Gamma_i|})
\]
is $L$-bilipschitz for every $i=0,\ldots, i_0$, where $L$ depends only on $n$, $\nu$, $\lambda$, and $\ell_\atom(M)$, and so
\[
\phi_{i_0} \circ \cdots \circ \phi_0 \colon (|\Gamma_0|,d_{|\Gamma_0|}) \to (\hat A, d_{\hat A})
\]
satisfies the conditions of the claim. This concludes the proof.
\end{proof}

\begin{corollary}
\label{cor:fRt}
Let $M=|\Gamma(\cA)|$ be a $(\nu,\lambda)$-molecule and let $\Gamma\subset \Gamma(\cA)$ be a subtree containing the root $\hat A$ of $M$. Then there exist an $L\ge 1$ depending only on $n$, $\nu$, $\lambda$, and $\ell_\atom(M)$, and an $L$-bilipschitz homeomorphism $\phi \colon (M,d_M) \to (|\Gamma|,d_{|\Gamma|})$ which is the identity on $|\Gamma|\cap \partial M$.
\end{corollary}

\begin{proof}
Let $\Gamma'$ be a component of $\Gamma(\cA)\setminus \Gamma$. Then $|\Gamma|$ is an $(\nu,\lambda)$-molecule. Thus the claim follows by applying Proposition \ref{prop:fRt} to components of $\Gamma(\cA)\setminus \Gamma$ followed by Lemma \ref{lemma:fRt_PL} on the roots of these trees.
\end{proof}

Before introducing dented atoms, we record a uniform bilipschitz equivalence result in spirit of Proposition \ref{prop:fRt}. A half-space in $\R^n$ appears as the normalized target; full details of the proof are left to the interested reader.

\begin{proposition}
\label{prop:adap_cont_2}
Let $\nu \ge 1$, $\lambda >1$, $\ell\ge 1$, and let $(M_m)$ be an increasing sequence of $(\nu,\lambda)$-molecules so that, for every $m\ge 1$,
\begin{itemize}
\item [(1)] $M_m - M_{m-1}$ is connected and contains the root of $M_m$,  
\item[(2)] $\ell_\atom(M_m) \le \ell$, and
\item[(3)] if $A$ and $A'$ are adjacent in $\Gamma(M_m)$ with $\rho(A)<\rho(A')$ then $\rho(A')=3\rho(A)$.
\end{itemize}
Let $M = \bigcup_{m\ge 0} M_m$. Then $(M,d_M)$ is $L$-bilipschitz equivalent to $\R^{n-1}\times [0,\infty)$, where $L$ depends only on $n$, $\nu$, $\lambda$, and $\ell$.
\end{proposition}

\begin{proof}[Sketch of proof]
Let $\Gamma$ be the tree $\bigcup_{m\ge 0} \Gamma(M_m)$, and let $\Gamma'$ be the unique branch passing through all roots $\hat M_m$ of $M_m$ for $m\ge 0$. We may consider $\Gamma'$ as a sequence of atoms with increasing side length, and for every $m\ge 0$ denote by $\Gamma'_m$ the part of $\Gamma'$ contained in $\Gamma(M_m)$.  

Following the idea of Corollary \ref{cor:fRt}, we obtain a sequence $(\psi_m)$ of $L'$-bilipschitz contractions $\psi_m \colon (M_m,d_{M_m}) \to (|\Gamma'_m|,d_{|\Gamma'_m|})$ so that $\psi_{m+1}|M_m=\psi_m$ for every $m\ge 0$, where $L'$ depends only on $n$, $\nu$, $\lambda$, and $\ell$. This produces an $L$-bilipschitz map $\psi \colon (M,d_M) \to (|\Gamma'|,d_{|\Gamma'|})$. 

It remains now to show that $(|\Gamma'|,d_{|\Gamma'|})$ is $L''$-bilipschitz equivalent to $\R^{n-1}\times [0,\infty)$, where $L''$ depends only on $n$ and $\ell$.

Let $A$ be the unique vertex in $\Gamma'$ with valence $1$. Since $\Gamma'$ is a branch, we may now enumerate the vertices in $\Gamma'$ as $A=a_0,a_1,a_2,\ldots$ with $a_k$ adjacent to $a_{k+1}$. By (3), $\rho(a_{k+1}) = 3\rho(a_k)$ for every $k\ge 0$. Thus $(|\Gamma'|,d_{|\Gamma'|})$ is $L'''$-bilipschitz equivalent, $L'''=L'''(n,\ell)$, to a cone 
\[
\{ (x_1,\ldots, x_n)\in \R^n\colon x_n^2 \le x^2_1 + \cdots x^2_{n-1} \},
\]
and hence $L''$-bilipschitz equivalent to $\R^{n-1}\times [0,\infty)$, where $L''$ depends only on $n$ and $\ell$.
\end{proof}

\subsection{Dented atoms}
\label{sec:dented_atoms}

\begin{definition}
Let $A$ be an atom in $\R^n$. A molecule $M$ contained in $A$ is \emph{on the boundary of $A$} if $A-M$ is an $n$-cell and for each $Q\in \Gamma^\icl(M)$
\begin{itemize}
\item[(i)] $Q$ is contained in a strictly larger cube of $\Gamma^\icl(A)$, and 
\item[(ii)] $Q\cap \partial A$ contains a face of $Q$.
\end{itemize} 
\end{definition}

\begin{definition}
\label{def:dRc_A}
Let $A$ be an atom in $\R^n$ and let $M_1,\ldots, M_\nu \subset A$ be pair-wise disjoint molecules on the boundary of $A$ each having side length at most $3^{-2} \rho(A)$. The $n$-cell $D=A-\bigcup_{i} M_i$ is a \emph{dented atom} if
\begin{itemize}
\item[(i)] each $M_i$ is contained in an $n$-cube in $\Gamma(A)$, and  
\item[(ii)] $\dist(Q,Q')\ge \min \{ \rho(Q),\rho(Q')\}$ for all $Q\in \Gamma^{\icl}(M_i)$ and $Q'\in \Gamma^{\icl}(M_j)$ for $i\ne j$.
\end{itemize}
The molecules $M_1,\ldots,M_\nu$ are called \emph{dents of $A$}, and $A$ is the \emph{hull of $D$},  $\hull(D)$.
\end{definition}

\begin{remark}
The reader may find the constant $3^{-2}$ curious, but this explicit constant is chosen to be compatible with constructions in Section \ref{sec:RP}, more specifically, Section \ref{sec:IC_init}. These constructions also have the property that each cube in $\Gamma(A)$ has at most $2$ dents. 
\end{remark}

By (ii), the hull and the dents of a dented atom are unique. Given a dented atom $D=A-\bigcup_{i=1}^\nu M_i$, we write $\Sigma(D) = \bigcup_i \Gamma(M_i)$, $\Sigma^\icl(D) = \bigcup_i \Gamma^\icl(M_i)$, and $\rho(D) = \rho(\hull(D))$. For notational consistency, we consider every atom as a (trivially) dented atom and define $\hull(A)=A$ for every molecule $A$. When $\hull(D)$ is a cube, $D$ is a \emph{dented cube}.

The main result on dented atoms is the following uniform bilipschitz restoration result. We note that neither the internal geometry of the hull nor the geometry of dents have a r\^ole in the statement. This is a consequence of confining the dents to be in cubes of the hull and the local nature of the construction of the homeomorphism.

\begin{proposition}
\label{prop:fRt_flat}
Suppose $D$ is a dented atom with hull $A$. Then there exists $L=L(n)$ and an $L$-bilipschitz homeomorphism $\phi \colon (D,d_D) \to (A,d_A)$ which is the identity on $D\cap \partial A$. 
\end{proposition}

For the proof, we introduce a useful neighborhood for cubes contained in the dents. Let $Q$ and $q=B_\infty(x_q,r_q)$ be $n$-cubes in $\R^n$ so that $q\subset Q$ and $q$ has a face in $\partial Q$. The set 
\[
\Cone(q,Q)= \{ x\in B_\infty(x_q,(7/6)r_q)\cap Q \colon 2\dist(x,q) \le \dist(x,\partial Q)\}
\]
is the \emph{truncated conical neighborhood of $q$ in $Q$}. 

\begin{figure}[h!]
\includegraphics[scale=0.60]{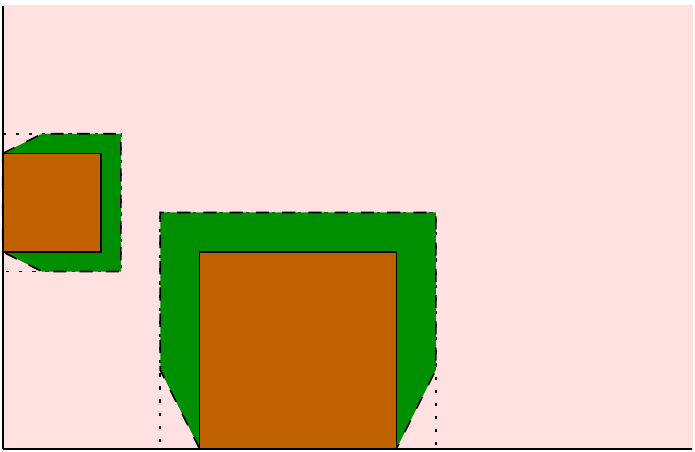}
\caption{Two cubes and their (truncated) conical neighborhoods in a larger cube.}
\label{fig:cones_xyz}
\end{figure}

\begin{lemma}
\label{lemma:local_boundedness}
Let $D = A-\bigcup_i M_i$ be a dented atom in $\R^n$. Then there exists $\mu>0$ depending only on $n$ so that 
\[
\# \{ q' \in \Sigma(D) \colon \Cone(q',Q) \cap \Cone(q,Q) \ne \emptyset\} \le \mu
\]
for all $q\in \Sigma(D)$.
\end{lemma}

\begin{proof}
Let $q$ and $q'$ be pair-wise disjoint $n$-cubes in an $n$-cube $Q$ so that $q$ and $q'$ have a face in $\partial Q$. If either $\rho(q) = \rho(q')$ or $\rho(q)\ge 3\rho(q')$ and $\dist_\infty(q,q')\ge \rho(q')$, then Definitions 3.9 and \ref{def:dRc_A} show that $\Cone(q,Q)\cap \Cone(q',Q)= \emptyset$.

Suppose now that $D = A - \bigcup_i M_i$ is a dented atom, and $n$-cubes $q$ and $q'$ in $\Sigma(D)$ are contained in $Q\in \Gamma(A)$. Then, by definition of dented atom and the first observation, $\Cone(q,Q)\cap \Cone(q',Q)\ne \emptyset$ if and only if $q\cap q'\ne \emptyset$.
Hence it suffices that $\mu$ be larger than the number of neighbors of $q$ of the same side length, so we may take $\mu = 3^n$.
\end{proof}

\begin{remark}
Let $D$ be a dented atom and consider cubes $Q,Q'\in \Gamma(\hull(D))$, $Q\ne Q'$. Then if $q, q' \in \Sigma(D)$ with $q\subset Q$ and $q'\subset Q'$ we have that $\Cone(q,Q)\cap \Cone(q',Q')= \emptyset$. 
\end{remark}

\begin{proof}[Proof of Proposition \ref{prop:fRt_flat}]
Recall that the atom $A$ is the hull of $D$ and that $A-D$ is a pair-wise disjoint union of molecules. The proof is an inductive collapsing of $A-D$ along the forest $\Sigma(D)$ removing leaves one by one. Let $m=\# \Sigma(D)$. 

Let $\Sigma$ be a subforest of $\Sigma(D)$, $q\in \Sigma$ a leaf, $Q\in \Gamma(A)$ be the cube containing $q$, and denote $\Sigma' = \Sigma\setminus \{q\}$. Then there exists a PL homeomorphism $\phi_{\Sigma,q} \colon A-|\Sigma| \to A-|\Sigma'|$ having support in $\Cone(q,Q)$; that is $\phi_{\Sigma,q}(x)=x$ for $x\not \in \Cone(q,Q)$. Clearly, we may take $\phi_{\Sigma,q}$ $L$-bilipschitz with $L$ depending only on $n$.

Using this observation, we find a sequence $\Sigma(D) = \Sigma_0 \supset \cdots \supset \Sigma_m = \emptyset$ of forests and $L$-bilipschitz PL-homeomorphisms $\phi_i \colon A-|\Sigma_{i-1}| \to A-|\Sigma_i|$ having support in the conical neighborhood of the leaf $\Sigma_{i-1}\setminus \Sigma_i$ for every $i=1,\ldots, m$.

Lemma \ref{lemma:local_boundedness} shows that the number of cones over cubes in $\Sigma$ is locally bounded, and thus
\[
\phi = \phi_m\circ \cdots \circ \phi_0 \colon (D,d_D) \to (A,d_A)
\]
is a bilipschitz homeomorphism with a bilipschitz constant depending only on $n$.
\end{proof}

\subsection{Dented molecules}
\label{sec:dented_molecules}

We end this section by defining dented molecules, which relate to dented atoms as molecules relate to atoms.

\begin{definition}
\label{def:prop_adj}
A dented atom $D'$ is \emph{properly adjacent to a dented atom $D$} if $\hull(D')\cup \hull(D)$ is a molecule and either 
\begin{itemize}
\item[(1)] $\hull(D')\subset \hull(D)$ and $D'\cap D = \hull(D')\cap D$, or
\item[(2)] $\hull(D')\cap \hull(D) = D'\cap D$.
\end{itemize}
\end{definition}

Note from (1) that proper adjacency is not a symmetric relation. However, we symmetrize this relation by saying that dented atoms $D$ and $D'$ are \emph{properly adjacent} if $D'$ is properly adjacent to $D$ or $D$ is properly adjacent to $D'$.

Let $\cD$ be a finite collection of dented atoms so that each pair of atoms in $\cD$ is either properly adjacent or pair-wise disjoint. Since dented atoms are $n$-cells, the adjacency tree $\Gamma(\cD)$ is well-defined. Let $U=|\Gamma(\cD)|$ and let $M = \bigcup_{D\in \cD} \hull(D)$. By proper adjacency of the dented atoms, $M$ is a molecule. 

\begin{definition}
\label{def:dented_molecule}
An $n$-cell $U$ is a \emph{dented molecule} if there exists a finite collection $\cD$ of pair-wise properly adjacent dented atoms so that $\Gamma(\cD)$ is a tree with $U=|\Gamma(\cD)|$. The $n$-cell $\hull(U)=\bigcup_{D\in \cD} \hull(D)$ is the \emph{hull of $U$}. The vertex $\hat D\in \cD$ is the \emph{root of $U$} if $\hull(\hat D)$ is the root of $\hull(U)$.
\end{definition}

\begin{figure}[h!]
\includegraphics[scale=0.35]{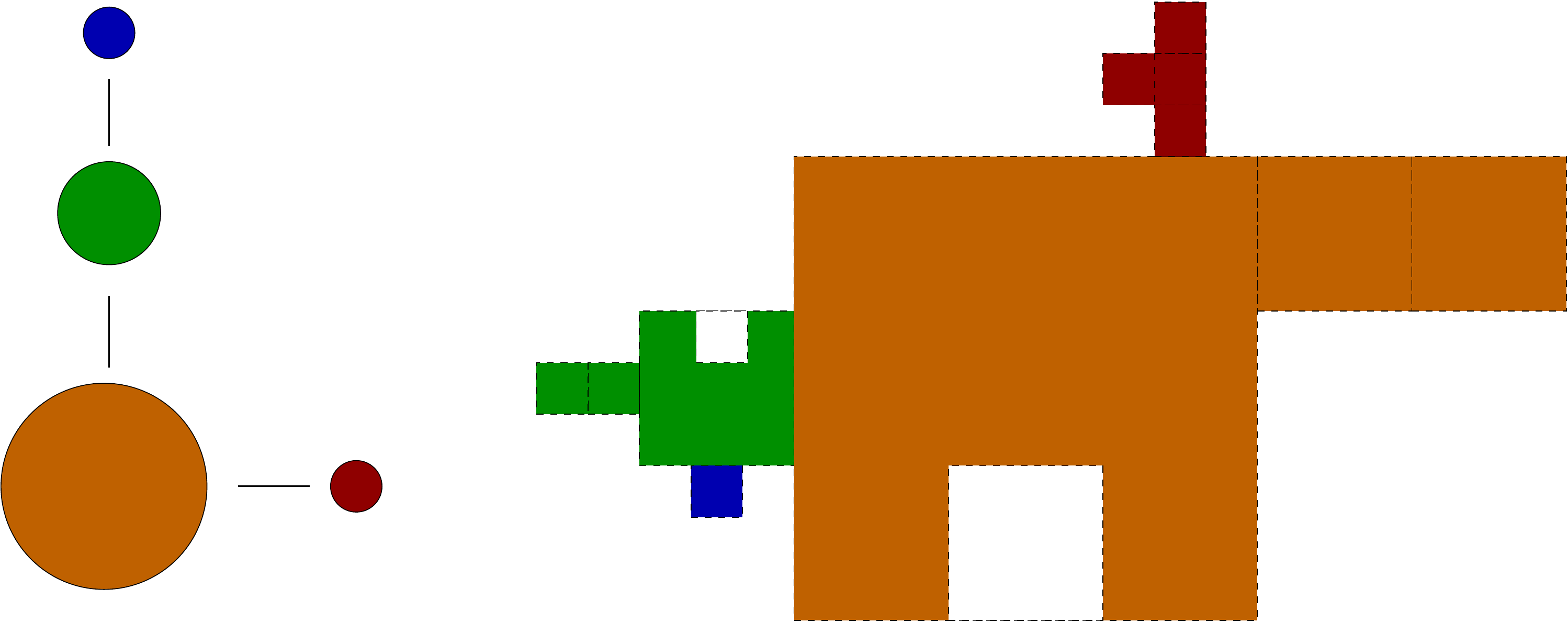}
\caption{A dented molecule $U$ with a tree $\Gamma(U)$.}
\label{fig:RT1}
\end{figure}

\begin{remark} 
Note that, given a dented molecule $U=|\Gamma(\cD)|$, the collection $\cD$ is uniquely determined. We call elements of $\cD$ the \emph{dented atoms of $U$} and  define $\Gamma(U)=\Gamma(\cD)$.
\end{remark}

Let $U$ be a dented molecule. We define internal and external vertices of $\Gamma(U)$ as follows.

\begin{definition}
\label{def:int_ext}
A dented atom $D\in \Gamma(U)$ is \emph{internal} if there exists strictly larger $D'\in \Gamma(U)$ whose hull contains $D$, i.e.\;$\rho(D')>\rho(D)$ and $D\subset \hull(D')$. A dented atom in $\Gamma(U)$ is \emph{external} if it is not internal. Denote by $\Gamma_I(U)$ the set of internal vertices of $\Gamma(U)$ and by $\Gamma_E(U)$ the set of external vertices.
\end{definition}

The motivation for this dichotomy is the following easy observation, which we record as a lemma.
\begin{lemma}
\label{lemma:dichotomy}
Let $U$ be a dented molecule. Then $D\mapsto \hull(D)$ is a tree isomorphism $\Gamma_E(U) \to \Gamma(\hull(U))$. In particular,
\[
\hull(U) = \bigcup_{D\in \Gamma_E(U)} \hull(D).
\]
\end{lemma}

We finish this section by introducing terminology related to dented molecules. Let $D$ be a dented molecule. 
\begin{definition}
\label{def:expanding}
A vertex $d\in \Gamma(D)$ is \emph{expanding in $D$} if the subtree $\Gamma(D)_{d}$ behind $d$ in $\Gamma(D)$ consists of atoms.
\end{definition}

Note that, if $d$ is expanding in $D$ then $d$ is an atom, since $d\in \Gamma(D)_d$.

\begin{definition}
\label{def:ph}
A dented molecule $D'$ is a \emph{partial hull of $D$} if there exist vertices $d_1,\ldots, d_m$ of $\Gamma(D)$ for which
\[
D' = D \cup \bigcup_{k=1}^m \hull(d_k).
\]
\end{definition}

\begin{remark}
In Section \ref{sec:RP} (e.g.\;in Section \ref{sec:IC_init}), we consider a sequence of dented molecules $(U_i)$ for which $\hull(U_i)$ is a $(\nu,\lambda)$-molecule with $\nu$ and $\lambda$ depending only on $n$, although the adjacency tree $\Gamma(U_i)$ no longer has uniformly bounded valence. 

We show there exist $L$-bilipschitz maps $U_i \to \hull(U_i)$ with $L$ depending only on $n$. This proof is based on a sequence of partial hulls from $U_i$ to $\hull(U_i)$.

Since we prove this statement only for particular dented molecules based on notions in the following section, we postpone this statement to Section \ref{sec:RP}. Nevertheless, we invite the interested reader to consider a general statement along the lines of Propositions \ref{prop:fRt} and \ref{prop:fRt_flat}.
\end{remark}


\section{Local rearrangements and the tripod property}
\label{sec:LRA}

In this section we develop tools to produce rough Rickman partitions; recall Section \ref{sec:intro}. Throughout this section we consider different kinds of repartitions in a single cube. These rearrangements are only tangentially related to the final essential partitions introduced in Section \ref{sec:RP}, so the reader may find these constructions unmotivated. Our aim is to simplify these later discussions by introducing these local modifications and their properties here before exploiting them later. Thus the reader should consider this section as preparation for Section \ref{sec:RP}. 

To motivate the r\^ole of our tools, consider the following example. Let $D_1$, $D_2$, and $D_3$ be the cubes $[0,1]^{n-1}\times [0,1]$, $[0,1]^{n-1}\times [-1,0]$, and $[1,2]\times [0,1]^{n-2}\times [0,1]$, respectively, $\mathbf D$ the essential partition $(D_1,D_2,D_3)$ of their union. 
\begin{figure}[h!]
\includegraphics[scale=0.7]{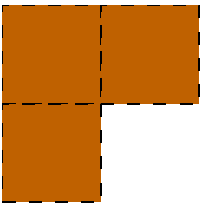}
\caption{Essential partition $\mathbf D$.}
\label{fig:EP_init}
\end{figure}

The Hausdorff distance of the common boundary $\partial_\cap \mathbf D$ and the pair-wise common boundary $\partial_\cup \mathbf D$ satisfy
\begin{equation}
\label{eq:haus_0}
\dist_\haus(\partial_\cup \mathbf D, \partial_\cap \mathbf D) = 1
\end{equation}
in the sup-metric.

Let $k>0$ and consider now the sets $V_i=3^kD_i$ for $i=1,2,3$, and associated essential partition $\mathbf V = (V_1,V_2,V_3)$. Of course, topological properties and bilipschitz equivalence of the cubes remain invariant under this scaling. The Hausdorff-distances in \eqref{eq:haus_0} scale accordingly, and so
\begin{equation}
\label{eq:haus_k}
\dist_\haus(\partial_\cup \mathbf V, \partial_\cap \mathbf V) = 3^k.
\end{equation}

We will show that in this case, as well as in more general situations, there exists an essential partition $\mathbf W = (W_1,W_2,W_3)$ of $\bigcup_i V_i$ into $n$-cells $(W_i,d_{W_i})$ uniformly bilipschitz to $[0,3^k]^n$ with
\begin{equation}
\label{eq:haus_W}
\dist_\haus(\partial_\cup \mathbf W, \partial_\cap \mathbf W) \le 6
\end{equation}
in the sup-metric.

Property \eqref{eq:haus_W} is a consequence of the so-called tripod property, informally mentioned in the introduction, which we now formally define. In further sections, we discuss other structures related to partitions. 

We first need an equivalence relation. Let $U$ be a $3$-fine cubical $n$-set in $\R^n$ and let $U^*$ be a $3$-fine subdivision of $U$. Suppose $\bU = (U_1,U_2,U_3)$ is an essential partition of $U$, and let $(\partial_\cup \bU)^\#$ be the unit subdivision of $\partial_\cup \bU$ as defined in Section \ref{sec:complexes}. Let $\Gamma_\cup(\bU)$ be the subgraph of the adjacency graph $\Gamma((\partial_\cup \bU)^\#)$ composed of vertices of $\Gamma((\partial_\cup \bU)^\#)$ and all edges $\{q,q'\}\in \Gamma((\partial_\cup \bU)^\#)$ for which $q\cup q'\subset U_i\cap U_j$ for a pair $i\ne j$. 
\begin{example}
In the discussion accompanying Figure \ref{fig:EP_init}, $\Gamma_\cup(\mathbf D)$ consists of two vertices $\{ [0,1]^{n-1}\times \{0\}, \{1\}\times [0,1]^{n-1}\}$ and has no edges, whereas $\Gamma_\cup(3^k\mathbf D)$, for $k\ge 1$, is a pair-wise disjoint union of two connected subgraphs. 
\end{example}

\begin{definition}
\label{def:U-eqv}
Cubes $q$ and $q'$ in $(\partial_\cup \bU)^\#$ are \emph{$\bU$-equivalent} if
\begin{itemize}
\item[(a)] $q$ and $q'$ are in the same component of $\Gamma_\cup(\bU)$ and
\item[(b)] $q\cup q'\subset Q$ for some $Q\in U^*$.
\end{itemize}
\end{definition}

Denote by $[q]$ the $\bU$-equivalence class of $q\in (\partial_\cup \bU)^\#$ and by $|[q]|$ the union $\bigcup_{q'\in [q]} q'$. For each pair $(i,j)$, $i\ne j$, the $\bU$-equivalence class $[q]$ of $q\in (\partial_\cup \bU)^\#$ is said to be \emph{between $U_i$ and $U_j$} when $q\subset U_i\cap U_j$.

\begin{remark}
Condition (b) in Definition \ref{def:U-eqv} implies that the equivalence class $[q]$ of $q\in (\partial_\cup \bU)^\#$ has diameter at most $3$ in the sup-metric. Note that equivalence classes are cubical $1$-fine sets of dimension $n-1$, and that the number of $(n-1)$-cubes in $[q]$ is uniformly bounded by a constant depending only on $n$.
\end{remark}

\begin{definition}
\label{def:tripod}
An essential partition $\bU$ of $U$ has the \emph{tripod property} if there exists an essential partition $\Delta$ of $\partial_\cup \bU$ into cubical $(n-1)$-cells satisfying 
\begin{itemize}
\item[($\Delta$1)] each $c\in \Delta$ is contained in a $\bU$-equivalence class, and
\item[($\Delta$2)] to each $c_1\in \Delta$ corresponds a unique pair $c_2,c_3\in \Delta$ for which $c_1\cap c_2\cap c_3$ contains an $(n-2)$-cell in $\partial_\cap \bU$ with  $c_1$, $c_2$, $c_3$  contained in different $\bU$-equivalence classes.
\end{itemize}
\end{definition}

The tripod property of an essential partition is most conveniently verified using the following local tripod property.
\begin{definition}
\label{def:tripod_Q}
Given an essential partition $\bU$ and a cube $Q\subset |\bU|$ of side length at least $3$, we say that $\bU$ \emph{has the tripod property relative to $Q$} if there exists an essential partition $\Delta$ of $Q\cap \partial_\cup \bU$ into $(n-1)$-cells satisfying ($\Delta$1) and ($\Delta$2).
\end{definition}

\begin{example}
\label{ex:tripod}
To give a simple example of an essential partition $\bU$ satisfying the tripod property we consider $\bU=(Q-A,A,Q')$, where $Q=[0,3]^3$, $Q'=[0,3]^2\times [-3,0]$, and $A$ the atom $A=\bigcup_{r=1}^4 q_r$, where $q_r = [r-1,r]\times [1,2]\times [0,1]$ for $r=1,2,3$ and $q_4 = [1,2]\times [2,3]\times [0,1]$; see Figure \ref{fig:block_4}.

\begin{figure}[h!]
\includegraphics[scale=0.6]{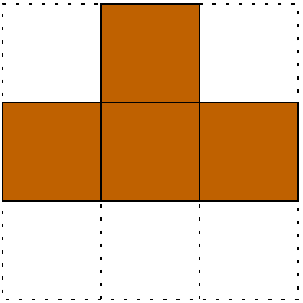}
\caption{Profile of $q_1$, $q_2$, $q_3$, $q_4$ on the face common to $Q$
and $Q'$.}
\label{fig:block_4}
\end{figure}

Note first that $(Q-A)\cap Q'$ has three components $f_1=[0,1]\times [2,3]\times \{0\}$, $f_2=[0,3]\times [0,1]\times \{0\}$, and $f_3=[2,3]\times [2,3]\times \{0\}$, whereas $A\cap (Q-A)$ and $A\cap Q'$ are $2$-cells. We organize the essential partition $\Delta$ of $\partial_\cup \bU$ into three triples $\Delta_1$, $\Delta_2$, and $\Delta_3$ by subdividing cells $A\cap (Q-A)$ and $A\cap Q'$ as follows.

For $r = 1, 3$, we set $\Delta_r = \{ f_r, q_r \cap (Q-A), q_r \cap Q'\}$. Let $\Delta_2 = \{ f_2, (q_2\cup q_4)\cap (Q-A), (q_2\cup q_4)\cap Q'\}$. For each $r$, we directly check that $\Delta_r$ is a triple of $(n-1)$-cells. In addition, $\bigcap_{c \in \Delta_r} c$ is an $(n-2)$-cell for every $r=1,2,3$. Hence $\Delta = \bigcup_{r=1}^3 \Delta_r$ is an essential partition of $\partial_\cup \bU$ satisfying conditions $(\Delta 1)$ and $(\Delta 2)$.
\end{example}

\subsection{Building blocks}
\label{sec:bb}

We introduce the elementary atoms which generate rough Rickman partitions. 

An $(n-1)$-cell $F$ in $\R^n$ is \emph{planar} if $F$ is congruent to an $(n-1)$-cell in $\R^{n-1}$. Suppose $P$ is an $r$-fine $n$-cell and $F$ a planar $(n-1)$-cell. Then $P$  is \emph{$F$-based} if there exists an $(n-1)$-cell $F'$ in $\R^{n-1}$ and a cubical $(n-1)$-cell $P'\subset F'$ so that $P\cup F$ is congruent to $(P'\times [0,r]) \cup F'\subset \R^n$.

Let $T_n=\{0, \pm e_1,\ldots, \pm e_n\}$ and let $\cT_n$ be the graph with vertices $T_n$ and edges $\{0,e_i\}$ and $\{0,-e_i\}$ for $i=1,\ldots, n$. 
\begin{definition}
\label{def:bb}
An atom $A$ is an \emph{($n$-dimensional) building block} if $\Gamma(A)$ is isomorphic to a proper subtree of $\cT_{n-1}$ having at least two vertices.
\end{definition}

\begin{figure}[h!]
\includegraphics[scale=0.6]{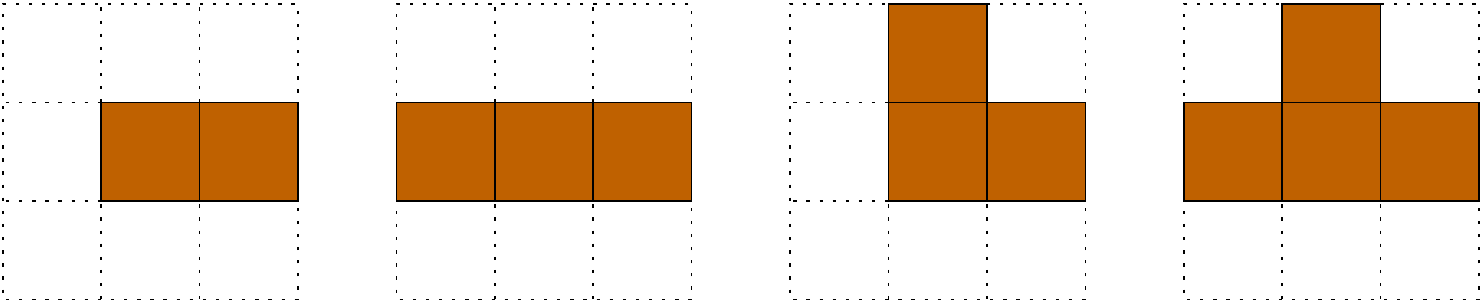}
\caption{Congruence classes of building blocks for $n=3$.}
\label{fig:blocks}
\end{figure}

The fundamental property used in what follows is that an $n$-dimensional building block is an $n$-cell. We record now some observations based on the combinatorial structure of building blocks.

Let $B$ be a building block in $\R^n$. Since $\Gamma(B)$ is a proper subtree of $\cT_{n-1}$, we observe that, for all $q\in \Gamma(B)$, the cubical set $q\cap \partial B$ is an $(n-1)$-cell which induces an essential partition to the faces of $q$, and the adjacency graph $\Gamma(q\cap \partial B)$ of these faces is connected. Moreover, $\Gamma(B)$ has valence less than $2(n-1)$ and contains at most one vertex $q\in \Gamma(B)$ having valence greater than $1$. Further, if $\#\Gamma(B)>2$ there exits a unique $n$-cube $q_B$ in $B$ which is a vertex in $\Gamma(B)$ with valence greater than $1$: this unique cube $q_B$ is the \emph{center of $B$}. 

A building block $B$ in $\R^n$ is \emph{$r$-fine} if $B$ is an $r$-fine atom for some $r>0$. 

Suppose $Q$ is a cube of side length $3r$ containing an $r$-fine building block $B$ along a face $F$ of $Q$. Then, for every cube $q\in \Gamma(B)$, $q\cap F$ is an $(n-1)$-cube and a face of $q$. For the following definition, recall (as in Section \ref{sec:RT}) that a \emph{barycenter of a $k$-cube $C$} is the unique point in $C$ equidistant from all vertices of $C$.

\begin{definition}
Suppose $Q\subset \R^n$ is an $n$-cube of side length $3r$ containing an $F$-based $r$-fine building block $B\subset Q$, where $F$ is a face of $Q$. Let $x_F$ be the barycenter of $F$.  The building block $B$ is \emph{centered in $Q$} if either of the following conditions is satisfied:
\begin{enumerate}
\item if $B$ has a center $q_B$ then $x_F$ is the barycenter of $q_B\cap F$, or
\item if $\#\Gamma(B)=2$, then $\Gamma(B)$ contains the cube $q$ with $x_F$ the barycenter of $q\cap F$. 
\end{enumerate}
\end{definition}

The significance of centered building block is motivated by the following observation. 
\begin{remark}
Let $Q\subset \R^n$ be a cube side length $3$ and $B$ a $1$-fine centered building block contained in $Q$ along the face $F$ of $Q$. Since $B$ is centered, the barycenter $x_F$ of $F$ is the barycenter $x_{f_0}$ of a face $f_0$ of a unique cube $q_0$ in $\Gamma(B)$. Suppose that $q\in\Gamma(B)$ is a cube adjacent to $q_0$. Since $Q$ has side length $3$ and the barycenter of $q_0$ is $x_F$, we have that $q\cap (\partial Q -F)$ is a face of $q$. In particular, the components of $B\cap (\partial Q - F)$ are unit $(n-1)$-cubes, which are in one to one correspondence with cubes in $B-q_0$, cf. Figure \ref{fig:blocks}.
\end{remark}

\begin{convention}
We assume from now on that every $r$-fine building block $B$ in a cube $Q$ is centered and based on a face of $Q$ whenever $Q$ has side length $3r$. We extend the notion of \emph{center} by defining  the unique cube in $B$ containing the barycenter of $F$ on its boundary to be the \emph{center of $B$}.
\end{convention}

Building blocks give rise to a local tripod property of the following form.
\begin{proposition}
\label{prop:bb_tripod}
Let $n\ge 3$, and let $Q$ and $Q'$ be $n$-cubes of side length $3$ with a common face $F=Q\cap Q'$, and let $B$ be an $F$-based building block in $Q$. Then $\bU = (Q-B,B,Q')$ has the tripod property.
\end{proposition}

We begin the proof of Proposition \ref{prop:bb_tripod} with a partition lemma.
\begin{lemma}
\label{lemma:pre_tt}
Let $n\ge 2$, and let $A$ be a $1$-fine atom in $Q=[0,3]^n$ containing the cube $[1,2]^n$ and  $\Gamma(A)$ isomorphic to a subgraph of $\cT_n$, with $1<\# \Gamma(A) \le 2n$. Then $Q - A$ has an essential partition $\cP$ into $n$-cells. Moreover, there exist cubes $\cC_{\cP} = \{q_C\in A^\#\colon C\in \cP\}$ so that $q_C\ne q_{C'}$ for cells $C\ne C'$ in $\cP$ and $q_C\cap C$ contains an $(n-1)$-cube for every $C\in \cP$.  
\end{lemma}

\begin{proof}
In the special case $\#\Gamma(A)=2$, we may take $\cP=\{Q-A\}$ and $\cC_\cP = \{[1,2]^n\}$.

The proof in the general case is by induction on the dimension $n$. The claim clearly holds for $n=2$; consider e.g\;variations of Example \ref{ex:tripod}. Suppose that $n\ge 3$ is a dimension for which the claim holds for $n-1$.

Let $A$ be a $1$-fine atom in $Q=[0,3]^n$ containing $[1,2]^n$ with $\Gamma(A)$ isomorphic to a subtree of $\cT_n$ and $1<\# \Gamma(A)\le 2n$. By rotation, we may assume that $[1,2]^n + e_1 \in \Gamma(A)$. Let $F=[0,3]^{n-1}$. Then $A\cap (F\times [1,2]) = A'\times [1,2]$, where $A'$ is an $(n-1)$-dimensional atom in $F$ where $1<\#\Gamma(A')\le 2(n-1)$. The adjacency graph $\Gamma(A')$ is isomorphic to a subgraph of $\cT_{n-1}$. By induction, $F-A'$ has an essential partition $\cP'$ into $(n-1)$-cells and, for each $C'\in \cP'$, we may fix $q_{C'}\in \cC_{\cP'}\subset (A')^\#$ so that each $C'\cap q_{C'}$ contains an $(n-2)$-cube.

Let $\cP'' = \{ C' \times [0,3] \colon C' \in \cP'\}$. We observe that $Q-(|\cP''|\cup A)$ consists of unit cubes in $(A'\times [0,3]-A)^\#$. It is now easy to find, for each $C' \in \cP'$ a cubical $n$-cell $\Omega_{C'}$ so that $C'\times [0,3]\subset \Omega_{C'}$, $\bigcup_{C'\in \cP'} \Omega_{C'} = Q-A$, and that the sets $\Omega_{C'}$ are pair-wise essentially disjoint. We set $\cP = \{ \Omega_{C'} \colon C'\in \cP'\}$ and $\cC_{\cP} = \{ q_{C'} \times [1,2]\colon C'\in \cP'\}$.
\end{proof}

The following corollary encapsulates the key consequence of Lemma \ref{lemma:pre_tt}.
\begin{corollary}
\label{cor:tt}
Let $n\ge 3$, $Q$ an $n$-cube of side length $3$ and $F$ a face of $Q$. Given an $F$-based building block $B$ in $Q$, the set $F-B$ has an essential partition $\cP$ into cubical $(n-1)$-cells and there exists a collection $\cC_{\cP}=\{q_C\in B^\#\colon C\in \cP\}$ of pair-wise essentially disjoint unit $n$-cubes so that $C\cap q_C$ contains an $(n-2)$-cube for every $C\in \cP$. 
\end{corollary}
\begin{proof}
We may assume $Q=[0,3]^n$ and $F=[0,3]^{n-1}$. Since $F\cap B$ is an $(n-1)$-dimensional atom containing $[1,2]^{n-1}$ and having an adjacency tree isomorphic to a (proper) subtree of $\cT_{n-1}$ with at least two vertices, the claim follows from Lemma \ref{lemma:pre_tt}.
\end{proof}

\begin{proof}[Proof of Proposition \ref{prop:bb_tripod}]
Clearly $\partial_\cup \bU$ consists of $\bU$-equivalence classes $(Q-B)\cap B$, $B\cap Q'$, and $(Q-B)\cap Q'$. The classes $(Q-B)\cap B$ and $B\cap Q'$ are $(n-1)$-cells meeting $\partial_\cap \bU$ in an $(n-2)$-cell. We construct now an essential partition of $\partial_\cup \bU$ into $(n-1)$-cells as required.

Let $\cP$ and $\cC_{\cP}$ be as in Corollary \ref{cor:tt}. Then there exists an essential partition $\{A_C\colon C\in \cP\}$ of $B$ into atoms $A_C$ satisfying $q_C\subset A_C$; consider, for example, the components of the graph $\Gamma(B^\#\setminus \cP)$. For every $C\in \cP$, take $\Delta_C = \{ A_C \cap (Q-B), A_C \cap Q', C\}$. Then $\Delta = \bigcup_{C\in \cP} \Delta_C$ is the required partition of $\partial_\cup \bU$.
\end{proof}

In what follows, Proposition \ref{prop:bb_tripod} is used to verify the tripod property for essential partitions obtained by rearrangements based on building blocks. 


\subsection{Flat (planar) rearrangements}
\label{sec:pc}

Although the notion of atom admits a large variety of possible constructions, we restrict ourselves to only a few basic constructions, all of which appear in this section. These choices yield a double edged sword: we avoid self-intersections and thus preserve the topology of the original essential partition after rearrangement, as a penalty we create neglected faces (discussed in Section \ref{sec:NF}).

In the next two sections we discuss local rearrangements, based on centered building blocks in a single $n$-cube. This section concerns \emph{flat rearrangements}, in that atoms are extended across a single face of a cube. The following section considers the case that atoms are extended across several faces of a cube.

With this objective in mind, we say that an atom $A$, which is a pair-wise essentially disjoint union of building blocks, \emph{consists of building blocks}. Note that planar atoms admit unique partitions into building blocks, but essential partitions of non-planar atoms into building blocks are not unique. Indeed, in each corner where two planar parts of a non-planar atom meet, there are two possible partitions if one of the building blocks consists of two cubes. This ambiguity is, however, not significant in our considerations, since in these cases we may take any possible partition. Keeping this ambiguity in mind, we give the following definition.

\begin{definition}
\label{def:acb}
Given an atom $A$ consisting of builing blocks, we denote by $\widetilde{\Gamma}(A)$ the adjacency graph $\Gamma(\cB)$, where $\cB$ is an essential partition of $A$ into building blocks. We also denote $\ell_{\mathrm{bb}}(A) = \ell(\Gamma(\cB))$.
\end{definition}

Thus, when the essential partition of $A$ into building blocks is clear from the context, we denote this adjacency graph by $\wt \Gamma(A)$. Note that $\wt\Gamma(A)$ is always a tree.  

We consider different cases, starting from simple and heading to more complicated constructions. 

Let $Q$ be an $n$-cube of side length $9$ and $F$ a face of $Q$. We subdivide $Q$ into $3^n$ congruent $n$-cubes of side length $3$, i.e.\;, we consider $Q^*$.
Then $Q^*$ induces a subdivision of $F$ into $3^{n-1}$ congruent $(n-1)$-cubes of side length $3$.  The collection of these $(n-1)$-cubes is $F^*$. Let $\cQ(Q;F)$ be the subset of cubes in $Q^*$ with a face in $F^*$.

\begin{definition}
\label{def:ID_planar}
A quadruple $(Q,F,\cQ'_0,q_0)$ forms \emph{initial data} if 
\begin{itemize}
\item[(a)] $q_0$ is an $n$-cube of side length $3$ so that $q_0 \cap Q$ is a face of $q_0$ and $q_0\cap F$ is an $(n-2)$-cube, and
\item[(b)] $\cQ'_0 \subset \cQ(Q;F)$ is a collection with
\begin{itemize}
\item[(i)] $\Gamma(\cQ'_0)$ connected and
\item[(ii)] $q_0\cap |\cQ'_0| = q_0 \cap Q$. 
\end{itemize}
\end{itemize}
\end{definition}

\begin{figure}[h!]
\includegraphics[scale=0.25]{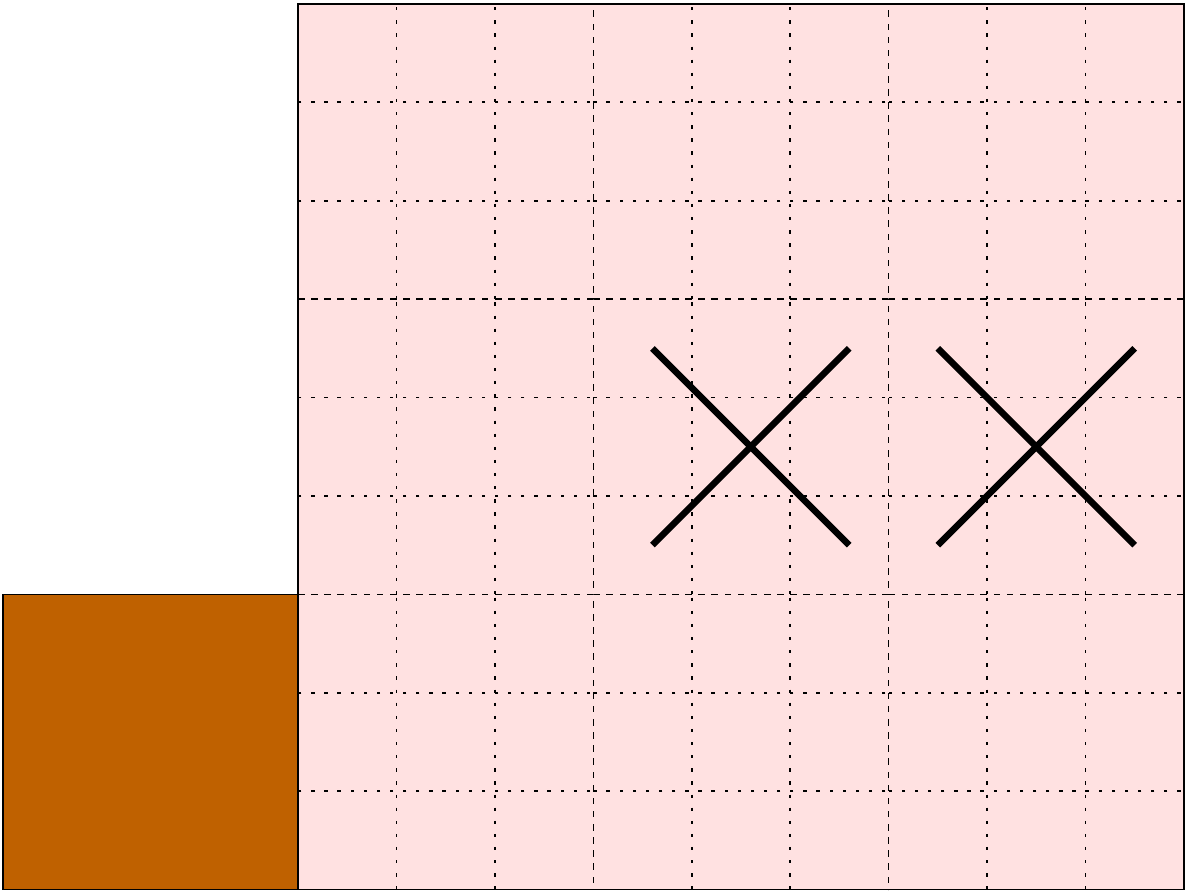}
\caption{An example of an initial data $(Q, F, \cQ'_0, q_0)$. The face $F$ (side length $9$) and cube $q_0$ (side length $3$) are viewed from above, cubes in $\cQ(Q;F)\setminus \cQ'_0$ marked with '\textsf{x}'; $n=3$.}
\label{fig:1face_init}
\end{figure}

\begin{definition}
\label{def:SF_planar}
Let $(Q,F,\cQ'_0,q_0)$ be initial data. A maximal tree $\Gamma\subset \Gamma(\cQ'_0\cup\{q_0\})$ is a \emph{spanning tree associated to this initial data} if $\Gamma$ has valence less than $2(n-1)$.
\end{definition}

The valence bound $2(n-1)$ in Definition \ref{def:SF_planar} was already anticipated in our valence bound for building blocks, recall Definition \ref{def:bb}.

The following simple lemma shows the existence of spanning trees in the configurations we consider here. Let $q_F$ be the unique cube of side length $3$ in $\cQ(Q;F)$ having valence $2(n-1)$ in $\Gamma(\cQ(Q;F))$; thus the barycenter of $q_F\cap F$ is the barycenter of $F$.

\begin{lemma}
\label{lemma:gv}
Suppose $(Q,F,\cQ'_0,q_0)$ forms initial data and $\Gamma(\cQ'_0\setminus \{q_F\})$ is connected. Then there exists a spanning tree $\Gamma\subset \Gamma(\cQ'_0)$.
\end{lemma}
\begin{proof}
Let $\Gamma'$ be a maximal tree in $\Gamma(\cQ'_0\setminus \{q_F\})$. Since $\Gamma(\cQ'_0\setminus \{q_F\}) \subset \Gamma(\cQ(Q;F))$ and $q_F$ is the unique vertex in $\Gamma(\cQ(Q;F))$ having valence $2(n-1)$, $\Gamma'$ is a spanning tree of $\Gamma(\cQ'_0\setminus \{q_F\})$. If $q_F\not \in \cQ'_0$, we may take $\Gamma=\Gamma'$.

If $q_F\in \cQ'_0$, let $q'\in \Gamma(\cQ'_0)$ be a vertex adjacent to $q_F$. We extend $\Gamma'$ to a tree $\Gamma$ containing $q_F$ by adding the edge $\{q',q_F\}$. Since the valence of $q'$ in $\Gamma'$ is less than $2(n-1)-1$, the claim follows.
\end{proof}

Spanning trees repartition $Q$ using atoms. 
 
\begin{lemma}
\label{lemma:ra_1}
Given initial data $(Q,F,\cQ'_0,q_0)$ and a spanning tree $\Gamma$, there exists a $1$-fine atom $A_\Gamma$ in $Q$ with the following properties:
\begin{itemize}
\item[(1)] $A_\Gamma\cap q'$ is an $F$-based building block for every $q'\in \cQ'_0$, 
\item[(2)] the adjacency graph $\wt\Gamma(A_\Gamma)$ of building blocks is $\Gamma\setminus \{q_0\}$,
\item[(3)] $A_\Gamma\cup q_0$ is an $n$-cell, and
\item[(4)] $A_\Gamma\cap \partial Q \subset F\cup q_0$.
\end{itemize}
\end{lemma}

We call $A_\Gamma$ the \emph{(unique) atom associated with spanning tree $\Gamma$ (and initial data $(Q,F,\cQ'_0,q_0)$)}.

\begin{figure}[h!]
\includegraphics[scale=0.25]{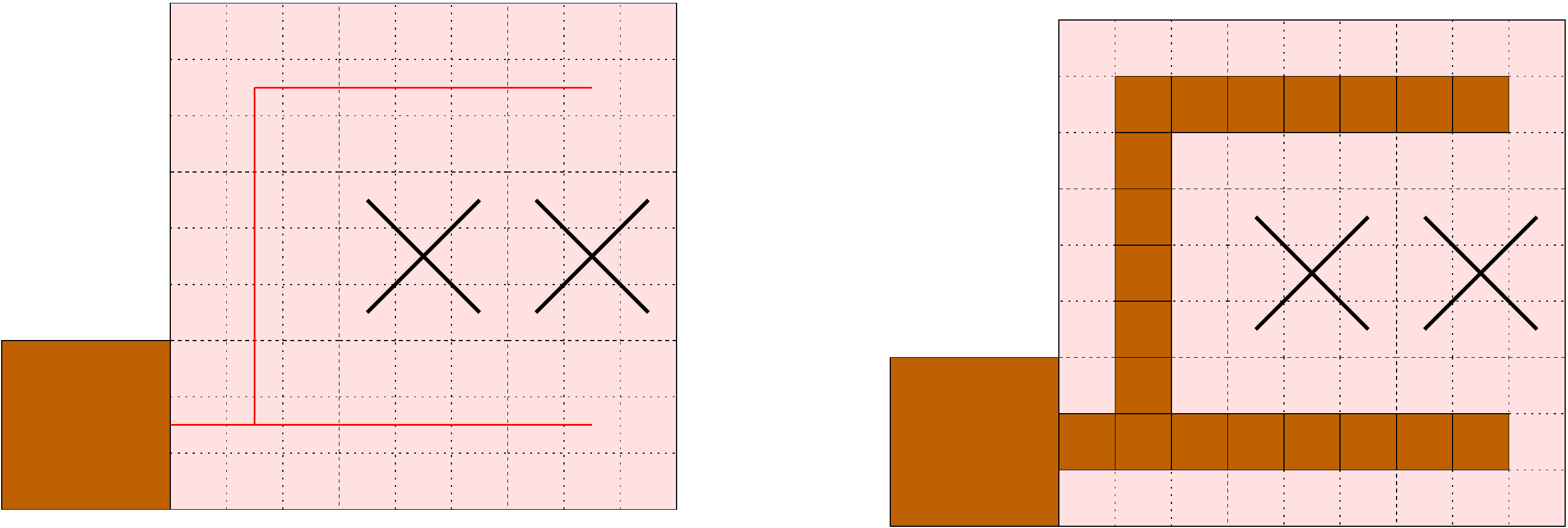}
\caption{A spanning tree (left) and the corresponding atom (right) associated to the initial data in Figure \ref{fig:1face_init}.}
\label{fig:1face}
\end{figure}

\begin{remark}
\label{rmk:ra_1}
Note that the atom $A_\Gamma$ in Lemma \ref{lemma:ra_1} is on the boundary of $Q$ as defined in Section \ref{sec:dented_atoms}. Thus $Q-A_\Gamma$ is a dented cube and, in particular, an $n$-cell.
\end{remark}

\begin{proof}[Proof of Lemma \ref{lemma:ra_1}]
To obtain the building blocks, we make the following observation. 

Suppose $q'\in \Gamma$ is a vertex other than $q_0$. Let $\Gamma_{q'}$ be the star of $q'$ in $\Gamma$, that is, the subgraph of $\Gamma$ containing only edges connecting to $q'$ and all vertices on these edges. We denote $E_{q'}=|\Gamma_{q'}|$. Then $E_{q'}$ is a building block.

To each $q'\in \cQ'_0$ corresponds a unique $F$-based centered building block $B_{q'}\subset q'$ which is a translation of $(1/3)E_{q'}$. These building blocks form an essential partition of the atom $A_\Gamma=\bigcup_{q'\in \cQ_0'} B_{q'}$, whose adjacency graph $\tilde \Gamma(A_\Gamma) = \Gamma(\{B_{q'}\colon q'\in \cQ'_0\})$ is isomorphic to $\Gamma$. 

Conditions (1), (2), and (4) are clearly satisfied by the construction. Since $\Gamma$ is a tree, $A_\Gamma$ is an atom. Since $q_0$ is a leaf in $\Gamma$ and $A_\Gamma\cap q_0$ is an $(n-1)$-cube, $A_\Gamma \cup q_0$ is an $n$-cell and (3) holds.
\end{proof}

Atoms associated to initial data and spanning trees immediately yield a local tripod property. 
\begin{lemma}
\label{lemma:tripod_planar}
Let $Q$ and $Q'$ be $n$-cubes of side length $9$ sharing the face $F$. Suppose $(Q,F,\cQ(Q;F),q_0)$ forms initial data with spanning tree $\Gamma$. Let $A_\Gamma$ be the atom associated to $\Gamma$ and $(Q,F,\cQ(Q;F),q_0)$. Then the essential partition $\bU=(Q-A_\Gamma,A_\Gamma,Q')$ of $Q\cup Q'$ has the tripod property.
\end{lemma}
\begin{proof}
Let $q$ be a cube in $\cQ(Q;F)$ and let $q_-$ be the unique cube in $Q'$ sharing a face with $q$. Denote by $B_q$ the building block $q \cap A_\Gamma$. By Proposition \ref{prop:bb_tripod}, $(q-B_q, B_q, q_-)$ satisfies the tripod property. Let $\Delta_q$ be an essential partition of $(\partial_\cup \bU) \cap q$ as in Definition \ref{def:tripod}. Since $\cQ(Q;F)$ is an essential partition of a cubical set having $\partial_\cup \bU$ (essentially) in its interior, $\Delta=\bigcup_{q\in \cQ(Q;F)} \Delta_q$ is a required essential partition of $\partial_\cup \bU$. 
\end{proof}

More generally, we may consider initial data $(Q,F,\cQ'_0,q_0)$, where $q_0\in \cQ(Q;F)$; this means $q_0\subset Q$ with $q_0\cap F$ a face of $q_0$. Initial data of this type is called \emph{internal initial data}. This notion of initial data is especially useful for extending a $3$-fine building block inside a cube of side length $9$. We formulate now this rearrangement procedure. 

\begin{corollary}
\label{cor:ra_1}
Let $Q$ be a cube of side length $9$, $F$ a face of $Q$. 
Let also $q_1,\ldots,q_p$ be pair-wise essentially disjoint cubes in $\cQ(Q;F)$. 
Suppose, for $1\le r \le p$, each $(Q,F,\cQ'_r,q_r)$ forms internal initial data with $\cQ'_r \subset \cQ(Q;F)$ and $\cQ'_t \cap \cQ'_s = \emptyset$ for $t\ne s$. Suppose $\Gamma_1,\ldots, \Gamma_p$, respectively, are spanning trees for these initial data. Then there exist pair-wise disjoint $1$-fine atoms $A_r$ associated to initial data $(Q,F,\cQ'_r,q_r)$ for $r=1,\ldots, p$. 
\end{corollary}

\begin{figure}[h!]
\includegraphics[scale=0.25]{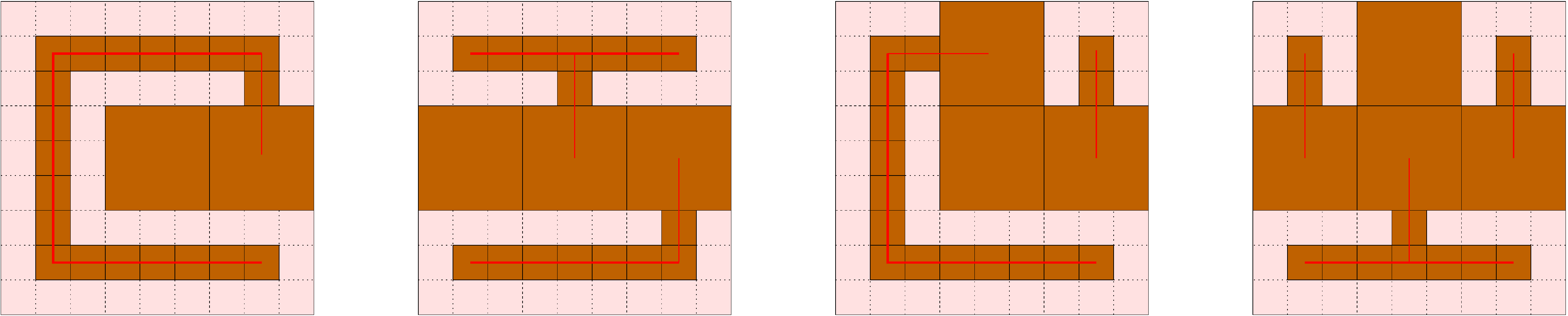}
\caption{Some examples of atoms $A_r$ for $r=1,\ldots, p$ each associated to an internal initial data; here $p=1,2,2,3$.}
\label{fig:1face3_molecule}
\end{figure}

It is easy to obtain a local tripod property for these repartitions. We leave the details, similar to those of the proof of Lemma \ref{lemma:tripod_planar}, to the interested reader.  
\begin{corollary}
\label{cor:tripod_planar}
Let $Q$ and $Q'$ be $n$-cubes of side length $9$ sharing the face $F$, and suppose that, for each $1\le r\le p$, $(Q,F,\cQ'_r,q_r)$ forms internal initial data as in Corollary \ref{cor:ra_1} so that in addition
\[
B:=|\cQ(Q;F)| - \bigcup_{r=1}^p |\cQ'_r|
\]
is a building block of side length $3$. For each $1\le r\le p$ let $\Gamma_r$ be a spanning tree for $(Q,F,\cQ'_r,q_r)$,  associate an atom $A_r $ to $\Gamma_r$ as in Corollary \ref{cor:ra_1} and define $A$ as the (disjoint) union of the atoms $A_r$. Then the essential partition
\[
\bU=(Q-(B\cup A), B\cup A, Q')
\]
of $Q\cup Q'$ has the tripod property.
\end{corollary}

\begin{convention}
Henceforth we do not differentiate between initial data and internal initial data, and refer to both as initial data.
\end{convention}

\subsection{Non-flat (non-planar) rearrangements}
\label{sec:npc}

We consider now local rearrangements in the non-flat case. For our purposes it suffices to consider rearrangements which occur in a single cube.

Let $Q$ be an $n$-cube of side length $9$ and $\cF$ a subset of the collection of all faces of $Q$. Let $\cF$ be partitioned into sets $\cF^1$ and $\cF^2$ so that $|\cF^r|$ is an $(n-1)$-cell for $r=1,2$. Note that each $|\cF^r|$, in particular, is a union of faces of $Q$.

Let $\cQ(Q;\cF)\subset Q^*$ be the cubes having a face in $|\cF|$; we denote by $\cQ(Q;\cF^r)\subset \cQ(Q;\cF)$ those with a face in $|\cF^r|$. Note that $\{\cQ(Q;\cF^1),\cQ(Q;\cF^2)\}$ is not (necessarily) a partition of $\cQ(Q;\cF)$. The following definition generalizes Definition \ref{def:ID_planar}.

\begin{definition}
\label{def:nfid}
A triple 
\[
(Q, (\cF^1,\cQ''_1,q_1), (\cF^2,\cQ''_2,q_2))
\]
forms \emph{non-flat initial data} if the following conditions are satisfied:
\begin{itemize}
\item[(a)] for every $r=1,2$, $q_r\subset \R^n-Q$ is an $n$-cube of side length $3$ with $Q\cap q_r$ a face of $q_r$ and $q_r \cap |\cF^r|$ an $(n-2)$-cube. 
\item[(b)] $\{\cQ''_1,\cQ''_2\}$ is a partition of $\cQ(Q;\cF)$ and for $r=1,2$ satisfies  
\begin{itemize}
\item[(0)] $\cQ''_r \subset \cQ(Q;\cF^r)$,
\item[(1)] $\Gamma(\cQ''_r)$ is connected, 
\item[(2)] $q_r\cap |\cQ''_r|$ is a face of $q_r$, and
\item[(3)] $q_r \cap |\cF^r|\cap |\cQ''_r|$ is an $(n-2)$-cube.
\end{itemize}
\end{itemize}
\end{definition}

\begin{remark}
Let $(Q,(\cF^1,\cQ''_1,q_1), (\cF^2,\cQ''_2,q_2))$ be as in Definition \ref{def:nfid}, and let $q_1^+$ and $q_2^+$ be the $n$-cubes in $Q^*$ sharing a face with $q_1$ and $q_2$, respectively. Since $q_1\cap Q$ is a face of $q_1$, condition (2) in (b) shows that $q_1^+\in \cQ''_1$. Clearly, the same argument holds for $q_2$ and we also have $q_2^+\in \cQ''_2$.
\end{remark}

Let $r\in \{1,2\}$, $\wh\Gamma \subset \Gamma(\cF^r)$ be a maximal tree and $\cQ'\subset \cQ(Q;\cF)\cup\{q_r\}$. A subgraph $\Gamma\subset \Gamma(\cQ')$ is \emph{dominated by $\wh\Gamma$} if, for each vertex $q\in \Gamma$ and the star $\Gamma_q$ of $q$ in $\Gamma$, either there exists a vertex $F_q\in \Gamma(\cF^r)$ satisfying $\Gamma_q \setminus \{q_r\} \subset \cQ(Q;F_q)$ or there exists an edge $\{F_q,F'_q\} \in \wh\Gamma$ satisfying $\Gamma_q\setminus\{q_r\}\subset \cQ(Q;F_q)\cup \cQ(Q;F'_q)$.

\begin{definition}
\label{def:sf}
Let $(Q, (\cF^1,\cQ''_1,q_1), (\cF^2,\cQ''_2,q_2))$ form non-planar initial data. A maximal forest $\Sigma=\Gamma_1\cup \Gamma_2 \subset \Gamma(\cQ''\cup \{q_1,q_2\})$ is a \emph{spanning forest associated to this data} if 
\begin{itemize}
\item[(i)] $\Sigma$ has valence less than $2(n-1)$,
\item[(ii)] for $r=1,2$, $\Gamma_r$ is a maximal tree in $\Gamma(\cQ''_r\cup \{q_r\})$ dominated by a maximal tree of $\Gamma(\cF^r)$.
\end{itemize}
\end{definition}

The proof of the following existence result for spanning forests is analogous to Lemma \ref{lemma:gv}, and we omit the details. Let $\cQ'_c(Q;\cF)$ be the collection of all cubes in $\cQ(Q;\cF)$ having valence $2(n-1)$. 
\begin{lemma}
\label{lemma:gv_non-flat}
Suppose $(Q, (\cF^1,\cQ''_1,q_1), (\cF^2,\cQ''_2,q_2))$ forms non-planar initial data for which $\Gamma(\cQ''_r\setminus \cQ_c(Q;\cF))$ is connected for $r=1,2$. Then there exists a spanning forest $\Sigma$ associated to $(Q, (\cF^1,\cQ''_1,q_1), (\cF^2,\cQ''_2,q_2))$. 
\end{lemma}

\begin{figure}[h!]
\includegraphics[scale=0.24]{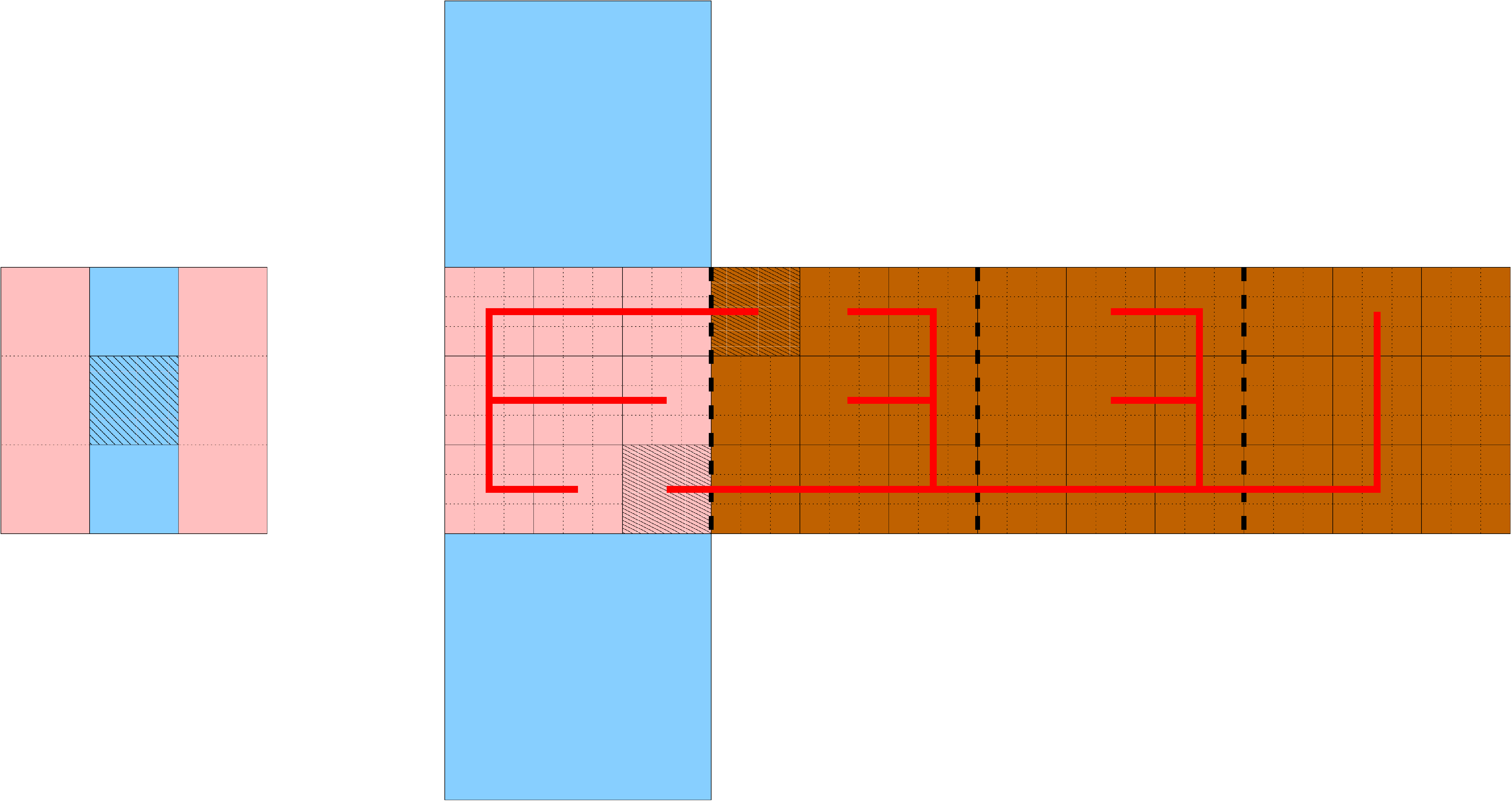}
\caption{A spanning forest on four faces $\cF$ of a cube $Q$; note that the forest enters each cube $q$ in $\cQ(Q;\cF)$. Here $Q$ is the center cube in a building block consisting of $3$ cubes (left figure).}
\label{fig:ra_non_flat_1}
\end{figure}

\begin{lemma}
\label{lemma:ra_2}
Let $(Q, (\cF^1,\cQ''_1,q_1), (\cF^2,\cQ''_2,q_2))$ form non-planar initial data, and let $\Sigma = \Gamma_1\cup \Gamma_2 \subset \Gamma(\cQ'\cup \{q_1, q_2\})$ be a spanning forest.
 
Then there exist a $1$-fine cubical set $A_\Sigma$ in $Q$ composed of pair-wise disjoint $1$-fine atoms $A_1$ and $A_2$ and for $r=1,2$ satisfying the following properties:
\begin{itemize}
\item[(1)] each $A_r$ is composed of building blocks,
\item[(2)] for every $q''\in \cQ''_r$, $A_r \cap q''$ is an atom having an essential partition into at most two building blocks,
\item[(3)] every building block in $A_r$ is $F$-based with $F\in \cF^r$,
\item[(4)] $A_r\cup q_r$ is an $n$-cell, 
\item[(5)] $A_r \cap \partial Q \subset |\cF^r|\cup q_r$, and
\item[(6)] the adjacency graph of cells $\{ A_r \cap Q'' \colon Q''\in \cQ''_r\}$ is isomorphic to $\Gamma_r$.
\end{itemize}
\end{lemma}

The set $A_\Sigma$ in Lemma \ref{lemma:ra_2} is said to be \emph{associated to this initial data and the spanning forest $\Sigma$}. Property (2) is a consequence of the trees $\Gamma_1$ and $\Gamma_2$ being dominated by $\Gamma(\cF^1)$ and $\Gamma(\cF^2)$ respectively. Property (3) asserts that $A_\Sigma$ is \emph{on the boundary of $Q$}.

\begin{figure}[h!]
\includegraphics[scale=0.24]{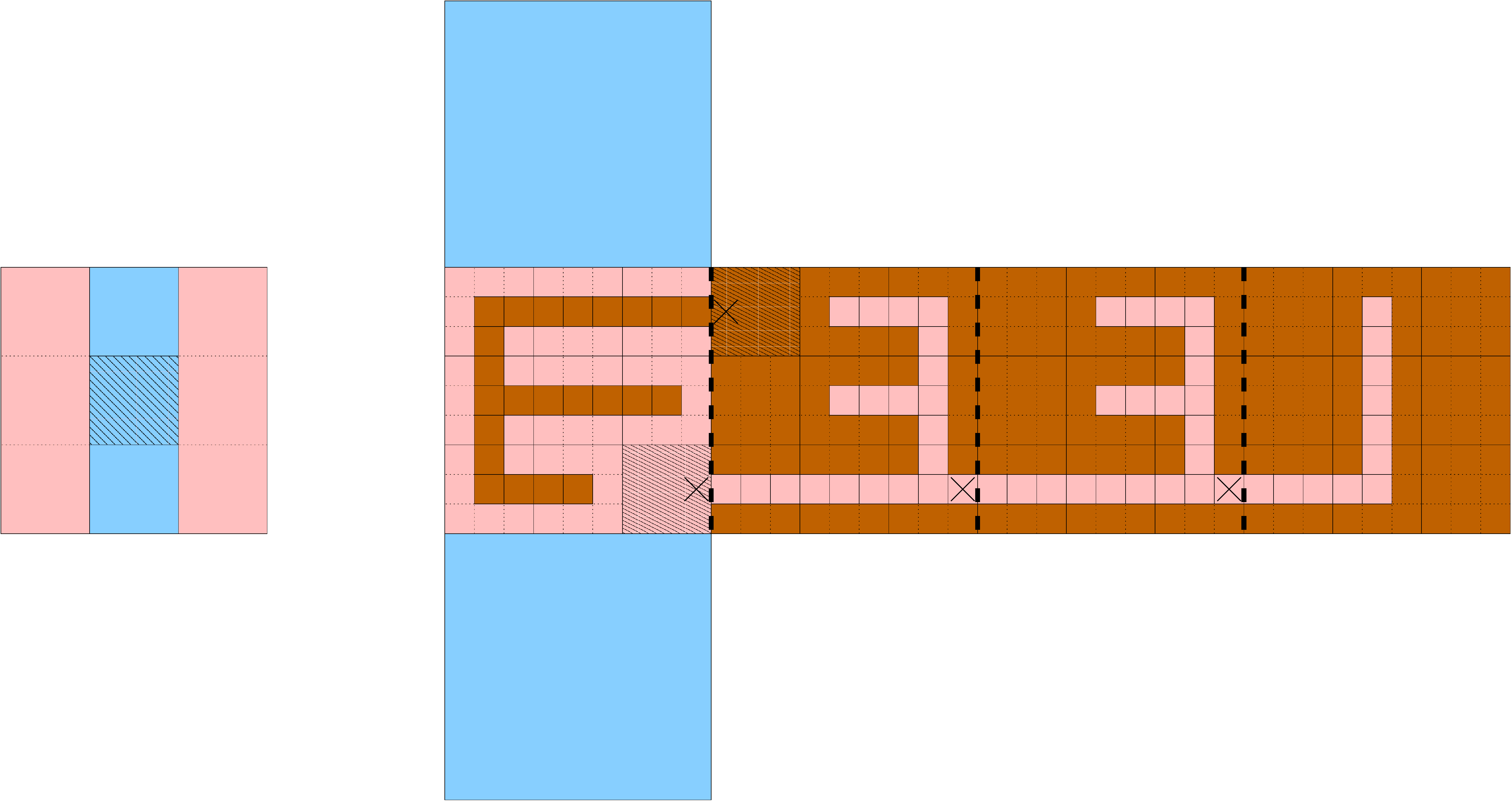}
\caption{Atoms $A_1$ and $A_2$ associated to the initial data in Figure \ref{fig:ra_non_flat_1}.}
\label{fig:ra_non_flat_2}
\end{figure}

\begin{remark}
As in Remark \ref{rmk:ra_1}, the components $A_1$ and $A_2$ of $A_\Sigma$ in Lemma \ref{lemma:ra_2} are atoms on the boundary of $Q$. In particularly, $Q-A_\Sigma$ is a dented cube.
\end{remark}

\begin{proof}[Proof of Lemma \ref{lemma:ra_2}]
Consider first the tree $\Gamma_1$. Let $q'\in\Gamma_1$ be an $F$-based cube, where $F\in \cF^1$, and let $\Gamma_{q'}$ be the star of $q'$ in $\Gamma_1$. 

If $|\Gamma_{q'}|$ is $F$-based, we fix a building block $B_{q'}$ as in Lemma \ref{lemma:ra_1}. Suppose, however, that $|\Gamma_{q'}|$ is not $F$-based. Then, by (ii) in Definition \ref{def:sf}, there exists a face $F'\in \cF^1$ so that each cube in $\Gamma_{q'}$ is either $F$-based or $F'$-based. Thus there exist an $F$-based building block $B_F$ and an $F'$-based building block $B_{F'}$ in $q'$ with the following properties:
\begin{itemize}
\item $B_F\cap B_{F'}$ is an $(n-1)$-cube and
\item $B_F\cup B_{F'}$ meeting the neighbors of $q'$ in $\Gamma_1$ in $(n-1)$-cubes. 
\end{itemize}

In this case, we take $B_{q'}=B_F\cup B_{F'}$, and define $A_1 = \bigcup_{q'\in \Gamma} B_{q'}$. The atom $A_2$ is defined similarly. It is easy to check that atoms $A_1$ and $A_2$ satisfy properties (1)-(5).
\end{proof}

These non-planar rearrangements satisfy the tripod property. 
\begin{lemma}
\label{lemma:tripod_non_planar}
 
Let $\bU=(U_1,U_2,U_3)$ be an essential partition and $Q\subset U_3$ an $n$-cube of side length $9$ sharing a face with both $U_1$ and $U_2$. Let 
\[
(Q, (\cF^1,\cQ''_1,q_1), (\cF^2,\cQ''_2,q_2))
\]
form non-planar initial data for which  
\begin{itemize}
\item[(i)] $q_r \subset U_r$ for $r=1,2$, 
\item[(ii)] $|\cF^r|\subset Q\cap U_{j_r}$,  where $\{j_r,r\}=\{1,2\}$, and
\item[(iii)] $|\cF^1|\cup |\cF^2| = Q\cap \partial_\cup \bU$.
\end{itemize}
Let $\Sigma$ be a spanning forest for this initial data and let $A_\Sigma = A_1\cup A_2$ be the union atoms associated to this initial data and spanning forest. 

Then the essential partition
\[
\bV=\left(U_1\cup A_1, U_2 \cup A_2, U_3 - A_\Sigma \right)
\]
has the tripod property in $Q$.
\end{lemma}

\begin{proof}
It suffices to verify that $\partial_\cup \bV$ satisfies the tripod property in every cube in $\cQ(Q;\cF)$.

Let $q\in \cQ(Q;\cF)$. We consider two cases. Suppose first that $b=q\cap A_\Sigma$ is a building block, with $A_\Sigma$ from Lemma \ref{lemma:ra_2}. Let $q'$ be the unique $n$-cube in $U_2\cup U_3$ sharing a side with $q$. By Proposition \ref{prop:bb_tripod}, the essential partition $(q-b,b,q')$ of $q\cup q'$ satisfies the tripod property.

Suppose next that $A=q\cap A_\Sigma$ has an essential partition into two building blocks, say $b_1$ and $b_2$. By (ii), there are exactly two $n$-cubes $q_1$ and $q_2$ in $U_2\cup U_3$ sharing a side with $q$. Let $f_1 = q\cap q_1$ and $f_2 = q\cap q_2$. By relabeling, we may assume that $b_r$ is $f_r$-based for $r=1,2$. Since the building blocks $b_1$ and $b_2$ are centered and do not contain common $n$-cubes, we may assume, by relabeling again if necessary, that $b_2 \cap f_1 = \emptyset$. Since $b_1\cup b_2$ is connected, it follows that $c_{bf}=b_1 \cap f_2$ must be an $(n-1)$-cube. We also note that the set $c_{bb}=(\partial b_1) \cap b_2$ is a unit $(n-1)$-cube and $(\partial b_1) \cap b_2 = b_1 \cap (\partial b_2)$. Define $E_1 = (\partial_\cup (q,q-b_1,q_1)-c_{bb})\cup c_{bf}$ and $E_2 = \partial_\cup (q,q-b_2,q_2)-(c_{bb} \cup c_{bf})$.

Thus, by elementary modifications to the proof of Proposition \ref{prop:bb_tripod}, there exists, for $r=1,2$, an essential partition $\Delta_r$ of $E_r$ satisfying the conditions of Definition \ref{def:tripod}, so that $\Delta = \Delta_1 \cup \Delta_2$ is an essential partition of $\partial_\cup (q,q-A,q_1\cup q_2)$ satisfying the conditions of Definition \ref{def:tripod}. The claim follows.
\end{proof}

\subsection{Neglected faces in $\cQ(Q;\cF)$}
\label{sec:NF}

We finish this section by a slight modification of our analysis for non-flat initial data. This is to compensate for the fact that while the spanning forest contains every subcube $q\in \cQ(Q;\cF)$, some cubes $q$ will have faces, contained in $\partial Q$, disjoint from atoms in $A_\Sigma = A_1\cup A_2$. For example, consider Figure \ref{fig:ra_non_flat_2}. It is easy to find a cube $q$ in $\cQ(Q;\cF)$ which meets more faces of $\partial Q$ than $q\cap A_\Sigma$. Such cubes $q$ are only of side length $3$, but this will create a problem in satisfying the tripod property when, in Section \ref{sec:RP}, we scale these configurations, and so preparations are given here. We make a formal definition.

\begin{definition}
\label{def:neglected_face}
Let $(Q, (\cF^1,\cQ''_1,q_1), (\cF^2,\cQ''_2,q_2))$ form non-flat initial data, $\Sigma$ be a spanning forest, and let $A_\Sigma=A_1 \cup A_2$ be the cubical set associated to $\Sigma$ from Lemma \ref{lemma:ra_2}. 
A cube $q\in \cQ(Q;\cF)$ has an $A_\Sigma$-neglected face if $q$ has more faces contained in $\partial Q$ than $q\cap A_\Sigma$ has building blocks.
\end{definition}

\begin{remark}
Note that, for each $q\in \cQ(Q,\cF)$, $q\cap A_\Sigma$ is either a building block or a union of two building blocks.
\end{remark}

Let $\cN(Q;A_\Sigma)$ denote the collection of all $A_\Sigma$-neglected faces in cubes in $\cQ(Q;\cF)$. 

\begin{definition}
\label{def:flat_ext}
Suppose $q\in \cQ(Q;\cF)$ has an $A_\Sigma$-neglected face $f$ and let $p\in \{1,2\}$ be such that $f\subset |\cF^p|$.  Then $f$ \emph{admits a flat extension of $A_\Sigma$} if there exists $q'\in \cQ(Q;\cF)$ adjacent to $q$ and a face $f'$ of $q'$ contained in $|\cF^p|$ so that $q'\cap A_p$ contains an $f'$-based atom and $f \cap f'$ is an $(n-2)$-cube. We call $f'$ a \emph{link of $A_\Sigma$ into $f$}.
\end{definition}

To motivate this terminology, consider a cube $q\in \cQ(Q,\cF)$ having a neglected face $f$ and let $q'\in\cQ(Q;\cF)$ be the cube adjacent to $q$ as in Definition \ref{def:flat_ext}. Then $q\cap A_\Sigma = q\cap A_r$ and $q'\cap A_\Sigma = q\cap A_p$, where $\{r,p\}=\{1,2\}$. Moreover, both cubes $q$ and $q'$ are $F$-based for $F\in \cF^p$. Thus using a flat rearrangement, the atom $q'\cap A_\Sigma$ may be extended to a molecule by adding an atom which enters the cube $q$ and is $f\cup f'$ based. This heuristics is made precise in Section \ref{sec:RP}.

Note that, in Figure \ref{fig:ra_non_flat_2}, all neglected faces admit a flat extension of $A_\Sigma$. In general this is, however, not the case; see Figure \ref{fig:BB_C_mods_small}. For this reason, we partition the neglected faces into collections, called pre-basins, so that each collection contains at least one neglected face admitting a flat extension. Note that pre-basins are always flat, in the sense that each pre-basin is contained in a single face in $\cF^1\cup \cF^2$.

Let $\cN_\ecl(Q;A_\Sigma)$ be the collection of all faces in $\cN(Q;A_\Sigma)$ admitting a flat extension of $A_\Sigma$.
\begin{definition}
Given $p\in \{1,2\}$, a collection $C\subset \cN(Q;A_\Sigma)$ is a \emph{pre-basin on $|\cF^p|$} if
\begin{itemize}
\item[(PB1)] $|C|\subset F$ for some $F\in \cF^p$,
\item[(PB2)] $\Gamma(C)$ is connected, and 
\item[(PB3)] $C \cap \cN_\ecl(Q;A_\Sigma)\ne \emptyset$,
\end{itemize}
\end{definition}

\begin{remark}
It is easy to observe that the components of the graph $\cN(Q;A_\Sigma)$ are pre-basins. Indeed, given a component $C\subset \cN(Q;A_\Sigma)$, by definitions of spanning forest and connected component, there exists a pair $\{f,f'\}$ where $f\in C$ and $f'$ is a link of $A_\Sigma$ to $f$.
\end{remark}

This formulation of pre-basins is sufficient for all forthcoming constructions in dimensions $n>3$. In Section \ref{sec:IC_dim3}, when $n=3$, we will also need to subdivide pre-basins. We formalize this with the notion of system of basins; however this procedure is (quite) general and need not be restricted only to dimension $n=3$. Note that, whereas a pre-basin always consists of neglected faces, a basin need not contain a neglected face; see Figure \ref{fig:ra_non_flat_2_basins}.

Given a pre-basin $C\subset \cN(Q;A_\Sigma)$, we introduce a cell $\sigma_C$, called a \emph{connecting cell}, as follows. By (PB3), we may fix $f_C\in C\cap \cN_\ecl(Q;A_\Sigma)$. Let $f'_C$ be a link into $f_C$, and let $q_C$ and $q'_C$ denote the unique cubes in $\cQ(Q;\cF)$ having $f_C$ and $f'_C$ as faces, respectively. Let $\sigma_C$ be the connected component of $f'_C-A_\Sigma$ meeting $f_C$ in an $(n-2)$-cell, and set $\Omega_C = |C|\cup \sigma_C$. The cell $\Omega_C$ is called an \emph{extension of $|C|$ to $f'_C$}.

\begin{figure}[h!]
\includegraphics[scale=0.24]{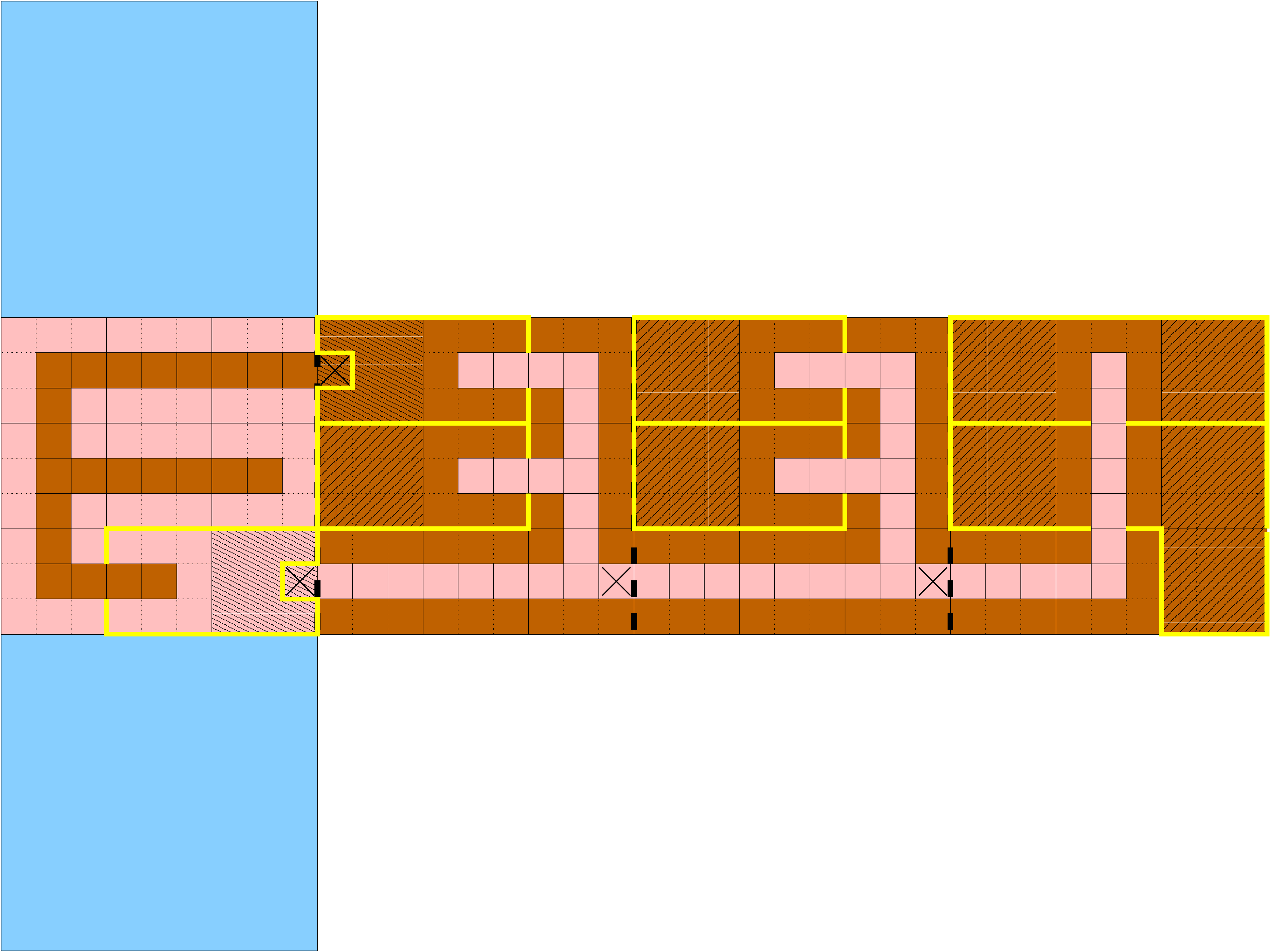}
\caption{Extended pre-basins for a partition of $\cN(Q;A_\Sigma)$ into $8$ pre-basins given the data in Figure \ref{fig:ra_non_flat_2}; the neglected faces are shaded.}
\label{fig:ra_non_flat_2_pre_basins}
\end{figure}

Let $\fP$ be a partition of $\cN(Q;A_\Sigma)$ into pre-basins, and suppose we have fixed, for each $C\in \fP$, an extension $\Omega_C$ of $|C|$; see Figure \ref{fig:ra_non_flat_2_pre_basins}. Let $\Omega_{\fP} = \bigcup_{C\in \fP} \Omega_C$.

\begin{definition}
An essential partition $\cB$ of $\Omega_{\fP}$ is a \emph{system of basins (associated to $\Omega_{\fP}$)} if 
\begin{itemize}
\item[(B1)] each $B\in \cB$ is a subset of $F\in \cF^1\cup \cF^2$,
\item[(B2)] $\Gamma(B^\#)$ is connected for every $B\in \cB$ 
\item[(B3)] $\Gamma(B^\#)$ admits a spanning tree,
\item[(B4)] $B\cap A_\Sigma$ contains a unit $(n-2)$-cube for every $B\in \cB$,
\item[(B5)] for every $B\in \cB$ there exists $C\in \fP$ so that $B-|\cN(Q;A_\Sigma)|$ is contained in a connecting cell $\sigma_C$.  
\end{itemize}
The elements of $\cB$ are called \emph{basins}. 
\end{definition}

Note that the condition (B5) is more flexible than requiring that $B-|C|\subset \sigma_C$, as can be observed by contrasting Figure \ref{fig:ra_non_flat_2_basins} with Figure \ref{fig:ra_non_flat_2_pre_basins}.
 .

\begin{figure}[h!]
\includegraphics[scale=0.24]{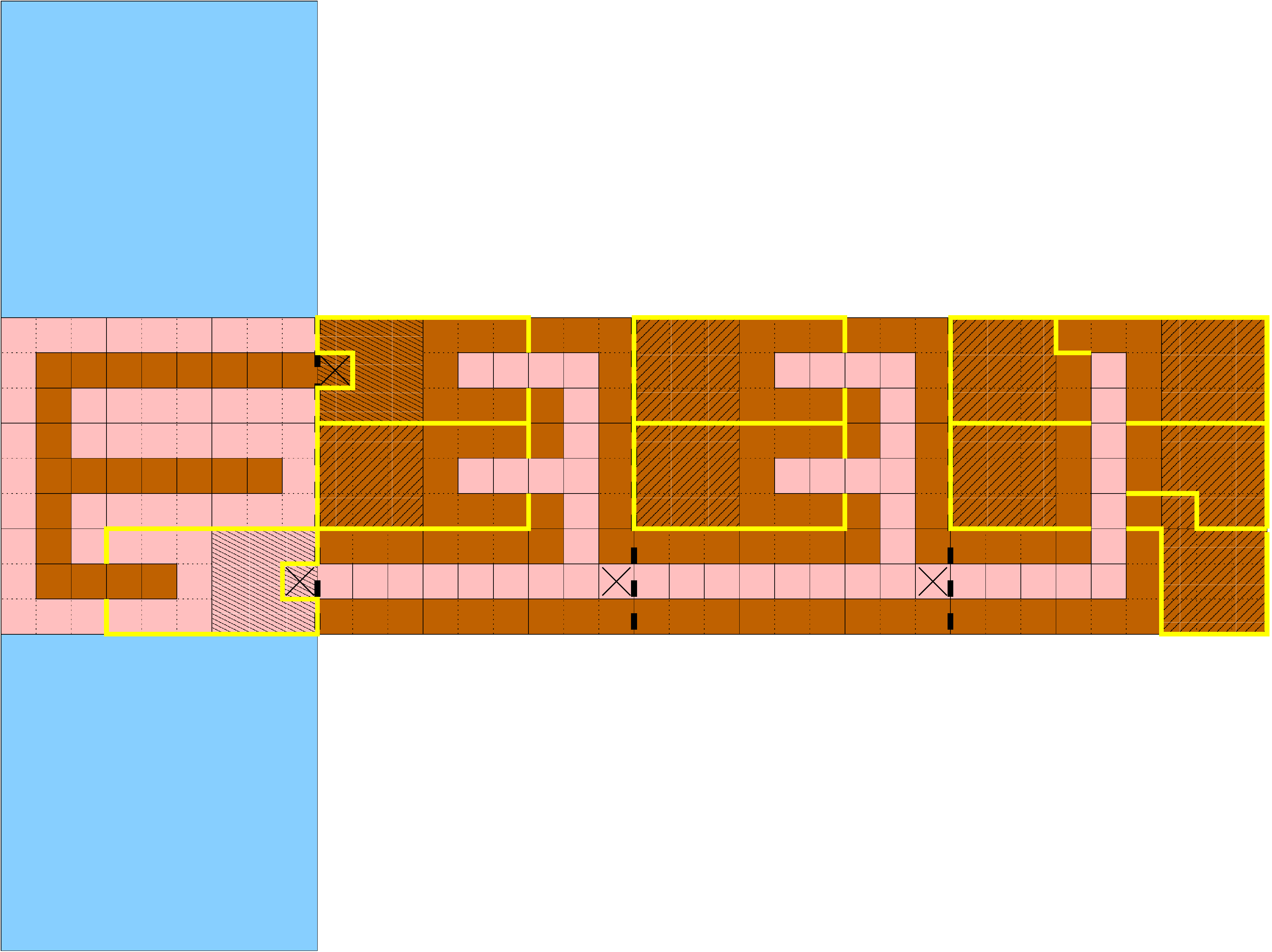}
\caption{A partition of $\cN(Q;A_\Sigma)$ into $10$ basins associated to the data of Figure \ref{fig:ra_non_flat_2_pre_basins}.}
\label{fig:ra_non_flat_2_basins}
\end{figure}

\begin{remark}
The existence of a system of basins is straightforward given a partition $\fP$ of $\cN(Q;A_\Sigma)$. Indeed, for every $C\in \fP$, fix $f_C\in C\cap \cN_\ecl(Q;A_\Sigma)$. Let $f'_C$ be a link into $f_C$ and let $\sigma_C$ be a connecting cell. We then subdivide $\bigcup_{C\in \fP} \sigma_C$ into pair-wise disjoint $1$-fine sets $\sigma'_C$ with connected graphs $\Gamma(\sigma'_C)$ so that the sets $B_C=|C|\cup \sigma'_C$ satisfy conditions (B2) and (B4) for every $C\in \fB$. Since $\Gamma(\sigma'_C)$ has valence less than $2(n-1)-1$ and $|C|$ is $3$-fine, it is also straightforward to show that $\Gamma( B_C^\# )$ admits a spanning tree. Clearly conditions (B1) and (B5) are satisfied. Thus $\cB=\{B_C \colon C\in \fP\}$ is a system of basins.
\end{remark}

Finally, we introduce a (flat) rearrangement along a system of basins. Let $\cB$ be a system of basins associated to $\Omega_{\fP}$, and let $B\in \cB$. By (B5) we may fix $q_B\in A_\Sigma$ so that $B\cap q_B$ is an $(n-2)$-cube.

Let $F_B\in \cF^1\cup \cF^2$ be the unique face of $Q$ satisfying (B1). Then the quadruple $(3Q,3F_B, 3B^\#, 3q_B)$ satisfies the conditions for flat initial data. The only modification is that $3Q$ and $F_B$ now have side length $27$. We call $(3Q,3F_B, 3B^\#, 3q_B)$ \emph{scaled flat initial data}.

By (B3), we may fix, for every $B\in \cB$, a spanning tree $\Gamma_B$ of $\Gamma(3B^\#\cup \{3q_B\})$. Similarly, as in the proof of Lemma \ref{lemma:ra_1}, we find a $1$-fine atom $A_{\Gamma_B}$ associated with the initial data $(3Q,3F_B,3B^\#, 3q_B)$ and the spanning tree $\Gamma_B$. This observation is formalized as the next lemma, with the details left to the interested reader. 

\begin{lemma}
\label{lemma:ra_1_ext}
Let $Q$ be a cube of side length $9$ and $A_\Sigma\subset Q$ a union of two atoms as in Lemma \ref{lemma:ra_2}. Suppose $\cB$ is a system of basins associated to $\Omega_{\fP}$, where $\fP$ is a partition of $\cN(Q;A_\Sigma)$ into pre-basins. For every $B\in \cB$, let $(3Q,3F_B,3B^\#, 3q_B)$ be a scaled flat initial data and $\Gamma_B$ a spanning tree of $\Gamma(3B^\#\cup \{3q_B\})$.

Then there exist $1$-fine pair-wise disjoint atoms $A_{\Gamma_B}$, $B\in \cB$, satisfying conditions (1)-(5) in Lemma \ref{lemma:ra_1} and so that $3A_\Sigma\cup \bigcup_{B\in \cB} A_{\Gamma_B}$ is a pair-wise disjoint union of two molecules.
\end{lemma}

\begin{figure}[h!]
\includegraphics[scale=0.35]{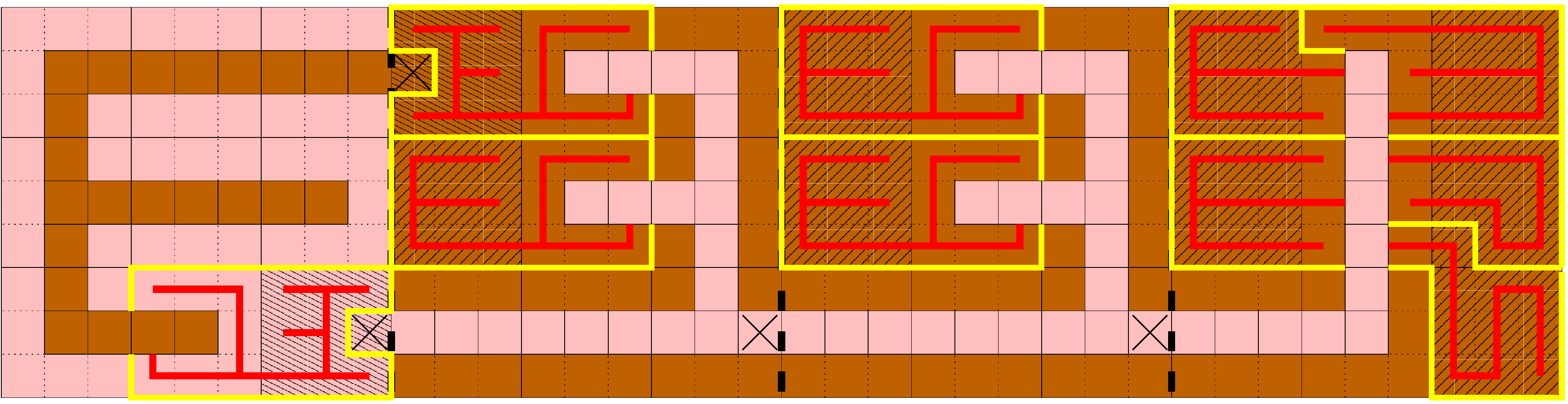}
\caption{A selection of spanning trees associated the configuration in Figure \ref{fig:ra_non_flat_2_basins}.}
\label{fig:ra_ext_flat}
\end{figure}




\section{Rough Rickman partitions}
\label{sec:RP}

This section applies the elementary constructions from Section \ref{sec:LRA} to produce domains $\Omega_1,\Omega_2,\Omega_3$ which form a rough Rickman of $\R^n$, and proves Theorem \ref{thm:3} for $p=2$. The proof is based on the existence of uniform essential partitions associated to the exhaustion of $[0,\infty)^{n-1}\times \R = \bigcup_{k\ge 0} 3^k Q_0$, where $Q_0 = [0,3]^{n-1}\times [-3,3]$. 

\begin{theorem}
\label{thm:RP}
For $m\ge 0$, there exist essential partitions 
\[
\bfOmega_m = (\Omega_{m,1},\Omega_{m,2},\Omega_{m,3})
\]
of $n$-cells $3^m\left( Q_0 \cup ([3,6]\times [0,3]^{n-1})\right)$ contained in $[0,\infty)^{n-1}\times \R$ with the following properties:
\begin{itemize}
\item[(1)] the sequence $(\bfOmega_m)$ is stable: 
\begin{itemize}
\item[(1a)] $\Omega_{m,p} \cap 3^{m-2}Q_0 = \Omega_{m',p} \cap 3^{m-2}Q_0$ for $m'>m>2$ and $p=1,2,3$,
\item[(1b)] $\Omega_{m,3} \subset \left(\interior [0,\infty)^{n-1}\right)\times \R = \bigcup_{m\ge 0}  \bfOmega_m$; 
\end{itemize}
\item[(2)] each $\Omega_{m,p}$ is a dented molecule satisfying
\begin{itemize} 
\item[(2a)] there exist $\nu\ge 1$, $\lambda>1$, and $\ell_0\ge 1$ depending only on $n$ so that each $\hull(\Omega_{m,p})$ is a $(\nu,\lambda)$-molecule with atom length at most $\ell_0$ and 
\item[(2b)] there exist $L\ge 1$ depending only on $n$ and an $L$-bilipschitz homeomorphism $(\Omega_{m,p},d_{\Omega_{m,p}}) \to (\hull(\Omega_{m,p}),d_{\hull(\Omega_{m,p})})$ which is the identity on $\partial \hull(\Omega_{m,p}) \cap \Omega_{m,p}$; 
\end{itemize}
\item[(3)] each $\bfOmega_m$, $m\ge 1$, satisfies the tripod property.
\end{itemize}
For $p=1,2,3$, each domain $\Omega_p = \bigcup_{m\ge 0}\Omega_{m,p}$ in its inner metric $d_{\Omega_p}$ is bilipschitz equivalent to $\R^{n-1}\times [0,\infty)$. Moreover, there exist bilipschitz homeomorphisms  $\phi_1 \colon [0,\infty)^{n-1}\times [0,\infty) \to (\Omega_1,d_{\Omega_1})$ and $\phi_2 \colon [0,\infty)^{n-1} \times (-\infty,0] \to (\Omega_2,d_{\Omega_2})$ which restrict to the identity mappings on $\partial [0,\infty)^{n-1}\times [0,\infty)$ and $\partial [0,\infty)^{n-1}\times (-\infty,0]$, respectively; the boundary $\partial [0,\infty)^{n-1}$ is understood relative to $\R^{n-1}$.
\end{theorem}

Conditions (1)-(3) have the following interpretations. Condition (1) refers to an induction process, which consists of two main steps: scaling and rearranging, and allows us to paste the essential partitions $\bfOmega_m$ together. Condition (2) yields that the domains $\Omega_{m,j}$ are uniformly bilipschitz equivalent to cubes $[0,3^m]^n$. Finally, (3) ensures that $\dist_\haus(\partial_\cup \bfOmega_m, \partial_\cap \bfOmega_m)\le 6$ in the sup-metric; compare with (\ref{eq:haus_W}). We also observe the following corollary; see Section \ref{sec:proof_RP}.

\begin{corollary}
\label{cor:John-domains}
Let $p=1,2,3$ and $m\geq 1.$ Then the domains $\Omega_{m,p}$ are John-domains with John-constant depending only on $n$. Furthermore, each $\Omega_p$ is a uniform domain.
\end{corollary}

\begin{proof}[Proof of Theorem \ref{thm:3} (for $p=2$) given Theorem \ref{thm:RP}]
Let $\bfOmega'=(\Omega'_1,\Omega'_2,\Omega'_3)$ be the essential partition of $[0,\infty)^{n-1}\times \R$ from Theorem \ref{thm:RP}. By (1a) and (3) in Theorem \ref{thm:RP}, $\bfOmega'$ satisfies the tripod property.

We subdivide $\R^n$ into $2^{n-1}$ congruent subsets $W_1,\ldots, W_{2^{n-1}}$, where $W_1 = [0,\infty)^{n-1}\times \R$. Since $\Omega'_3 \subset \interior W_1$, by reflecting $\Omega'_3$ with respect to the common sides of $W_1,\ldots, W_{2^{n-1}}$ we obtain pair-wise disjoint domains $\Omega'_4,\ldots, \Omega'_{2^{n-1}+2}$. The unions of the corresponding reflections of $\Omega'_1$ and $\Omega'_2$ are the domains $\Omega_1$ and $\Omega_2$ claimed in Theorem \ref{thm:3}. Thus $\Omega_1$ and $\Omega_2$ are connected. 

Let $\phi_1 \colon [0,\infty)^{n-1}\times [0,\infty) \to (\Omega'_1,d_{\Omega'_1})$ and $\phi_2 \colon [0,\infty)^{n-1}\times (-\infty,0] \to (\Omega'_2,d_{\Omega'_2})$ be bilipschitz homeomorphisms which reduce to the identity mapping on the boundary, a consequence of Theorem \ref{thm:RP}. Reflections across the pair-wise common sides of domains $W_1,\ldots, W_{2^{n-1}}$ extend $\phi_1$ and $\phi_2$ to bilipschitz homeomorphisms $\psi_1 \colon \R^{n-1}\times [0,\infty) \to (\Omega_1,d_{\Omega_1})$ and $\psi_2 \colon \R^{n-1}\times (-\infty,0] \to (\Omega_2,d_{\Omega_2})$. Finally, if
\[
\Omega_3 = \Omega'_3\cup \cdots\cup \Omega'_{2^{n-1}+2}, 
\]
and
\[
\bfOmega = (\Omega_1,\Omega_2,\Omega_3),
\]
condition (3) in Theorem \ref{thm:RP} ensures that $\bfOmega$ is a rough Rickman partition satisfying the tripod property.
\end{proof}

\subsection{Proof of Theorem \ref{thm:RP} -- first steps}
\label{sec:RP_1st_steps}

We begin the proof of Theorem \ref{thm:RP} in this section by explicitly giving the initial steps of the inductive construction of partitions $\bfOmega_m$. The general induction is based on rearrangements in three types of cubes and their successive scalings, and we consider these rearrangements in detail in Section \ref{sec:CD}. We complete the proof finally in Section \ref{sec:proof_RP}.

Let $\Omega$ be a $3$-fine $n$-cell and suppose that $\bU=(U_1,U_2,U_3)$ is an essential partition of $\Omega$ into $n$-cells. A cube $Q\in \Omega^*$ of side length $3$ is a \emph{$\bU$-cube} if there exists $i\in \{1,2,3\}$ for which $Q\subset U_i$. The index $i$ is the \emph{color of $Q$ in $\bU$}, and the indices $\{1,2,3\}\setminus \{i\}$ are \emph{complementary indices (of the color of $Q$)}. Let also
\[
\cQ_\partial(\bU)=\{ Q\in |\bU|^* \colon Q\cap \partial_\cup \bU\ \mathrm{contains\ an\ } (n-1)\mathrm{-cell}\}.
\]

\subsubsection{The initial step; step $0$}
\label{sec:0ra}

We begin with the $n$-cubes
\[
\Omega_1=[0,3]^n,\ \Omega_2=[0,3]^{n-1}\times [-3,0],\ \mathrm{and}\ \Omega_3=[3,6]\times [0,3]^{n-1}
\]
of side length $3$, and set 
\[
\bfOmega = (\Omega_1,\Omega_2,\Omega_3),
\]
$\Omega = |\bfOmega| = \Omega_1\cup\Omega_2\cup \Omega_3$; see Figure \ref{fig:0ra}. 

\begin{figure}[h!]
\includegraphics[scale=0.35]{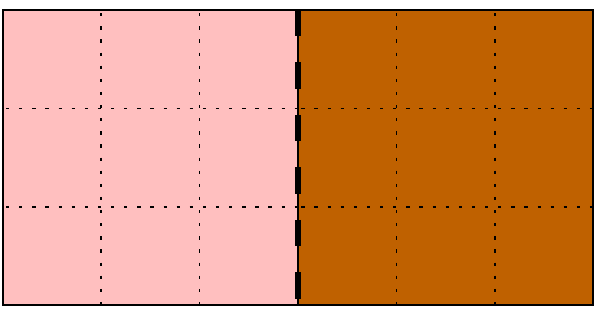}
\caption{Faces of $\Omega_1$ in $\partial_\cup \bfOmega$}
\label{fig:0ra}
\end{figure}

For consistency, let also 
\[
\bfOmega_0 = (\Omega_{0,1}, \Omega_{0,2},\Omega_{0,3}) = (\Omega_1, \Omega_2,\Omega_3).
\]
Note that $\bfOmega_0$ does not satisfy the tripod property for the (trivial) reason that $\Omega_2\cap \Omega_3$ is not $(n-1)$-dimensional. However, we note that $\partial_\cap \bfOmega = \Omega_1 \cap \Omega_2 \cap \Omega_3$ is an $(n-2)$-cube. 

In anticipation of the forthcoming induction step, we note that $\partial_\cup \bfOmega_0 \subset [0,3]^n$. Furthermore, the cube $[0,3]^n$ is contained in domain $\Omega_{0,1}$ but has one $(n-1)$-dimensional face contained in $\partial \Omega_{0,2}$ and one in $\Omega_{0,3}$. The cube $[0,3]^n$ will therefore be an example of a $\cC$-cube. This is one of the three general categories we will use $\cC$-cubes ('C' for color), $\cD$-cubes ('D' for dent), and $\cN$-cubes ('N' for neglected). They are formally introduced in Section \ref{sec:CD}, but have clear antecedents from various local rearrangements in Section \ref{sec:LRA}.

\subsubsection{First rearrangement}
\label{sec:1ra}

First scale $\bfOmega_0$ by $3$, and let 
\[
\bfOmega'_1=3\bfOmega_0=(3\Omega_{0,1},3\Omega_{0,2},3\Omega_{0,3}).
\]

To rearrange $\bfOmega'_1$ to achieve the tripod property and properties (1)-(3) in Theorem \ref{thm:RP}, we modify $\bfOmega'_1$ using atoms which allow respectively $3\Omega_{0,2}$ and $3\Omega_{0,3}$ to penetrate $3\Omega_{0,1}$; this will produce $\bfOmega_1$. We apply Lemma \ref{lemma:ra_2} to $C=3\Omega_{0,1}$ and thus obtain an essential partition 
\[
\bfOmega_1=(\Omega_{1,1},\Omega_{1,2},\Omega_{1,3}) = (3\Omega_{0,1}-(A_2\cup A_3), 3\Omega_{0,2} \cup A_2, 3\Omega_{0,3}\cup A_3)
\]
of $|\bfOmega'_1|$ into $n$-cells satisfying the tripod property, where $A_2$ and $A_3$ are atoms from the process of Lemma \ref{lemma:ra_2}, see Figure \ref{fig:1ra}; rearrangements of this type will be called $\cC$-modifications (Lemma \ref{lemma:ra_a}) in the inductive construction, since they are performed in a scaled copy of a $\cC$-cube.

\begin{figure}[h!]
\includegraphics[scale=0.35]{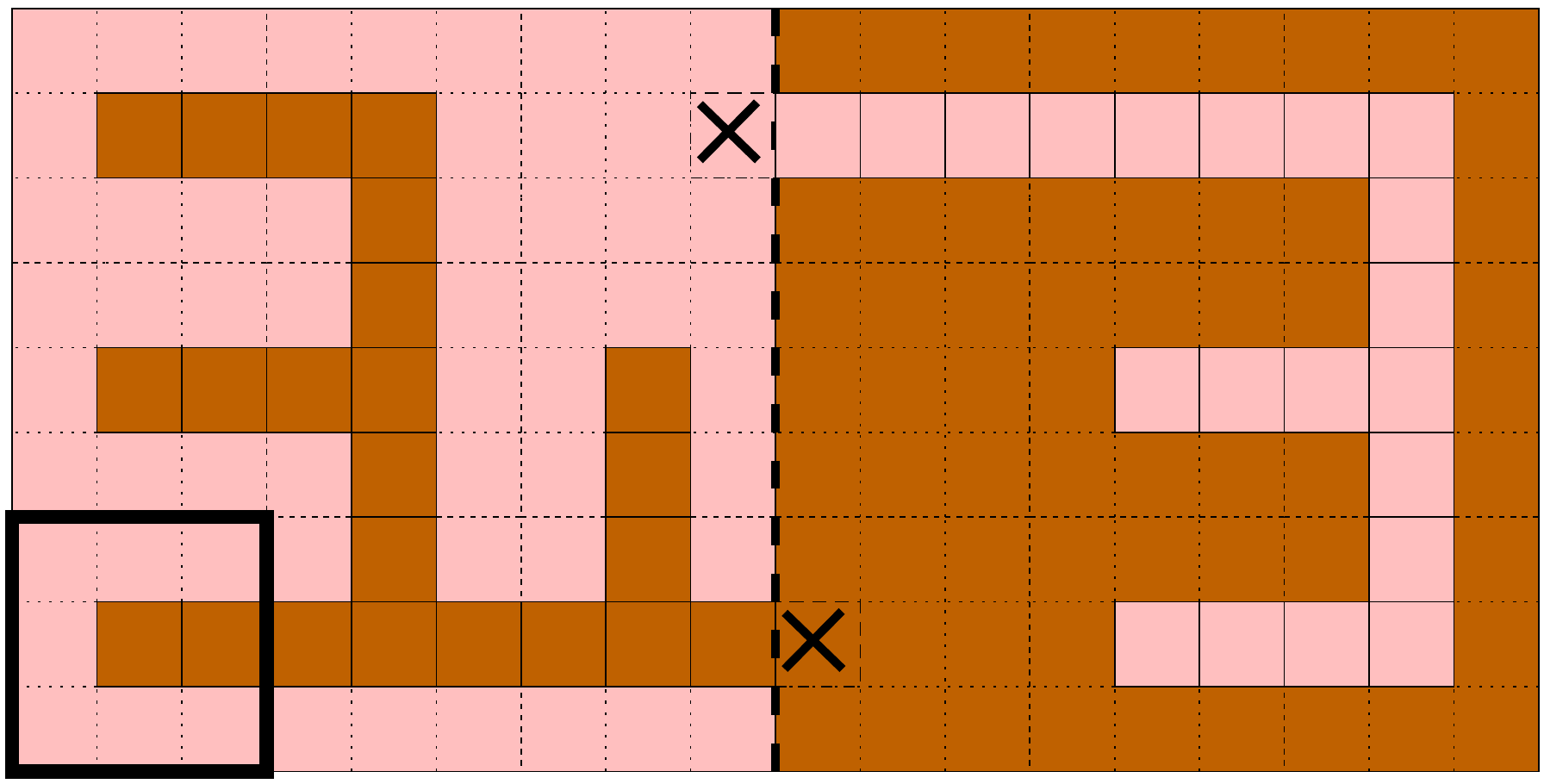}
\caption{An example of the evolution of $\partial_\cup \bfOmega_1$ with cube $[0,3]^3$ emphasized.} 
\label{fig:1ra}
\end{figure}

In the proof of Lemma \ref{lemma:ra_2}, we are free to use any maximal forest $\Sigma$. In particular, we may assume that $[0,3]^n$ is a leaf of $\Sigma$ as in Figure \ref{fig:1ra}; this choice is used to obtain stability condition (1a). Thus we arrive at  $\bfOmega_1$ in accord with the conditions of Theorem \ref{thm:RP}.

As orientation toward the general induction step, we note that $\partial_\cup \bfOmega_1$ is contained in a union of $n$-cubes of side length $3$ contained in $3\Omega_{0,1}$. Indeed, let 
\[
\cQ=\{ Q\in \cQ_\partial(\bfOmega'_1)\colon Q \subset 3[0,3]^n=3\Omega_{0,1}\}. 
\]
Then $\partial_\cup \bfOmega_1 \subset |\cQ|$. 

Furthermore, for all $Q\in\cQ$, there exists exactly one $j_Q\in \{2,3\}$ so that $\mathrm{cl}\left( \interior Q\cap \Omega_{1,j_Q}\right)$ is a building block. If $Q\cap \Omega_{1,j_Q} = \mathrm{cl}\left( \interior Q\cap \Omega_{1,j_Q}\right)$, $Q$ is a $\cD$-cube. Otherwise, $Q$ is an $\cN$-cube.


\subsubsection{The second step} 
\label{sec:2ra}

Whereas the essential partition $\bfOmega_0$ was explicitly chosen and $\bfOmega_1$ was described using Lemma \ref{lemma:ra_2}, at this point we only give heuristic description for $\bfOmega_2$. 

The essential partition $\bfOmega_2$ is obtained from $\bfOmega_1$ by first defining $\bfOmega_2' = 3\bfOmega_1$ and rearranging $3\Omega_{1,1}$, $3\Omega_{1,2}$, and $3\Omega_{1,3}$ with flat rearrangements (Lemma \ref{lemma:ra_1}) and by flat rearrangements in basins (Lemma \ref{lemma:ra_1_ext}); we attach atoms of side length $1$ to atoms $3A_2$ and $3A_3$ and, correspondingly, remove them from $3\Omega_{1,1}$. Figure \ref{fig:2ra} illustrates this step. This modification will be called a secondary $\cC$-modification and will be formalized in Lemma \ref{lemma:ra_c}. Note, however, that in order to satisfy the stability requirement (1a), we also impose the additional condition that $\bfOmega_2 \cap [0,3]^n = \bfOmega_1 \cap [0,3]^n$. This is possible, since $\Omega_{1,3}\cap [0,3]^n$ is a leaf.

\begin{figure}[h!]
\includegraphics[scale=0.20]{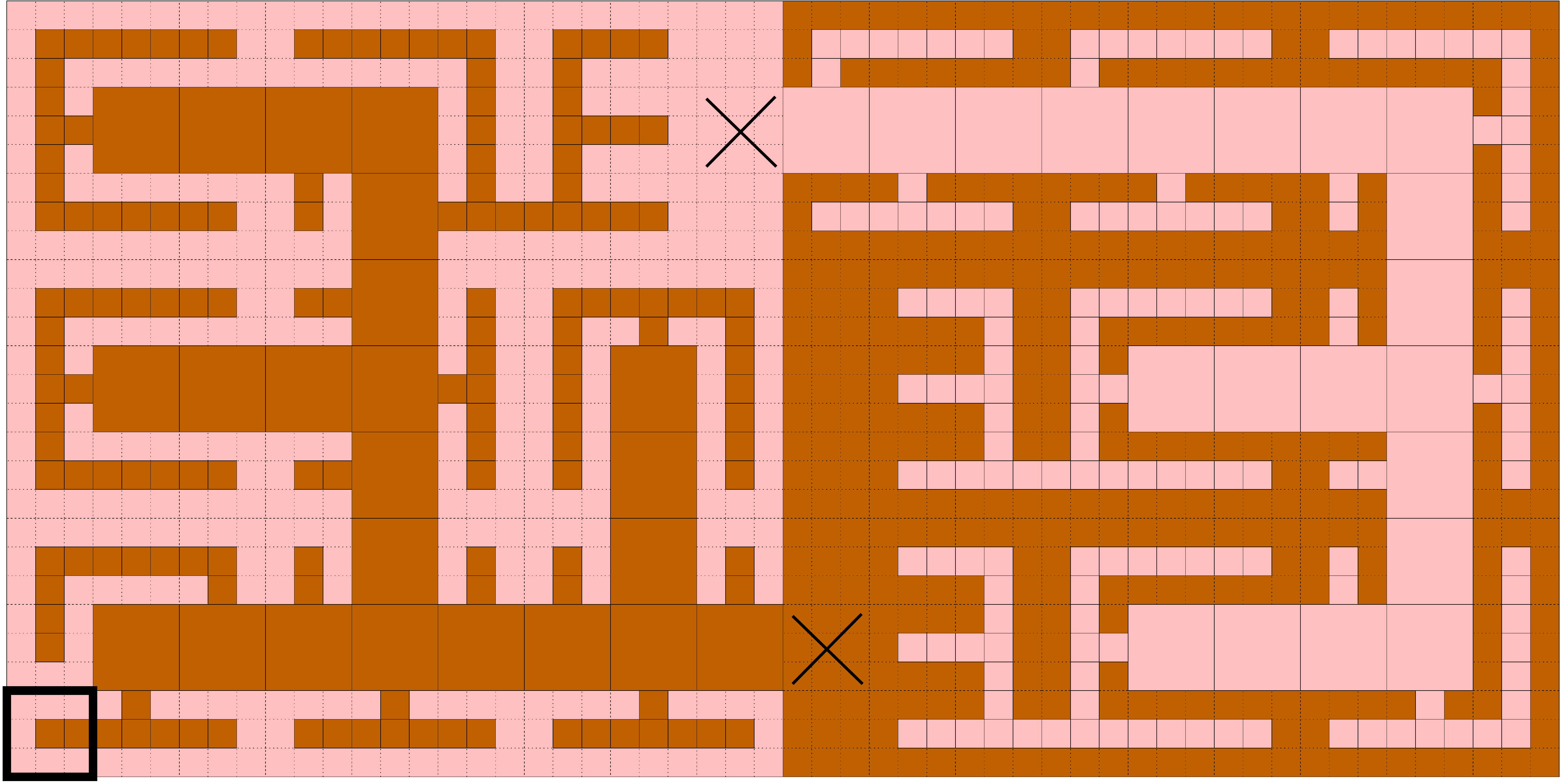}
\caption{An example of $\bfOmega_2$ with cube $[0,3]^3$ highlighted.}
\label{fig:2ra}
\end{figure}


\subsection{$\cC$-, $\cD$-, and $\cN$-cubes}
\label{sec:CD}

With this preparation, we formally define $\cC$-, $\cD$-, and $\cN$-cubes; primary cubes have side length $3$ and secondary cubes side length $9$. The corresponding rearrangements, based on flat and non-flat rearrangements in Section \ref{sec:LRA}, are then discussed then in the following sections. Note that, if $Q$ is a primary or a secondary cube, the corresponding rearrangement is performed in $3Q$.

Restricting exclusively to these cubes in the iteration process provides a systematic rearrangement process. We obtain the sequence $(\bfOmega_m)$, using scalings and rearrangements, in such a way that for each $m\ge 0$ there exists an essentially disjoint collection $\fL_m$ of primary and secondary cubes (of different types) which covers $\partial_\cup \bfOmega_m$. After scaling $\bfOmega_m$ by $3$, we perform appropriate rearrangements of the right type in each cube in $3\fL_m$. This yields a new essential partition $\bfOmega_{m+1}$ and a new list $\fL_{m+1}$ of essentially disjoint cubes which also have the property $\partial_\cup \bfOmega_{m+1} \subset |\fL_{m+1}|$. The rearrangements in these cubes are mutually independent, and it follows from properties of rearrangements in Section \ref{sec:LRA} that $\bfOmega_{m+1}$ satisfies the tripod property. We discuss the list $\fL_m$ and this inductive step in Section \ref{sec:prep_induction}.

Although there are \emph{a priori} six different types of cubes, only four types of rearrangements occur here. The reason is that all $\cN$-cubes are contained in secondary $\cC$-cubes and secondary $\cN$-cubes, and secondary $\cD$-cubes never appear. In Table \ref{table:modifications} we list the four types of rearrangements and descendants they produce; we use the subscript $2$ to denote secondary cubes.

\begin{table}[ht!]
\begin{tabular}{c|ccc}
$Q$  & modification in $3Q$ & descendant(s) in $3Q$ \\  
\hline
$\cC$-cube & $\cC$ & $\cC_2$ \\
$\cD$-cube & $\cD$ & $\cC$, $\cD$ \\
$\cC_2$ & $\cC_2$ & $\cC$, $\cD$, $\cN_2$ \\
$\cN_2$ & $\cN_2$ & $\cC$, $\cD$, $\cN_2$  
\end{tabular}
\smallskip
\caption{Cubes, modifications, and descendants.}
\label{table:modifications}
\end{table}

\subsubsection{$\cC$-cubes}

Let $U$ be an $n$-cell and $\bU=(U_1,U_2,U_3)$ an essential partition of $U$. Recall that $\cQ_\partial(\bU)$ is the collection of cubes $Q$ in $|\bU|^*$ for which $Q\cap \partial_\cup \bU$ is an $(n-1)$-cell.

Let $Q\in \cQ_\partial(\bU)$ be a $\bU$-cube of color $i\in \{1,2,3\}$, and $j$ and $k$ complementary indices. For $p=j,k$, let $\cQ'_p(Q)$ be the collection of all unit $n$-cubes in $Q^\#$ meeting $U_p$ in an $(n-1)$-cube, and let $\cQ'(Q)=\cQ'_j(Q) \cup \cQ'_k(Q)$. Let $\cQ'_c(Q)$ be the cubes in $\cQ'(Q)$ having (maximal) valence $2(n-1)$ in the adjacency graph $\Gamma(\cQ'(Q))$ as in Section \ref{sec:npc}.

\begin{definition}
\label{def:C_cube_extended}
Let  $\bU=(U_1,U_2,U_3)$ be an essential partition.
A $\bU$-cube $Q\in \cQ_\partial(\bU)$ of color $i$ is a \emph{$\cC$-cube in $\bU$} if, for complimentary colors $j$ and $k$, 
\begin{itemize}
\item[(i)] there are unit $n$-cubes $q_j\subset U_j$ and $q_k \subset U_k$ with $q_j\cap q_k = \varnothing$ such that both cubes $q_j$ and $q_k$ have a face contained in $\partial Q$; let $q'_j$ and $q'_k$ be the unique cubes in $\cQ'(Q)$ which share a face with $q_j$ and $q_k$, respectively, and
\item[(ii)] the adjacency graph $\Gamma(\{q_k\}\cup (\cQ'_j(Q)\setminus (\cQ'_c(Q)\cup \{q'_j\})))$ is connected.
\end{itemize}
\end{definition}

The collection of $\cC$-cubes in $\bU$ is denoted by $\cC(\bU)$. Note that each $Q\in \cC(\bU)$ satisfies $Q\cap \partial_\cup \bU \subset \partial Q$, since $\cC$-cubes are $\bU$-cubes.

\begin{remark}
Definition \ref{def:C_cube_extended} formalizes the heuristic properties of $\cC$-cubes discussed in Section \ref{sec:0ra}. First, a $\cC$-cube $Q$ is contained in one element of the essential partition, and, second, $Q$ meets the other two elements in a codimension $1$-set (item (i)). Item (ii) formalizes a necessary condition for rearrangement to extend color $k$ between $i$ and $j$ in the scaled copy of $Q$. For  $\bfOmega_0$ and $\bfOmega_1$ this condition could be simplified to the condition that $Q\cap \Omega_j$ is a union of faces of $Q$. 
\end{remark}

\begin{definition}
\label{def:2nd_C_cube}
Let $\bU$ and $\bV$ be essential partitions satisfying $|\bU|=3|\bV|$. A cube $Q$ of side length $9$ is a \emph{secondary $\cC$-cube in $\bU$ with respect to $\bV$} if $(1/3)Q$ is a $\cC$-cube with respect to $\bV$.
\end{definition}

The collection of secondary $\cC$-cubes in $\bU$ with respect to $\bV$ is denoted by $\cC_2(\bU;\bV)$. 

\begin{remark}
Note that in Definition \ref{def:2nd_C_cube} we do not require $\bU \cap Q = 3(\bV\cap (1/3)Q)$. In fact, if $\bU\cap Q$ is obtained by a $\cC$-modification in $Q$ (see Section \ref{sec:Cmod}),  $\bU\cap Q \ne (3\bV) \cap Q$.
\end{remark}

\subsubsection{$\cD$- and $\cN$-cubes}

Suppose $\bU=(U_1,U_2,U_3)$ and $\bV=(V_1,V_2,V_3)$ are essential partitions satisfying $|\bU|=3|\bV|$. We first discuss $\cD$-cubes.

\begin{definition}
\label{def:D_cube}
A cube $Q\in \cQ_\partial(\bU)$ of side length $3$ is a \emph{$\cD$-cube in $\bU$ relative to $\bV$} if 
\begin{itemize}
\item[(1)] $Q$ is a $3\bV$-cube of color $i$, 
\item[(2)] $Q$ is not a $\bU$-cube, 
\item[(3)] there exists complementary colors $j$ and $k$ for which $A:=Q \cap U_j$ is an $n$-cell and $Q$ has no neglected faces, and $(\interior Q)\cap U_k=\emptyset$,
\item[(4)] $A$ is either a $(Q\cap \partial 3V_i)$-based building block in $Q$ or a union of two building blocks based on different faces of $Q$, and
\item[(5)] $(Q-A,A,\Omega)$ has the tripod property, where $\Omega$ is the smallest $n$-cell consisting of $n$-cubes of side length $3$ for which $A\cap \partial Q \subset \Omega$.
\end{itemize}
\end{definition}
The collection of all $\cD$-cubes in $\bU$ with respect to $\bV$ is denoted $\cD(\bU;\bV)$. Note that in Definition \ref{def:D_cube}, $A$ in (3) is always a $1$-fine atom. If $A$ in (3) and (4) is a building block, we say that $Q$ is a \emph{$\cD$-cube of type $1$}. Otherwise, $Q$ is a \emph{$\cD$-cube of type $2$}. 

Since by definition $\cD$-cubes have no neglected faces, we also need $\cN$-cubes.

\begin{definition}
\label{def:N_cube}
A cube $Q\in \cQ_\partial(\bU)$ is an \emph{$\cN$-cube in $\bU$ relative to $\bV$} if 
\begin{itemize}
\item[(1)] $Q$ is a $3\bV$-cube of color $i$, 
\item[(2)] $Q$ is not a $\bU$-cube, 
\item[(3)] there exists unique complementary color $j$ so that $A:=\cl(\interior Q \cap U_j)$ is a $(Q\cap \partial 3V_i)$-based building block in $Q$ while $(\interior Q)\cap U_k=\emptyset$,
\item[(4)] $Q$ has a neglected face contained in $3V_j$, 
\item[(5)] $(Q-A,A,\Omega)$ has the tripod property, where $\Omega$ is the smallest $n$-cell consisting of $n$-cubes of side length $3$ for which $A\cap \partial Q \subset \Omega$.
\end{itemize}
\end{definition}
The collection of all $\cN$-cubes in $\bU$ with respect to $\bV$ is denoted $\cN(\bU;\bV)$. We define secondary $\cN$-cubes as follows. 
\begin{definition}
\label{def:2nd_cN-cube}
Let $\bU$, $\bV$, and $\bW$ be essential partitions satisfying $|\bU|=3|\bV|=9|\bW|$. A cube $Q$ of side length $9$ is a \emph{secondary $\cN$-cube in $\bU$ (with respect to $\cN(\bV;\bW)$)} if $(1/3)Q$ is an $\cN$-cube in $\bV$ relative to $\bW$.
\end{definition}
We let $\cN_2(\bU;\bV,\bW)$ to denote the collection of all secondary $\cN$-cubes in $\bU$ with respect to $\cN(\bV;\bW)$. 


\subsubsection{$\cD$-modifications}

We consider first $\cD$-cubes of type $1$. The first rearrangement is called a \emph{$\cD$-modification}, which has already been anticipated by Lemma \ref{lemma:ra_1} and Corollary \ref{cor:ra_1}.

\begin{lemma}[$\cD$-modification of type $1$]
\label{lemma:ra_b}
Let $V$ be an $n$-cell, $\bV=(V_1,V_2,V_3)$ and $\bW=(W_1,W_2,W_3)$ essential partitions satisfying $V=|\bV|=3|\bW|$, 
and let $Q\in \cD(\bV;\bW)$ be a $\cD$-cube of type $1$; let $i$ be the color of $(1/3)Q$ in $\bW$ and let $j$ be such that $A:=Q\cap V_j$ is an $F$-based building block, where $F$ is face of $Q$. Then there exists a pair-wise disjoint union of atoms $B_\Sigma\subset 3Q$ composed of $1$-fine $3F$-based building blocks on the boundary of $3Q$, so that the $n$-cells $U_i = 3V_i - B_\Sigma$, $U_j = 3V_j \cup B_\Sigma$, and $U_k = 3V_k$, where $k$ is the other complementary index, form an essential partition 
\[
\bU=(U_1,U_2,U_3)
\]
of $|3\bV|$ satisfying
\begin{itemize}
\item[(1)] $B_\Sigma \cap \partial (3Q) \subset \partial_\cup 3\bV$, 
\item[(2)] $\partial_\cup \bU \cap 3Q \subset |\cC(\bU)|\cup |\cD(\bU;\bV)|$, and
\item[(3)] $B_\Sigma$ is an atom for $n>3$ and consists of at most $3$ components for $n=3$.
\end{itemize}
Furthermore, $\bU$ has the tripod property in $3Q$.
\end{lemma}

\begin{figure}[h!]
\includegraphics[scale=0.30]{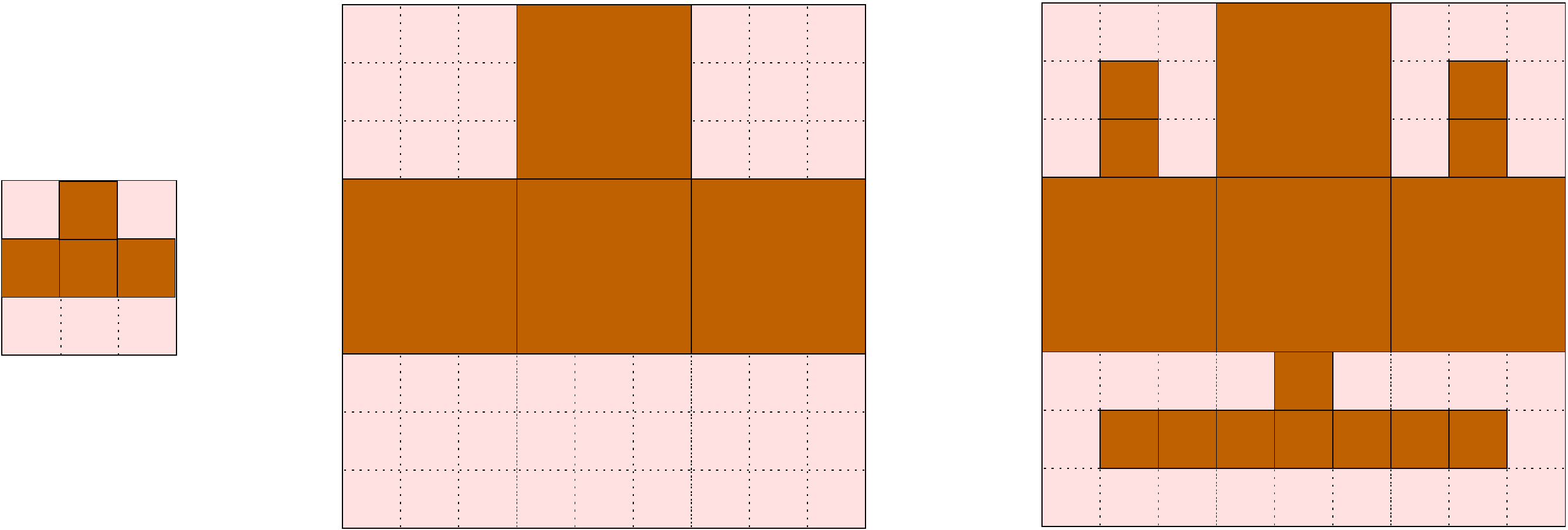}
\caption{An example of an essential partition $\bV$ in $Q$, and essential partitions $3\bV$ and $\bU$ in $D=3Q$ for a building block in Figure \ref{fig:1face3_molecule}.}
\label{fig:D_cubes}
\end{figure}

\begin{convention}
Before giving the proof of Lemma \ref{lemma:ra_b}, we emphasize that the figures in this section (e.g.\;in Figure \ref{fig:D_cubes}) use only the two complementary colors $j$ and $k$. The third color, the color $i$ of the cube itself, never appears.
\end{convention}

\begin{proof}[Proof of Lemma \ref{lemma:ra_b}]
It suffices to find planar initial data for Corollary \ref{cor:ra_1}, the claim then follows from Lemma \ref{lemma:tripod_planar}. Note that Corollary \ref{cor:ra_1} is necessary only for $n=3$, since for $n>3$ we may use Lemma \ref{lemma:ra_1}. 

Define $F' = 3F\cap 3V_i$, that is, $F'=3F-3A$. For $n>3$, $F'$ is connected. For $n=3$, $F'$ is at most three $2$-cells. Let $D=3Q$.

Let $D^*$ be the $3$-regular subdivision of $D$ and let $(D^*)'\subset D^*$ be the subset of cubes having a face contained in $F'$. For $n>3$ we fix a unit cube $Q_1$ in $A$. Then $Q_1\cap F'$ is an $(n-2)$-cell. For $n=3$, we fix unit cubes $Q_1,\ldots, Q_p$ in $A$, where $p$ is the number of components of $F'$. 

When $n>3$ we choose a maximal tree $\Sigma \in \Gamma((D^*)'\cup \{Q_1\})$, and for $n=3$, we fix a maximal forest $\Sigma = \Gamma_1\cup \cdots \cup \Gamma_p$ in $\Gamma((D^*)'\cup \{Q_1,\ldots, Q_p\})$. 
The vertex sets of trees $\Gamma_i$ give now a required partition for $(D^*)'$. Corollary \ref{cor:ra_1} yields a $1$-fine set $B_\Sigma$ whose components are $3F$-based atoms. The tripod property in $D$ for $\bU$ follows from Lemma \ref{lemma:tripod_planar}, and condition (4) in Lemma \ref{lemma:ra_1} shows that $B_\Sigma \subset F\cup \interior D$.  

Assertions (1) and (2) follow from Corollary \ref{cor:ra_1}, the fact that cubes in $(D^*)'$ are $\cD$-cubes in $\cD(\bU;\bV)$ and the observation that $3A=|\cC(\bU)|\cap D$.
\end{proof}

For $\cD$-cubes of type $2$, the corresponding arrangement is also called a \emph{$\cD$-mod\-ifi\-ca\-tion}.
\begin{lemma}[$\cD$-modification of type $2$]
\label{lemma:ra_b2}
Let $V$ be an $n$-cell, let $\bV=(V_1,V_2,V_3)$ and $\bW=(W_1,W_2,W_3)$ be essential partitions satisfying $V=|\bV|=3|\bW|$, and $Q\in \cD(\bV;\bW)$ be a $\cD$-cube of type $2$; let $i$ be the color of $(1/3)Q$ in $\bW$ and take $j$ so that $A:=Q\cap V_j=B\cup B'$ is an atom, where $B$ and $B'$ are essentially disjoint building blocks. Then there exists a pair-wise disjoint union $A_\Sigma\subset 3Q$ of $1$-fine atoms on the boundary of $3Q$ consisting of building blocks with
\[
\bU=(U_1,U_2,U_3)
\]
an essential partition of $3|\bV|$ by $n$-cells satisfying the tripod property in $3Q$. Here $U_i = 3V_i - A_\Sigma$, $U_j = 3V_j \cup A_\Sigma$, $U_k = 3V_k$, where $k$ is the remaining complementary index. Moreover,  
\begin{itemize}
\item[(1)] $A_\Sigma \cap \partial (3Q) \subset \partial_\cup 3\bV$, and
\item[(2)] $\partial_\cup \bU \cap 3Q \subset |\cC(\bU)|\cup |\cD(\bU;\bV)|$, and
\item[(3)] $A_\Sigma$ is an atom for $n>3$ and consists of at most $4$ components for $n=3$.
\end{itemize}
Finally, $\bU$ has the tripod property in $3Q$.
\end{lemma}

\begin{figure}[h!]
\includegraphics[scale=0.25]{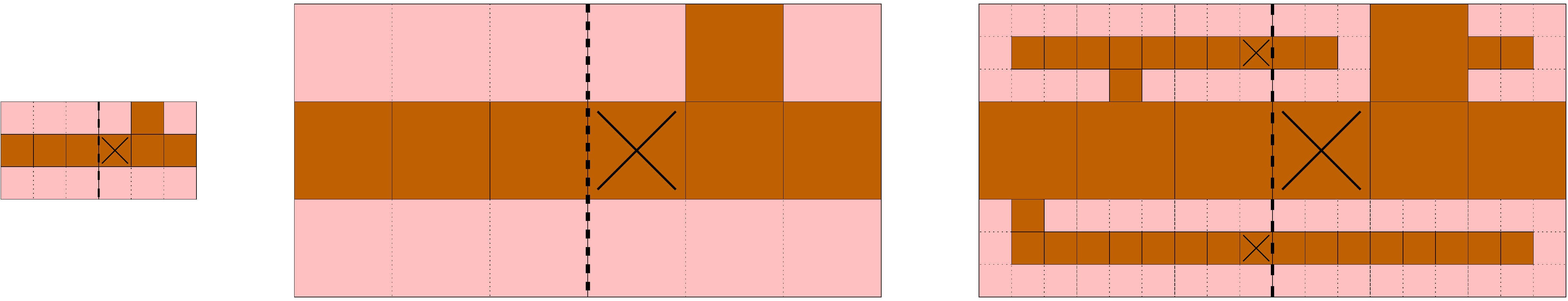}
\caption{Analogue of Figure \ref{fig:D_cubes} for $\cD$-cube of type $2$.}
\label{fig:D_example_3}
\end{figure}

\begin{proof}[Proof of Lemma \ref{lemma:ra_b2}]
This case uses Lemma \ref{lemma:ra_2} in place of Corollary \ref{cor:ra_1} and Lemma \ref{lemma:tripod_non_planar} in place of Lemma \ref{lemma:tripod_planar}. 

We may assume $i=1$, $j=2$, and $k=3$ and that $B$ and $B'$ are $f$- and $f'$-based, respectively. Let $D=3Q$. 

Let $\cQ'$ be the collection of cubes in $D^*$ meeting $f\cup f'$ and not contained in $3A^\#$. Recall that $\Gamma(\cQ')$ is the adjacency graph of cubes in $\cQ'$. For $n>3$, $\Gamma(\cQ')$ is connected, and we may fix a cube $q\in 3A^\#$ of side length $3$ and a maximal tree $\Sigma \subset \Gamma(\cQ'\cup \{q\})$. It is a simple observation that we may now apply Lemma \ref{lemma:ra_2} directly to $\Sigma$ and obtain a $1$-fine atom $A_\Sigma$ satisfying (1)-(5). 

For $n=3$, we observe that $\Gamma(\cQ')$ has at most $4$ components $\Gamma_1,\ldots, \Gamma_p$, $p\le 4$; Figure \ref{fig:D_example_3} illustrates $p=3$. It is easy to observe that we may fix pair-wise essentially disjoint cubes $q_1,\ldots, q_p$ in $3A^*$ so that $\Gamma(\Gamma_i\cup \{q_i\})$ is connected, and thus create a maximal forest $\Sigma = \Sigma_1\cup \cdots \cup \Sigma_p$ with $\Sigma_i \subset \Gamma(\Gamma_i\cup \{q_i\})$. A slight modification of Lemma \ref{lemma:ra_2} yields $A_\Sigma = A_{\Sigma_1}\cup \cdots \cup A_{\Sigma_p}$, where $A_{\Sigma_i}$ is a $1$-fine atom satisfying (1)-(5). 

In both cases, $\bU=(3V_1-A_\Sigma,3V_2\cup A_\Sigma,3V_3)$ satisfies the required conditions.
\end{proof}

The essential properties of $\cD$-modifications are summarized in the next two corollaries. 

\begin{corollary}
\label{cor:ra_b}
Let $\bV$, $Q$, $A$, $\bU$, and $\{i,j,k\}=\{1,2,3\}$ be as in Lemma \ref{lemma:ra_b} or as in Lemma \ref{lemma:ra_b2}. Then $\bU\cap 3Q$ satisfies the tripod property and in addition
\begin{itemize}
\item[(a)] $\partial_\cup \bU \cap 3Q \subset |\cC(\bU)| \cup |\cD(\bU;\bV)|$ and
\item[(b)] $\cC(\bU\cap 3Q) = (3A)^*$.
\end{itemize}
Furthermore, to each $f\in \left(((\partial_\cup \bV)\cap Q) - A\right)^\#$ corresponds one $3f$-based building block in $U_j$.
\end{corollary}

By condition (1) in Lemmas \ref{lemma:ra_b} and \ref{lemma:ra_b2}, $\cD$-modifications are performed independently in each cube of $3\cD(\bV;\bW)$ in the sense that,  given two adjacent $\cD$-cubes $Q$ and $Q'$ in $\cD(\bV;\bW)$, all $\cD$-modifications (of types 1 and 2) in $3Q$ and $3Q'$ leave the essential partition $3\bV$ unmodified on the common face $(3Q)\cap (3Q')$. This is summarized in the following definition and corollary; see Definition \ref{def:acb} for notations $\widetilde{\Gamma}(\cdot)$ and $\ell_{\mathrm{bb}}(\cdot)$. 

Let $\bV=(V_1,V_2,V_3)$ be an essential partition of an $n$-cell by $n$-cells so that $V_p$ is a dented molecule for $p=1,2,3$. Let $\bW=(W_1,W_2,W_3)$ be an essential partition satisfying $|\bV|=3|\bW|$ and let $\cD'\subset \cD(\bV,\bW)$ be a non-empty subfamily.

\begin{definition}
An essential partition $\bU$ of $3|\bV|$ into $n$-cells is \emph{obtained by $\cD$-modifications in $\cD'$ from essential partitions $\bV$ relative to $\bW$} if $\bU$ satisfies the  tripod property in each cube in $3\cD'$ and 
\begin{itemize}
\item[(a)] $\bU - |3\cD'| = 3\bV - |3\cD'|$,
\item[(b)] for each cube $C\in 3\cD'$, the essential partition $\bU\cap C$ is obtained by a $\cD$-modification, that is, $\bU$ has the properties (1)--(3) of  Lemma \ref{lemma:ra_b} or \ref{lemma:ra_b2} relative to $C$.
\item[(c)] each leaf $A \in \Gamma(U_i)$ is a $1$-fine atom adjacent to a $3$-fine atom $A'=3a'$, where $a'$ is a leaf in $\Gamma(V_i)$, and
\item[(d)] to each leaf $a\in \Gamma(V_i)$ correspond at most $3\max_{a'} \ell_{\mathrm{bb}}(a')$ leaves in $\Gamma(U_i)$ adjacent to $3a\in \Gamma(U_i)$, where the maximum is taken over the leaves $a'$ in $\Gamma(V_i)$.
\end{itemize}
For $\cD'=\cD(\bV;\bW)$, we say that $\bU$ is \emph{obtained by $\cD$-modification from $\bV$ relative to $\bW$}. 
\end{definition}

\begin{corollary}
\label{cor:ra_b_all}
Let $\bV=(V_1,V_2,V_3)$ be an essential partition of an $n$-cell by $n$-cells so that $V_p$ is a dented molecule for $p=1,2,3$, and let $\bW=(W_1,W_2,W_3)$ be an essential partition satisfying $|\bV|=3|\bW|$. Given a non-empty subfamily $\cD'\subset \cD(\bV;\bW)$ there exists an essential partition $\bU$ which is obtained by $\cD$-modification in $\cD'$ from $\bV$ relative to $\bW$.
\end{corollary}


\subsubsection{$\cC$-modification}
\label{sec:Cmod}

The following rearrangement is a \emph{$\cC$-modification}.
\begin{lemma}
\label{lemma:ra_a}
Let $V$ be an $n$-cell and $\bV=(V_1,V_2,V_3)$ an essential partition of $V$. Suppose $Q\in \cC(\bV)$ has color $i$ in $\bV$, and let $j$ and $k$ be complementary colors. Then there exist atoms $A_j$ and $A_k$ in $3Q$ which are composed of building blocks along $\partial (3Q)$ so that $U_i = 3V_i - (A_j\cup A_k)$, $U_j = 3V_j \cup A_j$, and $U_k = 3V_k \cup A_k$ are $n$-cells and 
\begin{equation}
\label{eq:C_cubes}
\bU = (U_1,U_2,U_3) 
\end{equation}
is an essential partition of $3V$ into $n$-cells having the tripod property in $3Q$. Moreover, 
\begin{itemize}
\item[(1)] $(A_j\cup A_k)\cap \partial (3Q) \subset \partial_\cup 3\bV$ and
\item[(2)] $(\partial_\cup \bU)\cap 3Q \subset |\cD(\bU;\bV)|\cup |\cN(\bU;\bV)|$.
\end{itemize}
\end{lemma}

\begin{figure}[h!]
\includegraphics[scale=0.30]{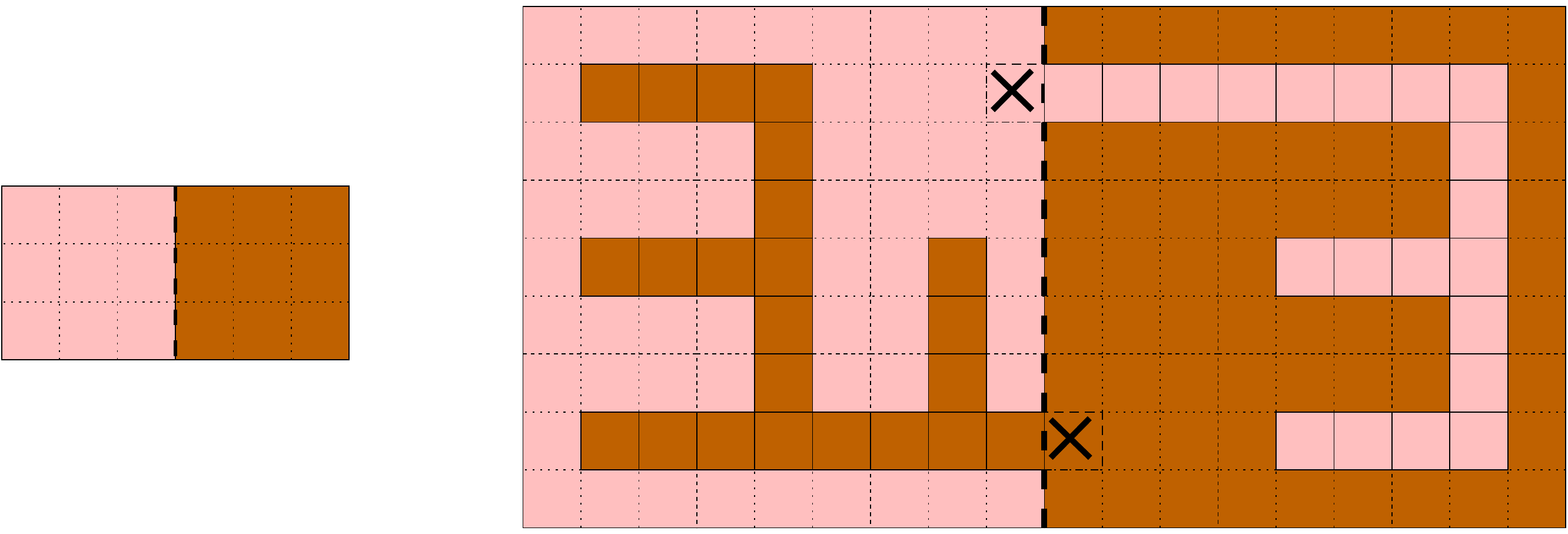}
\caption{Cube $Q$ and essential partition $\bU$ in $3Q$.} 
\label{fig:C_example_2b}
\end{figure}

\begin{proof}[Proof of Lemma \ref{lemma:ra_a}]
The proof is a straightforward application of Lemma \ref{lemma:tripod_non_planar} to appropriate non-planar initial data. 

For notational convenience, take $i=3$. Let $q_1\subset V_1$ and $q_2 \subset V_2$ be unit cubes as in Definition \ref{def:C_cube_extended}. For $p=1,2$, let $\cF^p$ be the collection of faces of $Q$ which meet $V_p$ in an $(n-1)$-cell. Then 
\[
(Q,(\cF^1, \cQ(Q;\cF^2),q_1),(\cF^2,\cQ(Q;\cF^1),q_2))
\]
are non-planar initial data; cf.\;Definition \ref{def:sf}. 

Let $\Sigma$ be a spanning forest as in Lemma \ref{lemma:gv_non-flat} and $A_1$ and $A_2$ be atoms associated to $\Sigma$ as in Lemma \ref{lemma:ra_1}. By Lemma \ref{lemma:tripod_non_planar}, the essential partition $(V_1\cup A_1, V_2\cup A_2,V_3-(A_1\cup A_2))$ satisfies the tripod property in $3Q$.

Property (1) follows immediately from (5) in Lemma \ref{lemma:ra_2}, and (2) from the observation that every cube in $3(\cQ(Q;\cF^1)\cup \cQ(Q;\cF^2))$ is either a $\cD$- or $\cN$-cube.
\end{proof}

It is obvious that $\cC$-modifications are performed independently. We formalize this in the following definition and corollary. Let $\bV$ be an essential partition of an $n$-cell into $n$-cells and let $\cC'\subset \cC(\bV)$ be non-empty.

\begin{definition}
An essential partition $\bU$ of $3|\bV|$ into $n$-cells is \emph{obtained by $\cC$-modification in $\cC'$ from $\bV$} if $\bU$ satisfies the tripod property in each cube in $3\cC'$ and 
\begin{itemize}
\item[(a)] $\bU - |3\cC'| = 3\bV - |3\cC'|$,
\item[(b)] if $C\in 3\cC'$, then $\bU\cap C$ is obtained from $3\bV$ by a $\cC$-modification, that is, $\bU$ satisfies the properties of Lemma \ref{lemma:ra_a} relative to $C$, and
\item[(c)] $\partial_\cup \bU \cap |3\cC'| \subset |\cD(\bU;\bV)\cup \cN(\bU;\bV)|$.
\end{itemize}  
\end{definition}

\begin{corollary}
\label{cor:ra_a_all}
Let $\bV=(V_1,V_2,V_3)$ be an essential partition of an $n$-cell so that $V_p$ is a dented molecule for $p=1,2,3$. Let $\cC'\subset \cC(\bV)$ be a non-empty subfamily. Then there exists an essential partition $\bU=(U_1,U_2,U_3)$ of $|3\bV|$ which is obtained by $\cC$-modification in $\cC'$.
\end{corollary}


\subsubsection{Secondary $\cC$- and $\cN$-modifications}
\label{sec:secondary_mods}

The $\cC$-modification in Lemma \ref{lemma:ra_a} is a 'primary' $\cC$-modification. To illustrate the necessity of 'secondary' $\cC$- and $\cN$-modifications, consider the following example. This is necessitated by the presence of neglected faces (Definition \ref{def:neglected_face}).

\begin{example}
\label{ex:C_problem}
Let $\bW=(W_1,W_2,W_3) = ([0,3]^3,[0,3]^2\times  [-3,0],[3,6]\times [0,3]^2)$. The cube $Q=[0,3]^3$ is a $\cC$-cube of color $1$ in $\bW$. 

Using Lemma \ref{lemma:ra_a} we perform a $\cC$-modification in $C=3Q$, that is, obtain the essential partition $\bV$ relative to $C$ as in Lemma \ref{lemma:ra_a}; see Figure \ref{fig:C_example_2b}. 

Consider now the essential partition $3\bV$. One possible essential partition $\bU''$ of $|3\bV|$ is in Figure \ref{fig:C_example_pre_final}. Notice in Figure \ref{fig:C_example_2b} that $\bV$ already has $3$ neglected faces each of which meets the vertical fold. Thus $\bU''$ cannot satisfy the tripod property. A glance at Figure \ref{fig:C_example_pre_final} also shows that in addition there are now $3$ cubes, in fact cubes in $3\cN(\bV;\bW)$, of side length $9$ in $\bU''$ which cannot be partitioned into $\cC$- and $\cD$-cubes.

\begin{figure}[h!]
\includegraphics[scale=0.20]{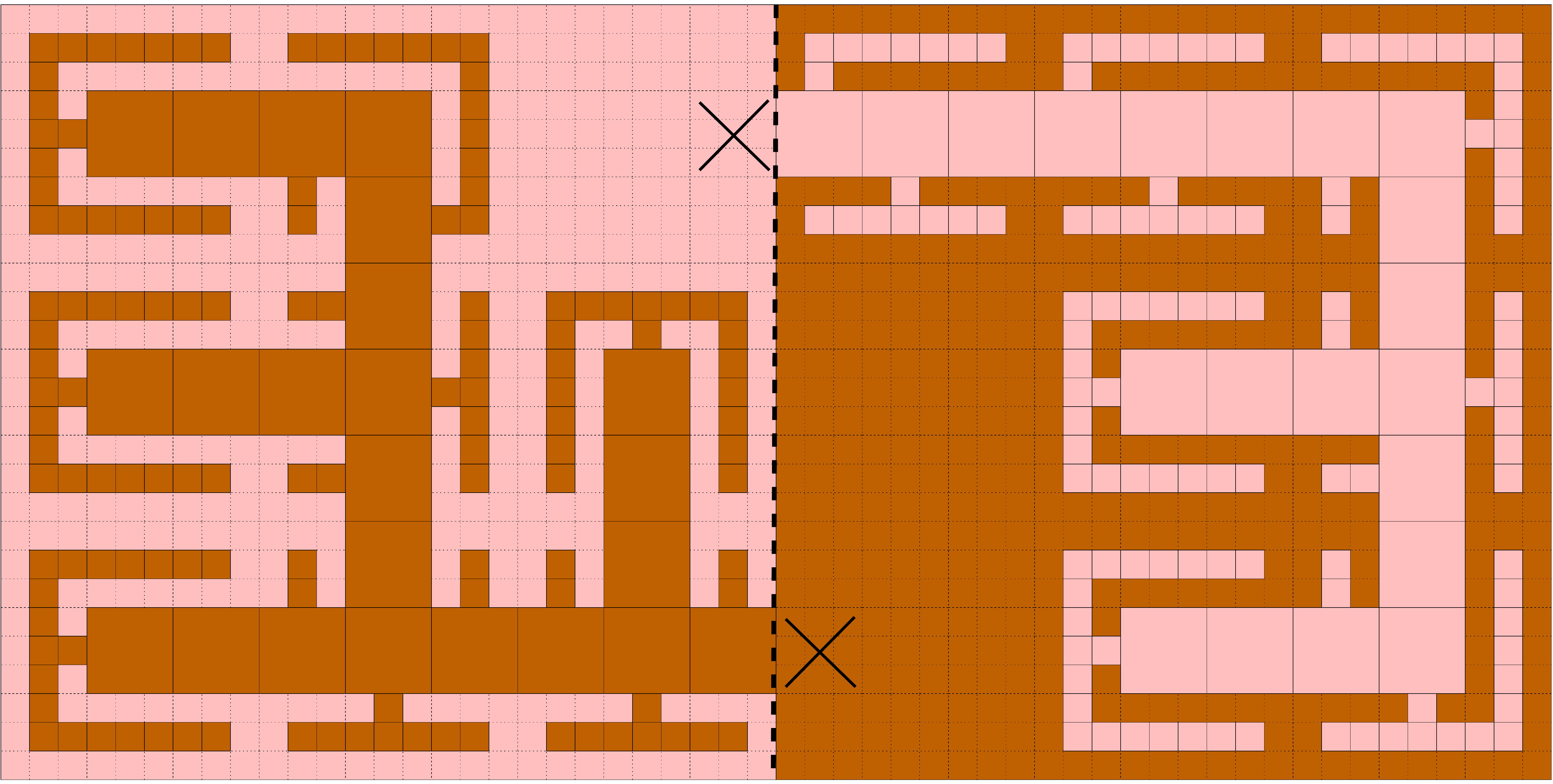}
\caption{Essential partition $\bU''$ in $9Q$.}
\label{fig:C_example_pre_final}
\end{figure}

In this situation, we achieve the tripod property with a 'secondary' $\cC$-modification in $3\bV$. The procedure imitates that of Lemma \ref{lemma:ra_b}, and takes into account both neglected faces and $\cD$-cubes; see Figure \ref{fig:C_example_final}. 

\begin{figure}[h!]
\includegraphics[scale=0.20]{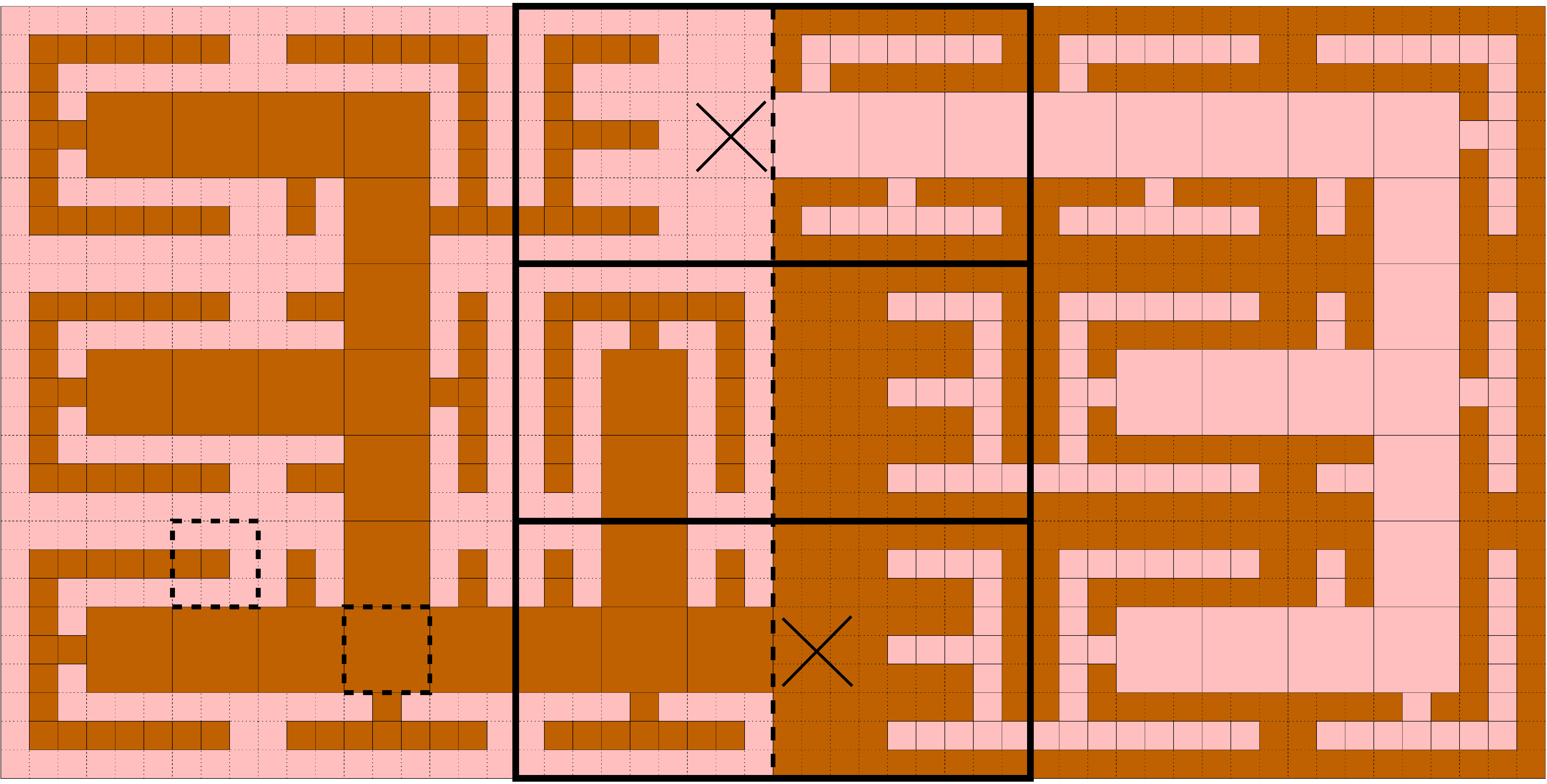}
\caption{An essential partition $\bU$ obtained by a secondary $\cC$-modification from $3\bV$. The three cubes of $\cN_2(\bU;\bV,\bW)$ are highlighted with solid lines and a $\cC$-cube and a $\cD$-cube highlighted with dashed lines.}
\label{fig:C_example_final}
\end{figure}

To be more precise, the essential partition $\bU$ is obtained by a secondary $\cC$-modification as follows. We observe first that $\partial_\cup \bV\cap 9Q \subset |\cD(\bV;\bW)|\cup |\cN(\bV;\bW)|$. Let also $A_j = C\cap V_j$ for each color $j=1,2$.

We perform now an independent $\cD$-modification in each cube in $3\cD(\bV;\bW)$. For each cube $Q' \in 3\cN(\bV;\bW)$, we extend either $3A_1$ or $3A_2$ from outside $Q'$ into $Q'$ using a $1$-fine atom; after this extension $Q'$ meets each color in its interior. This extension uses Lemma \ref{lemma:ra_1_ext}. Let $\bU$ be the essential partition obtained this way. Note that $3\cN(\bV;\bW) = \cN_2(\bU;\bV,\bW)$. 

Regarding the construction of the sequence $(\bfOmega_m)$, we may take $\bfOmega_2 = \bU$ if we have $\bV=\bfOmega_1$ and $\bW = \bfOmega_0$, since the tripod property is now satisfied.

The secondary $\cN$-modification in cubes $3\cN_2(\bU;\bV,\bW)$ is defined as follows. Let $Q'\in \cN_2(\bU;\bV,\bW)$. By Lemma \ref{lemma:ra_1_ext}, $U_j \cap Q'$ is a disjoint union of a molecule and a $1$-fine atom, each of a different complementary color. Moreover, there are three types of cubes in $Q'$, that is, sets $\cC(\bU\cap Q')$, $\cD(\bU\cap Q';\bV\cap (1/3)Q')$ and $\cN(\bU\cap Q';\bV\cap (1/3)Q')$ are all non-empty; note that all cubes of $\cN(\bU\cap Q';\bV\cap (1/3)Q')$ and none of $\cD(\bU\cap Q';\bV\cap (1/3)Q')$ meet the two distinguished faces of $Q'$. Secondary $\cN$-modification in $3Q'$ consists of independent $\cC$- and $\cD$-modifications and applications of Lemma \ref{lemma:ra_1_ext} and is similar to the case of a secondary $\cC$-modification.
\end{example}

We formalize now the secondary modifications in the form of lemmas. The proofs of those lemmas follow proofs of Lemmas \ref{lemma:ra_b}, \ref{lemma:ra_b2}, and \ref{lemma:ra_a}. Since the main difficulty is in their formulation, we leave the details of the proofs to the interested reader. For the statement, we give the following definition.

\begin{definition}
An essential partition $\bU$ is \emph{obtained by a secondary modification from $\bV$ with respect to $\bW$ in a cube $C=3Q$ of side length $27$} if $\bU$ satisfies the tripod property in $C$, $\bV$ and $\bW$ are essential partitions satisfying $|\bU|=3|\bV|=9|\bW|$, and there are colors $j$ and $k$ and molecules $M_j$ and $M_k$ in $3Q$ so that $U_i = 3V_i - (M_j\cup M_k)$, $U_j=3V_j\cup M_j$, and $U_k= 3V_k\cup M_k$ are $n$-cells, and these molecules satisfy the following conditions: 
\begin{itemize}
\item[(1)] $(M_j\cup M_k)\cap \partial (3Q) \subset \partial_\cup 3\bV$,
\item[(2)] $(\partial_\cup \bU)\cap 3Q \subset |\cC(\bU)|\cup |\cD(\bU;\bV)|\cup |\cN_2(\bU;\bV,\bW)|$,
\item[(3)] $3A_j\subset M_j$ and $3A_k \subset M_k$, 
\item[(4)] for $p=j,k$, $M_p-3A_p$ consists of pair-wise disjoint $1$-fine atoms made of building blocks,
\item[(5)] when $n>3$ and $p=1,2,3$, each building block in $\tilde \Gamma(3A_p)$ meets at most one atom in $M_p-3A_p$,
\end{itemize}
where $A_j$ and $A_k$ are atoms in $Q$ satisfying $V_i = 3W_i - (A_j\cup A_k)$, $V_j=3W_j\cup A_j$, and $V_k=3W_k\cup A_k$; here $\bV=(V_1,V_2,V_3)$ and $\bW=(W_1,W_2,W_3)$.
\end{definition}

\begin{lemma}[Secondary $\cC$-modification]
\label{lemma:ra_c}
Let $\bV$ and $\bW$ be essential partitions satisfying $|\bV|=3|\bW|$, and suppose that $\bV$ has been obtained by a $\cC$-modification in $(1/3)Q$ from $\bW$. Then there exists an essential partition $\bU$ of $3|\bV|$ which is obtained by a secondary modification from $\bV$ with respect to $\bW$ in $3Q$.
\end{lemma}

\begin{lemma}[Secondary $\cN$-modification]
\label{lemma:ra_n}
Let $\bV$, $\bW$, and $\bW'$ be essential partitions satisfying $|\bV|=3|\bW|=9|\bW'|$, and let $Q\in \cN_2(\bV;\bW,\bW')$ be a secondary $\cN$-cube so that $\bV$ is obtained by a secondary modification in $(1/3)Q$ from $\bW$ with respect to $\bW'$. Then there exists an essential partition $\bU$ of $3|\bV|$ so that $\bU$ is obtained by a secondary modification from $\bV$ with respect to $\bW$ in $3Q$.
\end{lemma}


\subsubsection{The Machine}
\label{sec:prep_induction}
Having all necessary modifications now at our disposal, we introduce the main induction step. Corollary \ref{cor:pre_induction} will summarize the process.

Let $\bV$, $\bW$, and $\bW'$ be essential partitions satisfying $|\bV|=3|\bW|=9|\bW'|$. Suppose that all secondary $\cC$-cubes in $\cC_2(\bV;\bW)$ and all secondary $\cN$-cubes $Q$ in $\cN_2(\bV;\bW,\bW')$ are obtained by secondary $\cC$- and $\cN$-modifications from $(1/3)Q$, respectively. Suppose also that $|\bV|\cap |3\cC(\bW)|$ is obtained by a series of independent $\cC$-modifications and $|\bV|\cap |3\cD(\bW;\bW')|$ by a series of independent $\cD$-modifications.

Based on the modifications introduced in this section (Lemmas \ref{lemma:ra_b}, \ref{lemma:ra_a}, \ref{lemma:ra_c}, and \ref{lemma:ra_n}) we note:
\begin{itemize}
\item[(a)] the sets $|\cC(\bV)|$ and $|\cD(\bV;\bW)|$ are essentially disjoint and
\item[(b)] sets $|\cC_2(\bV;\bW)|$ and $|\cN_2(\bV;\bW,\bW')|$ are essentially disjoint.
\end{itemize}
Indeed, by definition of $\cC$- and $\cD$-cubes, $\cC(\bV)\cap \cD(\bV;\bW)=\emptyset$. The claim (a) now follows from observation that cubes in $(1/3)\left( \cC(\bV)\cup \cD(\bV;\bW)\right)$ are unit cubes. The claim (b) follows now with a similar argument. 

It is essential to notice that cubes in $\cC_2(\bV;\bW)\cup \cN_2(\bV;\bW,\bW')$ contain cubes in $\cC(\bV)\cup \cD(\bV;\bW)$, that is, the intersection $|\cC(\bV)\cup \cD(\bV;\bW)| \cap |\cC_2(\bV;\bW)\cup \cN_2(\bV;\bW,\bW')|$ has non-empty interior; see e.g.\;Figure \ref{fig:C_example_final}. For this reason, and to have a well-defined order of independent rearrangements, we set 
\[
\mathfrak{R}(\bV;\bW,\bW') = \{ Q\in \cC(\bV)\cup \cD(\bV;\bW) \colon Q \not \subset  |\cC_2(\bV;\bW)\cup \cN_2(\bV;\bW,\bW')|\};
\]
here 'R' stands for 'remaining'.

The inductive process can now be organized in a form of a list of operations to be performed. Given a collection $\fL_\bW$ of essentially disjoint cubes in $|\bW|$, we define the \emph{list $\fL_{\bV}$ with respect to the history $(\bW,\bW',\fL_\bW)$} by  
\[
\begin{split}
\fL_\bV&=\fL(\bV;\bW,\bW',\fL_{\bW}) \\ 
&= \{ Q\in \cC_2(\bV;\bW)\cup \cN_2(\bV;\bW,\bW') \cup \mathfrak{R}(\bV;\bW,\bW') \colon Q\subset 3|\fL_\bW|\}.
\end{split}
\]
Note that cubes in $\fL(\bV;\bW,\bW',\fL_\bW)$ are  pair-wise essentially disjoint and have side length either $3$ or $9$.

\begin{remark}
\label{rmk:flat_top}
To see how the collection $\fL_\bW$ organizes the modification process, consider the essential partitions $\bfOmega_0$ and $\bfOmega_1$. The essential partition $\bfOmega_0$ has one $\cC$-cube $Q$ and, after scaling, we perform the (only possible) rearrangement in $3Q$. Thus the essential partition $\bfOmega_1$ has one secondary $\cC$-cube $3Q$. There are also several $\cC$-cubes in $|\bfOmega_1|-(3Q)$, contained in $\Omega_{1,2}\cup \Omega_{1,3}$, but all these $\cC$-cubes have one face in $3Q$. Since it suffices to perform a rearrangement on one side of $\partial_\cup 3\bfOmega_1$, we perform the secondary $\cC$-modification in $9Q$ and ignore the other $\cC$-cubes. Thus, if we set $\fL_{\bfOmega_0}=\{Q\}$ and $\fL_{\bfOmega_1}=\{3Q\}$, the rearrangement to obtain $\bfOmega_2$ is performed in $3|\fL_{\bfOmega_1}|$. We follow this general principle of nested rearrangements throughout the construction. For example, for the next rearrangement we define $\fL_{\bfOmega_2} = \fL(\bfOmega_2;\bfOmega_1,\bfOmega_0,\fL_{\bfOmega_1})$, and rearrangements take place in $3|\fL_{\bfOmega_2}|$; note that $\fL_{\bfOmega_2}$ consists of $\cC$- and $\cD$-cubes and secondary $\cN$-cubes as mentioned in Example \ref{ex:C_problem}. In particular, we have $9 |\fL_{\Omega_0}| = 3|\fL_{\bfOmega_1}|\supsetneq |\fL_{\bfOmega_2}|$.
\end{remark}

\begin{definition}
Given essential partitions $\bV$, $\bW$, and $\bW'$ satisfying $|\bV|=3|\bW|=9|\bW'|$ and a list $\fL_{\bW}$ of cubes in $\bW$, an essential partition $\bU$ is \emph{properly obtained (following $\fL_\bV=\fL(\bV;\bW,\bW',\fL_{\bW}))$} if $\bU$ is obtained 
\begin{itemize}
\item by $\cC$-modification in $3\left( \cC(\bV)\cap \fL_{\bV}\right)$, 
\item by $\cD$-modification in $3\left(\cD(\bV,\bW) \cap \fL_{\bV}\right)$, and 
\item by secondary modification in $3\left((\cC_2(\bV;\bW)\cup \cN_2(\bV;\bW,\bW'))\cap \fL_{\bV}\right)$.
\end{itemize}
\end{definition}

These modifications now yield the following corollary, which can be viewed as the induction step in the construction; the specific sequence $(\bfOmega_m)$ satisfying Theorem \ref{thm:RP} appears in the next section. 

\begin{corollary}
\label{cor:pre_induction}
Let $\bV$, $\bW$, and $\bW'$ be essential partitions for which $|\bV|=3|\bW|=9|\bW'|$ and let $\fL_{\bW}$ be list of cubes in $|\bW|$. Suppose $\bV$ is properly obtained following $\fL_{\bV}=\fL(\bV;\bW,\bW',\fL_{\bW})$ and suppose that $\bV$ satisfies the tripod property and $\partial_\cup \bV\subset |\fL_{\bV}|$. 

Then there exists a properly obtained essential partition $\bU$ satisfying $|\bU|=3|\bV|$ and $\partial_\cup \bU\subset |\fL_{\bU}|\subset 3|\fL_{\bV}|$, where $\fL_{\bU}=\fL(\bU;\bV,\bW,\fL_{\bV})$. In particular, $\bU$ satisfies the tripod property.
\end{corollary}

\begin{proof}
It suffices to note that $\bU$ is obtained by independent modifications in each cube $3Q$ for $Q\in \fL(\bV;\bW,\bW',\fL_{\bW})$, and is hence properly obtained. These modifications also yield a new list $\fL(\bU;\bV,\bW,\fL_{\bV})$; Lemmas \ref{lemma:ra_b}, \ref{lemma:ra_b2}, \ref{lemma:ra_a}, \ref{lemma:ra_c}, and \ref{lemma:ra_n} cover the possible situations of different modifications. Thus $\partial_\cup \bU\subset |\fL(\bU;\bV,\bW,\fL_{\bV})|$ and $\bU$ satisfies the tripod property.
\end{proof}

\subsection{Inductive construction} 

Throughout this section, $\bfOmega_0$ and $\bfOmega_1$ are essential partitions defined in Sections \ref{sec:0ra} and \ref{sec:1ra}.

\begin{proposition}
\label{prop:induction_clean}
Let $n\ge 3$, $\bfOmega_0=([0,3]^n, [0,3]^{n-1}\times [-3,0], [3,6]\times [0,3]^{n-1})$ and let $\bfOmega_1$ be an essential partition as in Section \ref{sec:1ra}. Then there exist essential partitions $\bfOmega_m=(\Omega_{m,1},\Omega_{m,2},\Omega_{m,3})$ for $m\ge 1$ satisfying the tripod property and the following conditions:
\begin{itemize}
\item[(a)] $|\bfOmega_m| = 3|\bfOmega_{m-1}|$,
\item[(b)] $\partial_\cup \bV_m \subset |\fL(\bfOmega_m; \bfOmega_{m-1},\bfOmega_{m-2},\fL_{\bfOmega_{m-1}})|$
\item[(c)] all cubes in $\fL_{\bfOmega_m}=\cL(\bfOmega_m; \bfOmega_{m-1},\bfOmega_{m-2},\fL_{\bfOmega_{m-1}})$ are properly obtained,
\item[(d)] $\bfOmega_\ell\cap 3^{m-2}|\bfOmega_0| = \bfOmega_m\cap 3^{m-2}|\bfOmega_0|$ for all $\ell>m$.
\end{itemize}
In addition, there exist $\nu\ge 1$ and $\lambda>1$, depending only on $n$, so that for all $m\ge 0$ and each $p=1,2,3$,
\begin{itemize}
\item[(e)] each $\hull(\Omega_{m,p})$ is a $(\nu,\lambda)$-molecule with the atom length of $\Gamma(\hull(\Omega_{m,p}))$ bounded by a constant depending only on $n$, and
\item[(f)] there exists $L=L(n)\ge 1$ and an $L$-bilipschitz map $\psi_{m,p} \colon (\Omega_{m,p}, d_{\Omega_{m,p}}) \to (\hull(\Omega_{m,p}),d_{\hull(\Omega_{m,p})})$ which is the identity on $\Omega_{m,p}\cap \partial \hull(\Omega_{m,p})$. 
\end{itemize}
\end{proposition}

We prove Proposition \ref{prop:induction_clean} first in dimensions $n>3$ and then consider the more complicated dimension $n=3$ separately; see Section \ref{sec:IC_dim3}. Proposition \ref{prop:induction_clean} is obtained in three steps. In higher dimensions, we first construct the sequence $\bfOmega_3, \bfOmega_4, \ldots$ by induction using Corollary \ref{cor:pre_induction} and then check conditions (a)--(d) and the tripod property. Property (e) is more subtle and considered separately in Section \ref{sec:cond_e}. Finally, we prove Property (f), the most demanding part, in Section \ref{sec:cond_f}. For $n=3$, the steps are similar with the exception that we use specific $\cC$- and secondary $\cC$-modifications to meet condition (e).

\subsubsection{Proof of Proposition \ref{prop:induction_clean} in dimension $n>3$}
\label{sec:IC_init}

Consider essential partitions $\bfOmega_1$, $\bfOmega_0$, and $(1/3)\bfOmega_0$ in the r\^ole of the essential partitions $\bV$, $\bW$, and $\bW'$ of Section \ref{sec:prep_induction}. 

We obtain $\bfOmega_1$ by one $\cC$-modification from $3\bfOmega_0$ and take $\bfOmega_2$ to be either the essential partition $\bU$ in Lemma \ref{lemma:ra_c}, or directly apply Corollary \ref{cor:pre_induction}. In particularly, $\bfOmega_2$ satisfies the tripod property and is properly obtained following $\fL_{\bfOmega_1}=\fL(\bfOmega_1;\bfOmega_0,(1/3)\bfOmega_0,\fL_{\bfOmega_0})$, where $\fL_{\bfOmega_0}=\{[0,3]^n\}$; cf.\;Remark \ref{rmk:flat_top}. We take $\fL_{\bfOmega_2}=\fL(\bfOmega_2;\bfOmega_1,\bfOmega_0,\fL_{\bfOmega_1})$.

To meet the stability requirement (d) of the proposition, let $Q_0=[0,1]^{n-1}\cup [-1,1]$, and note that $\bfOmega_2\cap 9Q_0$ is obtained by a single $\cD$-modification from $3\bfOmega_1\cap 9Q_0$. Thus, by making proper choices in the $\cC$-modification yielding $\bfOmega_1$ and secondary $\cC$-modification yielding $\bfOmega_2$, we may assume that $\bfOmega_2 \cap 3Q_0 = \bfOmega_1\cap 3Q_0$; compare with Figures \ref{fig:1ra} and \ref{fig:2ra} and the discussion in Sections \ref{sec:1ra} and \ref{sec:2ra}. Indeed, using the notation from Section \ref{sec:1ra}, $\Omega_{1,3} = 3\Omega_{0,3}\cup A_3$ and we may assume as in Figure \ref{fig:1ra} that the building block $\Omega_{1,3}\cap 3Q_0$ is a leaf in $\Gamma(A_3)$. Let $\bU$ be an essential partition given by Corollary \ref{cor:pre_induction}, and define $\bfOmega_2$ by $\bfOmega_2 - 3Q_0 = \bU-3Q_0$ and $\bfOmega_2 \cap 3Q_0 = \bfOmega_1\cap 3Q_0$. 

Since $\bfOmega_2 \cap 3Q_0 = \bfOmega_1 \cap 3Q_0$, it is easy to obtain the rest of the sequence $\bfOmega_0,\bfOmega_1,\bfOmega_2,\ldots$ by applying Corollary \ref{cor:pre_induction} to essential partitions $\bfOmega_{m-1}$, $\bfOmega_{m-2}$, and $3\bfOmega_{m-2}$ and modifying the obtained essential partitions as for $m=2$.

Corollary \ref{cor:pre_induction} yields that the essential partitions in the sequence $(\bfOmega_m)$ satisfy the tripod property and conditions (a)--(c). 

\begin{remark}
Recall that, as a direct consequence of the definition, each dented molecule $\Omega_{m,p}$ has a unique essential partition into dented atoms. Recall that the adjacency graph of this essential partition of $\Omega_{m,p}$ into dented atoms is $\Gamma(\Omega_{m,p})$.
\end{remark}

\begin{remark}
Whereas there is no simple inclusion relation between domains $\Omega_{m,p}$ and $\Omega_{m+1,p}$, the trees $\Gamma(\Omega_{m,p})$ and $\Gamma(\Omega_{m+1,p})$ are closely related. Indeed, formally, $\Gamma(\Omega_{m+1,p})$ is obtained by adding leaves to $\Gamma(\Omega_{m,p})$. At the same time, however, a vertex representing an atom of side length at least $3$ in $\Gamma(\Omega_{m,p})$ becomes a dented atom in $\Gamma(\Omega_{m+1,p})$.

Finally, the tree $\Gamma(\Omega_p)$ of the limit $\Omega_p = \bigcup_{m\ge 1} \Omega_{m,p}$ is an inverse limit of the trees $\Gamma(\Omega_{m,p})$.
\end{remark}

\subsubsection{Condition (e)}
\label{sec:cond_e}

We consider first some general properties of dented molecules $\Omega_{m,p}$ and their hulls $\hull(\Omega_{m,p})$ (cf.\;Section \ref{sec:dented_molecules}), and then obtain condition (e) in Proposition \ref{prop:induction_clean}.

\begin{remark}
The trees $\Gamma(\Omega_{m,p})$ and $\Gamma(\hull(\Omega_{m,p}))$ are related since $\Gamma(\hull(\Omega_{m,p}))$ is obtained by removing those (dented) atoms from $\Gamma(\Omega_{m,p})$ which are contained in $\hull(\hull(\Omega_{m,p})-\Omega_{m,p})$. Thus $\Gamma(\hull(\Omega_{m,p}))$ can be viewed as a subtree of $\Gamma(\Omega_{m,p})$ where the remaining vertices are (undented) atoms instead of dented atoms. 
\end{remark}

Recall that a vertex $D\in \Gamma(\Omega_{m,p})$ is internal if there exists a vertex $D'\in \Gamma(\Omega_{m,p})$ so that $\rho(D')>\rho(D)$ and $D\subset \hull(D')$,  and that a vertex is external if not internal (Definition \ref{def:int_ext}). 

\begin{remark}
\label{rmk:dents}
Although we focus one of the domains in the following lemma, it should be noted that the other two domains also have a r\^ole, since they create the dents. This is crucial for the combinatorics to settle (e) and (f). Suppose $D\in \Gamma(\Omega_{m,p})$ is an internal vertex and $D'\in \Gamma(\Omega_{m,p})$ is a vertex closest to $D$ in $\Gamma(\Omega_{m,p})$ satisfying $D\subset \hull(D')$. 

Then $D$ is contained in a dent $M'$ of $D'$. This dent is a vertex in $\Gamma(\Omega_{m,r})$ for $r\ne p$. Note also that since $D$ is contained in a dent of $M'$, we have $\rho(\hull(D'))\ge 3^2 \rho(\hull(M'))\ge 3^4 \rho(\hull(D))$.

\end{remark}

\begin{lemma}
\label{lemma:mods_meet_exts}
Let $m>1$. Suppose $A$ is a leaf in $\Gamma(\Omega_{m,p})$ and let $3^k\in \{1,3\}$ be the side length of $A$. Let $D$ be the vertex adjacent to $A$ in $\Gamma(\Omega_{m,p})$ satisfying $\rho(D)>\rho(A)$. Then $3^{-k}A$ is a leaf of $\Omega_{m-k,p}$. 

If $\rho(D)>3\rho(A)$, then the atom $3^{-k}A$ arose from a $\cC$-modification and $A$ is an internal vertex of $\Gamma(\Omega_{m,p})$,

Otherwise, $\rho(D)=3\rho(A)$ and $3^{-k}A$ came from a $\cD$-modification or a secondary modification. Furthermore, in this case, $A$ is an external vertex of $\Gamma(\Omega_{m,p})$ if and only if $D$ is an external vertex of $\Gamma(\Omega_{m,p})$.
\end{lemma}

\begin{remark}
Note that, whereas the number of atoms $A$ attached to $D$ in Lemma \ref{lemma:mods_meet_exts} satisfying $\rho(D)=3\rho(A)$ is uniformly bounded, there will be no upper bound for the number of atoms $A$ attached to $D$ in general. This follows from the observation that there is no upper bound for the size of a dent of a dented molecule and each cube in a dent is penetrated by a (dented) molecule which is an internal vertex attached to $D$. Note that trees $\Gamma(\Omega_{m,p})$ have internal vertices only for $m\ge 3$. 

Note also that the essential partition $\bfOmega_1$ is exceptional in the following sense. The molecule $\Omega_{1,p}$, for $p=2,3$, is obtained from $3\Omega_{0,p}$ by a $\cC$-modification but the leaf $\Omega_{1,p}-3\Omega_{0,p}$ is not contained in a dent of $\Omega_{1,p}$. This is the one case in the construction of the sequence $(\bfOmega_m)$ when a $\cC$-modification does not procude an internal vertex.
\end{remark}

\begin{proof}[Proof of Lemma \ref{lemma:mods_meet_exts}]
Since $A$ is a leaf, it is an atom. Moreover, $\rho(D)\ge 3\rho(A)$ by construction.

First observe that if $3^{-k}A$ is obtained by a $\cC$-modification, there exists a dent $M$ of $3^{-k}D$ with $3^{-k}A\subset M$. Since $M\subset \hull(3^{-k}D)$, it follows that $A\subset \hull(D)$. Thus in this case $A$ is internal and $\rho(A)\le 3^{-4}\rho(D)$. 

Since the ratio of side lengths in a $\cD$-modification and in secondary modifications is $3$, the atom $3^{-k}A$ is obtained by a $\cD$-modification or a secondary modification if and only if $\rho(D)=3\rho(A)$.

Suppose now that $\rho(D)=3\rho(A)$. We show that $A$ is an internal vertex if and only if $D$ is an internal vertex.

\emph{Suppose first that $A$ is an internal vertex.} Then there exists $D'\in \Gamma(\Omega_{m,p})$ containing $A$ in its hull. Let $M'$ be the dent of $D'$ containing $A$. By properties of modifications, we have either $D\subset M'$ or $M'\subset \hull(D)$, since $D$ is adjacent to $A$ and $A\subset M'$. Since $\rho(M')\ge 9\rho(A)=3\rho(D)$, we have $D\subset M'$. Thus $D$ is internal.

\emph{Suppose now that $D$ is an internal vertex.} Then there exists $D_0\in \Gamma(\Omega_{m,p})$ and a dent $M_0$ of $D_0$ satisfying $D\subset M_0 \subset \hull(D_0)$. We may assume that $D_0$ is minimal in the sense that, for every $D'\in \Gamma(\Omega_{m,p})$ between $D$ and $D_0$ in $\Gamma(\Omega_{m,p})$, $D\not\subset \hull(D')$. 

Let $D_1,\ldots, D_\ell$ be the shortest path in $\Gamma(\Omega_{m,p})$ from $D_0$ to $D$ so that $D_1$ is adjacent to $D_0$. Then $D_1\subset M_0$ and we note that $\rho(\hull(D_1))^{-1}\hull(D_1)$ has been obtained by a $\cC$-modification in a cube $Q$ of side length $9$.

Furthermore, by properties of modifications, we observe that all modifications to obtain dented molecules $D_1,\ldots D_\ell$ take place in cubes $3^jQ$ where $3^j \le \rho(\hull(D_1))$. Thus all dented atoms $D_1,\ldots, D_\ell$ are contained in the cube $Q':=\rho(\hull(D_1))Q\subset M_0$. In particular, $D\subset Q'$. Since $A$ is obtained from $D$ by either a $\cD$- or secondary modification, we also have $A\subset Q'\subset M_0$. Thus $A$ is internal.
\end{proof}

\begin{lemma}[Property (e)]
\label{lemma:hull_regularity}
Let $n>3$. There exist $\nu\ge 1$ and $\lambda>1$ depending only on $n$ so that 
the adjacency tree $\Gamma(\hull(\Omega_{m,p}))$ is a $(\nu,\lambda)$-molecule for every $m\ge 2$ and each $p$.
\end{lemma}

\begin{proof}
By Lemma \ref{lemma:dichotomy}, $\Gamma(\hull(\Omega_{m,p}))$ is isomorphic to the tree $\Gamma_E(\Omega_{m,p})$ of external vertices of $\Gamma(\Omega_{m,p})$, and Lemma \ref{lemma:mods_meet_exts} shows that external vertices arise from $\cD$-modifications or secondary modifications. Thus it suffices to estimate the number of atoms created by a $\cD$-modification or secondary modification for $m>2$.

Let $1<k<m$, and let $A$ be an atom in $\Gamma(\Omega_{k,p})$ created by a $\cD$-modification or secondary modification. Then $A$ has side length $1$ and it is contained in a union of at most two cubes of side length $9$. Since there exist $3^n$ essentially disjoint cubes of side length $3$ in a cube of side length $9$, the atom $A$ consists of strictly less than $2\cdot 3^n$ building blocks; see Remark \ref{rmk:why_3_is_special} below. Since we attach at most one atom to builing block of $3A$, this modification of $3A$ attaches strictly less than $2\cdot 3^n$ atoms. We conclude that the tree $\Gamma(\hull(\Omega_{m,p}))$ is at most $(2\cdot 3^n)$-valent.

To show that $\hull(\Omega_{m.p})$ is $\lambda$-collapsible for some $\lambda > 1$, let $M\in \Gamma(\hull(\Omega_{m,p}))$ be a molecule of side length $3^k$. Then $M$ is attached to at most $2\cdot 3^n$ molecules of side length $3^{k-1}$ and to one molecule $M'$ of side length $3^{k+1}$. Let $F'$ be the face of a cube in $M'$ where $M$ and $M'$ meet.

Let $\varepsilon>0$ to be fixed in a moment, and take  $\ell$ with $(1+\varepsilon)3^{k-1} \ell \le 3^{k+1} < (1+\varepsilon)3^{k-1}(\ell +1)$. Then there exist at least $\ell^{n-1}$ pair-wise disjoint $(n-1)$-cubes of side length $(1+\varepsilon)\cdot 3^{k-1}$ on $F$. Since
\begin{equation}
\label{eq:ell_eps}
\ell^{n-1} > \left( \frac{9}{1+\varepsilon}-1\right)^{n-1},
\end{equation}
we may fix $\varepsilon>0$ small enough, depending on $n$, so that 
\begin{equation}
\label{eq:ell}
\ell^{n-1} > 2\cdot 3^n
\end{equation}
when $n\ge 4$. We conclude that there exists $\lambda>1$, depending only on $n$, for which $M$, and hence $\hull(\Omega_{m,p})$, is $\lambda$-collapsible.
\end{proof}

\begin{remark}
\label{rmk:why_3_is_special}
Note that, although estimates \eqref{eq:ell_eps} and \eqref{eq:ell} hold also for $n=3$, the number of building blocks in an atom $A$ no longer is an upper bound for atoms attached to $3A$. In fact, $\cD$-modification in dimension $3$ may attach as many as $3$ atoms to a single building block; cf.\;Figure \ref{fig:1face3_molecule}.
\end{remark}

\subsubsection{Condition (f)}
\label{sec:cond_f}
It suffices to consider $m\ge 4$. Let $p\in \{1,2,3\}$. To simplify notation, set $V=\Omega_{m,p}$. 

\begin{lemma}
\label{lemma:f}
There exist $L=L(n)\ge 1$ and an $L$-bilipschitz map $\varphi \colon (V,d_V) \to (\hull(V),d_{\hull(V)})$ which is the identity on $V\cap \partial \hull(V)$.
\end{lemma}

We begin the proof of Lemma \ref{lemma:f} with two auxiliary lemmas. For the statements, we need some new notation and also use terms from Section \ref{sec:dented_molecules}. 

Let $d\in \Gamma(V)$ be a dented atom and let $D\in \Gamma(V)$ be the unique dented atom adjacent to $D$ satisfying $\rho(D)>\rho(d)$. Let $Q_d$ and $Q'_d$ be the unique (dented) cubes of side length $\rho(d)$ in $d$ and in $D$, respectively, having a common face $F'_d$. Note that $F'_d \subset Q_d\cap Q_D = d\cap D$.

Let $F_d$ be a face of $Q_d$ contained in $\partial d$ sharing an $(n-2)$-cube with $F'_d$. We call $J_d=F_d\star \{x_{Q_d}\}$ and $J'_d = F'_d \star \{x_{Q'_d}\}$ an \emph{internal} and the \emph{external join of $D$}, respectively. Note that $J_d \subset d$ and $J'_d \subset D$.

The first key ingredient in the proof of Lemma \ref{lemma:f} is a bilipschitz equivalence property for expanding descendants; recall Definitions \ref{def:expanding} and \ref{def:ph} of expanding descendants and partial hull, respectively, in Section \ref{sec:dented_molecules}.

\begin{lemma}
\label{lemma:expanding_removability}
Let $P$ be a partial hull of $V$ and let $D\in \Gamma(P)$ be a dented atom having only expanding descendants. Then there exist $L=L(n)\ge 1$ and an $L$-bilipschitz map $\varphi_D \colon (|\Gamma(P)_D|,d_{|\Gamma(P)_D|}) \to (\hull(D), d_{\hull(D)})$ which is the identity on $D\cap \partial \hull(D)$. Moreover, for any decendant $d\in \Gamma(P)$ of $D$, $\varphi_D(|\Gamma(P)_d|)\subset J_d$.
\end{lemma}
\begin{proof}
By Lemma \ref{lemma:hull_regularity}, for every descendant $d\in \Gamma(P)$, $|\Gamma(P)_d|$ is a collapsible $(\nu,\lambda)$-molecule with $\nu$ and $\lambda$ depending only on $n$. Thus, by Proposition \ref{prop:fRt}, there exist $L'=L'(n)\ge 1$ and an $L'$-bilipschitz mapping $\psi_D \colon (|\Gamma(P)_D|,d_{|\Gamma(P)_D|}) \to (D,d_D)$ which is the identity on $D-\bigcup_d J_d$, where the union is over descendants of $D$.

Proposition \ref{prop:fRt_flat} then produces $L''=L''(n)\ge 1$ and an $L''$-bilipschitz map $\phi_D \colon (D,d_D) \to (\hull(D),d_\hull(D))$ which is the identity on $D\cap \partial \hull(D)$. Furthermore, by a simple modification of the proof of Proposition \ref{prop:fRt_flat}, we may assume that $\phi_D$ is an isometry from $J_d$ to $J'_d$ for each descendant $d$ of $D$.  Thus $\varphi_D = \phi_D \circ \psi_D$ is the desired map.
\end{proof}

The second key ingredient in the proof of Lemma \ref{lemma:f} is a regrouping of joins associated to expanding descendants of large relative side length. We begin by counting the number of descendants, and again need some notation.

Let $P$ be a partial hull of $V=\Omega_{m,p}$. Let $D\in \Gamma(P)$ be a dented atom and $B\in \wt \Gamma(\hull(D))$ a building block. Let $\LE(P,D;B)$ denote the vertices of $\Gamma(P)$ adjacent to $D$ which have side length $3^{-4}\rho(D)$ and are contained in $B$. Note that there are no vertices adjacent to $D$ and contained in $B$ with side length greater than $3^{-4}\rho(D)$; recall Remark \ref{rmk:dents}. 

\begin{lemma}
\label{lemma:entry_counting}
Let $n>3$, $D\in \Gamma(P)$, and $B\in \wt\Gamma(\hull(D))$. Then 
\begin{equation}
\label{eq:LE}
\#\LE(P,D;B) \le 8n^2 3^n.
\end{equation}
\end{lemma}

\begin{proof}
The argument is similar to the counting argument in the proof of Lemma \ref{lemma:hull_regularity}. Let $\rho(D)=3^k$. Let $M_B = B \cap \hull(\hull(D)-D)$. We note first that given a cube $Q\in \Gamma^{\icl}(B)$, $Q\cap M_B$ is a pair-wise disjoint union of two molecules, since $Q\cap M_B$ stems from a $\cC$-modification performed in $3^{k-2}Q$. We also have that $\rho(M_B)=3^{k-2}$. 

Let $U_B\subset M_B$ be the union of the atoms of side length $3^{k-2}$ in $\Gamma(M_B)$. The dented molecules in $\LE(P,D;B)$ are in one-one correspondence with $\Gamma^\icl(U_B)$. Indeed, other cubes in $\Gamma^\icl(M_D)$ have side length at most $3^{k-3}$ and the dented molecules adjacent to $D$ which they contain have side length at most $3^{k-5}$.

Since an atom of side length $3^{k-2}$ in a cube of side length $3^k$ has at most $2n 3^n$ cubes, we have for each $Q\in \Gamma^{\mathrm{int}}(B)$ the estimate
\[
\# \Gamma^\icl(U_B\cap Q) \le 2 \cdot 2n \cdot 3^n = 4n 3^n.
\]

Since  $\#\Gamma^{\icl}(B)< 2n$, i.e\;an $n$-dimensional building block consists of less than $2n$ cubes, we have
\[
\# \Gamma^\icl(U_B) \le 2n \cdot 4n 3^n = 8n^2 \cdot 3^n.
\]
\end{proof}

In the proof of Lemma \ref{lemma:f} we construct a sequence of partial hulls from the dented molecule $V$ to the molecule $\hull(V)$. The crux of the proof is to contract inductively the leaves into joins associated to building blocks and then isometrically move these joins further. The estimate in Lemma \ref{lemma:entry_counting} is used to obtain the necessary collapsibility properties of the partial hulls. We formalize this step in the following lemma.

Let $D\in \Gamma(P)$ be a dented molecule and $B\in \wt\Gamma(\hull(D))$ a building block. Let $Q'_B\in \Gamma(B)$ denote the center of $B$ and by $F'_B$ the unique face of $Q'_B$ contained in $\partial \hull(D)$. Let $Q_B\subset 3^{k-1}(3^{-k}Q'_B)^\#$ be the unique cube of side length $3^{-1}\rho(B)$ having $F_B = Q_B\cap F'_B$ as a face of $Q_B$ with the same barycenter as $F'_B$. We call $J_B = F_B \star \{x_{Q_B}\}$ the \emph{join associated to $B$}.

\begin{lemma}
\label{lemma:Q_flat}
Let $P$ be a partial hull of $V$. Suppose $D\in \Gamma(P)$ is a dented atom. Then there exist $L=L(n)\ge 1$ and an $L$-bilipschitz map 
\[
\psi_D \colon (D,d_{D}) \to (\hull(D),d_{\hull(D)})
\]
which is the identity on $\partial D-\hull(D)$, and which for every $B\in \wt\Gamma(\hull(D))$ satisfies
\begin{itemize}
\item[(1)] $\psi_D(B\cap D)=B$ and
\item[(2)] for each $d\in \LE(P,D;B)$, $\psi_D|J_d$ is an isometric embedding from $J_d$ into $J_B$.
\end{itemize}

In addition, suppose $Q$ is the smallest cube having $D$ on the boundary, and let for every $B\in \wt\Gamma(\hull(D))$, $f_B$ be an $(n-1)$-cube of side length $3^{-4}\rho(B)$ in $B\cap \partial Q$ having distance at least $3^{-4}$ to $\partial B- \partial Q$ and to each $J_d$. Then $\varphi_D|f_B\star \{x_{q_B}\}$ is an isometry into $J_B$, where $x_{q_B}$ is the barycenter of the unique cube $q_B$ in $Q$ having $f_B$ as a face.
\end{lemma}

\begin{proof}
The argument is similar to the collapsing argument in Lemma \ref{lemma:hull_regularity}. Let $\rho(D)=3^k$ and $B\in \wt\Gamma(\hull(D))$. 
 
Since by definition $\rho(F_B)=3^{k-1}$, we may fix $26^{n-1}$ $(n-1)$-cubes of side length $(27/26)3^{k-4}$ in $F_B$. Since Lemma \ref{lemma:entry_counting} yields that 
\[
\#\LE(P,D;B)< 26^{n-1}
\]
and $\rho(J_d) = 3^{-4}\rho(D)=3^{k-4}$, there exists for each $d\in \LE(P,D;B)$ an $(n-1)$-cube $F''_d\subset F_B$ of side length $3^{k-4}$ so that the pair-wise distances of these $(n-1)$-cubes are at least $(1/26)3^{k-4}$. Thus there exist $L=L(n)\ge 1$ and an $L$-bilipschitz map $\psi_B \colon B \to B$ which is the identity on $B-\partial Q$ and which is an isometric embedding from $J_d$ to $F''_d \star \{x_{q''_d}\}$, where $q''_d$ is the unique $n$-cube in $Q$ having $F''_d$ as a face. 

The required mapping $\varphi_D$ is now the composition of the extensions of the various maps $\psi_B$ to all of $D$. We leave the modification of the argument in the case of additional $(n-1)$-cubes $f_B$ for the interested reader.
\end{proof}

\begin{proof}[Proof of Lemma \ref{lemma:f}]
We construct a sequence $P_0,\ldots, P_k$ of partial hulls of $V$ where $P_0=V$ and $P_k = \hull(V)$. In each stage, we remove one dented atom of smallest side length. 

Let $P_0 = V$ and $\cJ_0 =\emptyset$. Suppose that, for $k\ge 0$, we have constructed 
\begin{itemize}
\item[(a)] partial hulls $P_0,\ldots, P_k$ of $V$ so that $P_{\ell+1}$ is a partial hull of $P_\ell$ for $0 \le \ell < k-1$;
\item[(b)] collections $\cJ_0,\ldots, \cJ_k$ of joins associated to building blocks in these partial hulls so that joins $\cJ_\ell$ are contained in atoms of $\Gamma(P_{\ell})$ which are hulls of dented atoms $D$ in $\Gamma(P_{\ell-1})$ for $1\le \ell \le k$, and for such $D$, the number of joins contained in $|\Gamma(P_\ell)_D-D|$ is at most $3^n$, recall that $\Gamma(P_\ell)_D$ is the subtree in $\Gamma(P_\ell)$ behind vertex $D$;
\item[(c)] for every $1\le \ell < k$, an $L$-bilipschitz map $\psi_\ell \colon (P_\ell,d_{P_\ell}) \to (P_{\ell+1},d_{P_{\ell+1}})$, which is the identity on those atoms of $\Gamma(P_\ell)$ which are atoms of $\Gamma(P_{\ell-1})$, where $L$ is at most the product of bilipschitz constants in Lemmas \ref{lemma:expanding_removability} and \ref{lemma:Q_flat}.
\end{itemize}

If $P_k \ne \hull(V)$, we construct $P_{k+1}$ as follows. Since $\Gamma(V)$ is finite, this process terminates.

Since $P_k \ne \hull(V)$, there exist dented atoms in $\Gamma(P_k)$. Let $D_k \in \Gamma(P_k)$ be the dented atom having smallest side length, and $d\in \Gamma(P_k)$ an atom adjacent to $D_k$ in $\hull(D_k)$. By minimality of $D_k$,  $d$ is an expanding vertex (Definition \ref{def:expanding}) in $\Gamma(P_k)$. 

Let $\cJ_k(D_k)$ be the joins in $\cJ_k$ which are contained in $|\Gamma(P_k)_{D_k}-D_k|$. We treat these joins as (virtual) adjacent atoms. Thus each join $J\in \cJ_k(D_k)$ increases (virtually) the valence of $\Gamma(P_k)_{D_k}$ by $1$ at the dented atom containing it, and so when $n\ge 4$, the valence of $\Gamma(P_k)_{D_k}$ increases at each vertex by at most $3^{n}$.

We leave it to the interested reader to verify that $\Gamma(P_k)_{D_k}$ remains $\lambda$-collapsible with $\lambda$ depending only on $n$ even when joins $\cJ_k(D_k)$ are understood as (virtual) adjacent atoms; compare to Lemma \ref{lemma:hull_regularity}.

Let $\varphi_k \colon (|\Gamma(P_k)_{D_k}|, d_{|\Gamma(P_k)_{D_k}|}) \to (D_k,d_{D_k})$ be a bilipschitz map as in Lemma \ref{lemma:expanding_removability} with the property that, for each $J\in \cJ_k(D_k)$, $\varphi_k|J$ is an isometry. 

Let $\phi_k\colon (D_k,d_{D_k}) \to (\hull(D_k),d_{\hull(D_k)})$ be a bilipschitz map as in Lemma \ref{lemma:Q_flat} with the property that, for each descendant $d$ of $D_k$, $\phi_k$ is an isometry from $J_d$ into some $J_B$ for $B\in \tilde\Gamma(\hull(D_k))$.

Let $\psi_k$ be the composition of $\phi_k \circ \varphi_k$ and $P_{k+1} = P_k \cup \hull(D_k)$. To obtain $\cJ_{k+1}$, we remove the joins $\cJ_k(D_k)$ from $\cJ_k$ and add joins associated to builing blocks in $\hull(D_k)$.
This completes the induction step and the proof.
\end{proof}


\subsubsection{Proposition \ref{prop:induction_clean} in dimension $n=3$}
\label{sec:IC_dim3}

The essential partitions $\bfOmega_0$ and $\bfOmega_1$ fixed in Sections \ref{sec:0ra} and \ref{sec:1ra} are the starting point for the induction also in dimension $n=3$. To obtain partitions $\bfOmega_m$ for $m\ge 2$, we use explicit configurations of atoms in order to obtain branching estimates for the adjacency trees. We begin this section by introducing the particular modifications we use in the induction. 

When $n=3$ it is easy to exhibit an explicit catalog of $\cC$-modifications associated to building blocks. Similarly, the secondary modifications can be explicitly illustrated. These configurations are exhibited in figures and the estimates are obtained simply by counting building blocks and cubes in these configurations.

{\bf Visible faces.} Suppose $Q$ is a cube of side length $3$ in $\R^3$ and $F$ a face of $Q$, and $B$ an $F$-based building block in $Q$. Having Figure \ref{fig:blocks} at our disposal, we observe that for every $q\in B^\#$, $q\cap F$ is a unit square and $B\cap (Q-B)$ a $2$-cell consisting of at most $4$ faces of $q$.

\begin{figure}[h!]
\includegraphics[scale=0.45]{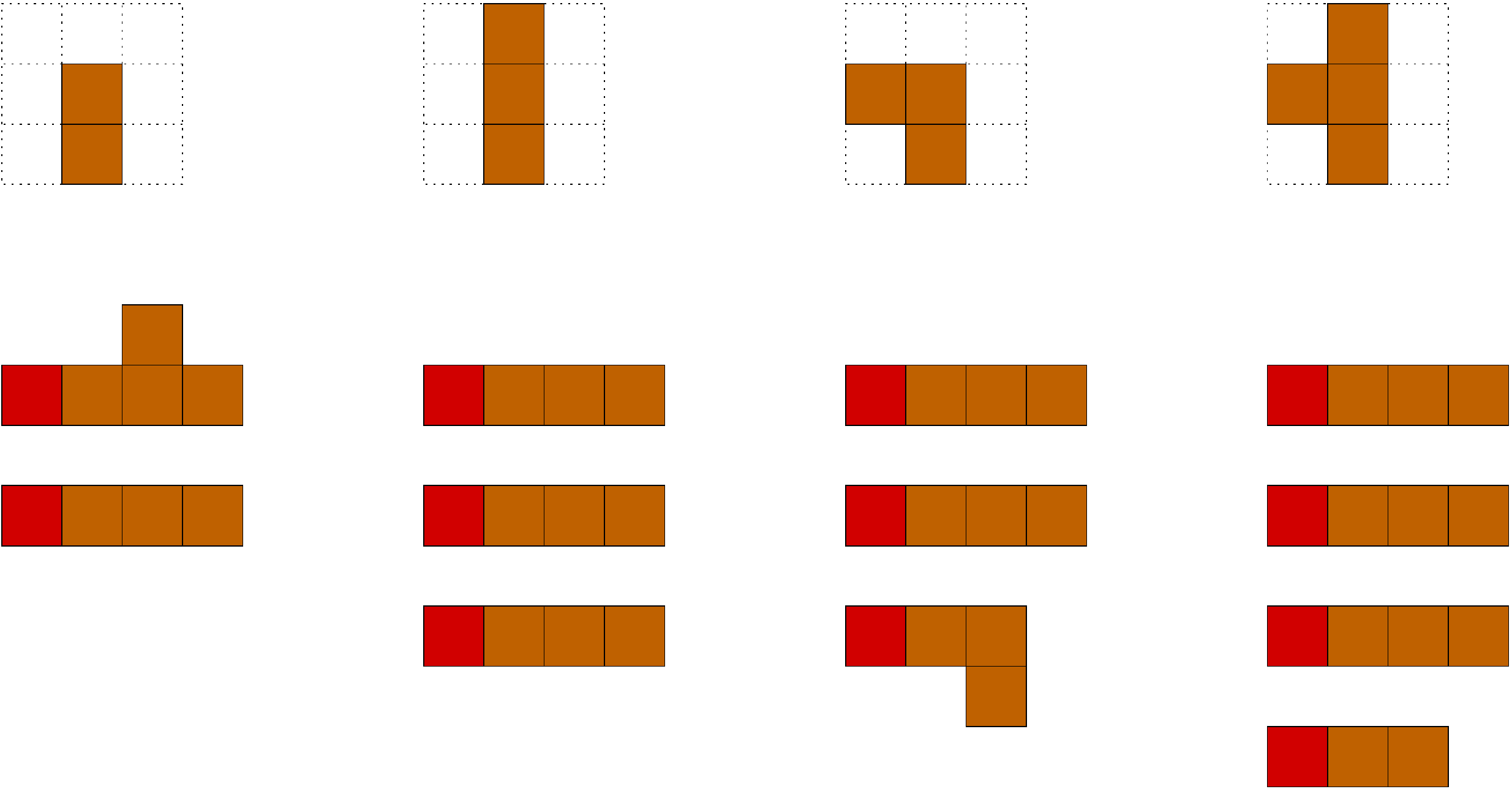}
\caption{Visible faces of building blocks.}
\label{fig:BB_foldouts}
\end{figure}

Figure \ref{fig:BB_foldouts} displays foldouts of faces of all (unit) cubes $q$ in building blocks $B$ which may occur in $Q$. Note that the foldout pictures show only faces of cubes $q$ contained in $F$ or $Q-B$. These faces are the \emph{visible} faces of $\partial q$; only these are in $\partial_\cup \bU$.

{\bf $\cC$-modification.} Based on the catalog in Figure \ref{fig:BB_foldouts}, we observe that in dimension $n=3$ it suffices to fix $4$ $\cC$-modifications which can be applied in all cubes in all building blocks of side length $9$. The case of $5$ visible faces is illustrated in Figure \ref{fig:BB_C_mods_4}. A comprehensive list of examples of $\cC$-modifications to cubes with $3$ or $4$ visible faces is given in Figure \ref{fig:BB_C_mods_small}.

\begin{figure}[h!]
\includegraphics[scale=0.25]{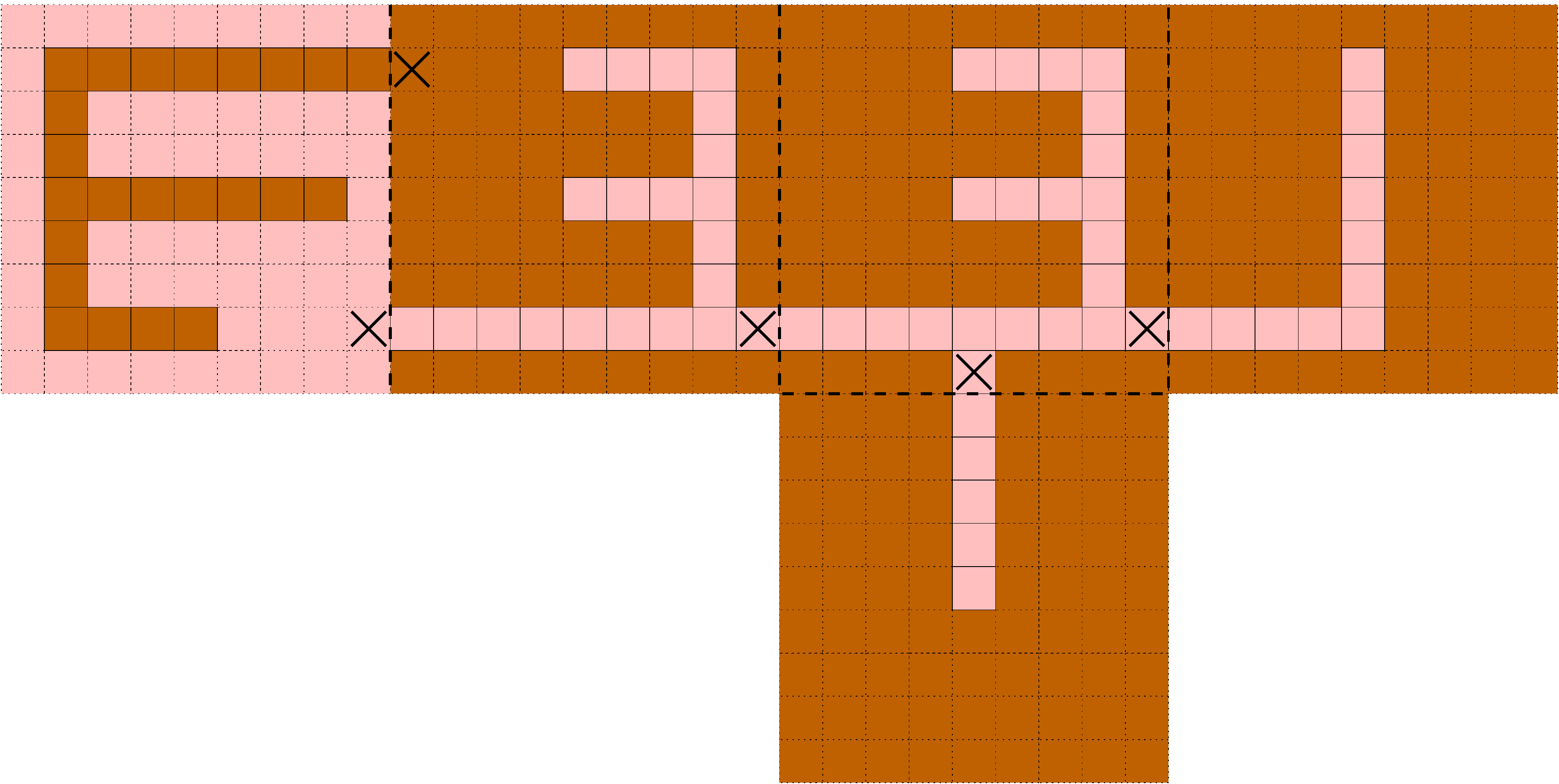}
\caption{Cube $q$ in $B$ with five visible faces.} 
\label{fig:BB_C_mods_4}
\end{figure}

\begin{figure}[h!]
\includegraphics[scale=0.25]{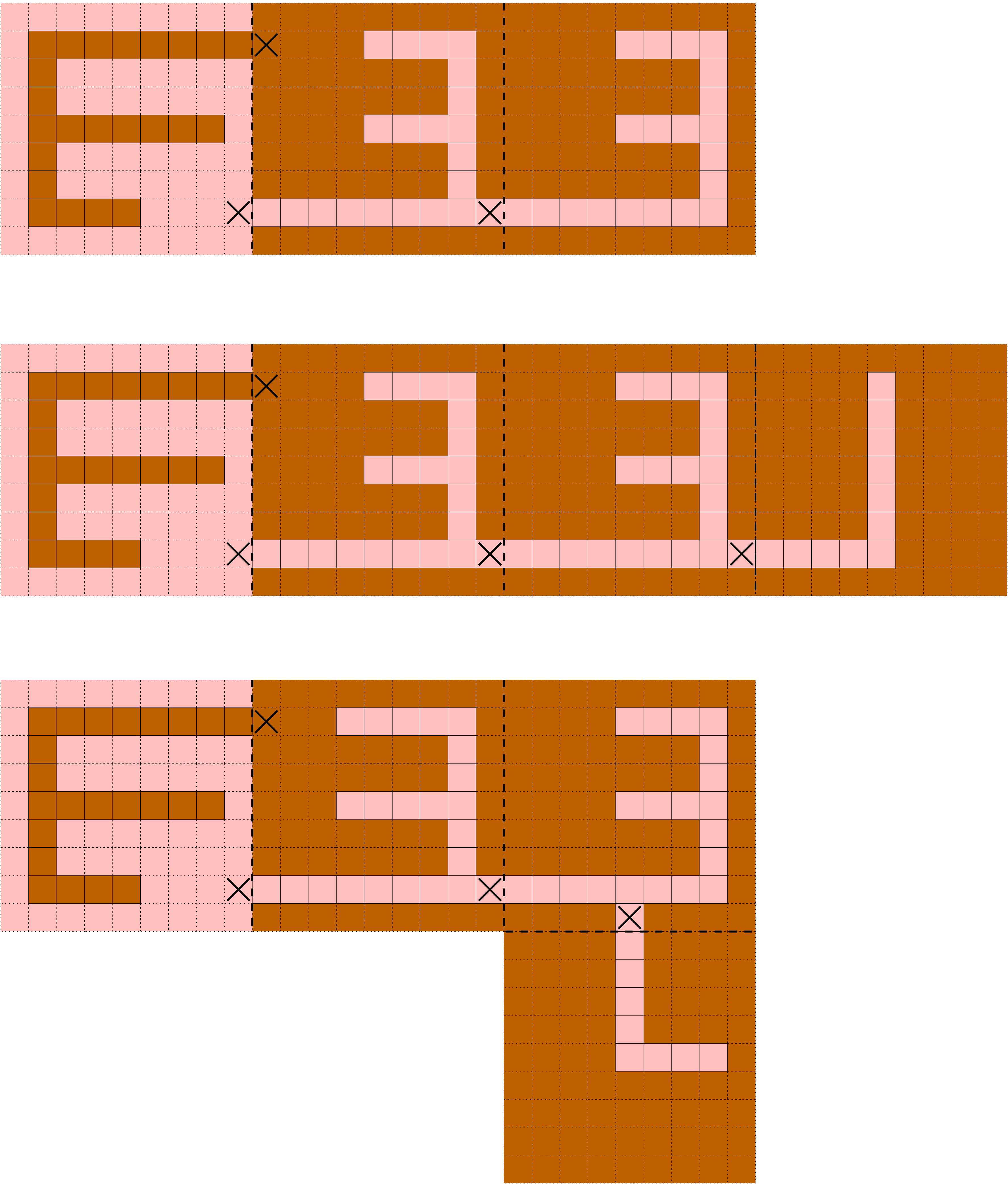}
\caption{}
\label{fig:BB_C_mods_small}
\end{figure}

\emph{Summary:} Let $3B$ be a building block of side length $9$ and suppose that in each $Q'\in 3(B^*)$ we have performed one of the $\cC$-modifications illustrated in Figures \ref{fig:BB_C_mods_4} and \ref{fig:BB_C_mods_small} and let $A_{Q',i}\subset Q'$, $i=1,2$, be the corresponding atoms; $\rho(A_{Q',i})=1$. Then 

\begin{itemize}
\item each atom $A_{Q',i}$ consists of at most $20$ building blocks and $56$ cubes;
\item in each cube $Q'\in 3(B^*)$, $A_{Q',1}\cup A_{Q',2}$ consists of at most $28$ building blocks and $79$ cubes;
\item $\bigcup_{Q'} \left( A_{Q',1}\cup A_{Q',2} \right) $ consists of at most $100$ building blocks and $285$ cubes.
\end{itemize}

{\bf Secondary modifications.}
Observe first that, while secondary $\cC$-modification may occupy as many as $4$ faces of a cube of side length $27$, secondary $\cN$-modification is confined to $2$ faces. Thus upper bounds for unit cubes and building blocks are achieved by secondary $\cC$-cubes, and so there is no need to consider explicit secondary $\cN$-modifications.

Let $Q$ and $B$ be as above and let $Q''$ be the unique cube sharing the face $F$ with $Q$. Let $\bV=(Q-B,Q'',B)$ and let $\bU=(U_1,U_2,U_3)$ be the essential partition of $|3\bV|$ obtained after $\cC$-modifications, based on Figures \ref{fig:BB_C_mods_4} and \ref{fig:BB_C_mods_small}. Note that components of $3A_1=U_1 - 3(Q-B)$ are atoms having $8$ building blocks.

Let $Q'\in 3(B^*)$. Figure \ref{fig:BB_C_mods_4_basins} presents an example of a system of basins in $Q'$ when $Q'$ has $5$ visible faces. For cubes with fewer visible faces, similar systems of basins can be found; these systems have fewer basins. Figure \ref{fig:BB_C_mods_4_basins_largest} illustrates an $\cC$-modification in the largest basin in Figure \ref{fig:BB_C_mods_4_basins}. The systems of basins for cubes in $3(B^*)$ with fewer visible faces can be chosen to have basins not larger than this basin in terms of the number of unit cubes in added atoms. We encourage the interested reader to verify these statements by illustrations.

\begin{figure}[!ht]
\includegraphics[scale=0.30]{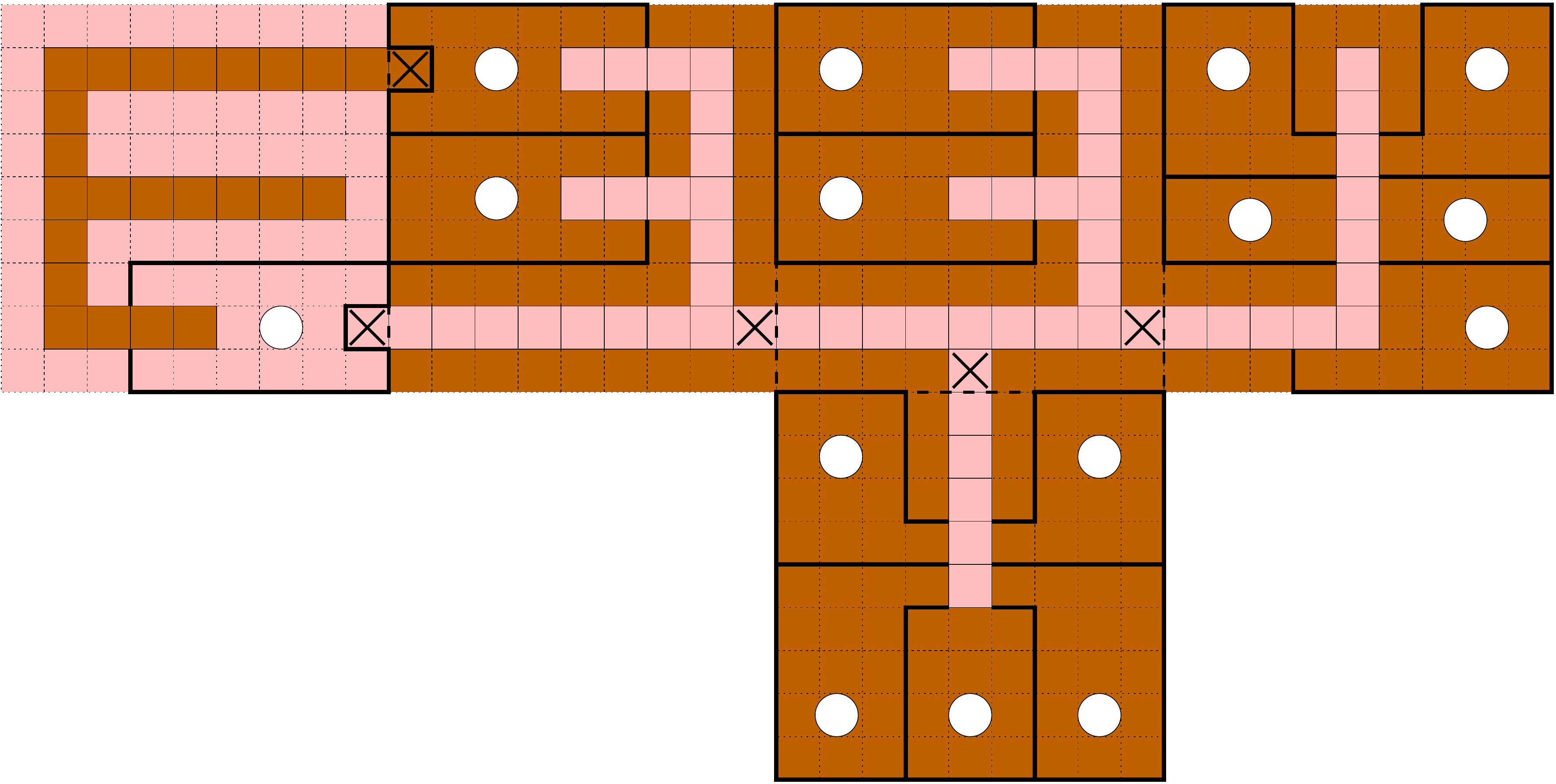}
\caption{An example of a system of basins. Basins indicated with (large) dots.}
\label{fig:BB_C_mods_4_basins}
\end{figure}

\begin{figure}[!ht]
\includegraphics[scale=0.30]{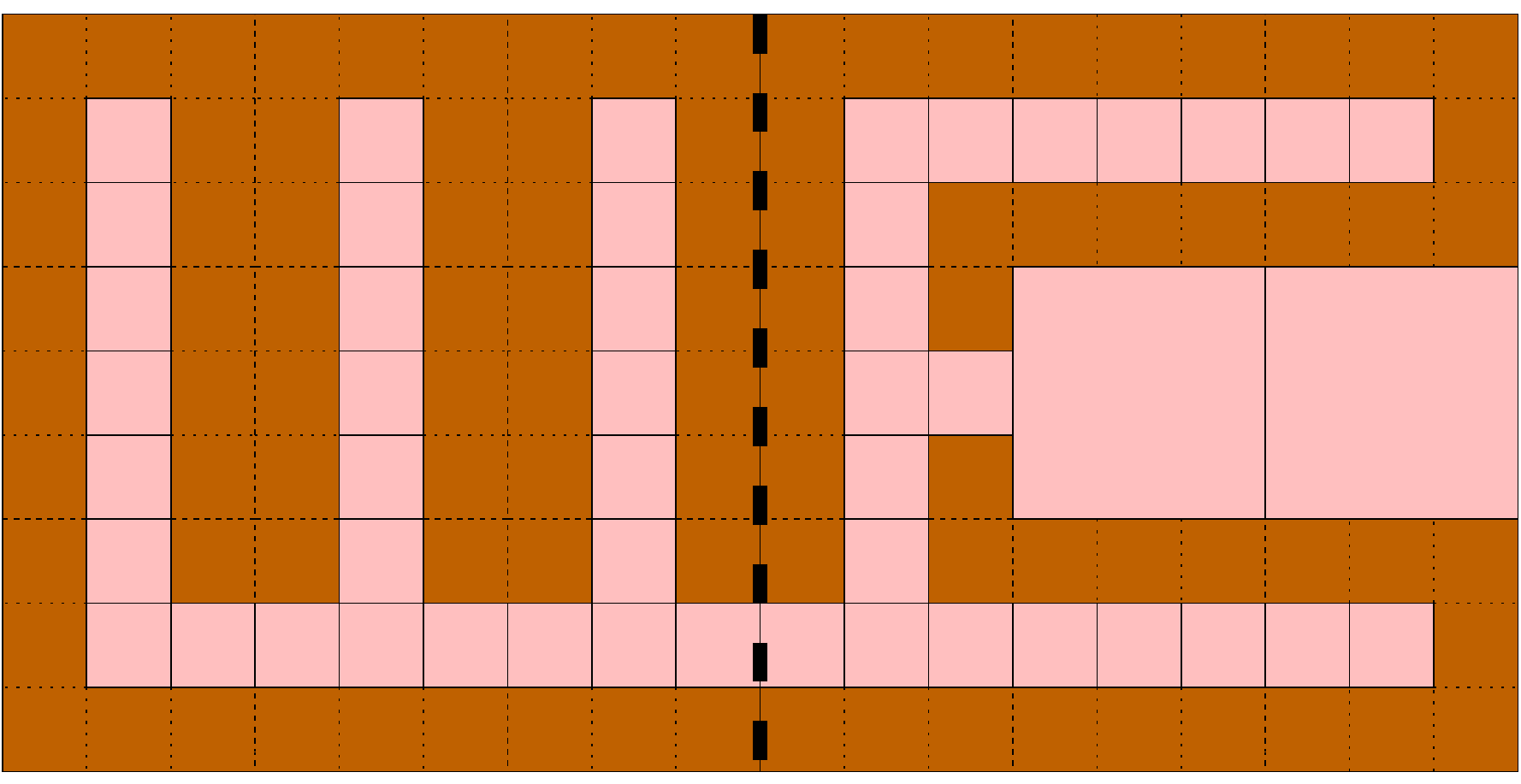}
\caption{}
\label{fig:BB_C_mods_4_basins_largest}
\end{figure}

\emph{Summary:}
Let $\bfOmega=(\Omega_1,\Omega_2,\Omega_3)$ be an essential partition of $3Q$ obtained from $\bU$ by a secondary modification. For $Q'\in 3(B^*)$, let $M_{Q',j} = \Omega_j \cap  Q'$, $j=1,2$. Then $M_{Q',j}$ is a molecule having $3A_{Q',j}$ as its root. Let $L_{Q',j} = M_{Q',j} - 3A_{Q',j}$ be the union of leaves of $M_{Q',j}$. Then
\begin{itemize}
\item each component of $L_{Q',j}$ consists of at most $16$ building blocks and $47$ cubes,
\item for each cube $Q'\in 3(B^*)$, $L_{Q',1}\cup L_{Q',2}$ has at most $31$ components and consists of at most $243$ building blocks,
\item the union $\bigcup_{Q'} \left( L_{Q',1}\cup L_{Q',2}\right)$  consists of at most $829$ building blocks.
\end{itemize}
Furthermore, $\Gamma(M_{Q',j})$ has valence at most $45$.

\subsubsection{Completion of Proposition \ref{prop:induction_clean} for $n=3$}
\label{sec:pic3}

We construct the sequence $(\bfOmega_m)$ of essential partitions using Corollary \ref{cor:pre_induction} iteratively as in Section \ref{sec:IC_init} with the only exception that for $\cC$- and secondary $\cC$-modifications, we use explicit configurations illustrated in Section \ref{sec:IC_dim3}. Thus, again, the essential partitions $\bfOmega_m$ satisfy conditions (a)--(d) and the tripod property.

To verify condition (e), we note first that Lemma \ref{lemma:mods_meet_exts} has no dimensional restrictions, and so it applies also for $n=3$. 

Regarding Lemma \ref{lemma:hull_regularity}, we note that, when $n=3$, the statistics in Section \ref{sec:IC_dim3} imply that $\Gamma(\hull(\Omega_{m,p}))$ has valence at most $20$ and every atom in $\Gamma(\hull(\Omega_{m,p}))$ consists of at most $56$ cubes. Since $9^2 > 8^2 > 56$, we may take $\varepsilon=8/9$ in the proof. Thus the claim of Lemma \ref{lemma:hull_regularity} holds also for $n=3$, and so $(\bfOmega_m)$ satisfies condition (e).

To verify condition (f), we observe that, using configurations in Section \ref{sec:IC_dim3} now show that $\# \cA(P,D;B)\le 285$ in Lemma \ref{lemma:entry_counting}. Since $285 < 26^2$, the process of Lemma \ref{lemma:Q_flat} is therefore also at our disposal, and thus it suffices to discuss the proof of Lemma \ref{lemma:f} for $n=3$. The bounds we need here are already available from Section \ref{sec:IC_dim3}.

Concerning item (b) in the proof of Lemma \ref{lemma:f}, the number of joins in our construction now is $16$. Furthermore, the maximal (virtual) valence of $\Gamma(P_k)_{D_k}$ is at most $31+16 = 47$, since the atoms in $\Gamma(P_k)_{D_k}-D_k$ are expanding and hence obtained by a $\cD$-modification, a secondary modification or by a $\cC$-modification over one face of a cube. Finally, since $D_k$ has at most $16$ descendants for $n=3$, the result of Lemma \ref{lemma:f} therefore holds also for $n=3$. This concludes the verification of condition (f) and the proof of Proposition \ref{prop:induction_clean} in dimension $n=3$.\qed

\subsection{Proof of Theorem \ref{thm:RP}}
\label{sec:proof_RP}

Let $(\bfOmega_m)$ be a sequence of essential partitions as in Proposition \ref{prop:induction_clean} and observe that conditions (2) and (3) in the claim of Theorem \ref{thm:RP} are satisfied.

We conclude the proof of Theorem \ref{thm:RP} by showing that, for $p=1,2,3$, $\Omega_p = \bigcup_{m\ge 0}\Omega_{m,p}$ is a bilipschitz equivalent to $\R^{n-1}\times [0,\infty)$ in its inner geometry.

By (2b), $(\Omega_p,d_{\Omega_p})$ is bilipschitz equivalent to $(\hull(\Omega_p),d_{\hull(\Omega_p)})$ for each $p$. Since $\hull(\Omega_3)$ is a monotone union of $(\nu,\lambda)$-molecules, where $\nu$ and $\lambda$ depend only on $n$, $(\Omega_3,d_{\Omega_3})$ is bilipschitz equivalent to $\R^{n-1}\times [0,\infty)$ by Proposition \ref{prop:adap_cont_2}.

Concerning $\hull(\Omega_2)$, we observe first that $\hull(\Omega_2) \cap [0,\infty)^{n-1}\times [0,\infty)$ consists of an infinite collection of pair-wise disjoint $(\nu,\delta)$-molecules. Thus $(\hull(\Omega_2),d_{\hull(\Omega_2)})$ is bilipschitz equivalent to $[0,\infty)^{n-1}\times (-\infty,0]$ as we may apply Proposition \ref{prop:fRt} to these molecules separately. Since components of $\hull(\Omega_2) \cap [0,\infty)^{n-1}\times [0,\infty)$ do not meet $\partial [0,\infty)^{n-2}\times \R$, we obtain a bilipschitz homeomorphism $[0,\infty)^{n-1}\times (-\infty,0] \to (\hull(\Omega_2),d_{\hull(\Omega_2)})$ which is the identity on the boundary $\partial [0,\infty)^{n-1}\times (-\infty,0]$.

We are left with $\hull(\Omega_1)$. Since $\hull(\Omega_{1,m}) = [0,3^{m+1}]^n$ for every $m\ge 1$, $\hull(\Omega_1) = [0,\infty)^n$.

This completes the construction of a rough Rickman partition of $[0,\infty)^{n-1}\times \R$ and the proof of Theorem \ref{thm:RP}.\qed

\begin{proof}[Proof of Corollary \ref{cor:John-domains}]
Domains $\Omega_{m,p}$ are John-domains with a constant depending only on $n$. This can be seen for example as follows. Let $a,b\in \interior \Omega_p$ be points and let $A,B\in \Gamma(\Omega_p)$ be the (dented) atoms containing $a$ and $b$, respectively. Let $D_1,\ldots, D_r$ be the geodesic in $\Gamma(\Omega_p)$ connecting $A$ and $B$. Since atoms $\hull(D_r)$ have uniformly bounded length and the function $r \mapsto \rho(\hull(D_r))$ satisfies the combinatorial John condition as noted in Section \ref{sec:RT}, we observe that there exists $C>1$ depending only on $n$ so that $a$ and $b$ can be connected with a path $\gamma\colon [0,1]\to \Omega_{m,p}$ satisfying $\min\{|a-\gamma(t))|,|\gamma(t)-b|\} \le C\dist(\gamma(t),\partial \Omega_{m,p})$ for $0\le t\le 1$.

Domains $\Omega_p$, for $p=1,2,3$, are uniform domains by the same argument.
\end{proof}




\section{From cubes to simplices}
\label{sec:triangulation}

In this section we introduce a particular triangulation of the pair-wise common boundary $\partial_\cup \bfOmega$ of a rough Rickman partition $\bfOmega = (\Omega_1,\Omega_2,\Omega_3)$. While the construction of the domains $\Omega_p$ is facilitated by using cubes as fundamental units, an Alexander-type mapping is  more naturally described using simplices. We wish to remind the reader that the rough Rickman partition $\bfOmega$ must be modified once more to obtain a Rickman partition $\wt\bfOmega$ supporting a suitable BLD-mapping on $\partial_\cup \wt\bfOmega$. The triangulation of $\partial_\cup \bfOmega$ and a parity function carried by it have important roles in the construction of $\wt\bfOmega$ in the next section.

The space $\R^n$ has a natural structure as a CW-complex with unit cubes $[0,1]^n+v$, $v\in \Z^n$, as $n$-cells, and the $k$-dimensional faces of these cubes as $k$-cells. Every $(n-1)$-cube $Q$ of this complex has a natural subdivision into $(n-1)$-simplices. In what follows the convex hull of points $v_0,\ldots, v_k$ in $\R^k$, $0\le k \le n-1$, is 
\[
[v_0,\ldots, v_k].
\]

Let $Q$ be an $(n-1)$-cube in $\R^n$ and, for $k=0,\ldots, n-1$, $Q_k$ a $k$-dimensional face of $Q$. The $n$-tuple $\cQ=(Q_0,\ldots, Q_{n-1})$ is a \emph{flag in $Q$} if
\begin{equation}
\label{eq:nested_cubes}
Q_0\subset Q_1\subset \cdots \subset Q_{n-1} = Q.
\end{equation} 
Each $k$-cell $Q_k$ has a uniquely defined barycenter $c_{Q_k}$ and, by \eqref{eq:nested_cubes}, the  vectors $c_{Q_0}-c_{Q_{n-1}},\ldots, c_{Q_{n-2}}-c_{Q_{n-1}}$ are linearly independent with 
\[
S_{\cQ} = [c_{Q_0},\ldots, c_{Q_{n-1}}]
\]
an $n$-simplex contained in $Q$. We say that $S_{\cQ}$ is the \emph{simplex induced by the flag $\cQ$}. Furthermore, 
\[
Q = \bigcup_{\cQ} S_{\cQ},
\]
the union over all flags $(Q_0,\ldots, Q_n)$ in $Q$. Two $(n-1)$-simplices $S_{\cQ}$ and $S_{\cQ'}$, determined by different flags $\cQ$ and $\cQ'$, may intersect but they have no common interior. Thus simplices induced by flags triangulate $Q$.

Since simplices induced by flags are determined by the barycenters of lower-dimensional faces of $(n-1)$-cubes, every $(n-1)$-dimensional subcomplex $\bX$ of $\R^{n-1}$, which is a union of its $(n-1)$-cells, admits a triangulation with simplices induced by flags. We call the simplicial complex associated to such triangulation the \emph{standard simplicial structure of $\bX$}. Note that, since simplices in the standard simplicial structure arise as a subdivision of unit cubes in $\R^n$, the $k$-simplices ($0<k\le n$) in the standard simplicial structure have diameter between $1/2$ and $\sqrt{n}/2$.

\begin{convention}
From now on we tacitly assume that a given $(n-1)$-simplex $\sigma$ in an $(n-1)$-dimensional cubical complex $\bX$ has the standard simplicial structure of $\bX$. 
\end{convention}

In particular, the pair-wise common boundary $\partial_{\cup} \bfOmega$ of a Rickman partition $\bfOmega$ admits this standard simplicial structure.

There is an elementary labeling function associated to the standard simplicial structure. Let $\bX$ be an $(n-1)$-dimensional subcomplex of $\R^n$ so that $\bX$ is a union of its $(n-1)$-cells and let $\bX^{(0)}$ be the vertices of the standard simplicial structure. Since every vertex $v$ in $\bX$ is a barycenter of a unique unit cube $Q_v$ in the cubical structure of $\R^n$, the map
\[
\vartheta_\bX \colon \bX^{(0)} \to \{0, \ldots, n-1\}, \ v\mapsto \dim Q_v, 
\]
is well-defined. Moreover, $\vartheta_\bX(\sigma) = \{0, \ldots, n-1\}$ for every $(n-1)$-simplex $\sigma$ in the standard simplicial structure of $\bX$. We call $\vartheta_\bX$ the \emph{labeling function of $\bX$}.

\subsection{Parity functions}
\label{sec:parity_function}

Let $\bfOmega = (\Omega_1,\Omega_2,\Omega_3)$ be a rough Rickman partition of $\R^n$ and let $\sigma$ be an $(n-1)$-simplex in $\left( \partial_\cup \bfOmega\right)^{(n-1)}$. Then $\sigma = [v_0,\ldots, v_{n-1}]$, where $0\leq k\leq n-1$ and $v_k$ is a barycenter of a $k$-cube in $\partial_\cup \bfOmega$.
Since $\partial_\cup \bfOmega$ is the pair-wise common boundary,  $\sigma$ lies on the boundary of exactly two domains in $\bfOmega$. We say that $\sigma$ is \emph{$\bfOmega$-positive} if there exist $i$ and $j$ with $\sigma\subset \Omega_i \cap \Omega_j$ and 
\begin{itemize}
\item[(1)] $j=i + 1 \mod 3$, and 
\item[(2)] there exists a unit vector $v\in \R^n$ with $v_{n-1} + v \in \Omega_i$ and 
\begin{equation}
\label{eq:orient}
\det((v_0-v_{n-1}),\ldots, (v_{n-2}-v_{n-1}), v) >0.
\end{equation}
\end{itemize}
Otherwise, $\sigma$ is \emph{$\bfOmega$-negative}.
A vector $v$ satisfying \eqref{eq:orient} is called an \emph{oriented normal of $\sigma$} if $v$ is orthogonal to $v_k-v_{n-1}$ for every $0\le k \le n-1$.

The \emph{parity function of $\bfOmega$ is the function $\nu_\bfOmega \colon (\partial_\cup \bfOmega)^{(n-1)} \to \{\pm 1\}$} defined by 
\[
\nu_\bfOmega(\sigma) = \left\{ \begin{array}{rl}
1, & \sigma\ \mathrm{is}\ \bfOmega\mathrm{-positive},\\
-1, & \sigma\ \mathrm{is}\ \bfOmega\mathrm{-negative}.
\end{array}\right.
\]

The next lemma describes the change of the parity on adjacent simplices. 
\begin{lemma}
\label{lemma:parity}
Let $\bfOmega=(\Omega_1,\Omega_2,\Omega_3)$ be a rough Rickman partition of $\R^n$. Suppose $\sigma$ and $\sigma'$ are adjacent $(n-1)$-simplices in $\partial\Omega_i$.
Then $\nu_\bfOmega(\sigma) = -\nu_\bfOmega(\sigma')$ if there exists $j\ne i$ so that $\sigma\cup \sigma'\subset \partial\Omega_j$, and $\nu_\bfOmega(\sigma) = \nu_\bfOmega(\sigma')$ otherwise.
\end{lemma}

\begin{proof}
Let $\sigma = [v_0,\ldots, v_{n-1}]$ and $\sigma'=[v'_0,\ldots, v'_{n-1}]$. Suppose first that $\sigma$ and $\sigma'$ are contained in an $(n-1)$-dimensional plane $P$. We claim that 
\begin{equation}
\label{eq:wedge_sign}
(v'_0-v'_{n-1})\wedge \cdots \wedge (v'_{n-2}-v'_{n-1}) = - (v_0-v_{n-1})\wedge \cdots \wedge (v_{n-2}-v_{n-1}).
\end{equation}
It is then easy to verify the claim of the lemma as the oriented normal vectors of $\sigma$ and $\sigma'$ will be opposite normals of $P$.

Let $\cQ=(Q_0,\ldots,Q_n)$ and $\cQ'=(Q'_0,\ldots, Q'_n)$ be flags defining $\sigma=S_{\cQ}$ and $\sigma'=S_{\cQ'}$ respectively. Since $\sigma$ and $\sigma'$ have a common side, there exists $0 \le k \le n-1$ so that $v_i = v'_i$ for $i \ne k$.  

\begin{figure}[h!]
\includegraphics[scale=0.30]{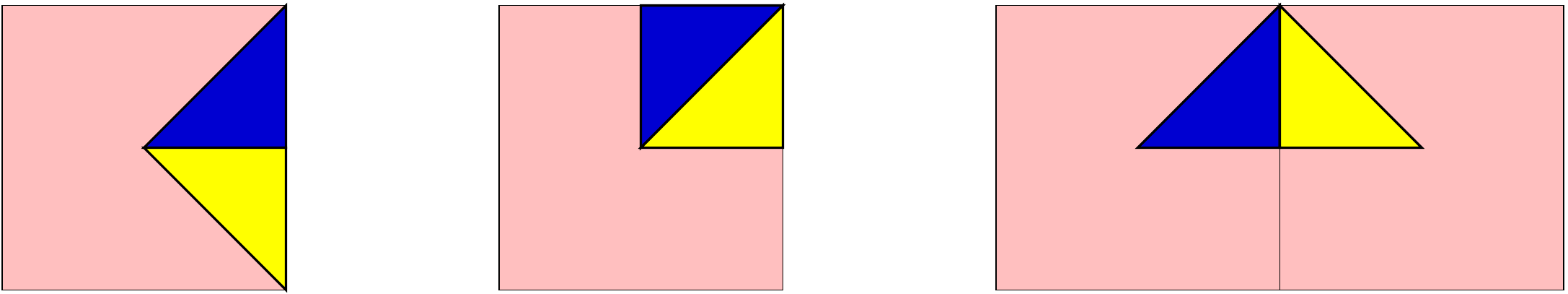}
\caption{Congruence classes of planar $\sigma\cup \sigma'$ for $n=3$ and $k=0,1,2$.}
\label{fig:PF}
\end{figure}

Suppose first that $0<k<n-1$. Then $Q'_k$ and $Q_k$ have a common face $Q_{k-1}$ and are contained in $Q_{k+1}$. Since 
\[
c_{Q_{k-1}} -c_{Q_{k+1}} = (c_{Q_{k-1}}-c_{Q'_k}) + (c_{Q'_k} - c_{Q_{k+1}})
= (c_{Q_k} - c_{Q_{k+1}}) + (c_{Q'_k}-c_{Q_{k+1}}), 
\]
it follows that
\begin{eqnarray*}
v'_k - v_{n-1} &=& v'_k - v_{k+1} + (v_{k+1}-v_{n-1}) \\
&=& v_{k-1}-v_{k+1} - (v_k - v_{k+1}) + (v_{k+1}-v_{n-1}) \\
&=& -(v_k - v_{n-1}) + (v_{k-1}-v_{n-1}) + (v_{k+1}-v_{n-1}),
\end{eqnarray*}
and so
\begin{eqnarray*}
&& (v'_0 - v'_{n-1}) \wedge \cdots \wedge (v'_k - v'_{n-1}) \wedge \cdots \wedge (v'_{n-2}-v'_{n-1}) \\
&&\quad = (v_0 - v_{n-1}) \wedge \cdots \wedge (v'_k-v_{n-1}) \wedge\cdots \wedge (v_{n-2}-v_{n-1}) \\ 
&&\quad = -(v_0 - v_{n-1}) \wedge \cdots \wedge (v_k - v_{n-1}) \wedge \cdots \wedge (v_{n-2} -v_{n-1}).
\end{eqnarray*}
Thus \eqref{eq:wedge_sign} holds. The cases $k=0$ and $k=n-1$ are similar.

Suppose now that $\sigma$ and $\sigma'$ are not contained in an $(n-1)$-dimensional hyperplane. In this case, using the notation above, $v'_{n-1} \ne v_{n-1}$ and $v'_k = v_k$ for $0\le k < n-1$. By construction of $\bfOmega$, there also exists an $n$-cube $Q$ having $\sigma$ and $\sigma'$ on its boundary. In particular, $w=c_Q-v_{n-1}$ and $w'=c_Q-v'_{n-1}$ are orthogonal to $\sigma$ and $\sigma'$, respectively.

\begin{figure}[h!]
\includegraphics[scale=0.30]{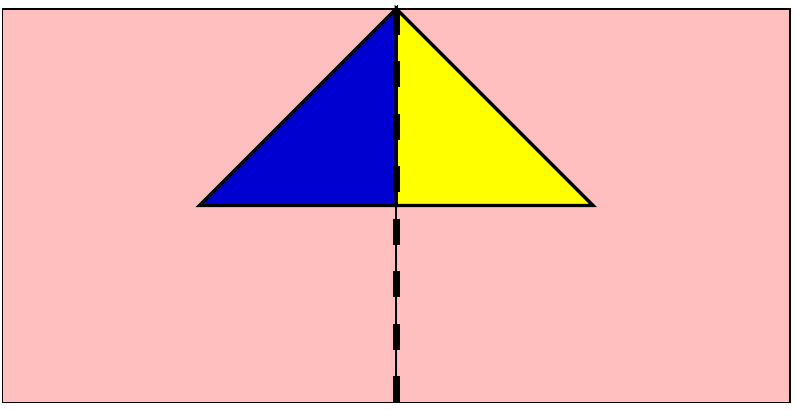}
\caption{Fold-out of the congruence class of $\sigma\cup \sigma'$ for $n=3$.}
\label{fig:PF2}
\end{figure}

Since the $n$-simplices $[v_0,\ldots,v_{n-1},c_Q]$ and $[v'_0,\ldots, v'_{n-1},c_Q]$ are planar in $\R^{n+1}$ and share an $(n-1)$-dimensional face, we have, by the previous argument,
\[
(v'_0 - c_Q) \wedge \cdots  \wedge (v'_{n-1}-c_Q) = - (v_0 - c_Q) \wedge \cdots  \wedge (v_{n-1}-c_Q),
\]
so that
\[
(v'_0 - v'_{n-1}) \wedge \cdots \wedge (v'_{n-2}- v'_{n-1})\wedge w' = - (v_0 -v_{n-1})\wedge \cdots \wedge (v_{n-2}-v_{n-1}) \wedge w.
\]
Since $Q$ is contained in one of the elements of the partition $\bfOmega$, the claim now follows by considering separately cases $Q\subset Q_i$ and $Q\subset Q_j$, where $j=i+1\mod 3$; in both cases the oriented normals for $\sigma$ and $\sigma'$ are either $w$ and $-w'$, or $-w$ and $w'$, respectively.
\end{proof}


\section{Pillows and pillow covers}
\label{sec:pillows}
In this section we establish the most significant case, $p=3$, of Proposition \ref{prop:4}. Using the ideas of Rickman \cite[Section 7]{R} we prove the following proposition.

\begin{proposition}
\label{prop:special_4}
Let $\bfOmega = (\Omega_1,\Omega_2,\Omega_3)$ be a rough Rickman partition of $\R^n$ supporting the tripod property. Then there exists a Rickman partition $\wt\bfOmega = (\wt\Omega_1,\wt\Omega_2,\wt\Omega_3)$ of $\R^n$ for which the Hausdorff distance of $\partial_\cup \bfOmega$ and $\partial_\cup \wt\bfOmega$ is at most $1$.
\end{proposition}

We would like to recall that the construction in Section \ref{sec:RP} yields a rough Rickman partition $\bfOmega=(\Omega_1,\Omega_2,\Omega_3)$ where $\Omega_1$ and $\Omega_2$ are connected and $\Omega_3$ has $2^{n-1}$ components. It should be, however, noted that we may construct the essential partition $\wt\bfOmega$ in Proposition \ref{prop:special_4} from any rough Rickman partition. Indeed, the construction of $\wt\bfOmega$ is local and the number of components of sets $\Omega_i$ have no r\^ole in the argument.

The proof of Proposition \ref{prop:4} is based on a construction of what we call a pillow cover of $\partial_\cup \bfOmega$, and yields the final essential partition $\wt\bfOmega$. The labeling and parity functions of $\bfOmega$ lead at once to a BLD-map $\partial_\cup \wt\bfOmega \to \wh \bS^{n-1}$, where $\wh \bS^{n-1} = \bS^{n-1}\cup \bB^{n-1}$. The bound on the Hausdorff distances of $\partial_\cup \bfOmega$ and $\partial_\cup \wt\bfOmega$ is immediate from the pillow construction.

\begin{remark}{(Idea of the pillow cover)}
We summarize the idea of the pillow cover construction as follows. Let $\bfOmega$ be a rough Rickman partition of $\R^n$ and consider a triangulation on $\partial_\cup \bfOmega$ as in Section \ref{sec:triangulation}. 

To construct a pillow cover of $\partial_\cup \bfOmega$ we (locally) replace each pair of adjacent $(n-1)$-simplices with a sextuplet of adjacent (Lipschitz) $(n-1)$-simplices. This sextuplet can be seen as a (branched) double cover of $\wh\bS^{n-1}$; note that $\wh\bS^{n-1}$ consists of three $(n-1)$-cells.

We use the tripod property of $\partial_\cup \bfOmega$, organize the adjacent $(n-1)$-simplices into directed trees, and modify the domains $\Omega_k$ by modifying their boundaries via this local replacement procedure of $(n-1)$-simplices. This process extends $\Omega_k$ between $\Omega_i$ and $\Omega_j$ for $\{i,j,k\}=\{1,2,3\}$ and, as a consequence, we obtain a new essential partition $\wt\bfOmega=(\wt\Omega_1,\wt\Omega_2, \wt\Omega_3)$ of $\R^n$.

Finally, the local (combinatorial) properties of $\partial_\cup \wt\bfOmega$ allow us to construct a BLD-map $\partial_\cup \wt\bfOmega\to \partial_\cup {\bf E}$ which shows that domains in $\wt\bfOmega$ are Zorich extension domains.
\end{remark}

We discuss first the pillow construction locally for planar $(n-1)$-cells contained in $\partial_\cup \bfOmega$. For notational convenience let $E\subset \partial_\cup \bfOmega$ be a cubical $(n-1)$-cell contained in a hyperplane $P=\R^{n-1}\times \{0\}$ of $\R^n$ so that $E\subset \Omega_i \cap \Omega_j$ for some $i\ne j$. Throughout Sections \ref{sec:pillow_simplex}--\ref{sec:ppc} we consider $E$ fixed but arbitrary and $E$ inherits a standard simplicial structure from $\partial_\cup \bfOmega$. We denote by $\nu = \nu_{E,\bfOmega} \colon E^{(n-1)} \to \{\pm 1\}$ the restriction of the parity function $\nu_\Omega$ to $E$. Similarly, $\vartheta = \vartheta_{E,\bfOmega} \colon E^{(0)} \to \{0,\ldots, n-1\}$ is the restriction of the labeling function $\vartheta_{\partial_\cup \bfOmega}$ to $E$.

Let $\cE$ be the adjacency graph $\Gamma(E^{(n-1)})$ and fix a maximal tree $\wh\cE$ in $\cE$. Contrary to the case of maximal trees of adjacency graphs of cubical complexes, we consider $\hat \cE$ as a directed tree, and fix orientation  on $\hat \cE$ so that $\hat \cE$ is connected and  all simplices in $\hat \cE$ have at most one outgoing edge and (possibly several or no) incoming edges. 

Suppose $\sigma$ is an $(n-1)$-simplex $\sigma$ of $\hat \cE$ and the $(n-2)$-simplex $\tau$ is a face of $\sigma$. Let $\sigma'$ be an  $(n-1)$-simplex in $\hat \cE$ adjacent to $\sigma$, that is, $\sigma'\cap \sigma=\tau$. Then $\tau$ is an \emph{entry face of $\sigma$} if the edge between $\sigma$ and $\sigma'$ is an incoming edge to $\sigma$, and $\tau$ is an \emph{exit face of $\sigma$} if it is the (unique) outgoing edge from $\sigma$. If $\tau$ is an entry or an exit face of a simplex,  $\tau$ is considered \emph{open}, otherwise $\tau$ is a \emph{closed face of $\sigma$}; in the configuration of Figure \ref{fig:entry-exit} the open faces are marked with dashed lines.

\begin{figure}[h!]
\includegraphics[scale=0.60]{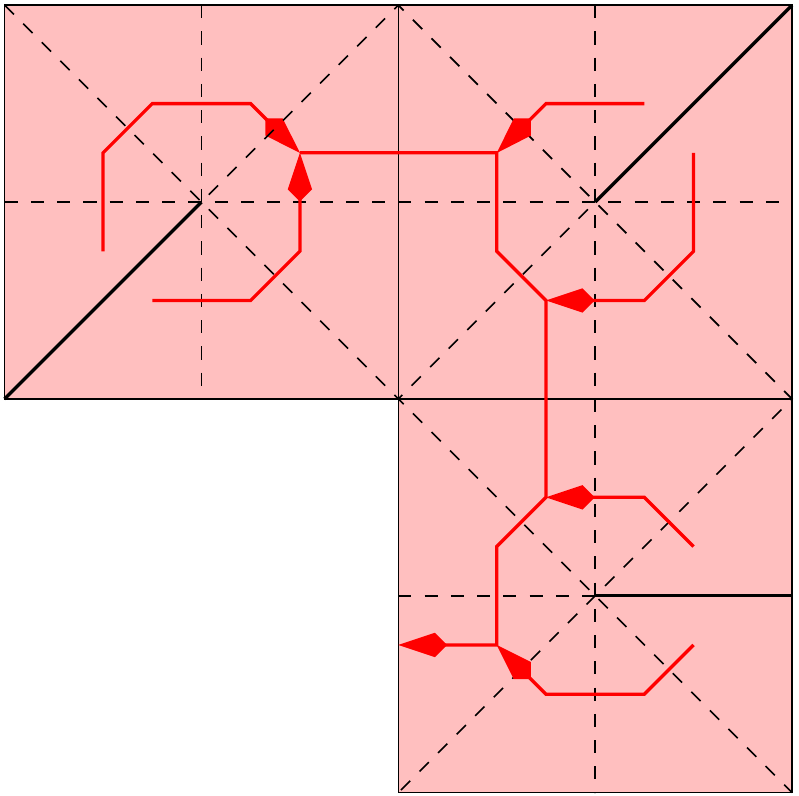}
\caption{}
\label{fig:entry-exit}
\end{figure}

\subsection{Pillow of a simplex}
\label{sec:pillow_simplex}

As a preparatory step, let $\tau=[v_1,\ldots, v_{n-1}]$ be an $(n-2)$-simplex in $\R^{n-1}$, and consider $\tau$ as a face of an $(n-1)$-simplex $\sigma$ in $E$. We define a subdivision $\tau_0,\ldots, \tau_{n-1}$ of $\tau$ as follows. 

Suppose first that $n\ge 4$. For $i=1,\ldots, n-1$, let 
\[
\tau_i = [(v_1+v_i)/2,\ldots, (v_{n-1}+v_i)/2]\subset \tau.
\]
Then $\tau_i$ is an $(n-2)$-simplex congruent to $\tau$ having diameter $(\diam \tau) /2$ and having $v_i$ as a vertex; see Figure \ref{fig:triangles}. Finally, let $\tau_0 = \tau - \bigcup_{i=1}^{n-1} \tau_i$; we use here and from now on the notation $\alpha-\beta = \overline{\alpha\setminus \beta}$ also for simplices.
For $n=4$, $\tau_0$ is an $2$-simplex, while when  $n>4$, $\tau_0$ is a more general polyhedron.

When $n=3$, $\tau$ is a line segment $[v_1,v_2]$. In this case, we set $\tau_1 = [v_1,v_1+(v_2-v_1)/3]$, $\tau_2 = [v_2,v_2+(v_1-v_2)/3]$, and $\tau_0 = \tau - (\tau_1\cup \tau_2)$; thus $\tau_0$ is the 'middle third' of $\tau$.

\begin{figure}[h!]
\includegraphics[scale=1]{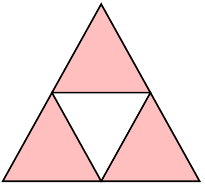}
\caption{$2$-simplices $\tau_1$, $\tau_2$, $\tau_3$ surrounding $\tau_0$ in a subdivision of $\tau$; $n=4$.}
\label{fig:triangles}
\end{figure}

\begin{definition}
Let $u \colon \tau \to [-1,1]$ be a continuous function on $\tau$. Then $u$ is an \emph{opening} if $u|\interior \tau_0 >0$ and $u|\tau\setminus \interior \tau_0 = 0$. Similarly,  $u$ is a \emph{shuffle} if 
\begin{itemize}
\item[(1)] $u|\interior \tau_0>0$,
\item[(2)] there exist $i\ne j$ in $\{1,\ldots, n-1\}$ so that $u|\interior \tau_i >0$ and $u|\interior \tau_j < 0$, and
\item[(3)] $u|\tau\setminus (\interior \tau_0 \cup \interior \tau_i\cup \interior \tau_j) = 0$.
\end{itemize}
\end{definition}

\begin{remark}
Note that, if $u\colon \tau \to [-1,1]$ is either an opening or a shuffle, $u|\partial \tau = 0$. 
\end{remark}

\subsubsection{Pillow cover functions}

To each $(n-1)$-simplex $\sigma$ in $E$, we set
\[
\ell_\sigma = \left\{ \begin{array}{ll}
2, & \nu(\sigma)=-1, \\
4, & \nu(\sigma)=1,
\end{array}\right.
\]
and introduce a family of Lipschitz functions 
\[
\Psi_\sigma\colon \sigma\times \{1,\ldots,  \ell_\sigma\} \to [-1,1],
\]
which will form the pillow covers. We consider the two parities separately.

\begin{remark}
For each $\sigma$ and both parities $\nu(\sigma)$, we may assume that the function $\Psi_\sigma$ satisfies the additional regularity condition
\[
\Psi_\sigma(x,i+1)-\Psi_\sigma(x,i) \ge \dist(x,\partial \sigma)/10
\]
for $x\in \sigma$ and $i \in \{1,\ldots, \ell_\sigma-1\}$.

We may also assume, from now on, that mappings $\Psi_\sigma$ are PL and uniformly Lipschitz, that is, there exists $L\ge 1$ (depending only on $n$) so that every $\Psi_\sigma$ is $L$-Lipschitz for every $\sigma$ in $\partial_\cup \bfOmega$ and, in particularly, in the cell $E$. 
\end{remark}

\medskip
\noindent
\emph{Case 1: Functions on negative simplices.}

Suppose $\nu(\sigma)=-1$. We define $u_\sigma \colon \partial \sigma \to [-1,1]$ as follows. Given a face $\tau$ of $\sigma$, we set $u_\sigma|\tau$ to be an opening if $\tau$ is either an entry or an exit face of $\sigma$. If $\tau$ is closed, $u_\sigma|\tau$ is the zero function. Thus we may fix $\Psi_\sigma \colon \sigma \times \{1,2\} \to [-1,1]$ satisfying 
\begin{itemize}
\item[(1)] $\Psi_\sigma(x,1) = 0$ and $\Psi_\sigma(x,2)=u_\sigma(x)$ for all $x\in \partial \sigma$, and
\item[(2)] $\Psi_\sigma(x,1)<0<\Psi_\sigma(x,2)$ for all $x\in \interior \sigma$.
\end{itemize}

\medskip
\noindent
\emph{Case 2: Function on positive simplices.}

For $\nu(\sigma)=1$, two functions $u_\sigma$ and $v_\sigma$ on $\partial \sigma$ will be used in a similar way. Given a face $\tau$ of $\sigma$, take $u_\sigma|\tau$ to be an opening if $\tau$ is either entry or exit face of $\sigma$, and $u_\sigma|\tau = 0$, otherwise. As for $v_\sigma$,  define $v_\sigma|\tau = 0$ for every face $\tau$ of $\sigma$ which is not the exit face, and take $v_\sigma$ to be a shuffle on the exit face of $\sigma$ if such exists. Note that $u_\sigma$ and $v_\sigma$ have (essentially) pair-wise disjoint supports.

We may now fix a function  $\Psi_\sigma \colon \sigma \times \{1,\ldots, 4\} \to [-1,1]$ so that, for $x\in \partial \sigma$, 
\begin{itemize}
\item[(1)] $\Psi_\sigma(x,1)=\Psi_\sigma(x,2)=0$ and $\Psi_\sigma(x,3)=\Psi_\sigma(x,4) = u_\sigma(x)$ if $v_\sigma(x)=0$, \\
\item[(2)] $\Psi_\sigma(x,1)=\Psi_\sigma(x,2)=\Psi_\sigma(x,3)=v_\sigma(x)$ and $\Psi_\sigma(x,4)=0$ if $v_\sigma(x)<0$, \\
\item[(3)] $\Psi_\sigma(x,1)=0$ and $\Psi_\sigma(x,2)=\Psi_\sigma(x,3)=\Psi_\sigma(x,4) = v_\sigma(x)$ if  $v_\sigma(x)>0$,
\end{itemize}
while for $x\in \interior \sigma$, 
\begin{itemize}
\item[(4)] $\Psi_\sigma(x,1)<\Psi_\sigma(x,2)<\Psi_\sigma(x,3)<\Psi_\sigma(x,4)$ and
\item[(5)] $\Psi_\sigma(x,1) < 0 < \Psi_\sigma(x,4)$.
\end{itemize}

\subsubsection{Sheets and a pillow cover}

The singular $(n-1)$-simplices
\begin{equation}
\label{eq:hat_sigma_i}
\hat \sigma_i = \{ (x,\Psi_\sigma(x,i)) \colon x\in \sigma\},
\end{equation}
where $i\in \{1,\ldots, \ell_\sigma\}$, constitute the \emph{sheets of $\sigma$} (as in \cite{R}), and the union of sheets 
\begin{equation}
\label{eq:hat_sigma}
\hat \sigma = \bigcup_i \hat\sigma_i
\end{equation}
forms a \emph{pillow cover on $\sigma$}. Note that a pillow cover of $\sigma$ consists of either $2$ or $4$ singular $(n-1)$-simplices depending on the parity $\nu(\sigma)$ of $\sigma$.

\begin{remark}
Observe that $\{\hat \sigma_1,\ldots \hat\sigma_{\ell_\sigma}\}$ is a (singular) triangulation of $\hat \sigma$ by singular $(n-1)$-simplices. This triangulation, however, does not induce a simplicial complex, since pair-wise intersections of these simplices are generally not unions of sides. For example, $\hat \sigma_1 \cap \hat \sigma_{\ell_\sigma}$ is not a union of faces of $\hat \sigma_1$.   
\end{remark}

\subsubsection{Pillows}

We consider next, in more detail, the complementary domains of $\hat \sigma$ in $\sigma\times \R$.
Let
\[
P_\sigma = \{ (x,t)\in \sigma \times \R \colon \Psi_\sigma(x,1)\le t \le \Psi_\sigma(x,\ell_\sigma)\}.
\]
We call $P_\sigma$ a ``pillow''. Let also
\[
U_\sigma = \{ (x,t)\in \sigma \times \R \colon t \ge \Psi_\sigma(x,\ell_\sigma)\} \\
\]
and
\[
L_\sigma = \{ (x,t)\in \sigma \times \R \colon t \le \Psi_\sigma(x,1))\}.
\]

Independent of the parity of $\sigma$,  $U_\sigma$ and $L_\sigma$ are bilipschitz equivalent to $\sigma \times [0,\infty)$ and $\sigma \times (-\infty,0]$, respectively. For example, for $U_\sigma$, there is the bilipschitz map 
\[
(x,t) \mapsto \left\{ \begin{array}{ll}
(x,2(t-\Psi_\sigma(x,\ell_\sigma)), & \Psi_\sigma(x,\ell_\sigma) \le t \le 2\Psi_\sigma(x,\ell_\sigma) \\
(x,t), & t \ge 2\Psi_\sigma(x,\ell_\sigma)
\end{array}\right.
\]
and similarly for $L_\sigma$ the map
\[
(x,t)\mapsto \left\{ \begin{array}{ll}
(x,2(t-\Psi_\sigma(x,1)), & \Psi_\sigma(x,1) \ge t \ge 2\Psi_\sigma(x,1) \\
(x,t), & t \le 2\Psi_\sigma(x,1)
\end{array}\right..
\]

Since $|\Psi_\sigma|\leq 1$, these homeomorphism are the identity outside $\sigma \times [-2,2]$, and the bilipschitz constant of these homeomorphisms depends only on $n$ and the Lipschitz constant of $\Psi_\sigma$. Similarly, $P_\sigma$ is bilipschitz to an $n$-cell independent of the parity of $\sigma$.

\subsubsection*{Pillows on a negative simplex} 

When $\nu(\sigma)=-1$, we observe that $\partial P_\sigma$ is an essentially disjoint union of $\hat \sigma = \hat \sigma_1\cup \hat \sigma_2$ together with a union of $(n-1)$-cells in $\partial \sigma \times \R$. 

\begin{figure}[h!]
\includegraphics[scale=0.40]{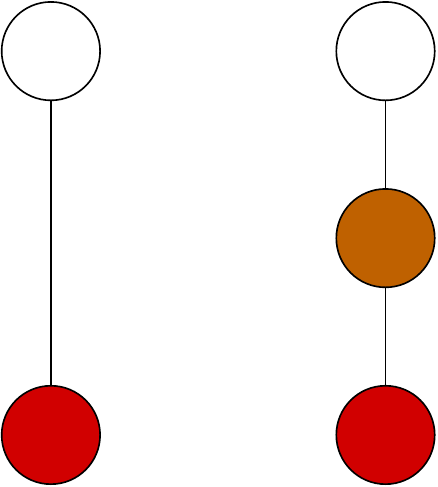}
\caption{Adjacency graphs $\Gamma((\sigma\times \R) \setminus \sigma)$ and $\Gamma((\sigma\times \R)\setminus \hat \sigma)$ for $\sigma$ with negative parity.} 
\label{fig:shuffle_2}
\end{figure}

\subsubsection*{Pillows on a positive simplex}

When $\nu(\sigma)=1$, the complementary domains have more complicated structure. Now $P_\sigma \setminus \hat \sigma$ has three components with closures $P_\sigma^U$, $P_\sigma^M$, and $P_\sigma^L$, respectively,
\begin{eqnarray*}
P_\sigma^U &=& \{ (x,t) \colon \Psi_\sigma(x,1) \le t \le \Psi_\sigma(x,2) \}, \\
P_\sigma^M &=& \{ (x,t) \colon \Psi_\sigma(x,2) \le t \le \Psi_\sigma(x,3) \},\ \mathrm{and} \\
P_\sigma^L &=& \{ (x,t) \colon \Psi_\sigma(x,3) \le t \le \Psi_\sigma(x,4) \}.
\end{eqnarray*}
Although the letters 'U', 'M', and 'L' refer to 'upper', 'middle', and 'lower', the domains are not named by their position along the $t$-axis; these names anticipate Lemma \ref{lemma:ULM} below. We have
\[
\hat \sigma \cap \partial P_\sigma^U = \hat\sigma_1\cup \hat\sigma_2, \
\hat \sigma \cap \partial P_\sigma^M = \hat\sigma_2\cup \hat\sigma_3, \ \mathrm{and}\
\hat \sigma \cap \partial P_\sigma^L = \hat\sigma_3\cup \hat\sigma_4;
\]
see Figure \ref{fig:shuffle_4}.

\begin{figure}[h!]
\includegraphics[scale=0.40]{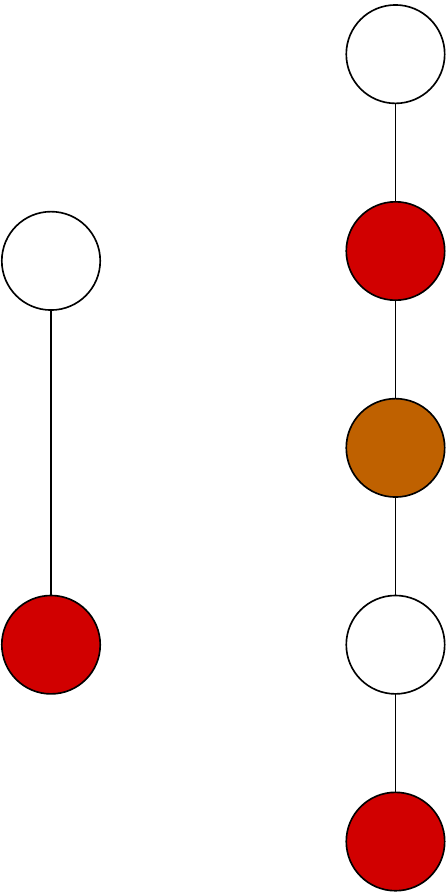}
\caption{Adjacency graphs $\Gamma((\sigma\times \R) \setminus \sigma)$ and $\Gamma((\sigma\times \R)\setminus \hat \sigma)$ for $\sigma$ with positive parity. The merge of domains $P_\sigma^U$ and $U_\sigma$ (as well as $P_\sigma^L$ and $L_\sigma$) in Lemma \ref{lemma:ULM} is anticipated by the choice of colors.}
\label{fig:shuffle_4}
\end{figure}

\subsection{Pillow covers of adjacent simplices}
\label{sec:pcas}

Recall that $E\subset \partial_\cup \bfOmega$ is a planar $(n-1)$-cell, and, to simplify the notation, we have assumed $E\subset \R^{n-1}\times \{0\}\subset \R^n$.
  
Let $\sigma$ be an $(n-1)$-simplex in $E$ as before and suppose that $\sigma'$ is another $(n-1)$-simplex in $E$ sharing an $(n-2)$-simplex with $\sigma$. By changing the r\^oles of $\sigma$ and $\hat \sigma$ if necessary, we may assume $\nu(\sigma')=-\nu(\sigma)=-1$. 

\begin{definition}
\label{def:compatible_pillows}
Pillow covers $\hat \sigma$ and $\hat \sigma'$ of $\sigma$ and $\sigma'$, respectively, are \emph{compatible} if $\Psi_\sigma(\cdot,2)=\Psi_{\sigma'}(\cdot,1)$ and $\Psi_{\sigma}(\cdot,3)=\Psi_{\sigma'}(\cdot,2)$ on $\tau$, where $\tau$ is the common face of $\sigma$ and $\sigma'$.
\end{definition}

From now on we assume that $\hat \sigma$ and $\hat \sigma'$ are compatible pillow covers. The following lemma recapitulates Rickman's idea on using two types of pillow covers to permute the local r\^oles of the three domains. 
\begin{lemma}
\label{lemma:ULM}
Let $\hat \sigma$ and $\hat \sigma'$ be compatible pillow covers of $\sigma$ and $\sigma'$, respectively. Then  
\[
\left( (\sigma\cup \sigma')\times \R\right) \setminus (\hat \sigma\cup \hat\sigma')
\]
has three components $\Omega^U$, $\Omega^M$, and $\Omega^L$ satisfying
\[
\overline{\Omega^U} = U_\sigma \cup P_\sigma^U \cup U_{\sigma'},\ \overline{\Omega^M} = P_\sigma^M \cup P_{\sigma'},\ \mathrm{and}\ \overline{\Omega^L} = L_\sigma \cup P_\sigma^L \cup L_{\sigma'}.
\]
\end{lemma}
\begin{proof}
It suffices to observe that the closures of $P_\sigma^U$ and $U_{\sigma'}$ meet in the $(n-1)$-cell
\[
\{(x,t)\colon \tau \times \R \colon \Phi_\sigma(x,3) \le t \le \Phi_\sigma(x,4)\}.
\]
Similarly, $P_\sigma^L\cap L_{\sigma'}$ is an $(n-1)$-cell.  
\end{proof}

Using the notation of Lemma \ref{lemma:ULM}, we make now few observations on the natural triangulation of $\hat\sigma \cup \hat\sigma'$ into sheets and domains $\Omega^U$, $\Omega^M$, $\Omega^L$.

For $\hat \sigma'$, the pair-wise intersections of domains $\Omega^L$, $\Omega^M$, $\Omega^U$ with $\hat \sigma'\times \R$ are (up to a closure) $L_{\sigma'}$, $P_{\sigma'}$, and $U_{\sigma'}$. Thus 
\[
\partial \Omega^L \cap \hat \sigma' = \hat \sigma'_1,\ 
\partial \Omega^M \cap \hat \sigma' = \hat \sigma'_1 \cup \hat\sigma'_2,\ \mathrm{and}\ 
\partial \Omega^U \cap \hat \sigma' = \hat \sigma'_2.
\]
 
The situation is slightly more complicated with $\hat \sigma$. Note first that $\Omega^M \cap (\sigma\times \R)$ is $P^M_\sigma$ up to closure. Thus
\[
\partial \Omega^M \cap \hat \sigma = \hat \sigma_2 \cup \hat \sigma_3,
\]
and we have 
\[
\hat \sigma_2 = \Omega^U\cap \Omega^M \cap (\sigma \times \R)\ \mathrm{and}\ 
\hat \sigma_3 = \Omega^L\cap \Omega^M \cap (\sigma \times \R).
\]

Moreover,
\[
\overline{\Omega^L\cap (\sigma\times \R)} = \overline{L_\sigma \cup P^L_\sigma}\ \mathrm{and}\ \overline{\Omega^U \cap (\sigma \times \R)} = \overline{U_\sigma \cup P^U_\sigma},
\]
\[
\partial \Omega^L \cap \hat \sigma = \partial L_\sigma \cup \partial P^L_\sigma= \hat \sigma_1 \cup \hat \sigma_3 \cup \hat \sigma_4,
\]
and
\[
\partial \Omega^U \cap \hat \sigma = \partial U_\sigma \cup \partial P^U_\sigma = \hat \sigma_4 \cup \hat \sigma_1 \cup \hat \sigma_2.
\]

This 'shuffle' will allow our domains $\{\Omega_\ell\}$ to connect near $\partial_\cup\bfOmega$. The proof of following lemma is left to the interested reader; the situation is captured by the suggestive figure in \cite[Fig 7.2]{R} and Figure \ref{fig:shuffle}.
\begin{lemma}
\label{lemma:bULM}
With the notation above, we have
\begin{eqnarray*}
\hat \sigma_1 \cup \hat \sigma_4 &=& \Omega^U\cap \Omega^L \cap (\sigma\times \R),\\ 
\hat \sigma_2 &=& \Omega^L\cap \Omega^M \cap (\sigma \times \R),\ \mathrm{and} \\  
\hat \sigma_3 &=& \Omega^M\cap \Omega^U \cap (\sigma \times \R). 
\end{eqnarray*}
Furthermore,
\begin{eqnarray*}
\hat \sigma'_1 &=& \Omega^L \cap \Omega^M \cap (\sigma' \times \R),\ \mathrm{and}\ 
\hat \sigma'_2 = \Omega^M \cap \Omega^U \cap (\sigma' \times \R).
\end{eqnarray*}
\end{lemma}

\begin{figure}[h!]
\includegraphics[scale=0.40]{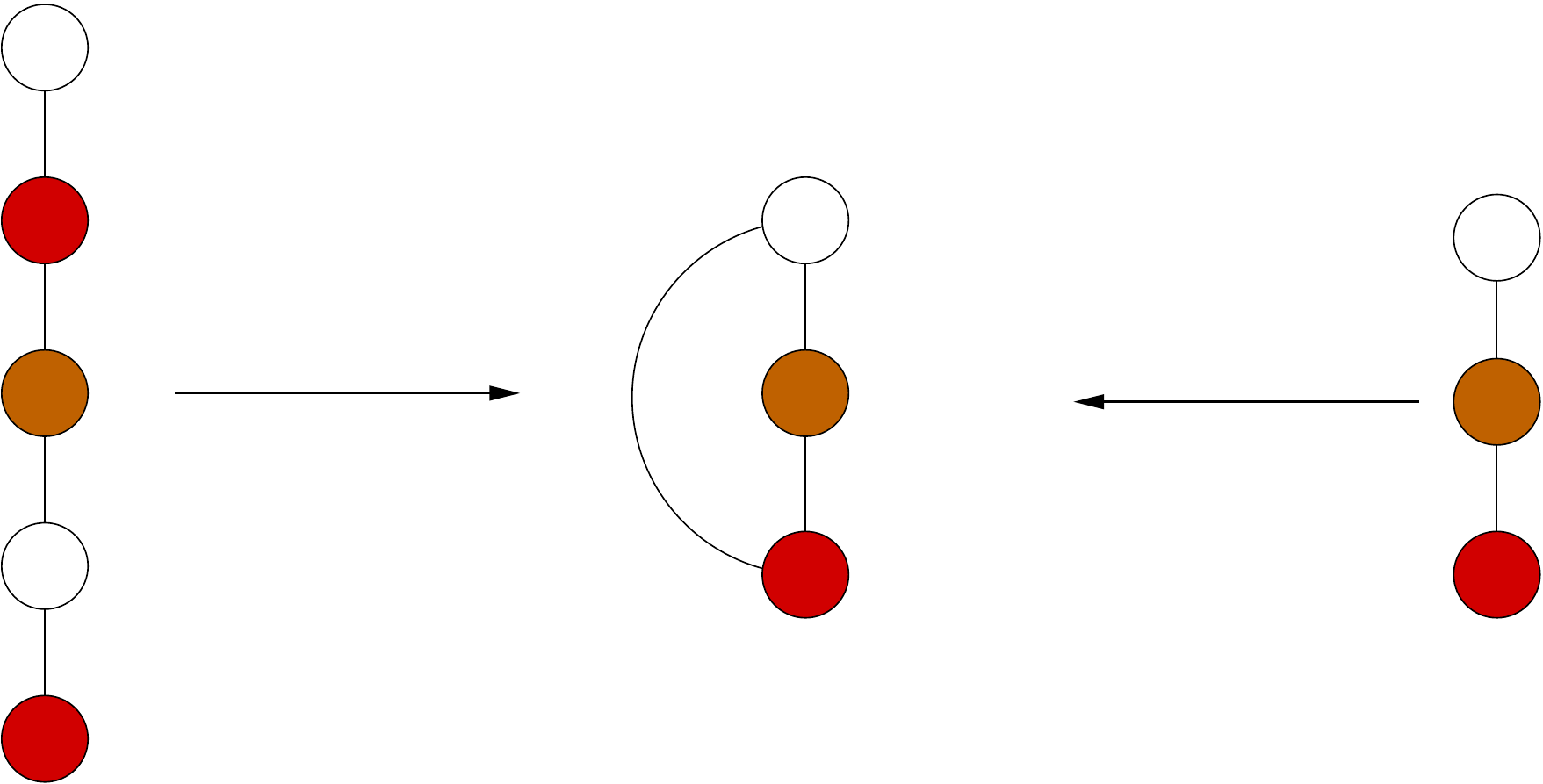}
\caption{Adjacency graphs $\Gamma((\sigma\times \R)\setminus \hat \sigma)$, $\Gamma(((\sigma\cup \sigma')\times \R)\setminus (\hat \sigma \cup \hat \sigma'))$ and $\Gamma((\sigma'\times \R)\setminus \hat \sigma')$ with maps of graphs induced by inclusions.} 
\label{fig:shuffle}
\end{figure}

Our discussion shows that the domains $\Omega^U$, $\Omega^M$, and $\Omega^L$ are bilipschitz equivalent to either $(\sigma\cup \sigma')\times (0,\infty)$, $(\sigma\cup \sigma')\times (-\infty,0)$, or to $\bB^n$. We formalize this observation as follows.
\begin{lemma}
\label{lemma:biliph}
Let $\hat \sigma$ and $\hat \sigma'$ be compatible Lipschitz pillow covers on $\sigma$ and $\sigma'$, respectively. Then
\begin{itemize}
\item[(1)] there exist bilipschitz homeomorphisms $h^U_{\sigma,\sigma'}\colon (\sigma\cup \sigma')\times (0,\infty) \to (\Omega^U,d_{\Omega^U})$ and $h^L_{\sigma,\sigma'} \colon (\sigma\cup \sigma')\times (-\infty,0) \to (\Omega^L,d_{\Omega^L})$ whose supports are contained in $\sigma\cup \sigma' \times [-1/2,1/2]$ so that $h^U_{\sigma,\sigma'}$ and $h^L_{\sigma,\sigma'}$ extend to BLD-maps $(\sigma\cup \sigma')\times [0,\infty) \to \overline{\Omega^U}$ and $(\sigma\cup \sigma')\times (-\infty,0] \to \overline{\Omega^L}$, respectively, and
\item[(2)] the closure of $\Omega^M$ is a bilipschitz $n$-cell.
\end{itemize}
The bilipschitz (and BLD) constants are quantitative in the sense that they depend only on $n$, the Lipschitz constants of $\Psi_\sigma$ and $\Psi_{\sigma'}$ and the minimal bilipschitz constants of homeomorphisms $\sigma \to \bB^{n-1}$ and $\sigma'\to \bB^{n-1}$.
\end{lemma}

\subsection{Maps on pairs of sheets}
\label{sec:LRM1}

The pillow construction on $\sigma\cup \sigma'$ gives rise to maps $\hat \sigma\cup \hat\sigma' \to \hat\bS^{n-1}$, where $\hat\bS^{n-1} = \bS^{n-1} \cup \bB^{n-1} \subset \R^n$. We discuss these local maps now in more detail.

We write $\bS^{n-1} = \bS^{n-1}_+ \cup \bS^{n-1}_-$, where $\bS^{n-1}_+$ and $\bS^{n-1}_-$ are the upper and lower hemispheres of $\bS^{n-1}$, i.e.\;$\bS^{n-1}_+ \cap \bS^{n-1}_- = \partial \bB^{n-1}$. Then $\R^n \setminus \hat\bS^{n-1}$ has three components denoted $D^U$, $D^L$, and $D^M$ so that $\partial D^U = \bS^{n-1}_+ \cup \bB^{n-1}$, $\partial D^L = \bS^{n-1}_- \cup \bB^{n-1}$, and $\partial D^M = \bS^{n-1}$. We fix $n$ points $\{y_0,\ldots, y_{n-1}\}$ on $\partial \bB^{n-1}$ and view $\hat \bS^{n-1}$ as a CW-complex having three $(n-1)$-cells $\bS^{n-1}_+$, $\bS^{n-1}_-$, and $\bB^{n-1}$ and  vertices $\{y_0,\ldots, y_{n-1}\}$.

Let $\sigma$ and $\sigma'$ be adjacent $(n-1)$-simplices in $E$ and let $\hat \sigma$ and $\hat \sigma'$ be compatible Lipschitz pillows on $\sigma$ and $\sigma'$, respectively. By changing the r\^oles of $\sigma$ and $\sigma'$ if necessary, we may assume $\nu(\sigma)=-\nu(\sigma')=1$. Let $\vartheta \colon (\sigma^{(0)} \cup \sigma'^{(0)}) \to \{0,\ldots,n-1\}$ be the labeling function of $\bfOmega$ restricted to $\sigma\cup \sigma'$.

Although the singular simplices $\Delta = \{\hat\sigma_1,\ldots, \hat\sigma_4, \hat \sigma'_1,\hat\sigma'_2\}$ again do not define a simplicial complex, there exists a continuous map $f \colon \hat \sigma \cup \hat \sigma' \to \hat \bS^n$ satisfying 
\begin{itemize}
\item[(S1)] $f$ maps each singular simplex in $\Delta$ to one of the simplices $\bS^{n-1}_+$, $\bS^{n-1}_-$, or $\bB^{n-1}$ in a bilipschitz manner, 
\item[(S2)] $f(v) = y_{\vartheta(v)}$ for all $v\in \sigma^{(0)} \cup (\sigma')^{(0)}$, and
\item[(S3)] if $\{X,Y\}\subset \{U,L,M\}$ is a pair then $f(\Omega^X \cap \Omega^Y) = D^X\cap D^Y$. 
\end{itemize}

Since $f$ is bilipschitz on singular simplices, it is discrete and  
\[
\frac{1}{\cL} \ell(\gamma) \le \ell(f\circ \gamma) \le \cL \ell(\gamma)
\]
for all paths $\gamma$ in $\sigma\cup \sigma'$, where $\cL$ is the maximum of the bilipschitz constants of $f$ restricted to simplices in $\Delta$. Furthermore, in the sense of the following lemma, $f$ is a branched cover in the interior of $\hat \sigma\cup \hat \sigma'$.

\begin{lemma}
\label{lemma:LRM1}
Let $O = (\hat \sigma\cup \hat \sigma') \cap \interior (\sigma\cup \sigma')\times \R$. Then $f|O \colon O \to \hat \bS^n$ is a branched cover and the branch set of $f|O$ is the set
\[
O \cap \{y\in \sigma\cap \sigma' \colon \Psi_\sigma(y,1)=\Psi_\sigma(y,4)\}\subset \R^n.
\]
In particular, $f|O$ is an open map.
\end{lemma}

\begin{proof}
Let $\tau$ be the common face of $\sigma$ and $\sigma'$. Let $S = \hat \sigma \cup \hat \sigma'$ and 
\[
G = \bigcup_{i=1}^4 \interior \hat \sigma_i \cup \bigcup_{j=1}^2 \interior \hat \sigma_j'.
\]
Then
\[
S = G \cup (S\cap (\tau\times \R)) \cup \left(S\cap \partial (\sigma\cup \sigma')\times \R\right).
\]
Clearly $G\subset O$ and $f|G\colon G\to \hat \bS^n$ is a local homeomorphism. Suppose now that $x=(y,t)\in O \cap (\tau\times \R)$. Then $f(x)\in \bS^n \cap \bB^n$.

There are four cases to consider. Suppose first that $y$ has a neighborhood $O'$ in $\tau$ so that $\Psi_\sigma(y',1)=\Psi_\sigma(y',2)$ for $y'\in O'$. Then also $\Psi_\sigma(y',1)=\Psi_{\sigma'}(y',1)$ and $\Psi_\sigma(y',3)=\Psi_\sigma(y',4)=\Psi_{\sigma'}(y',2)$ for $y'\in O'$ by compatibility, and so either $t = \Psi_\sigma(y,1)= \Psi_{\sigma'}(y,1)$ or $t=\Psi_\sigma(y,3)=\Psi_{\sigma'}(y,2)$. In either case, there are exactly three simplices $T_U$, $T_L$, $T_M$ among the simplices $\{\hat\sigma_1,\ldots, \hat\sigma_4, \hat\sigma'_1,\hat\sigma'_2\}$ with $x\in T_U\cap T_L\cap T_M$ and $f(T_U)=\partial D^U$, $f(T_L)=\partial D^L$, $f(T_M)=\partial D^M$. When $y$ has a neighborhood $O'$ with $\Psi_\sigma(y',1)=\Psi_\sigma(y',3)$ or $\Psi_\sigma(y',2)=\Psi_\sigma(y',4)$ for $y'\in O'$, the argument is similar. In all these cases, $f$ is a homeomorphism in a neighborhood of $x$.

In the remaining case, $x\in O \cap (\tau\times \R)$ and $\Psi_\sigma(y,1)=\Psi_\sigma(y,4)$. Then $x$ belongs to all six singular simplices, and $f$ is a branched double cover near $x$.
\end{proof}

\subsection{Pillow covers of cells}
\label{sec:ppc}

Suppose again that $E$ is a planar $(n-1)$-cell, i.e.\;$E$ is contained in an $(n-1)$-plane $P$. We may take $P=\R^{n-1}\times \{0\}$ as in the beginning of Section \ref{sec:pillows}.

Having $\nu=\nu_{E,\bfOmega}$ at our disposal, we fix, a maximal tree $\hat \cE \subset \Gamma(E^{(n-1)})$ and obtain, for every $\sigma\in E^{(n-1)}$, a pillow $\hat \sigma$ compatible with simplices adjacent to $\sigma$ in $E$. The set 
\[
\hat E = \bigcup_{\sigma \in E^{(n-1)}}\hat \sigma 
\] 
is called a \emph{pillow cover on $E$}. By our convention, all pillow covers $\hat \sigma$ for $\sigma\in E^{(n-1)}$ are $\cL$-Lipschitz for $\cL\ge 1$ depending only on $n$, so that $\hat E$ is an \emph{$\cL$-Lipschitz pillow cover}. 

Lemmas \ref{lemma:ULM} and \ref{lemma:biliph} on metric properties of the pillow cover construction for pairs of simplices have counterparts for an $(n-1)$-cell contained in a hyperplane. The proofs are verbatim so we merely state the results.

\begin{lemma}
\label{lemma:planar_complements}
Let $E$ be a cubical  $(n-1)$-cell in $\R^{n-1}$ and $\hat E\subset E\times [-1/2,1/2]$ an $\cL$-Lipschitz pillow on $E$. Then 
\[
E\times [-1,1] \setminus \hat E
\]
has three components $\Omega^U$, $\Omega^M$, and $\Omega^L$, each bilipschitz equivalent to $\bB^n$ in their inner metric, respectively, so that $\Omega^U \supset E\times \{1\}$ and $\Omega^L \supset E\times \{-1\}$. The bilipschitz constant is quantitative and depends only on $n$ and $\cL$.
\end{lemma}

\begin{figure}[h!]
\includegraphics[scale=0.25]{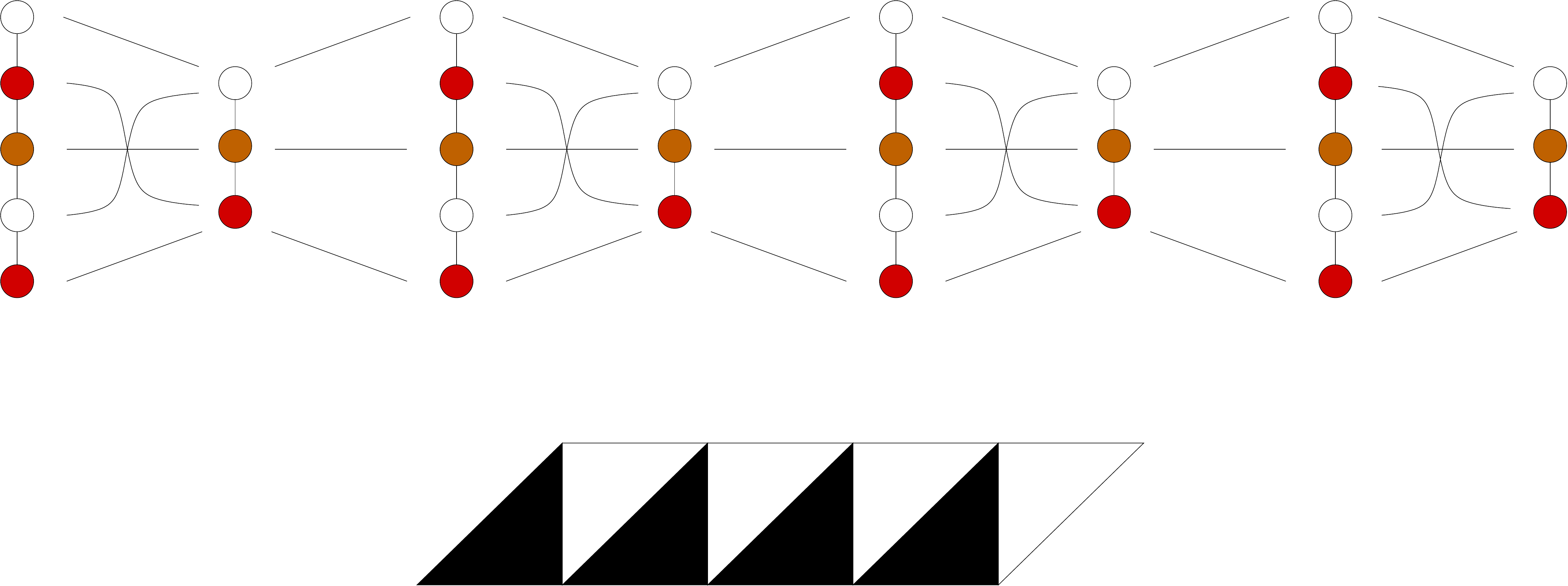}
\caption{The adjacency of domains $\Omega^U$,  $\Omega^M$, $\Omega^L$ over a $2$-cell in $\R^2\times \{0\}\subset \R^3$.}
\label{fig:shuffle_long}
\end{figure}

\begin{lemma}
\label{lemma:planar_BLD}
Let $E$ be a cubical  $(n-1)$-cell in $\R^{n-1}$ and $\hat E\subset E\times [-1/2,1/2]$ an $\cL$-Lipschitz pillow on $E$. Then
\begin{itemize}
\item[(1)] there exists a bilipschitz homeomorphism $h_E^U \colon E\times (0,1) \to (\Omega^U,d_{\Omega^U})$ having a BLD-extension $\bar h_E^U \colon E\times [0,1] \to \overline{\Omega^U}$ so that $\bar h_E^U$ is the identity on $E\times \{1\}\cup \partial E\times [0,1]$.
\item[(2)] there exists a bilipschitz homeomorphism $h_E^L \colon E\times (-1,0) \to (\Omega^L,d_{\Omega^L})$ having a BLD-extension $\bar h_E^L \colon E\times [-1,0] \to \overline{\Omega^L}$ so that $\bar h_E^L$  is the identity on $E\times \{-1\}\cup \partial E\times [-1,0]$. 
\end{itemize}
The statement is quantitative in the sense that the bilipschitz constant depends only on $n$ and $\cL$.
\end{lemma}

In order to define maps $\hat E \to \hat \bS^n$, we fix points $\{y_0,\ldots, y_{n-1}\}\subset \bS^{n-1}\cap \bB^{n-1}$, as in Section \ref{sec:LRM1}. The following lemma is a counterpart of the construction in Section \ref{sec:LRM1}.

\begin{lemma}
\label{lemma:planar_BLD-map}
Let $E$ be a cubical planar $n$-cell in $\R^n$ and $\hat E\subset E\times [-1/2,1/2]$ a pillow on $E$. Then there exists a map $f_E \colon \hat E \to \hat \bS^n$, which is a branched cover on $\interior \hat E = \hat E \cap (\interior E\times \R)$, so that $f_E|(\hat \sigma\cup \hat \sigma')$ satisfies (S1)-(S3) from Section \ref{sec:LRM1} for every pair of adjacent simplices $\sigma$ and $\sigma'$ in $E^{(n-1)}$. The BLD-constant of $f_E|\interior \hat E$ is quantitative in the sense that it depends only on $n$, $\cL$, and points $\{y_0,\ldots, y_{n-1}\}$.
\end{lemma}
\begin{proof}
It suffices to observe that $f_E$ is readily obtained from the discussion in Section \ref{sec:LRM1} and it suffices to discuss the uniformity of the BLD-constant of $f_E|\interior \hat E$. Since $E$ is given a standard simplicial structure, all simplices $\sigma$ in $E^{(n-1)}$ are congruent. For every $\sigma\in E^{(n-1)}$ faces of $\sigma$ are one of the three different types: \emph{entry}, \emph{exit}, and \emph{closed} faces. By fixing opening and shuffle functions invariant under congruences, we may assume that pillows over simplices, with the same combinatorics, are congruent. More precisely, there exist simplices $\sigma_1,\ldots, \sigma_r$ in $E^{(n-1)}$ and compatible pillows so that, for every $\sigma\in E^{(n-1)}$, there exists an isometry $I_\sigma$ of $\R^n$, preserving $\R^{n-1}\times [0,\infty)$, and $1\le i_\sigma \le r$ so that $I_\sigma(\sigma)=\sigma_{i_\sigma}$ and $I_\sigma(\hat \sigma) = \hat \sigma_{i_\sigma}$.   

Thus we fix a finite collection of Lipschitz maps $f_i \colon \hat \sigma_i \to \hat \bS^{n-1}$ and use the isometries $I_\sigma$ to obtain a map $f_E \colon \hat E \to \hat\bS^{n-1}$. The BLD-constant of $f_E|\interior \hat E$ then clearly depends only on the Lipschitz constants of this finite collection $f_1,\ldots, f_r$, depending only on $n$, $\cL$, and the choice of points $\{y_0,\ldots,  y_{n-1}\}$.
\end{proof}

\begin{remark}
\label{rmk:planar_BLD-map}
The standard simplicial structure of $E$ is not essential to the proof of Lemma \ref{lemma:planar_BLD-map}. In fact, given any simplicial complex $P$ in $\R^n$ with $|P|=E$, it is easy to observe that there exists a pillow $\hat E$ on $E$ consisting of compatible pillows $\hat \sigma$ for $\sigma\in P^{(n-1)}$, and a map $f_{E,P} \colon \hat E\to \hat\bS^{n-1}$ satisfying the properties of Lemma \ref{lemma:planar_BLD-map} with the only exception that the BLD-constant of $f_{E,P}|\interior \hat E$ now depends also on the bilipschitz constants of affine parametrizations $[0,e_1,\ldots,e_{n-1}] \to \sigma$ for $\sigma \in P^{(n-1)}$. Although, this observation is essential in what follows, we leave the simple modification of the proof of Lemma \ref{lemma:planar_BLD-map} to the interested reader.
\end{remark}

Suppose now that $E$ is a cubical $(n-1)$-cell in $\R^n$. Since $E$ is a PL $(n-1)$-cell, there exists a PL homeomorphism $E\to E'$, where $E'$ is an $(n-1)$-cell in $\R^{n-1}$. More precisely, there exists a simplicial complex $P$ so that $|P|=E$ and a simplicial homeomorphism $\varphi \colon E \to E'$ with respect to $P$.

Let $E$ be a cubical $(n-1)$-cell $E$ in $\R^n$ and let $\cQ(E)$ be the collection of all unit $n$-cubes $Q$ in $\R^n$ with $Q\cap \interior E\ne \emptyset$, and $|\cQ(E)|$ the union of these cubes. Set
\[
\cN(E) = B_\infty(E,1/3) \cap |\cQ(E)|.
\]
In particular, we have
\[
\cN(E') = E' \times [-1/3,1/3]
\]
for a planar $(n-1)$-cell $E'$ in $\R^n$, and the pair $(\cN(E),E)$ is PL-homeomorphic to proper cell pair $(\bar B^n,\bar B^{n-1})$; see \cite[Chapter 4]{RourkeC:Intplt}. 

We apply these observations to small $(n-1)$-cells in $\R^n$, and summarize the needed properties in the following lemma, omitting details. Note that the uniform bound of the bilipschitz constant follows directly from the finiteness of congruence classes of $(n-1)$-cells in statement.

\begin{lemma}
\label{lemma:E-E'}
Let $E$ be a cubical $(n-1)$-cell in a cube $Q\subset \R^n$ of side length $3$. Then there exist $\cL\ge 1$ depending only on $n$, a planar cubical $(n-1)$-cell $E'$, and an $\cL$-bilipschitz PL-homeomorphism $\varphi_E \colon \cN(E) \to \cN(E')$ so that $\varphi_E(E)=E'$. Moreover, there is a simplicial complex $P$ so that $|P|=E$ and $\varphi_E$ is piecewise affine with respect to $P$.   
\end{lemma}

Having Lemma \ref{lemma:E-E'} at our disposal, we may define pillow covers for small $(n-1)$-cells in $\R^n$. Let $E$ be a cubical $(n-1)$-cell contained in a cube of side length $3$. Suppose $E'$ is a planar $(n-1)$-cell and $\varphi_E \colon \cN(E)\to \cN(E')$ a PL-homeomorphism as in Lemma \ref{lemma:E-E'}. Then $\varphi_E(E^{(n-1)})$ is a triangulation of $E'$. Although $\varphi_E(E^{(n-1)})$ is not the standard triangulation of $E'$, we obtain, a pillow $\hat{E'}$ on $E'$ in $\cN(E')$ with respect to this triangulation, and call $\hat E = \varphi_E^{-1}(\hat{E'})$ a pillow cover of $E$.

Given an $(n-1)$-simplex $\sigma$ in $E$, we also say that $\hat \sigma = \varphi^{-1}(\hat E \cap (\varphi(\sigma)\times [-1,1]))$ is the \emph{pillow over $\sigma$ in $\wh E$}. By finiteness of congruence classes, we conclude that the results in the beginning of this section hold also for these pillow covers almost verbatim.


\subsection{Proof of Proposition \ref{prop:special_4}}
\label{sec:PC_RP}

Let $\bfOmega=(\Omega_1,\Omega_2,\Omega_3)$ be a rough Rickman partition of $\R^n$ having the tripod property.
Thus  $\partial_\cup \bfOmega$ has an essential partition into cubical $(n-1)$-cells $\Delta=\{E_\ell\}_{\ell\ge 0}$.

Given adjacent $E_\ell$ and $E_{\ell'}$ in $\Delta$ belonging to different $\bfOmega$-equivalence classes (recall Definition \ref{def:U-eqv}), there exists, by property ($\Delta$2) {of Definition \ref{def:tripod}, a unique $E_{\ell''}$ in $\Delta$ so that the cells $E_\ell$, $E_{\ell'}$, and $E_{\ell''}$ are mutually adjacent, contained in the same cube of side length $3$, and belong to different $\bfOmega$-equivalence classes. If we write  $E_\ell\sim E_{\ell'}$, the relation $\sim$ defines an equivalence relation in $\Delta$ which subdivides $\Delta$ into equivalence classes containing exactly three elements.

Let 
\[
\cN(\partial_\cup \bfOmega) = B_\infty(\partial_\cup \bfOmega, 1/3)
\]
be the $(1/3)$-neighborhood of $\partial_\cup \bfOmega$ in $\R^n$, and for each $\ell$ define
\[
\cN_\ell = \{ x\in \cN(\partial_\cup\bfOmega) \colon \dist(x,E_\ell)=\dist(x,\partial_\cup \bfOmega)\}.
\]
Then $\{\cN_\ell\}_{\ell\ge 0}$ is an essential partition of $\cN(\partial_\cup \bfOmega)$. Moreover, $\cN_\ell$ is PL-homeo\-morphic to $\cN(E_\ell)$ for every $\ell$. Due to finite number of congruence classes of $\cN_\ell$ and $\cN(E_\ell)$, we have that $\cN_\ell$ is bilipschitz to $\cN(E_\ell)$, the constant depending only on $n$. 

Suppose $E_{\ell_0}$, $E_{\ell_1}$, and $E_{\ell_2}$ are equivalent $(n-1)$-cells in $\Delta$. We create pillows $\hat E_{\ell_0}$, $\hat E_{\ell_1}$ and $\hat E_{\ell_2}$ simultaneously. Let $E_{[\ell]} =E_{\ell_0}\cup E_{\ell_1}\cup E_{\ell_2}$ and $\cN_{[\ell]} = \cN_{\ell_0}\cup \cN_{\ell_1}\cup \cN_{\ell_2}$.
We fix, for $m=0,1,2$, indices $\{i_m,j_m,k_m\}=\{1,2,3\}$ so that $E_{\ell_m}\cap \Omega_{k_m}$ is an $(n-2)$-cell and $E_{\ell_m} \subset \Omega_{i_m}\cap \Omega_{j_m}$.

Let 
\[
Y=\left(\R^{n-1}\times \{0\}\right) \cup \left(\{0\}\times \R^{n-2}\times [0,\infty)\right) \subset \R^n.
\]

Since $E_{[\ell]}=E_{\ell_0}\cup E_{\ell_1}\cup E_{\ell_2}$ is a union of equivalent elements in a $\Delta$, we may fix essentially disjoint $(n-1)$-cells $E'_{\ell_0},E'_{\ell,1},E'_{\ell,2}$ in $Y$ so that there exists a map $\phi_{[\ell]} \colon E'_{[\ell]}\to E_{[\ell]}$, where $E'_{[\ell]} = E'_{\ell_0}\cup E'_{\ell_1}\cup E'_{\ell_2}$, which is a PL homeomorphism $E'_{\ell_0}\cap E'_{\ell_1}\cap E'_{\ell_2} \to E'_{\ell_0}\cap E'_{\ell_1}\cap E'_{\ell_2}$ and $E'_{\ell_k}\to E_{\ell_k}$ for each $k$.

The map $\phi_{[\ell]}$ extends to a PL-map $\phi_{[\ell]} \colon \cN(E'_{[\ell]}) \to  \cN_{[\ell]}$ which is a homeomorphism from the interior of $\cN(E'_{[\ell]})$ to the interior of $\cN_{[\ell]}$, where $\cN(E'_{[\ell]}) = \bigcup_{m=0}^2 \cN(E'_{\ell_m})$. The connected components of $\cN(E'_{[\ell]})\setminus Y$ are $U_m = \psi_{[\ell]}(\interior \Omega_m \cap \cN_{[\ell]})$ for $m=0,1,2$.

Again, by finiteness of congruence classes, $\phi'_{[\ell]}=\phi_{[\ell]}|\interior E'_{[\ell]}\colon \interior E'_{[\ell]}\to \interior E_{[\ell]}$ is bilipschitz (in the inner metric) with constant depending only on $n$. Each map $\phi'_{[\ell]}$ induces a triangulation on $E'_{[\ell]}$, and we denote by $\nu$ the parity function $\sigma \mapsto \nu_\bfOmega(\phi_{[\ell]} \circ \sigma)$ defined on $(n-1)$-simplices in the induced triangulation of $E'_{[\ell]}$.

\begin{figure}[h!]
\includegraphics[scale=0.40]{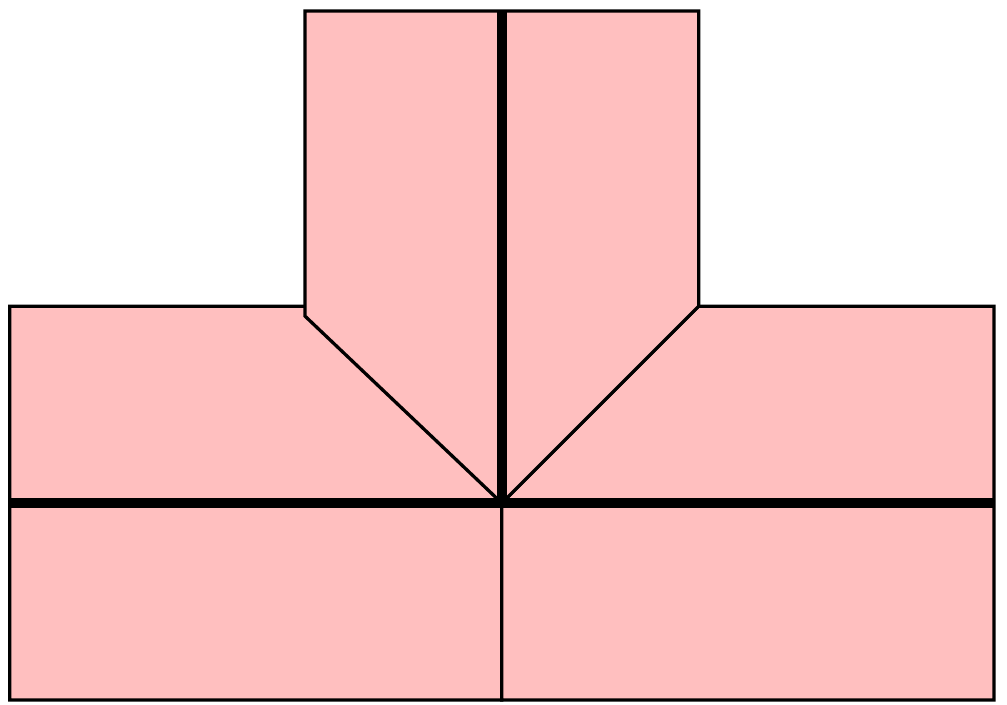}
\caption{Cells $E'_\ell$, $E'_{\ell'}$, and $E'_{\ell''}$ meeting at $\{0\}\times \R^{n-2}\times \{0\}$ and partition of $\cN(E'_{[\ell]})$.} 
\label{fig:N_E_3}
\end{figure}

In terms of this function $\nu$ on $E'_{[\ell]}$, we fix, for every $m=0,1,2$, a Lipschitz pillow $\hat E'_{\ell_m}\subset B_\infty(E'_{\ell_m},1/3)$. By Lemma \ref{lemma:planar_complements}, $\cN(E'_{\ell_m}) \setminus \hat E'_{\ell_m}$ has three components and there exists a unique component $D'_m\subset \cN(E'_{\ell_m})\setminus \hat E'_{\ell_m}$ which does not meet $\partial \cN(E'_{\ell_m})$ essentially; that is, the intersection $D'_m \cap \cN(E'_{\ell_m})$ does not contain $(n-1)$-simplices.

\begin{figure}[h!]
\includegraphics[scale=0.40]{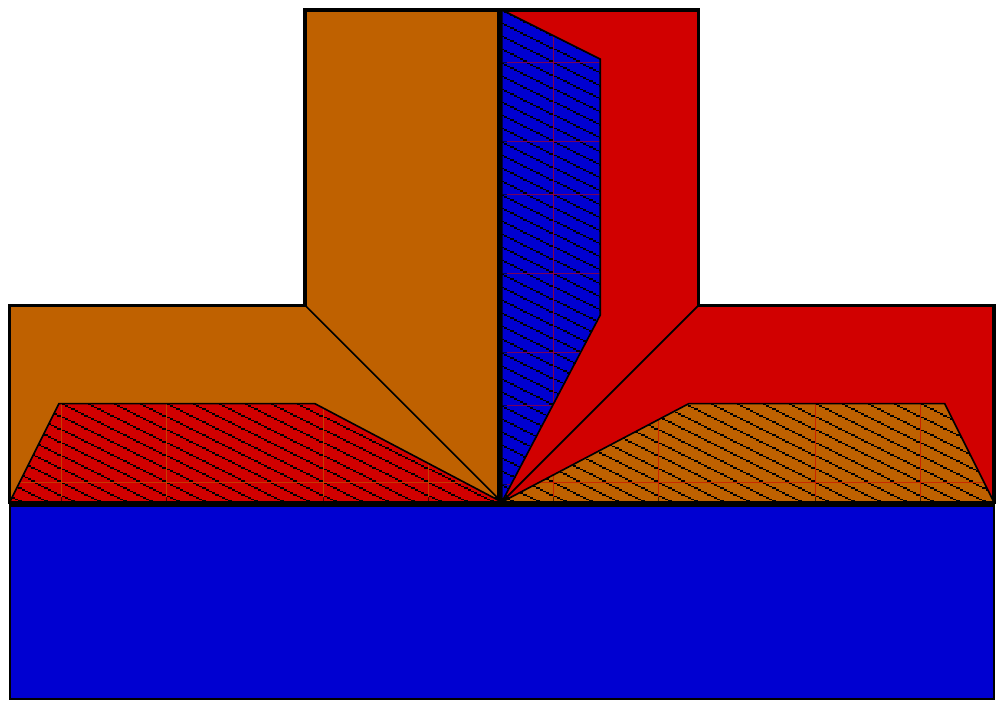}
\caption{Three components $D'_m$ waiting to be connected to corresponding components $U_m$.} 
\label{fig:N_E_3_C}
\end{figure}

We observe that the set $\bigcup_{m=0}^2 \hat E'_{\ell_m}$ has $6$ complementary components in $\cN(E'_{[\ell]})$; see Figure \ref{fig:N_E_3_C}. We now modify the pillows $\hat E'_{\ell_m}$; informally, by connecting each $D'_m$ to $U_m$, there will be only three complementary components.

For $m=0,1,2$, let $\sigma_{\ell_m}\subset E'_{\ell_m}$ be simplices meeting on a common face $\tau \subset \sigma_{\ell_0}\cap \sigma_{\ell_1}\cap \sigma_{\ell_2}$. By Lemma \ref{lemma:parity}, all simplices $\sigma_{\ell_m}$ have the same $\nu$-parity. For notational simplicity, we consider only the case $\nu(\sigma_{\ell_m}) = -1$; the case $\nu(\sigma_{\ell_m})=1$ is similar and is left to the reader. 

We fix three subsimplices by $\tau_0$, $\tau_1$, $\tau_2$ of $\tau$ by subdiving $\phi(\tau)$ into congruent subsimplices of side length $1/3$ and then fixing three of the preimages of these subsimplices in $\tau$.

Since $\nu(\sigma_{\ell_0})=\nu(\sigma_{\ell_1}) = -1$, the sheets $\hat \sigma_{\ell_0}$ and $\hat \sigma_{\ell_1}$ of $\sigma_0$ and $\sigma_1$, respectively, are determined by functions $\Psi_{\sigma_{\ell_0}}$ and $\Psi_{\sigma_{\ell_1}}$. We modify these functions so that 
\[
\Psi_{\sigma_{\ell_r}}(\interior\tau_r\times \{2\})\subset (0,1/3)
\]
for $r=0,1$, and denote the new sheets obtained in this manner as $\tilde \sigma_{\ell_r}$ for $r=0,1$. We denote also by $\tilde D'_r$ the component of $\cN(E'_{\ell_r})\setminus \tilde \sigma_{\ell_r}$ which does not meet $\partial \cN(E'_{\ell_r})$ essentially. 

For $r=0,1$, let $\tilde U_{k_r}$ be the components of $\cN(E'_{\ell_r})\setminus \tilde \sigma_{\ell_r}$ contained in $U_{k_r}$. It is now easy to observe that $\tilde D'_r \subset \tilde U_{k_r}$ is connected. Indeed, the $(n-2)$-cell 
\[
D_r = \{ (x,t)\in \tau_r \times \R \colon \Psi_{\sigma_{\ell_r}}(x,1) \le t \le \Psi_{\sigma_{\ell_r}}(x,2)\}
\]
for $r=0,1$, lies on the boundary of $\tilde D'_r$ and is contained in $\tilde U_{k_r}$. Furthermore, we have that the interior of $\cl(\tilde D'_r \cup U_{k_r})$ is bilipschitz to $\bB^n$ in the inner metric.

Without changing notation, we modify the sheet $\hat \sigma_{\ell_2}$ accordingly in order to preserve compatibility with other sheets after this change on $\tau_0\cup \tau_1$. The sheet modification is now applied to $\hat \sigma_{\ell_2}$ to obtain a new compatible sheet $\tilde \sigma_{\ell_2}$ so that the component $D'_2$ of $B_\infty(E'_2,1/3)\setminus \tilde \sigma_{\ell_2}$ is connected to $U_{k_2}$. We leave the details of this step to the interested reader.

\begin{figure}[h!]
\includegraphics[scale=0.40]{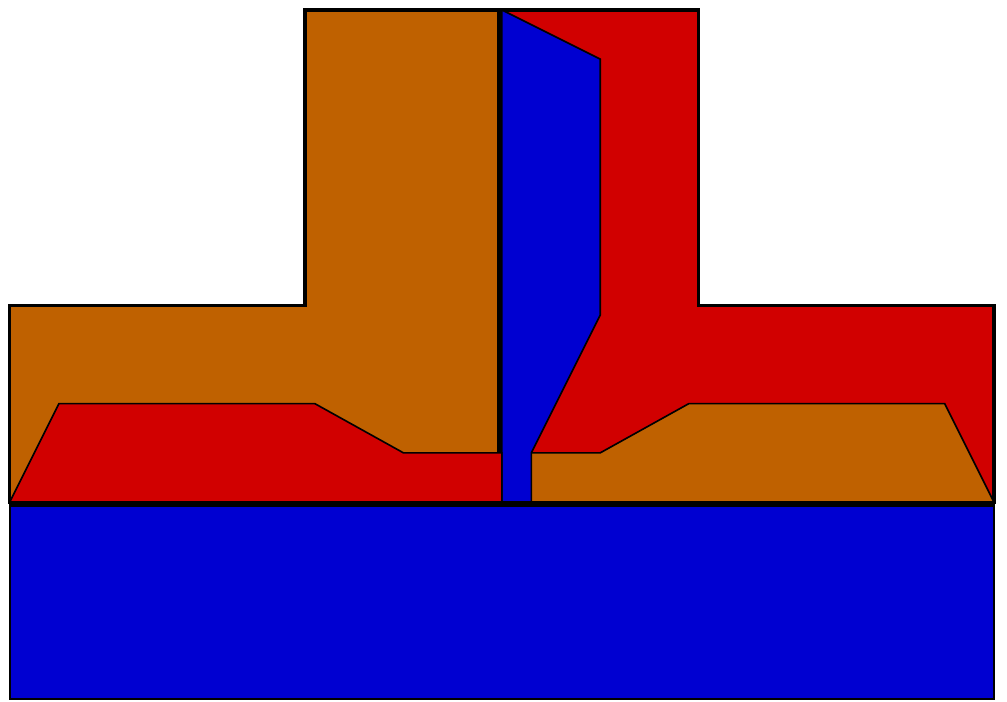}
\caption{Domains after modification; a side view.} 
\label{fig:N_E_3_C_2}
\end{figure}

We make some observations on the construction of the modified sheets $\tilde \sigma_{\ell_m}$ for $m=0,1,2$.  First note that, although $\tilde \sigma_{\ell_m}$ is not homeomorphic to $\hat \sigma_{\ell_m}$ there exist maps $h_{\ell_m} \colon \tilde \sigma_{\ell_m} \to \hat \sigma_{\ell_m}$ so that $h_{\ell_m}$ is a homeomorphism in the interior of $\tilde \sigma_{\ell_m}$ and $h_{\ell_m}|(\tilde \sigma_{\ell_m}\cap \hat \sigma_{\ell_m}) = \id$. In particular, $\tilde \sigma_{\ell_m}$ has the same number of singular simplices as does $\hat \sigma_{\ell_m}$ and the map $h_{\ell_m}$ restricts to a map between singular simplices. 

Second, let
\[
\tilde E'_{[\ell]} = \tilde \sigma_{\ell_0}\cup \tilde \sigma_{\ell_1}\cup \tilde \sigma_{\ell_2}.
\]
Then 
\[
\cN(E'_{[\ell]}) \setminus \tilde E'
\]
has three connected components $\tilde U_1,\tilde U_2,\tilde U_3$ with the property 
\[
\partial \tilde U_r \cap \partial U_r = B_\infty(Y,1/3) \cap \partial U_r,
\]
and for every $r=1,2,3$, there exists a bilipschitz homeomorphism 
\[
(\tilde U_r, d_{\tilde U_r}) \to (U_r,d_{U_r}),
\]
which is the identity on $\partial \tilde U_r \cap \partial U_r$. 

Let
\[
\tilde E_{[\ell]} = \phi_{[\ell]}(\tilde E'_{[\ell]}).
\]
Due to the convention on closed edges on the boundary of $\partial( \bigcup_{r=0}^2 E_{\ell_r})$, we have that
\[
\tilde E_{[\ell]} \cap \tilde E_{[\ell']} = E_{[\ell]} \cap E_{[\ell']}
\]
for all $\ell$ and $\ell'$.

Write
\[
X = \bigcup_{[\ell]} \tilde E_{[\ell]},
\]
the union over equivalence classes $[\ell]$ of indices. Then $\R^n\setminus X$ has components $\tilde \Omega_1, \tilde \Omega_2, \tilde \Omega_3$. Using the congruence classes of pillows $\hat E_{[\ell]}$, we may assume that pillows $\tilde E_{[\ell]}$ are uniformly Lipschitz. Then the components $\tilde \Omega_1$, $\tilde \Omega_2$, and $\tilde \Omega_3$ are bilipschitz equivalent to the components $\Omega_1$, $\Omega_2$, and $\Omega_3$ of our original Rickman partition, respectively, in their inner metric. Furthermore, these bilipschitz homeomorphisms $(\Omega_m,d_{\Omega_m}) \to  (\tilde \Omega_m,d_{\tilde\Omega_m})$, $m=1,2,3$, extend to BLD-maps $\cl(\Omega_m) \to \cl(\tilde \Omega_m)$. If we set $\tilde \bfOmega = (\tilde \Omega_1,\tilde \Omega_2,\tilde \Omega_3)$, then $X=\partial_\cup \tilde \bfOmega$.

Finally, we obtain a BLD-map $f \colon \partial_\cup \tilde \bfOmega \to \hat \bS^{n-1}$. Relabel the components of $\R^n\setminus \hat \bS^n$ by $D_1,D_2,D_3$ so that $D_1 = D^U$, $D_2=D^L$, and $D_3 = D^M$.

By Remark \ref{rmk:planar_BLD-map}, we may fix a map $g_{[\ell]} \colon \tilde E'_{\ell} \to \hat \bS^{n-1}$ as in Lemma \ref{lemma:planar_BLD-map}. By Lipschitz uniformity of the pillows $\tilde E'_{[\ell]}$, we may assume that $g_{[\ell]}|\interior \tilde E'_{[\ell]}$ is BLD with BLD-constant depending only on $n$. 

Let $f_{[\ell]} \colon \tilde E_{[\ell]} \to \hat \bS^{n-1}$ be the unique map satisfying $f_{[\ell]} \circ \phi_{[\ell]} = g_{[\ell]}$.

Given adjacent pillows $\tilde E_{[\ell]}$ and $\tilde E_{[\ell']}$, the mappings $f_{[\ell]}$ and $f_{[\ell']}$ are both defined on $\tilde E_{[\ell]}\cap \tilde E_{[\ell']}$. By uniformity of the BLD-constants, we may modify one of the mappings $f_{[\ell]}$ and $f_{[\ell']}$ slightly to obtain a new collection of uniformly BLD-mappings so that mappings $f_{[\ell]}$ and $f_{[\ell']}$ agree on $\tilde E_{[\ell]}\cap \tilde E_{[\ell']}$ for every $\ell\ne \ell'$. The map $f$, defined by $f|\tilde E_{[\ell]} = f_{[\ell]},$ is BLD.

This concludes the proof of Proposition \ref{prop:special_4}.



\section{Finishing touch}
\label{sec:FT}

In this section we prove Theorem \ref{thm:3} and Proposition \ref{prop:4}. The proofs are slight generalizations of Theorem \ref{thm:RP} and Proposition \ref{prop:special_4}. The proof of Proposition \ref{prop:4} is a straightforward modification, so we merely indicate the differences. For Theorem \ref{thm:3}, we introduce a particular class of rough Rickman partitions, called \emph{skewed Rickman partitions}, and show that the method to obtain a rough Rickman partition in the proof of Theorem \ref{thm:RP} may be modified to obtain skewed Rickman partitions. 
 
\subsection{Skewed Rickman partitions}

For general $p>2$, choose points $\{y_0,\ldots, y_p\}$ in $\bS^n$ as in the introduction, that is, $y_0 = e_{n+1}$ and $y_r = (0,t_r)\in \R^n$, where $-1/2=t_1 < 0 < t_2 < \cdots < t_p = 1/2$. Take $n$-cells $E_0,\ldots, E_p$ as in the introduction, i.e.\;$E_0 = \mathrm{cl}(\bS^n \setminus \bB^n)$, $E_1\cup \ldots \cup E_p = \bB^n$, $y_r \in \interior E_r$, so that $D_r = E_r\cap E_{r+1}$ is an $(n-1)$-cell for $r=0,\ldots p$ ($\mathrm{mod}(p+1)$). 
Note that $\partial E_r$ is an $(n-1)$-sphere consisting of $(n-1)$-cells $D_r\cup D_{r-1}$ ($\mathrm{mod}(p+1)$). 

Let
\[
\hat \bS^{n-1}_p = \bigcup_{r=0}^p \partial E_r.
\]
We emphasize that $E_i \cap E_j = \bS^{n-2}$ for $|i-j|>1$. Let $\cE_p = (E_0,E_1,\ldots, E_p)$. Then $\cE_p$ is an essential partition of $\bS^n$, $\partial_\cup \cE_p = \hat \bS^{n-1}_p$, $\partial_\cap \cE_p = \bS^{n-2}$ and the adjacency graph $\Gamma(\cE_p)$ is cyclic. 

\begin{figure}[h!]
\includegraphics[scale=0.40]{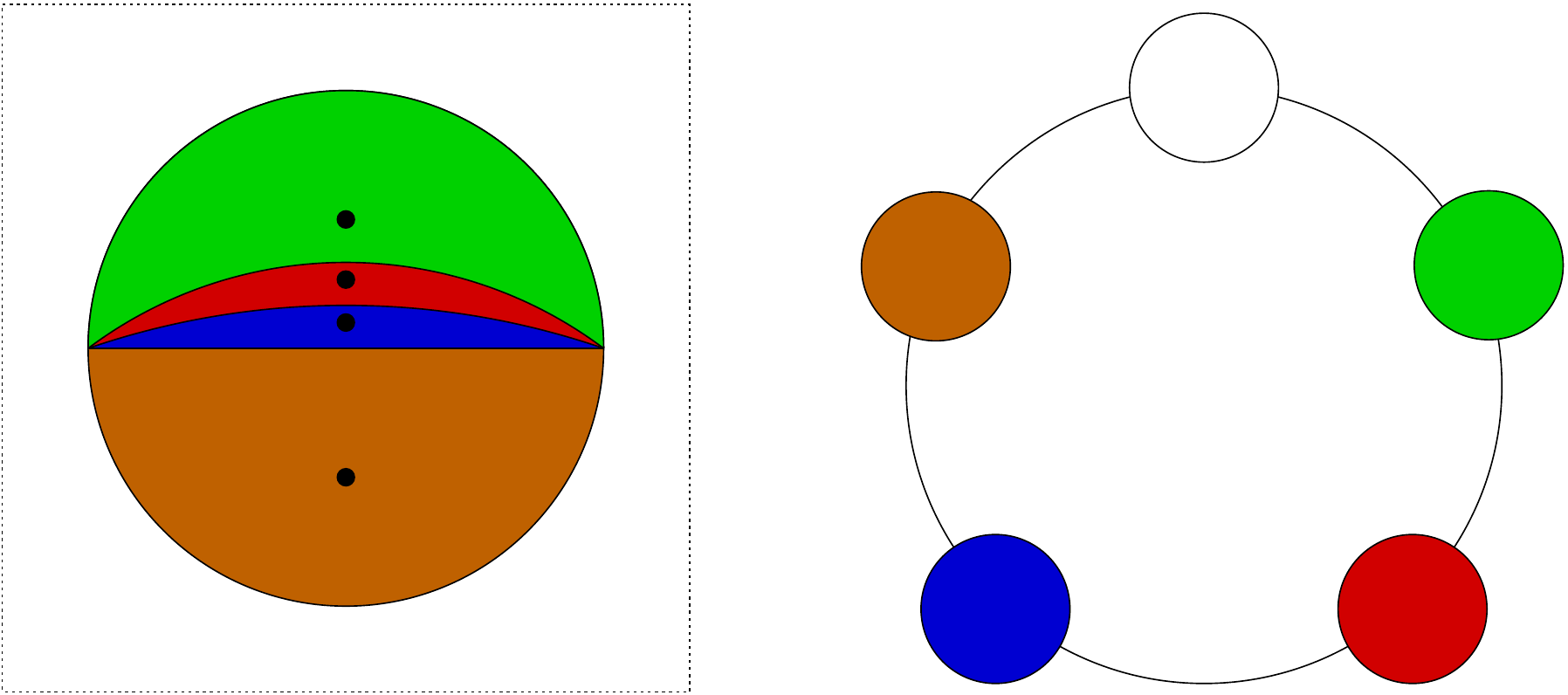}
\caption{The adjacency graph for cells in Figure \ref{fig:Gp_cells_intro}; $p=4$.} 
\label{fig:Gp_graph}
\end{figure}

Let $q$ be a $k$-cube. A PL embedding $\varphi\colon q\to \R^n$ is a \emph{PL $k$-cube} and a complex composed of PL cubes a \emph{skew complex} if the PL $k$-cubes are PL and uniformly bilipschitz equivalent. A Rickman partition $\bfOmega$ is \emph{skew} if $\partial_\cup \bfOmega$ is a skew complex.

The tripod property (Definition \ref{def:tripod}) admits a straightforward generalization to skew compexes $\bfOmega=(\Omega_0,\ldots, \Omega_p)$. Indeed, instead of having three elements in an equivalence class, we require that $\partial_\cup \bfOmega$ have an essential partition $\Delta$ into (skew) $(n-1)$-cells, and we require that $\Delta$ in turn admit a partition into groups of $p+1$ elements, each $(n-1)$-cell between different elements of $\bfOmega$, and all having a common intersection containing an $(n-2)$-cell. In this case we say that the skew complex satisfies a \emph{generalized tripod property}.

We show that $\R^n$ admits a skew Rickman partition $\bfOmega = (\Omega_0,\ldots, \Omega_p)$ for each $p>2$.
\begin{proposition}
\label{prop:3_skew}
Given $n\ge 3$ and $p>2$ there exists a skew Rickman partition $\bfOmega=(\Omega_0,\ldots,\Omega_p)$ supporting the (generalized) tripod property. 
\end{proposition}

\subsubsection{Skew structures on atoms and molecules}
\label{sec:skew}

An essential partition $\cS$ of an $n$-cell $C$ is \emph{skew} if elements of $\cS$ are skew $n$-cells. Before proceeding further, we discuss skew partitions for (flat) atoms and (non-flat) molecules.

Let $A$ be an $r$-fine $\R^{n-1}$-based atom in $\R^n$; let $F=A\cap \R^{n-1}$ and $C = \partial A - F$, where 'F' refers to 'floor' and 'C' to 'ceiling'. Note that $F$ and $C$ are $(n-1)$-cells. We partition $A$ into skew atoms $A_1,\ldots, A_{p-1}$ as follows.

Let $L_1 = F$, $L_p=C$, and define, for $j=2,\ldots, p-1$, $L_j = \{ (x,\delta_{B,j}(x)) \in A \colon x\in F\}$, where $\delta_{B,j} \colon F \to [0,r/3]$ is the function
\[
\delta_{B,j}(x) = \frac{j}{p} \max\{ \frac{r}{3}, \dist(x,F\cap C)\}
\]  
for $x\in F$. For every $j=1,\ldots,p-1$, $L_j\cup L_{j+1}$ bounds a unique $n$-cell $A_j$ with boundary $L_j\cup L_{j+1}$.

\begin{figure}[h!]
\includegraphics[scale=1]{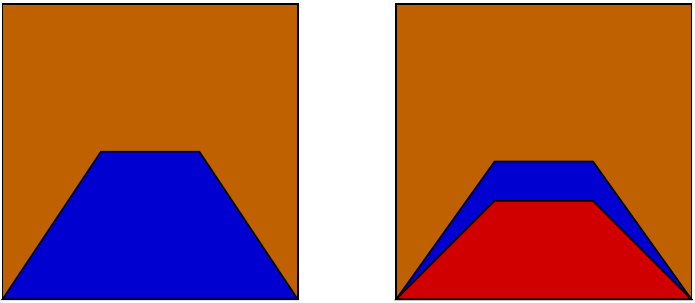}
\caption{Schematic figure on $n$-cells $A_j$ in $A$ for $p=3$ and $p=4$.} 
\label{fig:graphs_L_j}
\end{figure}

Now the essential partition 
\begin{equation}
\label{eq:skewpartition}
\cS(A) = (A_1,\ldots, A_{p-1})
\end{equation}
is a skew partition of $A$. Note that $F\subset \partial A_1$, $C\subset \partial A_{p-1}$, and $A_j\cap A_{j+1}$ is the $(n-1)$-cell $L_{j+1}$ for all $j=1,\ldots, p-1$. 

We leave the details of the following lemma to the interested reader.

\begin{lemma}
\label{lemma:skew_atoms}
Let $A$ be an $r$-fine $\R^{n-1}$-based atom in $\R^n$ and $\cS(A) = (A_1,\ldots, A_{p-1})$ a skew partition of $A$ in \eqref{eq:skewpartition} for $p>2$. Then there exist $L$-bilipschitz homeomorphisms $\varphi_j \colon A \to A_j$ for $j=1,\ldots, p-1$, where $L=L(n,p)$, so that on $F=A\cap \R^{n-1}$ and $C = \partial A-F$ we have
\begin{itemize}
\item[(i)] $\varphi_1|F = \id$, $\varphi_{p-1}|C = \id$, $\varphi_j|F\cap C = \id$ for each $j$, and 
\item[(ii)] $\varphi_j(F)=L_j$ and $\varphi_j(C)=L_{j+1}$ for each $j$.
\end{itemize}
\end{lemma}

Skew partitions of atoms merge to produce skew partitions of molecules.

\begin{lemma}
\label{lemma:skew_molecules}
There exists $L=L(n,p)$ with the following properties. Let $M$ be a molecule consisting of building blocks on the boundary of an $n$-cube $Q$ so that pair-wise unions of adjacent building blocks of different scales are planar. Then there exist an essential skew partition $\cS(M)=(M_1,\ldots, M_{p-1})$ of $M$ into $n$-cells and $L$-bilipschitz homeomorphisms $\psi_j \colon M \to M_j$, $j=1,\ldots, p-1$, for which 
\begin{itemize}
\item[(a)] $\partial M\cap \partial Q\subset \partial M_1$, $\partial M -  \partial Q\subset \partial M_{p-1}$,
\item[(b)] $\psi_i(M)\cap \psi_j(M)$ is an $(n-1)$-cell if $j=i+1$, and
\item[(c)] $\psi_i(M)\cap \psi_j(M) = \partial M \cap \partial Q$ for $|i-j|>1$.
\end{itemize}
\end{lemma}

\begin{proof}
It suffices to consider two cases: (i) a non-planar atom in $\Gamma(M)$, and (ii) two adjacent atoms in $\Gamma(M)$. 
 
Suppose first that $A$ is a non-planar atom in $\Gamma(M)$. Then $A$ consists of planar parts, all meeting in pairs of building blocks. Thus the general case follows from the special case of building blocks $B$ and $B'$ based on different faces of an $n$-cube, say $Q'$ (see Figure \ref{fig:Gp_blocks_join_1}). There exists a cube $q$ of side length $r$ in $B\cup B'$ contained in one of the building blocks, say $B$, so that $q\cap B'=B\cap B'$. Since $A' = B'\cup q$ is an atom, we find skew atoms $A'_j$ and $B_j$ for $j=1,\ldots, p-1$, in $A'$ and $B$, respectively, so that $A'_j\cup B_j$ is an $n$-cell for each $j$ and $A'_1$ and $B_1$ meet $\partial Q'$. Since $A'_j\cup B_j$ are $n$-cells for $j=1,\ldots, p-1$, it is now easy to define non-planar $n$-cells $A_1,\ldots, A_{p-1}$ forming an essential partition of $A$.

\begin{figure}[h!]
\includegraphics[scale=0.25]{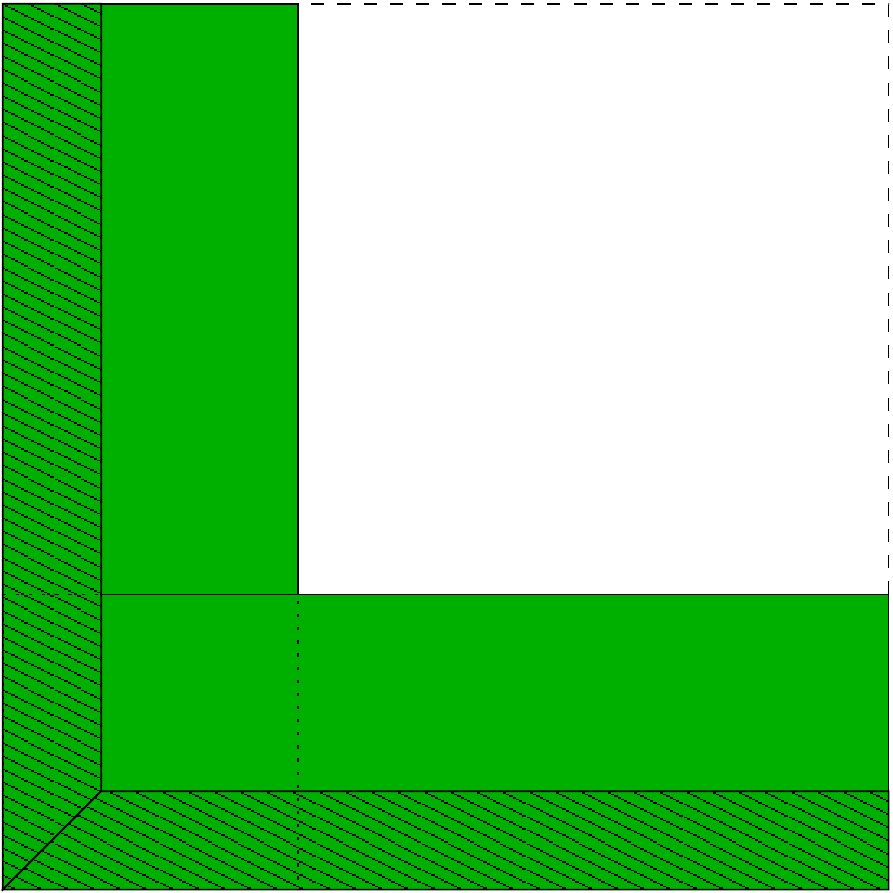}
\caption{Join of two skew non-planar building blocks.} 
\label{fig:Gp_blocks_join_1}
\end{figure}

Suppose now that $A$ is an $r$-fine atom adjacent to an $(r/3)$-fine atom $A'$. Again, there exist building blocks $B\subset A$ and $B'\subset A'$ so that $A'\cap A = B\cap B'$. We may assume that $B\cup B'$ is $\R^{n-1}$-based. Let $\cS(B)=(B_1,\ldots, B_{p-1})$ and $\cS(B')=(B'_1,\ldots, B'_{p-1})$ be skew partitions of $B$ and $B'$. Let $\varphi_j \colon B\to B_j$ and $\varphi'_j \colon B'\to B'_j$ be homeomorphisms as in Lemma \ref{lemma:skew_atoms}. It is now easy to modify these homeomorphisms on $A\cap A'$ to obtain homeomorphisms $\tilde \varphi_j$ and $\tilde \varphi'_j$ for $j=1,\ldots, p-1$, so that each $\tilde \varphi_j(B)\cup \tilde \varphi_j(B')$ is an $n$-cell. Since the modification is local, we may also assume that mappings $\tilde \varphi_j$ and $\tilde\varphi'_j$ are uniformly bilipschitz with constant depending only on $n$. We leave further details to the interested reader; see Figure \ref{fig:Gp_blocks_join_2}.
\end{proof}

\begin{figure}[h!]
\includegraphics[scale=0.35]{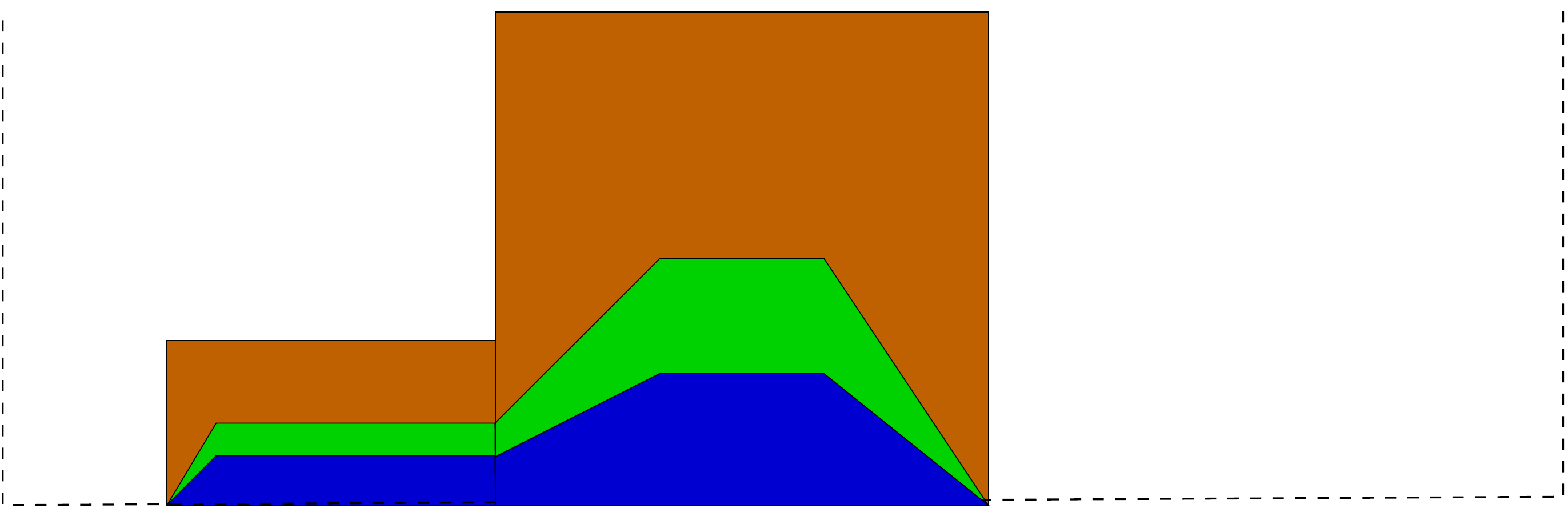}
\caption{A skew partitioned molecule; $p=4$.} 
\label{fig:Gp_blocks_join_2}
\end{figure}

\subsubsection{Coarsification of skew partitions}

In the proof of Proposition \ref{prop:3_skew} we use generalizations of primary and secondary modifications introduced in Section \ref{sec:RP}. The rearrangements are given using skew partitions having properties similar to cubical partitions. We introduce now the necessary terminology.

In this section, let $C$ be a cubical $n$-cell and $\cS=(S_1,\ldots, S_{p-1})$ a skew essential partition of $C$ into $n$-cells. 

Let $\alpha\in \Z_+$ and let $\cQ_\alpha(C)$ be a subdivision of $C$ into pair-wise disjoint $n$-cubes of side length $3^{-\alpha}$. We assign to each $q\in \cQ_\alpha(C)$ an index $i_q \in \{1,\ldots p-1\}$ with $i_q$ the minimal index for which $q\cap S_i$ has non-empty interior. Let $\cM_\alpha(\cS)$ be the set of cubes $q$ in $\cQ_\alpha(C)$ for which $\mathrm{int}(q\cap S_i)\ne \emptyset$ for more than $2$ indices $i\in \{1,\ldots, p-1\}$, and let 
\[
E_{\alpha,i}(\cS) = |\{ q\in \cQ_\alpha(C)\setminus \cM_\alpha(\cS) \colon i_q = i\}|.
\]

\begin{remark}
Clearly, $(E_{\alpha,1}(\cS),\ldots,E_{\alpha,p-1}(\cS))$ is an essential partition of $C-|\cM_\alpha(\cS)|$. Although the cubical sets $E_{\alpha,i}(\cS)$ need not be $n$-cells for all $\alpha\in \Z_+$, since $\cS$ is a skew partition, there exists $\alpha_0\in \Z_+$ for which $(E_{\alpha,1}(\cS),\ldots,E_{\alpha,p-1}(\cS))$ is an essential partition of $C-|\cM_\alpha(\cS)|$ into $n$-cells for $\alpha\ge \alpha_0$.
\end{remark}

\begin{definition}
Let $\alpha\in \Z_+$ and  $\cS=(S_1,\ldots, S_{p-1})$ be a skew partition of an $n$-cell $C$. The essential partition $\hat \cS_\alpha=(\hat S_1,\ldots, \hat S_{p-1})$ of $C$, where
\begin{equation}
\label{eq:cS_alpha}
\hat S_i = E_{\alpha,i}(\cS)\cup (|\cM_\alpha(\cS)|\cap S_i),
\end{equation}
is an \emph{$\varepsilon$-coarsification of $\cS$ for $\varepsilon>0$} if, for each $i=1,\ldots, p-1$, $\hat S_i$ is an $n$-cell, $E_{\alpha,i}(\cS)\ne \emptyset$, $\hat S_i \cap \hat S_{i-1}$ is an $(n-1)$-cell, and $\dist_\haus(S_i, \hat S_i) < \varepsilon$.
\end{definition}

\begin{figure}[h!]
\includegraphics[scale=0.35]{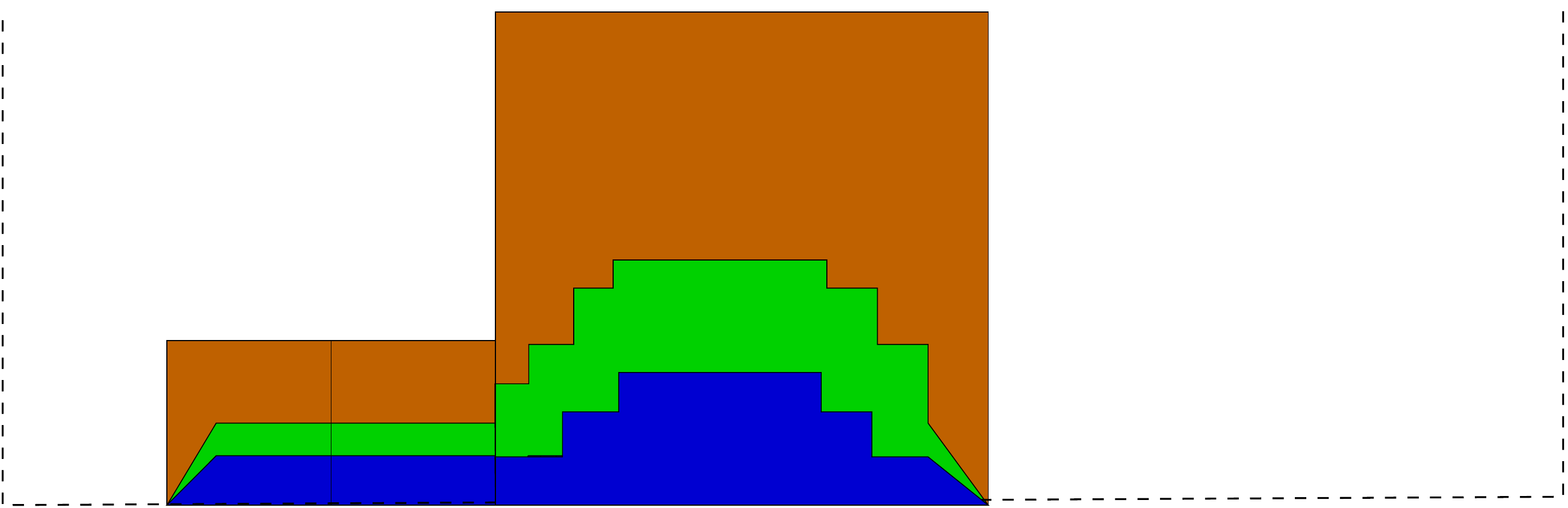}
\caption{A coarsification of the skew partition in Figure \ref{fig:Gp_blocks_join_2}.} 
\label{fig:coarsified}
\end{figure}

In the proof of Proposition \ref{prop:3_skew}, we modify the earlier $\cC$-modifications to produce skew partitions.  Heuristically, in generalized $\cC$-modification, we rearrange the domain $\hat S_i$ of a skew partition $\hat \cS_\alpha=(\hat S_1,\ldots,\hat S_{p-1})$ of a cube using atoms along common boundaries $\hat S_{i-1}\cap \hat S_i$. To obtain a repartition of a cube satisfying a collapsibility condition analogous to $\lambda$-collapsibility, we must restrict the combinatorial length of the atoms created by this generalized $\cC$-modification. With this aim in mind, we introduce now an additional condition for coarsified skew partions.

To motivate this condition, consider a $\cC$-cube $Q$ of color $i$ in a rough Rickman partition $\wt\bfOmega=(\tilde \Omega_1,\tilde \Omega_2,\tilde \Omega_3)$ of $\R^n$. Then $Q\cap \tilde \Omega_j$ is contained in at most $2n-2$ faces of $Q$ for $j\ne i$; see e.g.\;Figure \ref{fig:BB_C_mods_4} for $n=3$. Thus $Q\cap \tilde \Omega_j$ would meet at most $3^{\alpha(n-1)}(2n-2)$ cubes in $\cQ_\alpha(Q)$.

Now, let $\cS=(S_1,\ldots, S_{p-1})$ be a skew partition of an $n$-cube $Q$ and $\hat \cS_{\alpha}=(\hat S_1,\ldots, \hat S_{p-1})$ an $\varepsilon$-coarsification of $\cS$ for some $\alpha\in \Z_+$, and $\varepsilon\in (0,1)$. For $|i-j|=1$, let 
\[
P_{ij}(\hat \cS_\alpha) = \left\{ q\in \cQ_{\alpha}(Q) \colon q\subset \hat S_i\ \mathrm{and}\ q\cap \hat S_j\ \mathrm{contains\ a\ face\ of}\ q\right\}.
\]

\begin{definition}
\label{def:smallness}
The coarsification $\hat \cS_\alpha$ of $\cS$ is \emph{small} if, for each $|i-j|=1$, there exists a tree 
\begin{equation}
\label{eq:alpha_tree}
\Gamma \subset \Gamma\left(P_{i,j}(\hat \cS_\alpha) \cup P_{j,i}(\hat \cS_\alpha)\right)
\end{equation}
containing $P_{ij}(\hat \cS_\alpha)$ in its vertex set with $\# \Gamma < 3^{\alpha(n-1)}(2n-2)$.
\end{definition}

\begin{figure}[h!]
\includegraphics[scale=0.30]{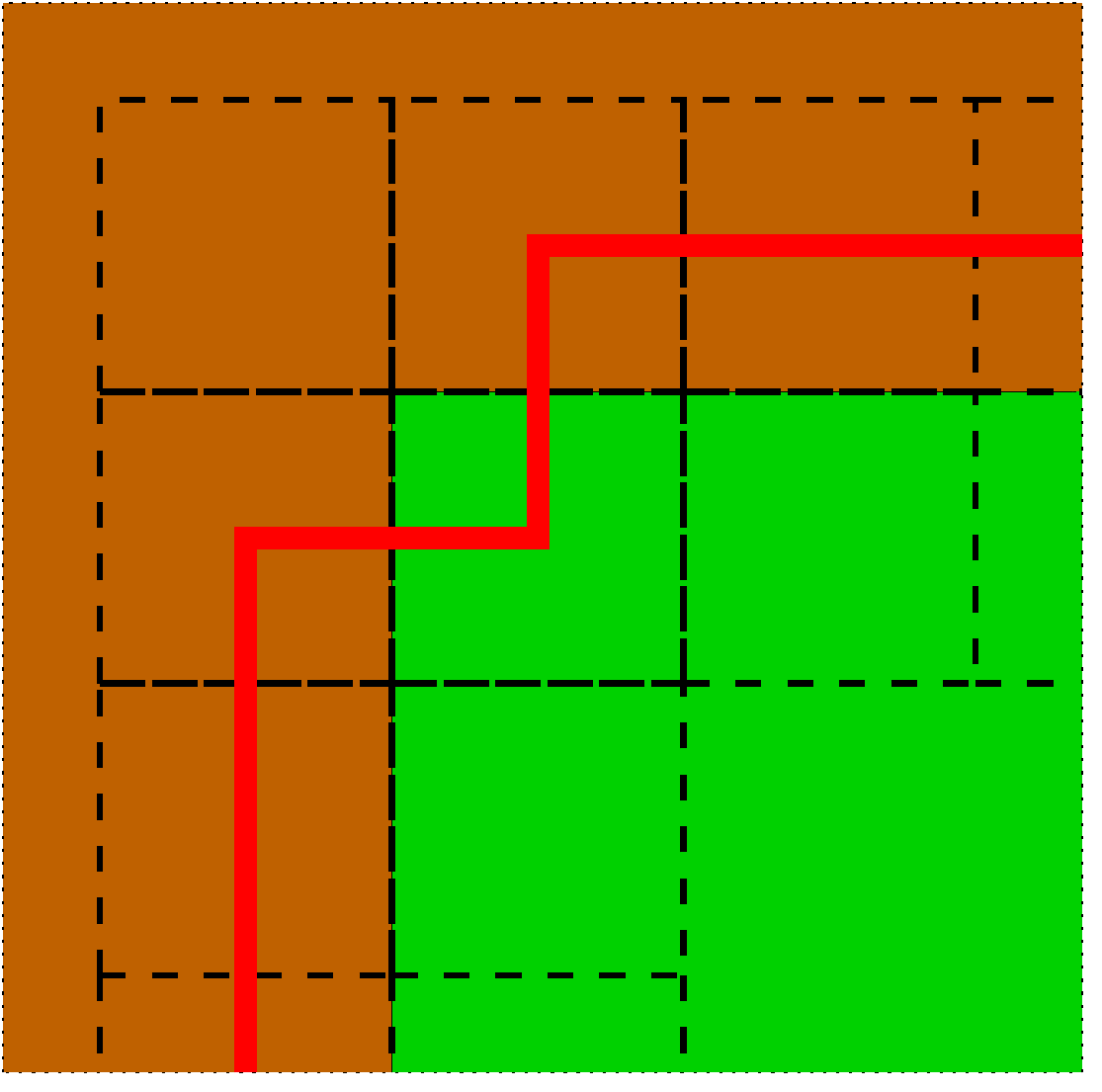}
\caption{Schematic figure on tree $\Gamma$; a detail.} 
\label{fig:c-tree}
\end{figure}

\begin{remark}
It is straightforward to check that for the skew partition $\cS(M)$ of Lemma \ref{lemma:skew_molecules} the coarsified partitions $\hat\cS_\alpha(M)\cap Q$ are small for all cubes $Q\in \Gamma^\icl(M)$.
\end{remark}

Before confronting the proof of Proposition \ref{prop:3_skew}, we introduce a combinatorial notion related to skew partitions. We say that an essential partition $\cS=(S_1,\ldots, S_{p-1})$ of a cube is \emph{linear} if the adjacency graph $\Gamma(\cS)$ is an arc, that is, each vertex in $\cS$ has valence at most $2$. Furthermore, we may also assume from now on that elements in $\cS$ are indexed so that $S_1$ and $S_{p-1}$ have valence $1$ in $\Gamma(\cS)$ and neighbors of $S_i$ are $S_{i-1}$ and $S_{i+1}$ for each $i=2,\ldots, p-2$.

\begin{remark}
The skew partition $\cS(M)$ in Lemma \ref{lemma:skew_molecules} is linear and $\cS(M)\cap Q$ is linear for each $Q\in \Gamma^{\icl}(M)$. Note also that, for $\alpha\in \Z_+$ large enough, coarsifications $\hat \cS_\alpha(M)$ of $\cS(M)$ are also linear.
\end{remark}

\subsubsection{Proof of  Proposition \ref{prop:3_skew}}

We construct the essential partition $\bfOmega$ with the same scheme as in Section \ref{sec:RP} but now with skew partitions and coarsification methods. Apart from coarsification, this approach is similar to Rickman's in \cite[Section 8.1]{R}. Since the methods are based on those of Section \ref{sec:RP} with modifications already introduced in Section \ref{sec:pillows}, we merely sketch the argument.

Mimicking the proof of Theorem \ref{thm:RP}, we construct a sequence $(\cS_m)$ of essential partitions of $n$-cells $3^m([0,3]^{n-1}\times [-3,3])$ analogous to the sequence $(\bfOmega_m)$. Recall that $\bfOmega_0= (\Omega_{0,1},\Omega_{0,2},\Omega_{0,3})= ([0,3]^n, [0,3]^{n-1}\times [-3,0], [3,6]\times [0,3]^{n-1})$ and $\bfOmega_1=(3\Omega_{0,1}-(A_2\cup A_3),3\Omega_{0,2}\cup A_2, 3\Omega_{0,3}\cup A_3)$, where $A_2$ and $A_3$ are atoms.

It is not necessary to define $\cS_0$, and we set directly
\[
\cS_1 = ([0,9]^n-A_3,[0,9]^{n-1}\times[-9,0], S_{1,2},\ldots, S_{1,p}),
\]
where $(S_{1,2},\ldots, S_{1,p})$ is the skew partition $\cS(3A_3)$  into $p-1$ $n$-cells as in \eqref{eq:skewpartition}.

\medskip
\noindent\emph{Construction of $S_2$; first generalized modifications}

We construct $\cS_2$ from $\cS_1$ by independent generalized $\cD$-modifications; note that $\bfOmega_2$ is obtained from $\bfOmega_1$ by a secondary $\cC$-modification, as observed in Remark \ref{rmk:flat_top}. In this particular case it suffices to observe that, in the construction of $\bfOmega_2$, we extend atom $3A_3$ to a molecule $M$ by attaching $1$-atoms. 
Thus, to obtain $\cS_2$ from $\cS_1$, it is enough to extend the skew partition $3(S_{1,2},\ldots,S_{1,p})$ to a skew partition of $M$; cf.\;Lemma \ref{lemma:skew_molecules}. This extension of the skew partition $3\cS_1$ into each $1$-fine atom is the \emph{generalized $\cD$-modification}. Thus 
\[
\cS_2 = (S_{2,0},\ldots, S_{2,p}) = ([0,27]^n - M, [0,27]^{n-1}\times [-27,0], S_{2,2},\ldots, S_{2,p}),
\]
where $(S_{2,2},\ldots, S_{2,p})$ is a skew partition of $M$. 

In later steps, we also use similar generalizations of secondary modifications. Note that we use these generalizations alongside with (original) $\cD$-modifications and secondary modifications.

\smallskip
\noindent\emph{Construction of $\cS_3$; generalized $\cC$-modifications.}

To obtain $\cS_3=(S_{3,0},\ldots, S_{3,p})$ from $\cS_2$, we use generalized $\cD$-modifications and generalized $\cC$-modifications in rescaled $\cC$-cubes. Note that, for $Q\in \Gamma^\icl(3A_3)$, the essential partition $Q\cap \cS_2$ is a skew partition of $Q$ into $p-1$ skew $n$-cells meeting the remaining two elements of $\cS_2$ analogously as in the situation with a $\cC$-cube; recall that $3A_3\subset \Omega_{2,3}$ was adjacent to domains $\Omega_{2,1}$ and $\Omega_{2,2}$ in $\bfOmega_2$ (see Section \ref{sec:2ra}). We call $Q$ therefore a \emph{generalized $\cC$-cube}.

Using notations related to $Q$ and the skew partition $\cS_2$, we now describe generalized $\cC$-modification in $3Q$. For this modification, we consider cubes in two scales $3^{-\beta}$ and $3^{-\alpha}$ for $\alpha >\beta\ge p$. Thus we divert here from the convention that side lengths of cubes are at least $1$.

First, let $\beta\ge p$ be an integer, to be determined later, for which we may choose, for $i=2,\ldots, p$, a cube $q_i\in\cM_{\beta}(3(Q\cap \cS_2))$ so that $\dist_{\!\infty}(q_i,q_j)\ge 3^{-\beta}$ for $i\ne j$. Note that each cube in $\cM_\beta(3(Q\cap \cS_2))$ is adjacent to $3S_{2,0}$ and $3S_{2,1}$. 

Second, let $\alpha > \beta$, to be determined later, so that $\hat \cS_\alpha =  (\hat S_1,\ldots,\hat S_{p-1})$, where
\[
\hat S_i = E_{\alpha,i}(3(Q\cap \cS_2))\cup (|\cM_\alpha(3(Q\cap \cS_2))|\cap S_i),
\]
is a $1$-coarsification of $\cS=3(Q\cap \cS_2) = (S_1,\ldots, S_{p-1})$ as in \eqref{eq:cS_alpha}. By increasing $\alpha$, if necessary, there exists for each $i=1,\ldots, p-1$ adjacent cubes $q'_i, q''_i\in \cQ_{\alpha}(Q)$ so that $q'_i \subset q_i$ and $q''_i \in P_{i,i-1}(\hat \cS_{\alpha})$; when $i=1$, we assume that $q''_1$ meets $\partial (3Q)$.

We modify now the cells $3\hat S_2,\ldots, 3\hat S_{p-1}$ in $3Q$ as follows; modification of $3\hat S_1$ is similar and postponed to the end of the process.

For each $i=2,\ldots, p-1$, let $\Gamma_i$ be a maximal tree as in \eqref{eq:alpha_tree}. Let $a'_i$ be the associated $3^{-\alpha-2}$-fine atom, and let $a_i = a'_i \cup q'_i$; then this allows $a_i$ to enter both $\hat S_i$ and $\hat S_{i-1}$, see Figure \ref{fig:c-atom}. Fix also a small skew partition $(a_{i,1},\ldots, a_{i,p-1})$ of $a_i$ so that $(\hat S_{i-1},a_{i,1},\ldots, a_{i,p-1},\hat S_i-a_i)$ is a skew partition of $\hat S_{i-1}\cup \hat S_i$ with cyclic adjacency graph; cf.\;Lemma \ref{lemma:skew_molecules}.

\begin{figure}[h!]
\includegraphics[scale=0.50]{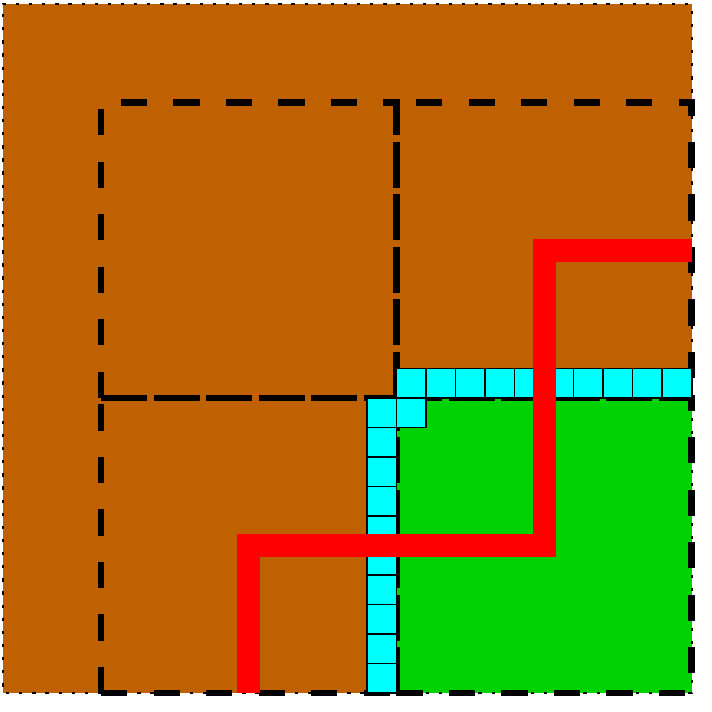}
\caption{Schematic figure on an atom $a_i$ associated to a tree $\Gamma_i$; a detail.} 
\label{fig:c-atom}
\end{figure}

To connect cells $a_{i,k}$ to cells $\hat S_j$ for $j\ne i$ and $k\in \{1,\ldots,p-1\}$, note that for each $i=2,\ldots, p-1$ there exists a unique graph isomorphism 
\[
\theta_i \colon \Gamma(\hat \cS_\alpha) \to \Gamma(\hat S_{i-1}-a_i,a_{i,1},\ldots, a_{i,p-1},\hat S_i-a_i)
\]
satisfying $\theta_i(\hat S_{i-1})=\hat S_{i-1}-a_i$ and $\theta_i(\hat S_i)=\hat S_i-a_i$.

We fix now, on each cube $q_i$, a small skew partition $(q_{i,1},\ldots, q_{i,p-1})$ so that $q_{i,j} \cup (\hat S_j-q_i)$ and $q_{i,j}\cup \theta_i(a_j)$ are skew $n$-cells. 

Then, by attaching cells $q_{i,j}\cup \theta_i(a_j)$ to cells  $\hat S_j$ for $2\le j \le  p-1$, we obtain cells $Q_j$ for which the system $(\hat S_1,Q_2,\ldots Q_{p-1})$ produces the desired skew partition of $3Q$ after we make an analogous extension of both $\hat S_1$ and $Q_j$ along $\partial (3Q)$. We leave this last detail to the interested reader.

We conclude by noting that, since the atoms $a_i$ have side length $3^{-\alpha}$, we do not need to rearrange their scaled copies before constructing $\cS_{3+(\alpha+2)}$. At that stage, cubes in $\Gamma^\icl(3^{\alpha+1}a_i)$ are generalized $\cC$-cubes. A similar comment applies to cubes $q_i$ and the construction of $\cS_{3+(\beta+2)}$. Note also that it suffices to fix, up to an isometry, one essential partition for a cube of side length $3^{-\beta}$ for all generalized $\cC$-modifications. In particular, we may fix parameters $\alpha$ and $\beta$ to depend only on $n$ and $p$.

\medskip
\noindent\emph{Construction of $\bfOmega$; inductive process}

With these generalized primary and secondary rearrangements at our disposal, we proceed as in Section \ref{sec:RP} and obtain an essential partition $\cS_m$ from $\cS_{m-1}$ for every $m>3$. Similarly as in the proof of Theorem \ref{thm:3} (for $p=2$) we may arrange that these essential partitions yield an essential partition $\bfOmega=(\Omega_0,\ldots, \Omega_p)$ of $\R^n$ satisfying the (generalized) tripod property; again $\Omega_0$ and $\Omega_1$ are connected and $\Omega_2,\ldots, \Omega_p$ have $2^{n-1}$ components each. 

Note that the combinatorial length estimate for small skew partitions yields that $\Omega_0$, $\Omega_1$, and each component of $\Omega_r$ for $r\ge 2$ are $\lambda$-collapsible in a natural generalized sense; in dimension $n=3$ we use again particular configurations illustrated in Section \ref{sec:IC_dim3} to obtain collapsibility. Analogously as in Sections \ref{sec:cond_f} and \ref{sec:proof_RP}, we obtain that $\Omega_0$, $\Omega_1$, and each component of $\Omega_r$ are bilipschitz to $\R^{n-1}\times [0,\infty)$. Thus $\bfOmega$ is a skew Rickman partition satisfying the (generalized) tripod property and we have proved Proposition \ref{prop:3_skew}. $\qed$

\subsection{Proof of Proposition \ref{prop:4}}
\label{sec:FINAL}

Let $\bfOmega = (\Omega_0,\ldots, \Omega_p)$ be a skew Rickman partition as in Proposition \ref{prop:3_skew}. Then $\partial_\cup \bfOmega$ carries a uniformly bilipschitz triangulation into $(n-1)$-simplices together with an associated labeling function. 

Due to the cyclic combinatorics of domains in $\bfOmega$, that is, since $\Omega_j \cap \Omega_{j+1}$ is locally an $(n-1)$-cell for $j=0,\ldots, p$ ($\!\!\!\!\mod(p+1)$), we define a parity function $\nu_{\partial_\cup \bfOmega} \colon (\partial_\cup \bfOmega)^{(n-1)} \to \{\pm 1\}$  for $p>2$ analogous to the case $p=2$ in Section \ref{sec:triangulation}.

To construct a pillow cover over the triangulation of $\partial_\cup \bfOmega$ it suffices to discuss pillows over pairs of adjacent $(n-1)$-simplices. We merely describe differences from the case $p=2$; apart from these slight modifications we proceed as in Section \ref{sec:pillows}.

Let $\sigma$ and $\sigma'$ be an adjacent pair of $(n-1)$-simplices on $\partial_\cup \bfOmega$ and suppose $\nu_{\partial_\cup \bfOmega}(\sigma)=-1$. We may also assume, to simplify  notation, that $\sigma\cup \sigma'\subset \R^{n-1}\times \{0\}$. In this case the sheets $\hat \sigma_1,\ldots, \hat \sigma_p$ on $\sigma$ are given by the graph of a function $\Psi_\sigma \colon \sigma \times \{1,\ldots, p\} \to \R$ similarly as in Section \ref{sec:pillow_simplex}. Sheets $\hat \sigma'_1,\ldots \hat \sigma'_{p+2}$ on $\sigma'$ are similarly given by the graph of a function $\Psi_{\sigma'} \colon \sigma' \times \{1,\ldots, p+2\}\to \R$. We require that these pillows satisfy compatibility conditions analogous to those of Definition \ref{def:compatible_pillows} in Section \ref{sec:pcas}. Since local modifications of pillows are similar to the case $p=2$, we leave the finer details to the interested reader and discuss in detail only the 'shuffle' of domains. 

Suppose for now that we have fixed functions $\Phi_\sigma$ and $\Phi_{\sigma'}$ providing us sheets for simplices $\sigma$ and $\sigma'$, respectively. Let $D_0,D_1,\ldots, D_p$ be the components of 
\[
\sigma \times \R\setminus (\bigcup_{i=1}^p \hat \sigma_i)
\]
so that $\hat\sigma_1 \subset \partial D_0$, $\hat\sigma_i \cup \hat\sigma_{i+1}\subset \partial D_i$ for $i=1,\ldots, p-1$, and $\hat \sigma_p \subset \partial D_p$. Let $D'_0, \ldots, D'_{p+2}$ be the components of 
\[
\sigma'\times \R \setminus (\bigcup_{j=1}^{p+2} \hat \sigma'_j)
\]
in the same order, that is, $\hat \sigma'_0 \subset \partial D'_0$, $\hat\sigma'_j \cup \hat\sigma'_{j+1}\subset \partial D'_j$ for $j=1,\ldots, p+1$, and $\hat \sigma_{p+2}\subset \partial D'_{p+2}$.

Following the method in Section \ref{sec:pcas}, we may assume that, for functions $\Phi_\sigma$ and $\Phi_{\sigma'}$, the sets $D_0 \cup D'_0\cup D'_{p+1}$, $D_i \cup D'_{p+1-i}$ for $i=1,\ldots, p-1$, and $D_p\cup D'_{p+2}\cup D'_1$ are connected components of 
\begin{equation}
\label{eq:merge}
(\sigma\cup \sigma')\times \R  \setminus ( \bigcup_{i=1}^p \hat \sigma_i \cup \bigcup_{j=1}^{p+2} \hat \sigma_j ).
\end{equation}
Note that in order to merge the sets $D_i$ and $D'_j$ this way it suffices to subdivide the set $\tau_0\subset \tau = \sigma\cap \sigma'$, defined in Section \ref{sec:pillow_simplex}, into $(n-2)$-simplices and to define several openings this way.

This 'shuffle' allows the domains $D_p$ and $D'_p$ to be connected across $\bar \sigma\cup \sigma'$ and preserves the global adjacency structure on these domains when passing from $\bfOmega$ to $\tilde \bfOmega$.

To fix notation, suppose that simplices $\sigma$ and $\sigma'$ in $\partial_\cup \bfOmega$ are between domains $\Omega_\ell$ and $\Omega_{\ell+1}$ for $\ell \in \{0,\ldots, p\}$, where we understand $\ell+1=0$ if $\ell=p$. We may assume that locally near $\sigma\cup \sigma'$, $\Omega_\ell$ is contained in $(\sigma \cup \sigma')\times (-\infty,0]$.

We begin with the negative simplex $\sigma$. The adjacency graph $\Gamma(\bfOmega\cap (\sigma\times \R))$ near $\sigma$ consists only of an edge between $\Omega_\ell$ and $\Omega_{\ell+1}$. The adjacency graph of domains $D_0,\ldots, D_p$, on the other hand, is an arc from $D_0$ to $D_p$. By construction of the essential partition $\tilde \bfOmega$, the sets $D_0,\ldots, D_p$ are contained in elements of the essential partition $\tilde \bfOmega$. Since $\tilde \bfOmega$ has the same cyclic adjacency graph as $\bfOmega$ and $\Gamma(\tilde \bfOmega \cap (\sigma\times \R))$ is an arc of length $p$, we note that domains $D_0,\ldots, D_p$ belong to sets $\tilde \Omega_\ell, \tilde \Omega_{\ell-1},\ldots, \tilde \Omega_1, \tilde \Omega_p,\ldots, \tilde \Omega_{\ell+1}$, in this order.  

For the positive simplex $\sigma'$, we note that, by \eqref{eq:merge} and by the same argument, the domains $D'_0,\ldots, D'_{p+2}$ are contained in domains $\tilde \Omega_\ell, \tilde \Omega_{\ell+1}, \ldots, \tilde \Omega_p, \tilde \Omega_1,\ldots, \tilde \Omega_{\ell},\tilde \Omega_{\ell+1}$ in this order.

As a remark, we note that if we merge graphs $\Gamma(\tilde \bfOmega \cap (\sigma \times \R))$ and $\Gamma(\tilde \bfOmega\cup (\sigma'\times \R))$ by identifying vertices corresponding to domains $D_0$ and $D_p$ with $D'_0$ and $D'_{p+2}$, respectively, we obtain a cyclic graph which is a natural double cover of $\Gamma(\tilde \bfOmega)$.

\begin{figure}[h!]
\includegraphics[scale=0.20]{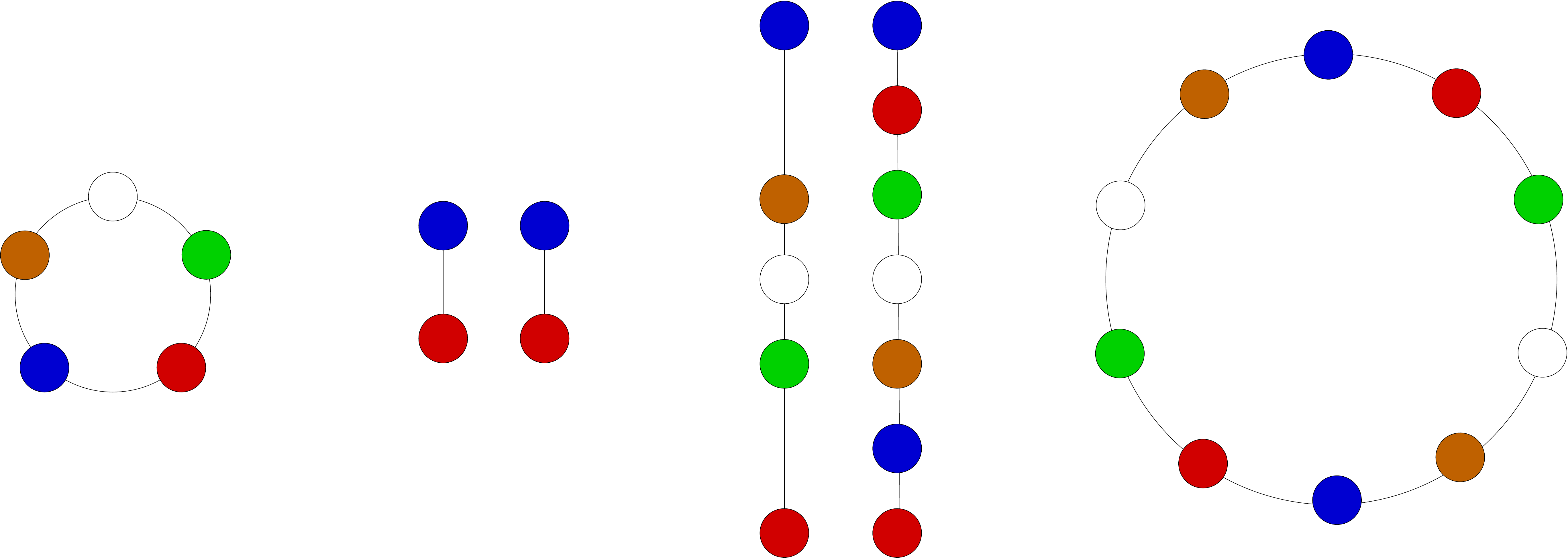}
\caption{Case $p=5$. From left to right: cyclic adjacency graph of $\tilde \bfOmega$, the adjacency graphs $\Gamma(\bfOmega\cap (\sigma\times \R))$ and $\Gamma(\bfOmega\cap (\sigma'\times \R))$,  the adjacency graphs $\Gamma(\tilde\bfOmega \cap (\sigma \times \R))$ and $\Gamma(\tilde \bfOmega\cap (\sigma'\times \R))$, and the merge of $\Gamma(\tilde\bfOmega \cap (\sigma \times \R))$ and $\Gamma(\tilde \bfOmega\cap (\sigma'\times \R))$.}
\label{fig:new_local_adjacency}
\end{figure}
 
This remark concludes the construction of the essential partition $\tilde \bfOmega$ and the proof of Proposition \ref{prop:4}. $\qed$

\begin{corollary}
\label{cor:John-final}
The domains $\interior \wt\Omega_0,\ldots,\interior \wt\Omega_p$, as well as $\interior \Omega_0,\ldots, \interior \Omega_p$, are uniform domains.
\end{corollary}
\begin{proof}
Since domains $\interior \Omega'_1,\interior \Omega'_2, \interior \Omega'_3$ are uniform domains by Corollary \ref{cor:John-domains}, we have that domains $\interior \Omega_0,\ldots, \interior \Omega_p$ are uniform domains by bilipschitz invariance of the uniformity condition. Since $\interior \Omega_k$ is bilipschitz to $(\interior \wt\Omega_k,d_{\interior \wt\Omega_k})$, we have that $\interior \wt\Omega_k$ is a uniform domain for each $k=0,\ldots, p$. 
\end{proof}



\def\cprime{$'$}

\end{document}